\documentclass[11pt,reqno]{amsart}

\pdfoutput=1
\usepackage{hyphenat}
\usepackage{LPPSS1_Macros}

\providebibmacro*{bbx:dashcheck}{}

\allowdisplaybreaks

\title[AS of the degree-one vortex in the abelian YMH model: Linear Theory]{Asymptotic stability of the degree-one vortex in the abelian Yang-Mills-Higgs model: Linear Theory}

\author[J. L\"uhrmann]{Jonas L\"uhrmann}
\address{Department of Mathematics and Computer Science, University of Cologne, Cologne, Germany}
\email{jonas.luehrmann@uni-koeln.de}

\author[J. M. Palacios]{Jos\'e M. Palacios}
\address{Institute of Mathematics, \'Ecole Polytechnique F\'ed\'erale de Lausanne (EPFL), Lausanne, Switzerland}
\email{jose.palaciosarmesto@epfl.ch}

\author[F. Pusateri]{Fabio Pusateri}
\address{Department of Mathematics, University of Toronto, Toronto, Ontario, Canada}
\email{fabiop@mail.math.toronto.edu}

\author[W. Schlag]{Wilhelm Schlag}
\address{Department of Mathematics, Yale University, New Haven, Connecticut, USA}
\email{wilhelm.schlag@yale.edu}

\author[S. Shahshahani]{Sohrab Shahshahani}
\address{Department of Mathematics and Statistics, University of Massachusetts Amherst, Amherst, Massachusetts, USA}
\email{sshahshahani@umass.edu}

\thanks{
J.L.\  was partially supported by NSF CAREER grant DMS-2235233.
F.P.\ is supported in part by a start-up grant from the University of Toronto and NSERC grants RGPIN-2018-0648 and RGPIN-2025-06419. J.M.P.\ was partially supported by NSERC grants RGPIN-2018-0648, and is also supported by the Swiss National Science Foundation grant 225701.
W.S.\ is partially supported by the United States NSF DMS-2350356. S.S. was supported by the Simons Foundation grant 639284.
}

\begin{document}

\begin{abstract}
We study the linearized dynamics near the degree--one vortex of the $(1+2)$--dimensional abelian Yang--Mills--Higgs model at the self--dual coupling, restricted to equivariant perturbations in the orthogonal gauge. The linearized operator is a selfadjoint matrix Schr\"odinger operator $\bfM$ on radial $L^2_{\mathrm{rad}}(\R^2;\R^4)$ with continuous spectrum $[1,\infty)$ and a two--dimensional internal mode at a unique gap eigenvalue $\lambda^2\in(0,1)$, as established in Part~I of our three-paper series on asymptotic stability of the ground state vortex. In this second part of the series, we prove linear estimates for $\bfM$ for applications in Part~III. Specifically, we prove dispersive and local-energy decay estimates, as well as a \emph{transference} relation which allows us to implement the space-time resonance method with respect to the flat Klein-Gordon operator in the nonlinear analysis in Part~III.

The engine for proving linear estimates for $\bfM$ in our approach is the distorted Fourier transform associated with $\bfM$. 
The construction of the distorted Fourier transform together with a detailed analysis of the underlying generalized eigenfunctions occupy the first half of this paper. For this we exploit the super-symmetric factorization of $\bfM$, and the diagonal structure of the super-symmetric partner operator, to relate the problem to the Weyl--Titchmarsh theory of two strongly singular scalar half--line Schr\"odinger operators. 
\end{abstract}

\maketitle

\vspace{-0.5em}

\setcounter{tocdepth}{2}

\tableofcontents

\section{Introduction}

This paper is the second part of a three-part series, cf.~\cites{LPPSS1PartI}{LPPSS3}, devoted to the asymptotic stability problem for equivariant perturbations of the degree-one vortex in the $(1+2)$--dimensional abelian Yang--Mills--Higgs (AYHM) model at the critical (self--dual) coupling. At this critical point the static energy admits a Bogomolny factorization, and finite--energy degree-one vortices are precisely the solutions of the associated first order Bogomolny equations. The aim of the present paper is to develop the linear theory needed for the nonlinear analysis carried out in the third part \cite{LPPSS3} of this series. The construction of the distorted Fourier transform relative to the linearized operator lies at the heart of this linear analysis. The description of the spectrum of the linearized operator, which was obtained in the first part \cite{LPPSS1PartI} of the series with the aid of rigorous numerics, is the starting point of the current paper.

\subsection{The vortex profiles and linearization}
The AYHM Lagrangian action is
\begin{equation}
    \calL_\lambda[\phi,A]:=\int_{\bbR^{1+2}}\Big(\frac{1}{4}F_{\mu\nu}F^{\mu\nu}+\frac{1}{2}\bfD_\mu\phi\overline{\bfD^\mu\phi}+\frac{\lambda}{8}(1-|\phi|^2)^2\Big)\ud x \,\ud t.
\end{equation}
Here $\phi$, the Higgs field, is a complex valued scalar field and $A$, the electromagnetic potential, is a real-valued $1$-form. $F:=\ud A$ denotes the curvature $2$-form, and $\bfD:=\nabla-iA$ the covariant derivative. Greek indices run over $0,1,2$ corresponding to the spacetime variables. Roman indices, run over $1,2$, and denote the spatial variables. The Einstein summation convention is in place, and indices are lowered and raised with respect to the Minkowski metric and its inverse.  In this work we focus exclusively on the choice of the coupling constant $\lambda=1$, known as the self-dual case. The corresponding Euler-Lagrange equations are
\begin{equation}\label{eq:AYMH1}
    \bfD^\mu\bfD_\mu\phi+\frac{1}{2}(1-|\phi|^2)\phi=0,\qquad \nabla^\mu F_{\mu\nu}+\Im\big(\overline{\phi}\bfD_\nu\phi\big)=0.
\end{equation}
An important feature of the Lagrangian action is its invariance under gauge transformations of the form
\begin{equation}
    (\phi,A)\mapsto (e^{i\chi}\phi,A+\ud \chi),\qquad \chi\in C^2(\bbR^{1+2}).
\end{equation}
As a consequence one needs to fix a gauge to arrive at a well-posed Cauchy formulation of the equations. Note that the gauge invariance above includes the case where $\chi$ is a constant.

A special class of solutions to \eqref{eq:AYMH1} are those that arise as critical points of the static energy functional
\begin{equation}
    \calE(\phi,A)=\int_{\bbR^{2}}\Big(\frac{1}{4}B^2+\frac{1}{2}\bfD_j\phi\overline{\bfD^j\phi}+\frac{1}{8}(1-|\phi|^2)^2\Big)\ud x.
\end{equation}
Here $B=F_{12}$ is the magnetic field, which completely determines the curvature form in $\bbR^2$. Remarkably, in the self-dual regime the energy functional admits a Bogomolny factorization, so that the critical points satisfy the first order equations 
\begin{equation}
    (\bfD_1+i\bfD_2)\phi=0,\qquad B-\frac{1}{2}(1-|\phi|^2)=0.
\end{equation}
Finite energy solutions of these equations have been completely classified in \cite{bookJT}. We refer the reader to \cite{bookJT,bookManSut}, as well as the introduction to our companion works \cites{LPPSS1PartI}{LPPSS3} for more details. In this article we are interested in the non-trivial ground state, the degree-one vortex, which  in polar coordinates and the Coulomb gauge $\nabla^jA_j=0$ takes the form
\begin{equation}
\underline{\Phi}(r,\theta)=e^{i\theta}U(r),
\qquad
\underline{A}(r,\theta)=a_\theta(r)\,\ud\theta.
\end{equation}
The profiles $(U,a_\theta)$ solve the Bogomolny system
\begin{equation}\label{eq:Ua_bogo}
U'(r)=\frac{1-a_\theta(r)}{r}U(r),
\qquad
a_\theta'(r)=\frac{r}{2}\bigl(1-U(r)^2\bigr),
\end{equation}
with boundary conditions $U(0)=0$, $a_\theta(0)=0$, $U(r)\to1$, $a_\theta(r)\to 1$ as $r\to\infty$. They satisfy the asymptotics
\begin{align}
U(r)& = \uprz\, r-\frac{\uprz}{8}r^3+O(r^5), & &  a_\theta(r)= \frac{1}{4}r^2-\frac{\uprz^2}{8}r^4+O(r^6),  & &   r\to 0,\qquad \label{Uaexp0}
\\ 
U(r)&= 1-mr^{-\frac12}e^{-r}+O(r^{-\frac32}e^{-r}), & &    a_\theta(r)= 1-mr^{1/2}e^{-r}+O(r^{-\frac12}e^{-r}) , & &   r\to\infty,\qquad \label{Uaexpinfty}
\end{align}
for some constants $\uprz,m>0$, cf. \cite{LPPSS1PartI}.
In the context of the time-dependent equations~\eqref{eq:AYMH1}, an important question is the stability of the ground state. A key step in this direction is the analysis of the linearized equations about the vortex. As shown for instance in \cite{LPPSS3}, in Stuart's gauge, cf. \cite{Stu}, $\nabla^j A_j=\Im(\overline{\underline{\Phi}}\phi)$, and under the equivariance ansatz
\begin{equation}
    \phi(t,r,\theta)\equiv \big(U(r)+\alpha(t,r)+i\beta(t,r)\big)e^{i\theta},\qquad A=\eta_0(t,r)\,\ud t+\eta_r(t,r)\,\ud r+ \big(\eta_\theta(t,r)+a_\theta(r)\big)\,\ud \theta,
\end{equation}
the linearized equations take the form
\begin{equation}\label{eq:lin_sys_intro}
(\pt^2+\bfM)\bmr=\bmN,
\qquad
\bfM:=\begin{pmatrix}\bfL&0\\0&\bfL\end{pmatrix},
\qquad
\bfL:=\begin{pmatrix}L_1&-2U'\\-2U'&L_2\end{pmatrix}.
\end{equation}
Here $\bfN$ denotes the nonlinearity, $\bmr^t:=(\alpha,r^{-1}\eta_\theta,\beta,-\eta_r)$, and the equations \eqref{eq:lin_sys_intro} are coupled to an elliptic constraint 
\begin{equation}
    (-\Delta+U^2)\eta_0=\calN_{\eta_0}
\end{equation}
which allows one to recover $\eta_0$ from the nonlinearity $\calN_0$, which in turn involves $\bmr$. The operators $L_1$ and $L_2$ are the radial Schr\"odinger operators
\begin{equation}\label{eq:bL1L2}
L_1:=-\Delta+b(r)^2-\frac{\partial_r a_\theta(r)}{r}+U(r)^2,
\qquad
L_2:=-\Delta+\frac{1}{r^2}+U(r)^2,
\end{equation}
where
\begin{equation}\label{eq:bdef1}
     b(r):=\frac{1-a_\theta(r)}{r}.
\end{equation}
A nice feature of Stuart's gauge is that the operator $\bfM$ has a block diagonal structure with identical diagonal blocks. Moreover, in this gauge, the elliptic equation for $\eta_0$ is decoupled from the evolution equation for $\bmr$ at the linear level. But perhaps most significantly, as described in \cite[Section 2]{Stu}, the motivation for the choice of gauge is that formally it makes the perturbation orthogonal to the zero modes corresponding to the gauge invariance of the problem. The only other invariances of equations~\eqref{eq:AYMH1} are those generated by the action of the Poincar\'e group. Since spatial translations do not preserve the equivariance ansatz, one then expects that in Stuart's gauge zero is not an eigenvalue of the linearized operator. This is indeed the case. In fact with the aid of rigorous numerics, \cite[Theorem~1.1]{LPPSS1PartI} established the following spectral theorem for~$\bfM$.

\begin{theorem}[Spectrum of the linearized operator]\label{maintheo1}
Let $\bfM$ be the self-adjoint operator defined by \eqref{eq:lin_sys_intro}--\eqref{eq:bL1L2} on $(L^2_{\mathrm{rad}}(\R^2))^4$. Then:
\begin{itemize}
\item[(i)] The restriction $\bfM\vert_{\mathrm{Ran\,}\bfP_c}$ has purely absolutely continuous spectrum $[1,\infty)$, where $\bfP_c$ denotes the projection onto the continuous spectrum. In particular, the threshold $1$ is neither an eigenvalue nor a resonance. 
\item[(ii)] $\bfM$ has a unique gap eigenvalue of multiplicity two (an internal mode), which we denote by $\lambda^2\in(0,1)$ (numerically $\lambda^2\approx 0.7774$). The associated eigenspace is spanned by
\begin{align*}
    \bmY_1 := \lambda^{-1} (U\psi, -\psi',0,0)^T, \qquad \bmY_2 := \lambda^{-1} (0, 0, U\psi , -\psi')^T, \qquad \bfM \bmY_j = \lambda^2\bmY_j,
\end{align*}
where $\psi$ is the $L^2_{r \ud r}$-normalized\footnote{This means $\int_0^\infty \psi(r)^2r\, dr=1$ without a $2\pi$ factor. That $\|\bmY_1\|_{L^2_{r \ud r}} = \|\bmY_2\|_{L^2_{r \ud r}} = 1$ follows from the eigenvalue equation and integration by parts. } radial ground state of $-\Delta + U^2$ with energy $\lambda^2$.
\end{itemize}

\end{theorem}

Theorem~\ref{maintheo1} serves as the starting point of this work, where we establish  decay estimates for the linear  evolution for $\bfM$. In our approach, the distorted Fourier transform for $\bfM$ plays a crucial role in deriving these estimates. The decay estimates are used in our companion paper \cite{LPPSS3} to establish nonlinear asymptotic stability of the degree-one vortex under regular and localized equivariant perturbations. We will discuss the distorted Fourier transform and the linear estimates in the following subsection.

\subsection{Main results}
The distorted transform provides (i) a diagonalization of $\bfM$ on $\mathrm{Ran}(\bfP_c)$, with $\bfP_c$ the projection onto the continuous spectral subspace, (ii) a functional calculus for $\bfM$, and (iii) an oscillatory integral representation of the linear flow against an explicit spectral measure. In this work we prove decay estimates using a detailed asymptotic analysis of this oscillatory integral representation.  The following theorem summarizes our findings about the distorted Fourier transform. We use $\calM_{n}(\bbC)$ to denote the space of $n\times n$ complex matrices. 

\begin{theorem}\label{propdFT}
There exists a  function $E:[0,\infty)\times[0,\infty)\to\calM_2(\bbC)$, $\small{E=\pmat{E_{1,1}&E_{2,1}\\ E_{1,2}&E_{2,2}}}$ such that with
\begin{equation}
    \begin{aligned}
        \bfE(r,\freq) := 
        \begin{pmatrix} 
            E(r,\freq) & 0_{2\times 2 } \\ 0_{2 \times 2} & E(r,\freq)
        \end{pmatrix}  , 
    \end{aligned}
\end{equation}
and the distorted Fourier transform $\wtilcalF$ relative to $\bfM$ and its inverse $\wtilcalF^{-1}$ defined as
\begin{equation}
    \wtilcalF[ \bmg ](\freq) \equiv \widetilde{\bmg}(\freq) := \int_0^\infty \overline{\bfE(r,\freq)}^{t}  \bmg(r) r\, \ud r, 
\end{equation}
and
\begin{equation}
    \wtilcalF^{-1}[\bmh](r) := \int_0^\infty \bfE(r,\freq)  \bmh(\freq) \freq\, \ud \freq,
\end{equation}
where $\bmg \in L^2_{r \ud r}(\bbR_+;\bbC^4)$ and $\bmh\in L^2_{k \ud k}(\bbR_+;\bbC^4)$, the following properties are satisfied:
\begin{enumerate}
\item Diagonalization of $\bfM$:
\begin{equation}
    \wtilcalF[ \bfM \bmg ](\freq) = (1+\freq^2) \widetilde{\bmg}(\freq).
\end{equation}
\item Fourier inversion:
\begin{equation}
    \bfP_c = \wtilcalF^{-1} \circ \wtilcalF.
\end{equation}
\item Plancherel: 
\begin{equation}
    \bigl\| \widetilde{\bmg} \bigr\|_{L^2_{\freq \ud \freq}(\bbR_+;\bbC^4)} = \bigl\| \bfP_c\,\bmg \bigr\|_{L^2_{r \ud r}(\bbR_+;\bbC^4)}.
\end{equation}
\item $rk\leq 1$ asymptotics: The components $E_{\ell,m}(r,k)$ of $E(r,k)$ satisfy the following asymptotics when $rk\leq 1$:
\begin{equation}\label{eq:Eklein}
E(r,k)=
\pmat{\dfrac{k r}{\jap{r}}\phi_{1,1}(r,k) &  \dfrac{\log(2+k)}{\jap{k}\jap{\log k}}\dfrac{r\log(2+r)}{\jap{r}}\phi_{2,1}(r,k)\\[3ex]
\dfrac{k r^3}{\jap{r}^2}\phi_{1,2}(r,k) & \dfrac{\log(2+k)}{\jap{\log k}}\dfrac{r}{\jap{r}^{2}}\jap{k}\big(1+(rk)^2\log(2+r)\big)\phi_{2,2}(r,k)},
\end{equation}
where $\phi_{\ell,m}$ satisfy $|(r^{-1}\partial_k)^j\phi_{\ell,m}(r,k)|\lesssim 1$, $j=0,1,2$.
\item $rk\geq 1$ asymptotics: The components $E_{\ell,m}(r,k)$ of $E(r,k)$ satisfy the following asymptotics when $rk\geq 1$: 
\begin{equation}\label{eq:Egross}
E(r,k)=\sum_{\pm}e^{\pm ikr}
\pmat{\dfrac{k^{\frac{1}{2}}}{r^{\frac{1}{2}}\jap{k}}\psi_{1,1}^\pm(r,k)&\dfrac{r^{\frac{1}{2}}}{\jap{r}k^{\frac{1}{2}}\jap{k}}\psi_{2,1}^\pm(r,k)\\[3ex]
\dfrac{r^{\frac{1}{2}}}{\jap{r}k^{\frac{1}{2}}\jap{k}}\psi_{1,2}^\pm(r,k) & \dfrac{k^{\frac{1}{2}}}{r^{\frac{1}{2}}\jap{k}}\psi_{2,2}^\pm(r,k)},
\end{equation}
where $\psi_{\ell,m}^\pm$ satisfy $|(k\partial_k)^j\psi_{\ell,m}(r,k)|\lesssim 1$, $j=0,1,2$.
\end{enumerate}
\end{theorem}

The full power of the distorted Fourier transform, in particular the asymptotics of the Fourier basis, are not be used in our nonlinear analysis in \cite{LPPSS3}. In fact, the distorted Fourier transform  appear there only in the ODE analysis for the contribution of the internal mode to the evolution. Instead, our main use of the distorted Fourier transform is to derive linear decay estimates that are used in an essential way in \cite{LPPSS3}.

We now turn to the description of the linear decay estimates which we divide into two categories: (1) Dispersive, or pointwise, and local energy decay estimates. (2) Transference relations. For (1) the statements are more or less standard and we will not replicate them in the introduction. They are proved in Section~\ref{sec:linear_decay_estimates}. Broadly speaking there are two types of estimates here. First, are the dispersive or pointwise bounds. They are derived as corollaries of the $L^1\to L^\infty$ dispersive estimate for the linear evolution proved in Proposition~\ref{propKG1dec}. In \cite{LPPSS3} they are used in the form of Corollaries~\ref{cor:Lpdisp1} and \ref{cor34KG1dec}, and most crucially in  estimating spatially localized terms appearing in the analysis of the internal mode contribution.  Estimate \eqref{eq:Linftyweighteddispersivebound1} in Corollary~\ref{cor:Lpdisp1} are also used in \cite{LPPSS3} to control the contribution of the initial data. The second type of estimates are local energy decay estimates. In \cite{LPPSS3} they are used crucially in estimating various spatially localized errors that result from the difference of $\bfM$ and $-\Delta+1$. They are also used in treating some spatially localized errors in the ODE analysis for the contribution of the internal mode. The corresponding results are given in Lemmas~\ref{locdec_lem2},~\ref{locdec_lem3},~\ref{lem:LEDallfereqs1},~\ref{lem:ILEDDuhamel}, and~\ref{lem:LAP1}. Finally, for reference in \cite{LPPSS3}, the known analogues of some of the dispersive and local energy decay estimates for the standard Klein-Gordon equation on $\bbR^{1+2}$ are recalled in the appendix. 

The second type of estimates for which Theorem~\ref{maintheo1} is used are the transference estimates. Since these are less standard we provide a bit of context before stating the result. The reader is referred to  \cite[Section~2]{LPPSS3} for a more detailed discussion of how they arise in the nonlinear analysis. In the setup of \cite{LPPSS3} we derive pointwise decay estimates by controlling (see Lemma~\ref{lem:decay})
\begin{equation}\label{eq:dxiprofileintro1}
    \|\nabla_\xi e^{-it\jap{\xi}}(2i\jap{\xi})^{-1}(\partial_t+i\jap{\xi})\hat{\bmu} \|_{L^2_\xi(\bbR^2)}.
\end{equation}
Here the Fourier transform is the standard, not distorted, Fourier transform on $\bbR^2$ and $\bmu=e^{i\theta}\bfP_c\bmv$, where
\begin{equation}\label{eq:lin_sys_intro2}
    (\partial_t^2+\bfM)\bfP_c\bmv=\bfP_c\bmF,
\end{equation}
with $\bmF$ spatially localized and enjoying favorable time decay. To be precise we need to consider the real and imaginary parts of $e^{i\theta}\bfP_c\bmv$ separately. See \cite[Section~2]{LPPSS3} and the statement of Theorem~\ref{thm:transference1} below for exactly how this arises in the analysis of \eqref{eq:lin_sys_intro}. The quantity $\bmf:=e^{-it\jap{D}}(2i\jap{D})^{-1}(\partial_t+i\jap{D})\bmu$, $\jap{D}:=\sqrt{-\Delta+1}$, is  referred to as the profile (associated to $-\Delta+1$ not $\bfM$) of $\bmu$ in the language of the space-time resonance approach.  In relation to the classical vector field method, estimating $\|\nabla_\xi\hat\bmf\|_{L^2_\xi(\bbR^2)}$ is in effect analogous to controlling the $L^2_x(\bbR^2)$ norm of one Lorentz boost vector field applied to $\bmu$. The quantity in \eqref{eq:dxiprofileintro1} is in terms of the profile of $\bmu$ associated with $-\Delta+1$ while $\bfP_c\bmv=e^{-i\theta}\bmu$ satisfies \eqref{eq:lin_sys_intro2} which has $\bfM$ as the linear operator. In our approach, we obtain control of \eqref{eq:dxiprofileintro1} by \emph{transferring} the estimate to the estimate for the profile defined with respect to $\bfM$, that is, to estimating
\begin{equation}
    \|\partial_k e^{-it\jap{k}}(2i\jap{k})^{-1}(\partial_t+i\jap{k})\widetilde\calF[\bfP_c\bmv] \|_{L^2_{k\ud k}}.
\end{equation}
The final resulting linear estimate, which is used in \cite{LPPSS3}, is recorded in our next theorem. While our proof relies on the distorted Fourier transform the result is stated without reference to it.  Note that in the statement we have introduced factors of $e^{i\theta}$ to view various functions of $r$ as equivariant functions on $\bbR^2$. This is the format that is most directly applicable in \cite{LPPSS3}. As above, we  use the notation $\hat{\bmg}$ to denote the standard Fourier transform of a function $\bmg$ on $\bbR^2$. We also use the notation $L_x^{\infty-}$ to denote the $L^p$ norm on $\bbR^2$ for $1\ll p<\infty$.
\begin{theorem}\label{thm:transference1}
    Suppose 
    \begin{equation} \label{equ:weighted_estimate_black_box_evol_equ}
        \left\{ \begin{aligned}
            &(\pt^2 + \bfM) \bfP_c \bmv = \bfP_c \bmF, \quad (t,r) \in [0,T] \times [0,\infty), \\
            &\bigl( \bfP_c \bmv(0), \bfP_c \pt \bmv(0) \bigr) = (\bma, \bmb).
        \end{aligned} \right.
    \end{equation}
    Define the profiles
    \begin{equation} \label{equ:weighted_estimate_black_box_def_profile}
        \begin{aligned}
            \bmg_\real(t) &:= e^{-it\jD} (2i\jD)^{-1} (\pt + i\jD) \Re\bigl( e^{i\theta} \bfP_c \bmv(t) \bigr), \\ 
            \bmg_\imag(t) &:= e^{-it\jD} (2i\jD)^{-1} (\pt + i\jD) \Im\bigl( e^{i\theta} \bfP_c \bmv(t) \bigr).
        \end{aligned}
    \end{equation}
    Then we have for $0 \leq t \leq T$,
    \begin{equation} \label{equ:weighted_estimate_black_box_bound}
        \begin{aligned}
            &\sum_{\ast \in \{\real, \imag\}} \bigl\| \jxi^2 \nabla_\xi \widehat{\bmg}_\ast(t,\xi) \bigr\|_{L^2_\xi} \\
            &\quad \quad \lesssim \bigl\| \jx e^{i\theta} \bma \bigr\|_{H^2_x} + \bigl\| \jx e^{i\theta} \bmb \bigr\|_{H^1_x} + \bigl\| \jx e^{i\theta} \bmF(0,r) \bigr\|_{L^2_x} + \bigl\| \jx e^{i\theta} (\pt \bmF)(0,r) \bigr\|_{L^2_x} \\ 
            &\quad \quad \quad + \bigl\| \jx e^{i\theta} \bmF(t,r) \bigr\|_{L^2_x} + \sum_{\ast \in \{\real, \imag\}} \Bigl( \bigl\| \bmg_\ast(t) \bigr\|_{H^2_x} + \jt \bigl\| \jD^2 e^{it\jD} \bmg_\ast(t) \bigr\|_{L^{\infty-}_x} \Bigr) \\ 
            &\quad \quad \quad + \sum_{j=0,1} \bigl\| s \cdot \jx \bigl( e^{i\theta} \partial_s^j \bmF(s,r) \bigr) \bigr\|_{L^2_s([0,t]; L^2_x)} + \bigl\| s \cdot \jx \jD \bigl( e^{i\theta} \bmF(s,r) \bigr) \bigr\|_{L^2_s([0,t]; L^2_x)} \\
            &\quad \quad \quad + \sum_{j=1,2} \bigl\| \jap{x} \bigl( e^{i\theta} \partial_s^j \bmF(s,r) \bigr) \bigr\|_{L^1_s([0,t]; L^2_x)} + \bigl\| \jx \jD \bigl( e^{i\theta} \bmF(s,r) \bigr) \bigr\|_{L^1_s([0,t]; L^2_x)}.
        \end{aligned}
    \end{equation}
\end{theorem}
In the next subsection we outline some of the main ideas that go into the proofs of Theorems~\ref{propdFT} and~\ref{thm:transference1}.
\subsection{Outline of the proofs of Theorems~\ref{propdFT} and~\ref{thm:transference1}}
We begin with a discussion of the proof of Theorem~\ref{propdFT}. A well-known consequence of the Bogomolny structure is the super-symmetric factorization of the linearized operator,
\begin{equation}\label{eq:LBstarB_intro}
\bfL=\calB^*\calB,
\end{equation}
with
\begin{equation}\label{eq:calB_intro}
\calB=\begin{pmatrix} \partial_r-b & U \\ U & \partial_r+\frac1r\end{pmatrix},
\qquad
\calB^*=\begin{pmatrix} -\partial_r-\frac{1}{r}-b & U \\ U & -\partial_r\end{pmatrix}.
\end{equation}
As already observed in \cite{LPPSS1PartI}, the super-symmetric partner $\calL:=\calB\calB^\ast$ is diagonal. That is,
\begin{equation}\label{eq:partner_intro}
\calL:=\calB\calB^*=\mathrm{diag}(\calL_1,\calL_2),
\qquad
\calL_1=-\Delta+V_1,
\quad
\calL_2=-\Delta+V_2,
\end{equation}
with
\begin{equation}\label{eq:V1V2_intro}
V_1(r)=\frac{(2-a_\theta(r))^2}{r^2}+\frac12(1+U(r)^2),
\qquad
V_2(r)=U(r)^2.
\end{equation}
Conjugation by $r^{1/2}$ then leads to two strongly singular scalar Schr\"odinger operators $\calH_j:=r^{1/2}\cdot\calL_j\cdot r^{-1/2}$, $j=1,2$, on the half--line. Explicitly
\begin{equation*}
    \calH:=r^{1/2}\cdot\calL\cdot r^{-1/2}=\mathrm{diag}(\calH_1,\calH_2),
\end{equation*}
where
\begin{align*}
\mathcal{H}_1 = -\partial_r^2 - \frac{1}{4r^2} + 1 + \big( V_1(r) - 1\big),\qquad \mathcal{H}_2 = -\partial_r^2 - \frac{1}{4r^2} + 1 + \big( U^2(r)-1 \big).
\end{align*}
The Weyl--Titchmarsh theory of $\calH_j$, which we carry out in the spirit of \cite{GZ} and \cite{KST}, then leads to the distorted Fourier transform for $\calL_j$ and hence of $\calL$. The starting point of the Weyl--Titchmarsh theory of $\calH_j$ is  the understanding of the spectra of these operators. For this we again rely on the properties established in \cite{LPPSS1PartI} with the aid of rigorous numerics. In particular, the edge of the continuous spectrum is regular for both $\calH_1$ and $\calH_2$, in the sense that there are no eigenvalues or resonances at the bottom of the continuous spectrum. Nevertheless, the low-energy analysis of the Jost solutions as well as the spectral density for $\calH_1$ are delicate in view of the critical $r^{-2}$ rate of decay of the potential for large $r$. This can already be seen at the level of the original operator $\bfL$. Indeed, let us decompose $\bfL$ in two different ways as
\begin{align}\label{eq:Lzerodefintro1}
\begin{split}
    &\bfL=\bfL_0+\bfV_0,\\
    &\bfL_0:=\pmat{-\Delta+\frac{1}{r^2}+1&0\\0&-\Delta+\frac{1}{r^2}+1},\quad \bfV_0:=\pmat{\tfrac{a_\theta^2-2a_\theta}{r^2}-\tfrac{3}{2}(1-U^2(r))&-\tfrac{2(1-a_\theta)U}{r}\\-\tfrac{2(1-a_\theta)U}{r}&U^2-1},
\end{split}
\end{align}
and
\begin{align}
\begin{split}\label{eq:Linfdefintro1}
    &\bfL=\bfL_\infty+\bfV_\infty,\\
    &\bfL_\infty:=\pmat{-\Delta+1&0\\0&-\Delta+\frac{1}{r^2}+1},\quad \bfV_\infty:=\pmat{\tfrac{(1-a_\theta)^2}{r^2}-\tfrac{3}{2}(1-U^2(r))&-\tfrac{2(1-a_\theta)U}{r}\\-\tfrac{2(1-a_\theta)U}{r}&U^2-1},
    \end{split}
\end{align}
Then $\bfV_0$ is smooth at the origin but decays at the rate $r^{-2}$ at infinity. Conversely, $\bfV_\infty$ has exponential decay at infinity but has an $r^{-2}$ singularity at $r=0$. 
It is then reasonable to expect that for small $r$ the distorted Fourier basis for $\bfL$ can be approximated by Bessel functions of order one, while for large $r$ Bessel functions of orders both zero and one are needed. At the level of $\calL$, the Jost solutions for $\calH_0$ are well-approximated by Bessel functions of order zero for both small and large $r$. But, for $\calH_1$ at small $r$ Bessel functions of order two provide a good approximation for small $r$ and Bessel functions of order one provide a good approximation at large $r$. The discrepancy between $\bfL$ and $\calL$ comes from the fact that the operator $\calB$ changes the degree of a function by one. The formal expectations above turn out to be correct in terms of the behavior in $r$, but at low frequencies modifications are needed in the $k$ behavior. This can be seen for instance by looking at the upper right entry in \eqref{eq:Eklein} for $r$ and $k$ small. In this range this entry is of order $(\log k)^{-1}r$, which in comparison with $J_1(kr)$ has the same behavior in $r$, but worse by a factor of $(k\log k)^{-1}$ in $k$. Consistent with the  uncertainty principle, this small $k$ asymptotics corresponds to the large $r$ asymptotics of the operator\footnote{That this behavior is seen in the upper right entry of $E(r,k)$ is a consequence of the matrix algebra of the problem. See Section~\ref{sec:spectral_fourier_bfH}.}. The logarithmic improvement over the small $k$ asymptotics of $J_0(kr)$ is a reflection of the absence of threshold resonances. While the small $k$ leading asymptotics is still no worse, in fact better, compared with $J_0(kr)$, the subtle low frequency behavior of the Fourier basis requires special care in using Theorem~\ref{propdFT} to prove estimates, especially in relation to the proof of Theorem~\ref{thm:transference1}.

Formally, the distorted Fourier transform for $\bfL$, and hence of $\bfM$, can then be recovered from that of $\calL$ by inverting the Darboux transform~$\bfL\mapsto \calL$. Alternatively, one can directly develop the Weyl--Titchmarsh theory for the strongly singular, but non-diagonal, matrix operator $\bfH:=r^{1/2}\cdot \bfM\cdot r^{-1/2}$, for instance along the lines of \cite{KS06} or \cite{LSS1}. In this approach the Weyl--Titchmarsh theory of $\calH_j$ allows one to easily compute the connection coefficients relating the Weyl solutions at large and small $r$, which is the most delicate point in the argument. In this work we pursue this alternate approach.


There is an important difference between the present construction and the
one used in \cite[Proposition~5.6, Proposition~5.7]{KST}.
There the outgoing Weyl solution is obtained by factoring out its oscillation
at infinity and constructing a Poincar\'e expansion in inverse powers of
$q=r\sqrt{\xi}$, where $\xi$ denotes the energy parameter.  This gives a
precise description in the oscillatory region $q\gtrsim1$.  The connection
argument in \cite{KST} is made at $r=\delta\xi^{-1/2}$ with $\delta>0$
fixed, and therefore does not by itself give estimates uniformly far to the
left of the transition scale.

Here we retain the inverse-square term in the comparison equation and use
Bessel and Hankel functions as the model solutions.  For $\calH_1$, the
model \eqref{eq:Psi0-def} and the Volterra equation
\eqref{eq:Psi-Volterra} yield
\eqref{eq:Psi10dom}, \eqref{eq:Psip10dom} down to
$R_k=k^{-1/2}$.  Since the turning scale is $r_t\simeq k^{-1}$, one has
$R_k/r_t=k^{1/2}\to0$.  Thus the comparison extends far past the turning
point into the region $kr\ll1$.  For $\calH_2$ there is no genuine turning
point, but the order-zero Hankel model \eqref{eq:Psi20-def} and the Volterra
equation \eqref{eq:Psi2-low-k-Volterra} extend through the transition
$kr\simeq1$ down to the radius \eqref{eq:Rk-H2-small}, where
$kR_k=|\log k|^{-2}\to0$.  The classical connection formulas are therefore
already built into the comparison functions.  In particular, the
small-argument logarithm of $H_0^{(1)}$ is present from the outset, whereas
it is not visible in an oscillatory expansion constructed at infinity.

The overlap at the matching radii also determines the leading complex
phases of the connection coefficients, rather than only their absolute
values.  The formulas \eqref{eq:a-smallk} and \eqref{eq:a2-smallk}, together
with \eqref{eq:eta12-def}, give the cancellations
\eqref{eq:eta12-small}.  These enter explicitly in
\eqref{eq:qY01klargrenz}, \eqref{eq:qY01schranke} and hence in the
low-frequency analysis of the transference identity \eqref{eq:trans1}.
Bounds for $|a_j(k)|$ alone would not retain the relative phases needed for
these cancellations.

The factor $\jap{\log k}^{-1}$ in the second channel is a consequence of
the threshold non-resonance of $\calH_2$.  Indeed
$\widetilde c_{1,*}\neq0$ in \eqref{Phi2_0}, and
\eqref{eq:a2-smallk} give
\[
  a_2(k)^{-1}=O\bigl(|\log k|^{-1}\bigr),\qquad k\to0,
\]
in agreement with \eqref{propSMrho} and the Fourier-basis estimates
\eqref{eq:Eklein}.  This is precisely a borderline gain: it supplies the missing factor in the logarithmic Hardy inequality
\eqref{eq:Hardydeg0}, and thereby gives the one-weight estimate
\eqref{eq:dkbounds2} for the operator \eqref{eq:tildecalMdef1}.  The same
endpoint mechanism occurs in the proof of \eqref{eq:calTcalRbounds1}: after
the changes of variables $r=e^\tau$ and $k=e^{-\sigma}$, the model kernel
\[
 \chi_{\geq2}(r)\chi_{\leq1}(rk)
 \frac{1}{rk\jap{\log k}}
\]
reduces to the bounded Hardy operator
$g\mapsto\int_\tau^\infty g(\sigma)\sigma^{-1}\,\ud\sigma$.
Without the logarithmic gain, both estimates would have a logarithmic
divergence.  A small power loss would replace the single spatial weight in
\eqref{equ:weighted_estimate_black_box_bound} by
$\jap{x}^{1+\epsilon}$.  Since that $x$-weight corresponds to the single
Lorentz field underlying \eqref{eq:dxiprofileintro1}, such a loss would
require more weighted control than is available in the nonlinear analysis 
of~\cite{LPPSS3}.


The basic idea of the proof of Theorem~\ref{thm:transference1} can already be clearly explained in a scalar problem in dimension one. Indeed, suppose $\bfM$ is replaced by $H:=-\frac{\ud^2}{\ud x^2}+1+V(x)$ where $V(x)$ is smooth and exponentially decaying as $x\to\pm\infty$. Moreover, assume that $H$ has no eigenvalues or resonances. Then for $u$ satisfying 
\begin{align}
    (\partial_t^2+H)u(t,x)=f(t,x), \quad u(0,x)=\partial_t u(0,x)=0
\end{align}
we want to estimate $\|\partial_\xi e^{-it\jap{\xi}} \jap{\xi}^{-1}(\partial_t+i\jap{\xi})\hatu\|_{L^2_\xi}$. We have chosen zero initial data for simplicity of exposition. As mentioned above, this is in effect equivalent to estimating $\|\Omega u\|_{L^2_x}$, with $\Omega:=x\partial_t+t\partial_x$. For a more precise statement of this fact see the proof of Theorem~\ref{thm:transference1}, in particular equation \eqref{equ:weighted_energy_g_jxi_nablaxi_comp}. To estimate $\|\Omega u\|_{L^2_x}$ we write
\begin{equation}\label{eq:u1ddistortedinversionintro1}
    u(t,x)=\int_{-\infty}^\infty \phi(x,k) \tilde{u}(t,k)\,\ud k,
\end{equation}
where $\phi(x,k)$ is the distorted Fourier basis satisfying $H\phi(x,k)=\jap{k}^2\phi(x,k)$, and $$\tilde{u}(t,k):=\int_{-\infty}^\infty\overline{\phi(x,k)}u(x)\,\ud x$$ is the distorted Fourier transform of $u$. By the variation of parameters formula
\begin{equation}
    \phi(x,k)=e^{ixk}s_+(x,k)+e^{-ixk}s_-(x,k),
\end{equation}
where for suitable functions $q_\pm(k)$,
\begin{equation}
    s_+(x,k)=q_+(k)-i\int_{-\infty}^xe^{-iyk}V(y)\phi(y,k)\,\ud y,\qquad s_-(x,k)=q_-(k)-i\int_x^\infty e^{iyk}V(y)\phi(y,k)\,\ud y.
\end{equation}
It follows that
\begin{align*}
    &\partial_x\phi(x,k)=ike^{ixk}s_+(x,k)-ike^{-ixk}s_-(x,k),\\
    &x\phi(x,k)=-i\partial_k(e^{ixk}s_+(x,k)-e^{-ixk}s_-(x,k))+i(e^{ixk}\partial_ks_+(x,k)-e^{-ixk}\partial_ks_-(x,k)).
\end{align*}
Combining these relations with \eqref{eq:u1ddistortedinversionintro1} and integrating by parts once in $k$ gives
\begin{align}\label{eq:transferenceintro1}
    \begin{split}
    \Omega u(t,x)&=i\int_{-\infty}^\infty\big(e^{ikx}s_+(x,k)-e^{-ixk}s_-(x,k)\big)\widetilde\Omega \tilde{u}(t,k)\ud k\\
    &\quad+ i\int_{-\infty}^\infty\big(e^{ikx}\partial_ks_+(x,k)-e^{-ixk}\partial_ks_-(x,k)\big)\partial_t\tilde{u}(t,k)\ud k,
    \end{split}
\end{align}
where $\widetilde{\Omega}:=\partial_t\partial_k+tk$. We can then hope to be able to bound the $L^2_x$ norm of the first integral by the $L^2_k$ norm of $\widetilde\Omega\tilde{u}$, and treat the second integral as an error term. Note that $[\widetilde\Omega,\partial_t^2+\jap{k}^2]=0$ so that
\begin{equation}
    (\partial_t^2+\jap{k}^2)\widetilde\Omega\tilde{u}=\widetilde\Omega\tilde{f}.
\end{equation}
This is a restatement of the fact that $[\Omega,\partial_t^2-\partial_x^2+1]=0$, written on the Fourier side, where for such algebraic calculations there is no difference between the standard and distorted Fourier transforms. Moreover, as with the standard Fourier transform, we expect that $\partial_k$ and $k$ appearing in $\widetilde{\Omega}$ on the distorted Fourier side correspond to an $x$ weight and a derivative, respectively, on the physical side. Assuming that we have access to  energy and (dual) integrated local energy decay (ILED) estimates for $H$, it follows that with $\tilde{f}$ denoting the distorted Fourier transform of $f$, and for a suitable $\gamma>0$ coming from the ILED estimate,
\begin{equation}
    \sup_t\|\Omega u\|_{L^2_x}\lesssim \|\jap{r}^\gamma t\jap{D}f\|_{L^2_t L^2_x}+\|\jap{x}\partial_tf\|_{L^1_tL^2_x}.
\end{equation} 
This already explains the appearance of the main terms on the right-hand side of \eqref{equ:weighted_estimate_black_box_bound}, the remaining terms appearing from the contribution of the initial data or the error analysis.

In our two dimensional equivariant setting, Bessel functions of the first and second kinds play the role of the exponentials  $e^{\pm ixk}$ in the outline above. In this context identity \eqref{eq:trans1} in Lemma~\ref{lem:trans1} provides  the analogue of \eqref{eq:transferenceintro1}. The explanation for the appearance of Bessel functions of different orders on the two sides of \eqref{eq:trans1} (see the definition of $\calT_\ast$ in the lemma) is that in two dimensions, the Lorentz boosts $\Omega_{0j}:=x^j\partial_t+t\partial_{x^j}$ change the degree of an equivariant function to which they get applied. Finally, we point out that as mentioned above the critical decay rate of the potential\footnote{The critical decay is related to the existence of finite energy equivariant vortices. Indeed, for $|(\nabla-i\underline{A})\underline{\Phi}|$ to be square integrable with $\underline{\Phi}(r,\theta)=e^{i\theta}U(r)$ and $\lim_{r\to\infty}U(r)=1$, we must have $\lim_{r\to\infty}\underline{A}_\theta(r)=1$. The critical decay of the potential comes precisely from the contribution of $r^{-2}\underline{A}_\theta$ to the magnetic Laplacian. This explains the cancellation in $\bfL_1$ of the $r^{-2}$ factor from restricting the Laplacian to the first harmonic. See \eqref{eq:bL1L2} and \eqref{eq:bdef1}, where $a_\theta=\underline{A}_\theta$.} in $\bfL$ and the fact that the Fourier basis for $\bfL_\infty$ in \eqref{eq:Linfdefintro1} involves Bessel functions of different orders, call for delicate error analysis in implementing our scheme. See for instance the different types of singularities in the different rows of the operator $\widetilde\calM_k$ defined in \eqref{eq:tildecalMdef1} which appears in the error analysis in Lemma~\ref{lem:calTcalRbounds1}.

A final remark is in order about the foregoing scheme for deriving a transference relation. Our approach above was based on relating the Lorentz boost operators on the physical and Fourier sides, namely $\Omega$ and $\widetilde\Omega$. This is natural in view of the well-known commutation property $[\partial_t^2-\partial_x^2+1,\Omega]=0$. However, it is in principle possible to work directly at the level of the profiles,  at the cost of a more complicated  transference relation, which would lead to a more difficult error analysis. To wit, in the one-dimensional model from above and upon commuting $\partial_\xi$ and $e^{-it\jap{\xi}}\jap{\xi}^{-1}(\partial_t+i\jap{\xi})$, suppose we want to compare the quantities
\begin{equation}
    \widehat{\Xi}\hat{u}(t,\xi)\qquad \mathrm{and}\quad \widetilde{\Xi}\tilde{u}(t,k),\qquad \widehat{\Xi}:=\partial_\xi-\frac{it\xi}{\jap{\xi}},\quad \widetilde{\Xi}:=\partial_k-\frac{itk}{\jap{k}}.
\end{equation}
Here $\hat{u}$ and $\tilde{u}$ denote the standard and distorted Fourier transforms of $u$. Write
\begin{equation}
    \hat{u}(t,\xi)=\int_{-\infty}^\infty\int_{-\infty}^\infty \tilde{u}(t,k)e^{-ix\xi}\big(e^{ixk}s_+(x,k)+e^{-ixk}s_-(x,k)\big)\,\ud k \,\ud x.
\end{equation}
Then using the relations $\partial_\xi e^{-ix\xi}=-ix\xi$, and $\frac{it\xi}{\jap{\xi}}e^{-ix\xi}=-t(\sqrt{1-\partial_x^2})^{-1}\partial_xe^{-ix\xi}$, a similar computation as in the derivation of \eqref{eq:transferenceintro1} gives
\begin{align*}
    \widehat\Xi\hat{u}(t,\xi)&=\int_{-\infty}^\infty\int_{-\infty}^\infty \widetilde{\Xi}\tilde{u}(t,k) e^{-ix\xi}\big(e^{ixk}s_{+}(x,k)-e^{-ixk}s_{-}(x,k)\big)\,\ud k \,\ud x\\
    &\quad+\int_{-\infty}^\infty\int_{-\infty}^\infty \tilde{u}(t,k) e^{-ix\xi}\big(e^{ixk}\partial_ks_{+}(x,k)-e^{-ixk}\partial_ks_{-}(x,k)\big)\,\ud k \,\ud x\\
    &\quad+\int_{-\infty}^\infty\int_{-\infty}^\infty itk\tilde{u}(t,k) e^{-ix\xi}\big([(1-\partial_x^2)^{-\frac{1}{2}},s_{+}]e^{ixk}-[(1-\partial_x^2)^{-\frac{1}{2}},s_{-}]e^{-ixk}\big)\,\ud k \,\ud x.
\end{align*}
This is the analogue of \eqref{eq:transferenceintro1}, if we work directly at the level of profiles without using $\Omega$ and $\widetilde{\Omega}$.

\subsection{Related works} 

The theory of static vortices in the self-dual case was largely developed by Jaffe and Taubes, and much of the related material is collected in \cite{bookJT}. We also refer to \cite[Chapter 7]{bookManSut} as well as the introductions of our companion works \cites{LPPSS1PartI}{LPPSS3} for a more comprehensive review of the corresponding literature. The local in time vortex dynamics of the AYMH near the self-dual regime was studied in \cite{Stu}. In particular, Stuart's gauge, which we use in this work, was introduced in \cite{Stu}. Away from the self-dual regime linearized spectral theory and effective dynamics were studied in \cite{GusSig2,GusSigVort,GusVort0,GusVort}. The distorted Fourier transform and linear decay estimates for the related non-gauged vortices in both the Schr\"odinger and wave settings were studied in \cite{LSS1, PPGL,CGP}. Our construction of the distorted Fourier transform in this work follows most closely the method of \cite{KST}, although we also use ideas from \cite{BusPer92,KriSchNLS,LSS1,SSS10p1, SSS10p2}. Since our focus in this article is the distorted Fourier theory and linear decay estimates in the gauged self-dual setting, we refer to \cites{LPPSS1PartI}{LPPSS3} for a more thorough discussion of the developments in other related linear and nonlinear contexts.

To the best of our knowledge, our scheme for using transference relations in the nonlinear analysis is different from existing works. However, use of transference relations in the analysis of solutions near solitons has a long history. Closest to our implementation are the works \cite{CGV13,DKSW16,Stewart24,KSW25, DonKri16}. The works \cite{CGV13,Stewart24,KSW25} use transference relations to study nonlinear Schr\"odinger equations, while \cite{DonKri16, DKSW16} study a massless wave equation. In \cite{DonKri16, DKSW16} the transference relation is proved for the scaling vector field, and estimates on the Lorentz boost vector fields are deduced from this.  Analogous transference relations have been used in the analysis of singularity formation near solitons, for instance in \cite{KST} for the wave equation and in \cite{KriSch05} for the Schr\"odinger equation. For  related approaches in the case of the Schr\"odinger equation see \cite{Naumkin16,Naumkin18,RodTao15}. More recently, starting with \cite{GHW}, an alternative approach has been used by several authors where the vector field, or space-time resonance, method is used directly on the distorted Fourier side. This leads to  multi-linear analysis on the distorted Fourier side. A key object of study in this approach is the nonlinear spectral distribution, which is a measure used to capture nonlinear interactions on the distorted Fourier side. See for instance \cite{Li25,GerPus22, ChenPus24, ChenLuhr24,PalPus24,PusSof20,LiLuhr24,GPR18,ColGer25,GermPusZhang22}, and the references therein. Another common approach is based on applying the wave operator to the entire nonlinear problem, thereby reducing the linear operator to the unperturbed one at the cost of introducing non-local nonlinearities. See for instance \cite{Delort16,DMKink}. For the massless wave equation we also mention the approach developed in \cite{DR1,Mos1,Schlue1} which combines ILED and a vector field analysis purely on the physical side. For related works on pointwise decay for wave equations with variable coefficients see \cite{OPT26,BVW18,LukOh24,AAG18,BVW15, DSS11, DSS12, Hintz22, MTT12} and references therein. 
\subsection{Organization of the paper}
Section~\ref{sec:spectral_fourier} contains the construction of the distorted Fourier transform for $\bfM$. After recalling the Weyl--Titchmarsh framework for the scalar half--line operators arising from the super-symmetric reduction, we construct the Weyl solutions $\Phi_j(r,k)$ near $r=0$ and $\Psi_j(r,k)$ near $r\to\infty$, and identify the corresponding spectral densities $a_j(k)$. The analysis is carried out separately for $\mathcal{H}_1$ and $\mathcal{H}_2$ in Subsections~\ref{subsec:calH1} and~\ref{subsec:calH2}. These scalar constructions are then combined in Subsection~\ref{sec:spectral_fourier_bfH} to obtain the matrix of generalized eigenfunctions $\bfE(r,k)$ and to prove Theorem~\ref{propdFT}.
The linear decay estimates are established in Section~\ref{sec:linear_decay_estimates}. First, in Subsection~\ref{ssecLP} we define and prove estimates on the Littlewood-Paley projections associated with $\bfL$. Then in Subsections~\ref{ssecKG} and~\ref{subec:localenergy} we prove dispersive and local energy decay estimates, respectively.
Section~\ref{sec:working_title_transference} is devoted to the transference relations and the proof of Theorem~\ref{thm:transference1}.
Finally, in the appendix we collect a number of linear estimates for the Klein-Gordon equation on $\bbR^2$ that are used in the nonlinear analysis in \cite{LPPSS3}.


\section{Spectral and Distorted Fourier Theory} \label{sec:spectral_fourier}

In this section we develop the spectral and distorted Fourier theory for the  selfadjoint matrix Schr\"odinger operator
\begin{equation}
    \bfM := \begin{pmatrix} \bfL & 0 \\ 0 & \bfL \end{pmatrix},
    \qquad
    \bfL := \begin{pmatrix} L_1 & -2U' \\ -2U' & L_2 \end{pmatrix},
\end{equation}
where $L_1$ and $L_2$ are given by \eqref{eq:bL1L2}. In particular we prove Theorem~\ref{propdFT}. The basic operator--theoretic properties of $\bfM$ are recalled in Theorem~\ref{maintheo1} in the introduction. 
We have stated Theorem~\ref{propdFT} in terms of the operator $\bfM$ rather than $\bfL$ because this is how it will be mostly used in our in our companion paper~\cite{LPPSS3}. However, for proving linear estimates it is more convenient to work with $\bfL$. For this we define 
\begin{equation}
    \wtilcalF_\bfL[\bmg](k):=\int_0^\infty \overline{E(r,k)}^{\,t}
    \bmg(r)r\,\ud r,
\end{equation}
where $\bmg\in L^2_{r \ud r}(\bbR_+;\bbC^2)$. Applying Theorem~\ref{propdFT} to $\bbC^2$ valued functions $(\bmg,0)^t$ we find the inverse of $\wtilcalF_\bfL$ to be 
\begin{equation}
    \wtilcalF_\bfL^{-1}[\bmh](r)=\int_0^\infty E(r,k)\bmh(k)k \ud k,
\end{equation}
where $\bmh\in L^2_{k \ud k}(\bbR_+;\bbC^2)$. By this we mean that $P_c^\bfL=\wtilcalF_{\bfL}^{-1}\circ\wtilcalF_\bfL$, where $P_c^\bfL$ denotes the orthogonal projection to the continuous spectrum of $\bfL$. Moreover, we have
\begin{align*}
    \|\wtilcalF_\bfL[\bmg]\|_{L^2_{k\ud k}}=\|P_c^\bfL\bmg\|_{L^2_{r\ud r}},\qquad\mathrm{and}\qquad \wtilcalF_\bfL[\bfL \bmg](k)=\jap{k}^2\wtilcalF_\bfL[\bmg](k).
\end{align*}

The remainder of this section is organized as follows. First, in Sections~\ref{subsec:calH1} and~\ref{subsec:calH2}, we analyze the distorted Fourier transform for the scalar operators $\calH_1$ and $\calH_2$. Then, using the super-symmetric factorization of $\bfM$, in Section~\ref{sec:spectral_fourier_bfH} we relate the distorted Fourier transform for $\bfM$ to those of $\calH_1$ and $\calH_2$ and prove Theorem~\ref{propdFT}.

\subsection{General spectral theory}
Here, following the general procedure in \cite{GZ,KST,KMS}, we analyze the distorted Fourier transform for the operator $\calH_j=r^{1/2}\calL_j r^{-1/2}$. Denote by
\begin{align}\label{def_Phij}
\Phi_j(r,k) \quad \hbox{ and } \quad \Theta_j(r,k)
\end{align}
a real-valued fundamental system of generalized eigenfunctions for the problem
\begin{align*}
\mathcal{H}_j f = (1+k^2) f,
\end{align*}
with $\Phi_j\in L^2((0,1))$, and let $\Psi_j(r,z)$, $\Im z>0$, be the
Weyl solution at infinity satisfying
\[
  \mathcal H_j\Psi_j(\cdot,z)=(1+z^2)\Psi_j(\cdot,z).
\]
For $k>0$ we denote its boundary value by $\Psi_j(r,k)$.
In the sequel we always assume that $\Phi_j$, $\Theta_j$, and $\Psi_j$ are chosen so that $W(\Theta_j,\Phi_j) = 1$ and $W(\overline{\Psi_j},\Psi_j) = 2i$, where
\[
W(f,g) = fg' - f'g.
\]
The general theory of \cite{GZ} gives
\begin{align}\label{def_aj}
\Phi_j(r,k)=2\Re\bigl(a_j(k)\Psi_j(r,k)\bigr),
\qquad
a_j(k)=\frac{W(\Phi_j,\overline{\Psi_j})}
{W(\Psi_j,\overline{\Psi_j})}.
\end{align}
Writing $\lambda=k^2$ for the spectral parameter above the threshold,
the corresponding spectral density is
\[
  \frac{d\rho_j}{d\lambda}(\lambda)
  =\frac{1}{4\pi|a_j(\sqrt\lambda)|^2}.
\]
Equivalently, in the momentum variable $k$,
\[
  \frac{d\rho_j}{dk}(k)
  =\frac{k}{2\pi|a_j(k)|^2}.
\]
Thus, if
\[
  \widetilde f(k)=\int_0^\infty \Phi_j(s,k)f(s)\,ds,
\]
then
\[
  P_c^{\mathcal H_j}f(r)=\int_0^\infty
  \Phi_j(r,k)\widetilde f(k)
  \frac{k}{2\pi|a_j(k)|^2}\,dk.
\]
\subsection{Linear spectral analysis for \texorpdfstring{$\mathcal{H}_1$}{H1}}\label{subsec:calH1}
Recall that $\mathcal{H}_1$ can be written  explicitly as  
\begin{align*}
\mathcal{H}_1 &= -\partial_r^2 - \frac{1}{4r^2} + 1 + \big( V_1(r) - 1\big),
\end{align*}
where 
\begin{align}\label{FouV1}
V_1(r)-1=\frac{1}{r^2} - \frac{1}{2}(1-U^2) + \frac{1}{r^2}(1-a)(3-a).
\end{align}
Note that $V_1(r)=4 r^{-2}+O(1)$ as $r \rightarrow 0$ and $V_1(r)=1 + r^{-2} + O(r^{-1/2}e^{-r})$ as $r \rightarrow \infty$.
We seek to construct generalized eigenfunctions of
\begin{align}\label{L1evalue}
\mathcal{H}_1 f = (k^2 + 1)f, \quad k \geq 0,
\end{align}
and obtain precise asymptotics as $r\to 0$ and $r\to+\infty$ for the Weyl solution at zero, as well as for the Weyl solution at infinity. For the former, we first solve the equation at $k=0$ and then proceed perturbatively for $rk \lesssim 1$ with $0<k\le 1$.  For large $k$ it is more efficient to approximate by Bessel functions on $0<r\le 1$. The Weyl solutions around $r=\infty$ we will find as perturbations of Hankel functions. Throughout, we rely on the asymptotic information on $(U,a_\theta)$ stated in~\eqref{Uaexp0}, \eqref{Uaexpinfty}. The key feature here is the appearance of two angular momenta: $\ell=2$ near $r=0$ and $\ell=1$ for large~$r$. This will determine the correct choice of basis functions depending on the energy~$k^2+1$. 

\subsubsection{Fundamental solutions at zero energy}
We look for $\Phi_1^{(0)}$, the zero-energy solution, which solves $\mathcal{H}_1 \Phi_1^{(0)} = \Phi_1^{(0)}$ and is square-integrable near $r=0$. For large $r$ the equation is asymptotic to
\[
-y''+\frac{3}{4r^2}y=0,
\]
whose fundamental system is given by
\begin{align}\label{fsys_h1_inf_app}
y_1(r)=r^{3/2}, \qquad y_2(r)=\tfrac{1}{2}r^{-1/2}.
\end{align}
We choose the normalization so that $W[y_2,y_1](r)=1$. 

\begin{lem}\label{lem_p1_t1}
There exists a fundamental system of real-valued solutions of $\mathcal{H}_1f=f$ with the following asymptotic behavior
\begin{align}\label{Phi1_0}
\Phi_{1}^{(0)} (r) & = \begin{cases}
r^{5/2}+O(r^{9/2}), & r\lesssim 1,
\\ c_*r^{3/2}+O(re^{-r}),  &  r \gtrsim 1, 
\end{cases}
\\ \Theta_{1}^{(0)} (r) & = \begin{cases}
 \tfrac{1}{4}r^{-3/2} + O(r^{1/2}), & r\lesssim 1,
\\ \tfrac1{2c_*}r^{-1/2}+O(r^{-1}e^{-r}), &   r \gtrsim 1, 
\end{cases} \label{Theta1_0}
\end{align}
for some $c_*\neq0$. These functions are normalized so that $W[\Theta_{1}^{(0)},\Phi_{1}^{(0)}]=1$. Furthermore, the symbol-type bounds
\[
|(r\partial_r)^m \Phi_{1}^{(0)} (r) |\le C_m \, r^{\frac52}\langle r\rangle^{-1} 
\]
and 
\[
|(r\partial_r)^m \Theta_{1}^{(0)} (r) |\le C_m\,  r^{-\frac32}\langle r\rangle
\]
hold for all $m\ge1$ and $r>0$. 
\end{lem}

\begin{proof}
Equation \eqref{eq:AYMH1} for $(\Phi(r,\theta),A(r,\theta))=(U(r)e^{i\theta},a_\theta(r)\ud \theta)$  implies that $\calH_1(r^{3/2}U)=r^{3/2}U$. Therefore $\Phi_1^{(0)}(r):=(U'(0))^{-1}r^{3/2}U(r)$ satisfies \eqref{Phi1_0}. Let
\begin{equation}\label{Theta10}
\Theta_1^{(0)}(r):=-\Phi_1^{(0)}(r)\int_r^\infty (\Phi_1^{(0)}(s))^{-2}\, ds.
\end{equation}
Then $W[\Theta_1^{(0)},\Phi_1^{(0)}]=1$ and \eqref{Theta1_0} is satisfied. The derivative estimates for $\Phi_1^{(0)}$ and $\Theta_1^{(0)}$ follow from the corresponding estimates for $U$.
\end{proof}

\subsubsection{Solving from $r=0$}
In this subsection we construct solutions $\Phi_1(r,k)\sim r^{5/2}$ as $r\to0$. We begin by deriving a power series expansion for $\Phi_1(r,k)$ by perturbing from the fundamental system at energy $0$. The following result holds for all $k>0$, but it is only efficient in the range $0<k\le1$. 

\begin{lem}\label{Phi1Op1_1}
For all $r>0$ and $0<k$ we have the expansion
\begin{align}\label{Phi1Op1_2}
\Phi_1(r,k) = \Phi_{1}^{(0)}(r) + r^{3/2} \sum_{j\geq 1} (rk)^{2j} \Phi_{1,j}(r),
\end{align}
which converges absolutely. The expansion converges uniformly if $rk \leq 1$, and the functions $\Phi_{1,j}$ are smooth. In the case $j=1$ we have
\begin{align}\label{Phi1Op1_10}
\Phi_{1,1}(r) = \begin{cases} 
-\tfrac{1}{12}r  + O(r^3), &  r\lesssim 1,
\\ -\tfrac{c_*}{8} + O(r^{-2}), &  r \gtrsim 1.
\end{cases}
\end{align}
Moreover, for some absolute $C>0$, for all $j\geq 1$
\begin{equation}\label{Phi1Op1_3}
 | \Phi_{1,j}(r) | \leq\frac{C^j}{j!}  r\langle r\rangle^{-1}
\end{equation}
In \eqref{Phi1Op1_3} we also have consistent symbol-type bounds for the derivatives, that is, for some constants $C_{m}>0$
\begin{equation}\label{Phi1Op1_4}
 | (r\partial_r)^m\Phi_{1,j}(r) | \leq     \frac{C_{m}^j}{j!}  r\langle r\rangle^{-1}
\end{equation}
for all $m\ge1$ and $j\geq1$. 
\end{lem}

\begin{proof}
To solve $\mathcal{H}_1\Phi_1(r,k) = (k^2+1) \Phi_1(r,k)$, we begin with the ansatz
\begin{align*}
\Phi_1(r,k) = r^{3/2} \sum_{j\geq 0} k^{2j} f_j(r), \qquad f_0(r) := r^{-3/2}\Phi_1^{(0)}(r).
\end{align*}
Substituting it into the equation yields
\begin{align}\label{Phi1Op1_5}
\mathcal{H}_1 (r^{3/2} f_j(r))-r^{3/2}f_j(r) = r^{3/2} f_{j-1}(r), \quad j\geq 0, \quad f_{-1}(r) = 0.
\end{align}
Comparing with \eqref{Phi1Op1_2}, we have $f_j(r) = r^{2j} \Phi_{1,j}(r)$. From \eqref{Phi1Op1_5} and the backward Green's function
\begin{align*}
(\mathcal{H}_1-1)^{-1}(r,s)=\Phi_1^{(0)}(r)\Theta_1^{(0)}(s) - \Theta_1^{(0)}(r)\Phi_1^{(0)}(s)  \quad  \hbox{ for }  \quad r>s,
\end{align*}
we obtain the recursion formula
\begin{align}\label{Phi1Op1_13}
r^{3/2} f_j(r) = \int_0^r 
\Big( \Theta_1^{(0)}(s)\Phi_1^{(0)}(r)-\Theta_1^{(0)}(r)\Phi _1^{(0)}(s)\Big) s^{3/2} f_{j-1}(s)\, ds.
\end{align}
Recalling \eqref{Theta10}, it follows that the recursion for $f_j(r)$ can be conveniently rewritten as 
\begin{align}\label{Phi1Op1_6}
\begin{split}
& f_j(r) =- \int_0^r K(r,s) f_{j-1}(s)\, ds, \quad j\geq 1
\\
& K(r,s) := \Big(\dfrac{s}{r}\Big)^{3/2}\Phi_1^{(0)}(r)\Phi_1^{(0)}(s) \int_s^r \dfrac{1}{\Phi_1^{(0)}(t)^2} \, dt
\end{split}
\end{align}
Using the asymptotics in \eqref{Phi1_0}, \eqref{Theta1_0} for small and large $r$, we obtain
\begin{align}\label{Phi1Op1_7}
K(r,s)
= \left\{ 
\begin{array}{ll}
\dfrac{1}{4r^3} \big(r^4-s^4 \big) +O\big(r^{2}(r-s)\big) & s < r \lesssim 1,
\\ 
\\
\dfrac{s}{2}\big(1-\dfrac{s^2}{r^2}\big) + O\big(s^{\frac12}e^{-s} [(r-s)\wedge 1]\big) &  r > s \gtrsim 1.
\end{array} 
\right.
\end{align}
In the region $0<s\le 1\le r$ we have the upper bound $0<K(r,s)\le K_0$, for some absolute constant $K_0$. The error bounds are stable under differentiation with respect to~$r$ in the following sense:
\[
|(r\partial_r)^m O\big(r^{2}(r-s)\big)| \leq C_m\, r^3 \qquad  s < r \lesssim 1
\]
as well as
\[
|(r\partial_r)^m O\big(s^{\frac12}e^{-s} [(r-s)\wedge 1]\big)| \leq C_m\, r e^{-s} \qquad  1\lesssim s < r
\]
Similar estimates hold for derivatives in $s$, but we do not need them.

\medskip
\noindent
{Proof of \eqref{Phi1Op1_10}.}
From \eqref{Phi1Op1_6}, when $r \lesssim 1$, we can use \eqref{Phi1_0} again to see that
\begin{equation}\label{Phi1Op1_14}\begin{aligned}
f_1(r) & = -r^{-3/2}\Phi_{1}^{(0)}(r) \int_0^r \Phi_1^{(0)}(s)^2 \Big( \int_s^r \frac{1}{\Phi_1^{(0)}(t)^2} \, dt \Big) \, ds \nonumber
\\
  & = -(r + O(r^2)) \int_0^r (s^5 + O(s^7)) \Big( \int_s^r \frac{1}{t^5} (1+O(t^2)) \, dt \Big) \, ds
\\ & = -\frac{1}{12} r^3 + O(r^5). \nonumber
\end{aligned}\end{equation}
Since $f_1(r) = r^{2} \Phi_{1,1}(r)$, the asymptotic formula for $r\lesssim 1$ in \eqref{Phi1Op1_10} follows.
When $r \gg 1$, we use \eqref{Phi1_0} and \eqref{Phi1Op1_7} to infer that
\begin{align*}
\int_{1}^r K(r,s)f_0(s)\, ds  & = \int_{1}^r  \Big(\dfrac{s}{2}\big(1-\dfrac{s^2}{r^2}\big) + O\big(s^{\frac12}e^{-s} [(r-s)\wedge 1]\big) \Big)\dfrac{1}{s^{3/2}}\Phi_1^{(0)}(s)ds
\\
 & = \dfrac{c_*}{8}r^2+O(1). \nonumber
\end{align*}
Since the integration from $0$ to $1$ gives a bounded contribution, it follows that
\begin{align*}
f_1(r) & = -\dfrac{c_*}{8}r^{2}+ O(1), \qquad r \gtrsim 1,
\end{align*}
which gives the second asymptotic formula for $\Phi_{1,1}(r)$.

\medskip

\noindent
{Proof of \eqref{Phi1Op1_3}.}
We proceed by induction, having already proved the case $n=1$. We assume that there exists a constant $C>0$ such that, for a given $n\ge2$ and all $r>0$,
\begin{equation}
    \label{eq:fkbd}
    0<|f_{m}(r)|\leq \dfrac{C^{m}}{m!}r^{2m+1} \langle r\rangle^{-1},\qquad m \leq n-1.
\end{equation}
Substituting this into the identity for $f_n(r)$ in terms of $f_{n-1}(r)$, for $r\geq 1$, we find that
\begin{align*}
|f_n(r)| & =\Big|\int_{0}^r K(r,s)f_{n-1}(s)\,ds\Big|
\\ & \leq \dfrac{1}{2}\int_{1}^r \Big(s\big(1-\dfrac{s^2}{r^2}\big)  + O\big(s^{\frac12}e^{-s} [(r-s)\wedge 1]\big) \Big) \dfrac{C^{n-1}}{(n-1)!}s^{2n-2}\, ds 
+ \int_0^1 K_0 \dfrac{C^{n-1}}{(n-1)!}s^{2n-1}\, ds \\
&= \dfrac{1}{2}\int_{0}^r  s\big(1-\dfrac{s^2}{r^2}\big)\dfrac{C^{n-1}}{(n-1)!}s^{2n-2}\, ds    + C_1 \int_1^r  s^{\frac12}e^{-s} [(r-s)\wedge 1]  \dfrac{C^{n-1}}{(n-1)!}s^{2n-2}\, ds  \\
&\qquad\qquad  + \int_0^1 \Big[ K_0 +  \frac{1}{2}\big(1-\dfrac{s^2}{r^2}\big)\Big] \dfrac{C^{n-1}}{(n-1)!}s^{2n-1}\, ds 
\\ & \leq  \dfrac{C^{n-1}r^{2n}}{4\, n!} + C_1 \frac{C^{n-1} r^{2n-\frac32}}{n!}+ \big (K_0+\frac12\big)\dfrac{C^{n-1}}{2\, n!}\le \dfrac{C^{n}r^{2n}}{n!}.
\end{align*}
This implies the second estimate in \eqref{Phi1Op1_3}. To pass to the last line, we bounded $e^{-s} [(r-s)\wedge 1]\le 1$.
If $0<r\le 1$, then
\begin{align*}
|f_{n}(r) |& =\Big| \int_0^r K(r,s) f_{n-1}(s) \, ds\Big|
\\ & \leq \int_0^r  \left(\dfrac{1}{4r^3} \big(r^4-s^4 \big) + C_1r^{-1}(r^4-s^4)\right)\dfrac{C^{n-1}}{(n-1)!}s^{2n-1} \, ds
  \\ & =\dfrac{C^{n-1}}{8\, n!} r^{2n+1}+ C_1\dfrac{C^{n-1} }{2\,  n!} r^{2n+3} \leq \dfrac{C^{n}}{ n!} r^{2n+1}.
\end{align*}
This concludes the proof of \eqref{Phi1Op1_3}.
To see that $f_j$ is smooth on $(0,\infty)$, one can proceed by induction. The smoothness of $f_0(r)$ follows directly from that of $\Phi_1^{(0)}(r)$ and the identity $f_0(r)=r^{-3/2}\Phi_1^{(0)}(r)$. The general inductive step follows from the recursive representation \eqref{Phi1Op1_6} and the smoothness of $K(r,s)$ on $\{(r,s)\in\R^2: \, 0<s<r\}$.

\medskip
\noindent
{Proof of \eqref{Phi1Op1_4}.}
We proceed by a double induction in $j$ and $m$. First note that, from \eqref{Phi1Op1_6}, after changing variables we infer that for all $j\ge1$,
\begin{align}\label{Phi1Op1_15}
\Phi_{1,j}(r\lambda)=-\dfrac{1}{\lambda r}\int_0^1 K( r\lambda , \lambda ru) u^{2j-2} \Phi_{1,j-1}(\lambda ru)\, du,
\end{align}
for $\lambda\simeq1$. Moreover, by \eqref{Phi1Op1_7},
\begin{align*}
\dfrac{1}{\lambda r} K(\lambda r, \lambda ru)
= \left\{ 
\begin{array}{ll}
\dfrac{1}{4} \big(1-u^4 \big) +O\big(r^{2}\lambda^2(1-u)\big), & u < 1  \lesssim (\lambda r)^{-1},
\\ 
\\
\dfrac{1}{2} u \big(1- u^2 \big) + O\big((\lambda ru)^{\frac12}e^{-\lambda ru} [(1-u)\wedge (\lambda r)^{-1}]\big), &  1 > u \gtrsim (\lambda r)^{-1}.
\end{array} 
\right.
\end{align*}
Note that, at leading order, the above expression does not depend on $\lambda$.

We now show that the result holds for all $j\geq1$ and $m=1$ fixed. Indeed, assuming \eqref{Phi1Op1_4} holds for $(j-1,m)$ with $m=1$, then, from \eqref{Phi1Op1_15}, we obtain that for $r\lesssim 1$, \begin{align*}
    \big\vert r \partial_r \Phi_{1,j}(r)\big\vert & \leq \int_0^1 \Big|\partial_\lambda \Big[\dfrac{1}{\lambda r}K( r\lambda , \lambda ru)\Big]_{\lambda=1}\Big| \; u^{2j-2} \big\vert \Phi_{1,j-1}( ru)\big\vert \, du
    \\ & + \int_0^1 \bigg\vert \dfrac{1}{ r}K( r , ru) \bigg\vert \, u^{2j-2} \,  \Big\vert \partial_\lambda \Phi_{1,j-1}(\lambda r u)\Big\vert_{\lambda=1} \, du
    \\ & \lesssim \int_0^1 \bigg(\dfrac{1}{4}(1-u^4) +O\big(r^2(1-u)\big)\bigg) u^{2j-2} \dfrac{1}{(j-1)!} \dfrac{r u}{\langle r u\rangle } \, du
    \\ & \lesssim \dfrac{1}{2j} \, \dfrac{1}{(j-1)!}\, r. 
\end{align*}
In the case $r\gtrsim 1$ we obtain \begin{align*}
    \big\vert r \partial_r \Phi_{1,j}(r)\big\vert & \leq \int_0^1 \Big|\partial_\lambda \Big[\dfrac{1}{\lambda r}K( r\lambda , \lambda ru)\Big]_{\lambda=1}\Big| \; u^{2j-2} \big\vert \Phi_{1,j-1}( ru)\big\vert \, du
    \\ & + \int_0^1 \bigg\vert \dfrac{1}{ r}K( r , ru) \bigg\vert \, u^{2j-2} \,  \Big\vert \partial_\lambda \Phi_{1,j-1}(\lambda r u)\Big\vert_{\lambda=1} \, du
    \\ & \lesssim \int_0^1 \bigg(\dfrac{1}{2}u (1-u^2)+O\big(( ru)^{\frac12}e^{-ru} [(1-u)\wedge  r^{-1}] \big)\bigg) u^{2j-2} \dfrac{1}{(j-1)!} \dfrac{r u}{\langle r u\rangle } \, du
    \\ & \lesssim \dfrac{1}{2j} \, \dfrac{1}{(j-1)!},
\end{align*}
which is consistent with \eqref{Phi1Op1_4}. Proving the base inductive hypothesis for $j=1$ fixed and $m\geq1$ follows similar direct calculations (assuming it holds for $(j,m-1)$ with $j=1$) and our previous bounds on $\Phi_{1,1}$ (see \eqref{Phi1Op1_14}).

Now we seek to prove the main inductive step. Taking the derivative bounds on the errors in \eqref{Phi1Op1_7} into account, we infer that for $j\ge1$
\begin{align*}
| r^m\partial_r^m \Phi_{1,j}(r)| & \leq \sum_{\ell=0}^m \binom{m}{\ell}\int_0^1 \Big|\partial_\lambda^{m-\ell} \Big[\dfrac{1}{\lambda r}K( r\lambda , \lambda ru)\Big]_{\lambda=1}\Big| \; u^{2j-2} \Big\vert \partial_\lambda ^\ell \Phi_{1,j-1}(\lambda ru)\Big\vert_{\lambda=1}\, du.
\end{align*}
Then, if $r\lesssim 1$, using the inductive hypothesis, we get that 
\begin{align*}
    | r^m\partial_r^m \Phi_{1,j}(r)| & \leq C_m \int_0^1 \Big|\dfrac{1}{4}(1-u^4) + O\big(r^2(1-u)\big) \Big| \; u^{2j-2} \dfrac{C_m^{j-1}}{(j-1)!} \dfrac{r u}{\langle r u \rangle} \, du 
    \\ & \leq C_m^j \dfrac{1}{2j}\, \dfrac{1}{(j-1)!} r,
\end{align*}
having used that for $\ell<m$ the kernel gives lower order terms (i.e., when at least one $\partial_\lambda$ hits the kernel) with a different constant only depending on $m$, and hence they can be absorbed by the leading order term by enlarging the constant $C_m$. In the case $r \gtrsim 1$, similarly we have that 
\begin{align*}
    & | r^m\partial_r^m \Phi_{1,j}(r)| & 
    \\ & \leq C_m \int_0^1 \Big| \dfrac{1}{2} u \big(1- u^2 \big) + O\big((\lambda ru)^{\frac12}e^{-\lambda ru} [(1-u)\wedge (\lambda r)^{-1}]\big) \Big| \; u^{2j-2} \dfrac{C_m^{j-1}}{(j-1)!} \dfrac{r u}{\langle r u \rangle} \, du 
    \\ & \leq C_m^j \dfrac{1}{2j}\, \dfrac{1}{(j-1)!},
\end{align*}
having used again that the terms coming from the cases $\ell<m $ can be absorbed by the leading order term for $m=\ell$ by simply enlarging the constant $C_m$. This concludes the proof of \eqref{Phi1Op1_4}.
\end{proof}

For $k\ge1$ it becomes inefficient to perturb off the $0$-energy solutions. Instead, we start from a Bessel approximation that remains accurate on $(0,1]$. This will be crucial for large $k$ since we solve the connection problem at $r=R_k:=k^{-1/2}$. To set up the analysis, observe that \eqref{L1evalue} is equivalent to
\begin{equation}\label{eq:W0}
\Bigl(-\partial_r^2+\frac{15}{4r^2}-W_0(r)\Bigr)f(r)=k^2 f(r),
\qquad
W_0(r)=\frac{4a_\theta(r)-a_\theta(r)^2}{r^2}+\frac12\bigl(1-U(r)^2\bigr). 
\end{equation}
Furthermore, \eqref{Uaexp0} implies that  $W_0(r)=\frac32+O(r^2)$ as $r\to0$.
For $k\geq1$ define  
\begin{equation}\label{eq:Phi10_J2}
\Phi_{1,0}(r,k):=8k^{-2}\sqrt r\,J_2(kr).
\end{equation}
and the Green kernel
\[
G_{k,2}(r,s):=\frac{\pi}{2}\sqrt{rs}\,\Bigl[J_2(kr)Y_2(ks)-Y_2(kr)J_2(ks)\Bigr], 
\qquad 0<s\le r\le1.
\]
Then on $(0,1]$ the real-valued Weyl solution at zero $\Phi_1(r,k)$ admits  the Volterra representation
\begin{equation}\label{eq:Phi1-high-k-Volterra}
\Phi_1(r,k)=\Phi_{1,0}(r,k)+\int_0^r G_{k,2}(r,s)W_0(s)\Phi_1(s,k)\,ds.
\end{equation}
The following lemma controls the error between $\Phi_1(r,k)$ and the Bessel approximation. 

\begin{lem}\label{lem:Phi1-high-k-J2}
As $k\to\infty$,
\begin{align}
\sup_{0<r\leq 1}\Bigl|\Phi_1(r,k)-\Phi_{1,0}(r,k)\Bigr|&\lesssim k^{-7/2},\label{eq:Phi1-high-k-J2-bd1}
\\
\sup_{0<r\leq 1}\Bigl|\partial_r\Phi_1(r,k)- \partial_r\Phi_{1,0}(r,k)\Bigr|&\lesssim k^{-5/2}.\label{eq:Phi1-high-k-J2-bd2}
\end{align}
\end{lem}

\begin{proof}
The functions $\sqrt r\,J_2(kr)$ and $\sqrt r\,Y_2(kr)$ solve
\[
\Bigl(-\partial_r^2+\frac{15}{4r^2}\Bigr)f=k^2 f,
\]
and their Wronskian is
\[
W\bigl(\sqrt r\,J_2(kr),\sqrt r\,Y_2(kr)\bigr)=\frac{2}{\pi}.
\]
Hence $G_{k,2}$ is the Green kernel for the comparison operator $-\partial_r^2+\frac{15}{4r^2}-k^2$ with regular boundary condition at $r=0$, and variation of constants yields~\eqref{eq:Phi1-high-k-Volterra}.

To estimate the Volterra equation, write $x:=kr$ and
\[
\phi(x,k):=k^{5/2}\Phi_1(x/k,k),\qquad 0<x<k.
\]
Then \eqref{eq:Phi1-high-k-Volterra} becomes
\begin{equation}\label{eq:phi-scaled-Volterra}
\phi(x,k)=8\sqrt x\,J_2(x)+k^{-2}\int_0^x \mathcal G(x,y)W_0(y/k)\phi(y,k)\,dy,
\end{equation}
where
\[
\mathcal G(x,y):=\frac{\pi}{2}\sqrt{xy}\,\Bigl[J_2(x)Y_2(y)-Y_2(x)J_2(y)\Bigr],
\qquad 0<y\le x<k.
\]
Using the standard bounds
\[
|J_2(z)|\lesssim
\begin{cases}
z^2,&0<z\le1,
\\
z^{-1/2},&z\ge1,
\end{cases}
\qquad
|Y_2(z)|\lesssim
\begin{cases}
z^{-2},&0<z\le1,
\\
z^{-1/2},&z\ge1,
\end{cases}
\]
and the analogous estimates for one derivative, we obtain
\[
|\mathcal G(x,y)|\lesssim
\begin{cases}
x^{5/2}y^{-3/2},&0<y\le x\le1,
\\
y^{-3/2},&0<y\le1\le x,
\\
1,&1\le y\le x,
\end{cases}
\qquad
|\partial_x\mathcal G(x,y)|\lesssim
\begin{cases}
x^{3/2}y^{-3/2},&0<y\le x\le1,
\\
y^{-3/2},&0<y\le1\le x,
\\
1,&1\le y\le x.
\end{cases}
\]
Let
\[
\omega(x):=
\begin{cases}
x^{5/2},&0<x\le1,
\\
1,&x\ge1.
\end{cases}
\]
Then $|8\sqrt x\,J_2(x)|\lesssim \omega(x)$, and the previous kernel bounds imply
\[
\omega(x)^{-1}\int_0^x |\mathcal G(x,y)|\omega(y)\,dy\lesssim 1+x\lesssim k,
\qquad 0<x\leq k.
\]
Since $W_0(y/k)$ is uniformly bounded for $0<y\leq k$, the Volterra operator on the right-hand side of \eqref{eq:phi-scaled-Volterra} has norm $O(k^{-1})$ on the Banach space
\[
X:=\Bigl\{f\in C((0,k])\:\Big|\: \|f\|_X:=\sup_{0<x\leq k}\omega(x)^{-1}|f(x)|<\infty\Bigr\}.
\]
Therefore,
\[
\|\phi(\cdot,k)-8\sqrt{\cdot}\,J_2(\cdot)\|_X\lesssim k^{-1}.
\]
Differentiating \eqref{eq:phi-scaled-Volterra} and using the bound for $\partial_x\mathcal G$ in exactly the same way yields
\[
\sup_{0<x\leq k}\Bigl|\partial_x\phi(x,k)-8\partial_x\bigl(\sqrt x\,J_2(x)\bigr)\Bigr|\lesssim k^{-1}.
\]
Rescaling back to $r=x/k$ gives \eqref{eq:Phi1-high-k-J2-bd1}--\eqref{eq:Phi1-high-k-J2-bd2}.
\end{proof}
So far we have focused on  $\Phi_1(r,k)$ in the region where $rk\leq1$. For future reference we also let
\begin{equation}\label{eq:Theta1varpar1}
    \Theta_1(r,k):=\Phi_1(r,k)\int_r^{1/k}\frac{\ud s}{(\Phi_1(s,k))^2},
\end{equation}
where the normalization is chosen so that $W(\Theta_1,\Phi_1)=1$. Note that $\Theta_1(r,k)$ also satisfies $(\calH_1-\jap{k}^2)\Theta_1(r,k)=0$.
\subsubsection{Constructing the outgoing Jost solution}
In this section we construct the Weyl--Titchmarsh functions $\Psi_1(r,k)$ coming from $r=\infty$. 
We proceed via approximation by Hankel functions. We again distinguish between small and large $k$.

\begin{lem}\label{lemWeyl1}
Let $0<k\le 1$ and define  
\begin{equation}\label{eq:Psi0-def}
\Psi_{1,0}(r,k):=e^{i3\pi/4}\sqrt{\frac{\pi r}{2}}\,H_1^{(1)}(kr) \sim k^{-\frac12} e^{ikr}
\end{equation}
as $r\to\infty$. 
Then,  for all $r\ge R_k:=k^{-\frac12}$,
\begin{align}
\Psi_1(r,k)&=\Psi_{1,0}(r,k)\bigl(1+O(e^{-r/2})\bigr),\label{eq:Psi10dom}\\
\Psi_1'(r,k)&=\Psi_{1,0}'(r,k)\bigl(1+O(e^{-r/2})\bigr).\label{eq:Psip10dom}
\end{align}
The $O(e^{-r/2})$ terms are stable under $(k\partial_k)^\ell$, as well as under any fixed number of derivatives in~$r$. 
\end{lem}

\begin{proof}
Write
\begin{equation}\label{eq:Wtail-def}
\Bigl(-\partial_r^2+\frac{3}{4r^2}-W_\infty(r)\Bigr)\Psi_1(r,k)=k^2\Psi_1(r,k),
\qquad
W_\infty(r):=\frac12(1-U(r)^2)-\frac{(1-a_\theta(r))(3-a_\theta(r))}{r^2}.
\end{equation}
By \eqref{Uaexpinfty},
\[
W_\infty(r)=O(r^{-1/2}e^{-r}) \qquad r\to\infty.
\]
Let
\begin{equation}\label{eq:Gk-def}
\begin{aligned}
    G_k(r,s) &:=\frac{\Psi_{1,0}(r,k)\overline{\Psi_{1,0}(s,k)}-\overline{\Psi_{1,0}(r,k)}\Psi_{1,0}(s,k)}{-2i} \\
&=\frac{\pi}{2}\sqrt{rs}\,\Bigl[J_1(kr)Y_1(ks)-Y_1(kr)J_1(ks)\Bigr],
\end{aligned}
\end{equation}
for $R_k\le r\le s$. 
Then $G_k$ is the Green kernel for the index-$1$ Bessel equation, and $\Psi_1$ is given on $[R_k,\infty)$ by the Volterra equation
\begin{equation}\label{eq:Psi-Volterra}
\Psi_1(r,k)=\Psi_{1,0}(r,k)-\int_r^\infty G_k(r,s)W_\infty(s)\Psi_1(s,k)\,d s.
\end{equation}

The standard small- and large-argument bounds for $H_1^{(1)}$, together with the corresponding bounds for $J_1$ and $Y_1$, imply that for $r\ge R_k$
\begin{align}
|\Psi_{1,0}(r,k)|&\simeq k^{-1}\bigl(r^{-1/2}+k^{1/2}\bigr),\label{eq:H1-model-size}\\
|\Psi_{1,0}'(r,k)|&\simeq k^{-1}\bigl(r^{-3/2}+k^{3/2}\bigr),\label{eq:H1-model-der-size}
\end{align}
and for $R_k\le r\le s$
\begin{align}
|G_k(r,s)|&\lesssim k^{-2}(r^{-1/2}+k^{1/2})(s^{-1/2}+k^{1/2}),\label{eq:H1-Gk-size}\\
|\partial_rG_k(r,s)|&\lesssim k^{-2}(r^{-3/2}+k^{3/2})(s^{-1/2}+k^{1/2}).\label{eq:H1-Gk-der-size}
\end{align}
A standard Volterra iteration   now yields
\[
\Psi_1(r,k)=\Psi_{1,0}(r,k)\bigl(1+O(e^{-r/2})\bigr),
\]
which proves \eqref{eq:Psi10dom}.
For the derivative, differentiate \eqref{eq:Psi-Volterra}; because $G_k(r,r)=0$, we obtain
\[
\Psi_1'(r,k)-\Psi_{1,0}'(r,k)
=-
\int_r^\infty \partial_r G_k(r,s)W_\infty(s)\Psi_1(s,k)\,ds.
\]
Using \eqref{eq:H1-model-size}, \eqref{eq:H1-model-der-size}, and \eqref{eq:H1-Gk-der-size} in the same way as above gives~\eqref{eq:Psip10dom}.

Finally, stability of the $O(e^{-r/2})$ errors under $(k\partial_k)^\ell$ and under any fixed number of $r$-derivatives follows by differentiating \eqref{eq:Psi-Volterra}. Every derivative falling on $G_k$, $\Psi_{1,0}$, or $W_\infty$ produces only polynomial factors in $r$ and $k$, and these are absorbed by the exponentially decaying weight $e^{-r/2}$.
\end{proof}

Next, we treat $k\ge1$. In that case we use a different comparison solution, adapted to the turning point at scale $r\sim k^{-1}$. 

\begin{lem}\label{lem:Weyl1-high-k-H2}
For $k\ge1$, define
\[
\Psi_{1,0}(r,k):=e^{i5\pi/4}\sqrt{\frac{\pi r}{2}}\,H_2^{(1)}(kr)\sim k^{-\frac12} e^{ikr}.
\]
As $k\to\infty$, uniformly for $r\ge R_k:=k^{-1/2}$,
\begin{align}
\Psi_1(r,k)&= 
\Psi_{1,0}(r,k)\bigl(1+O(k^{-1})\bigr),\label{eq:Psi1-high-k-approx1}
\\
\Psi_1'(r,k)&= 
\Psi_{1,0}'(r,k)\bigl(1+O(k^{-1})\bigr).\label{eq:Psi1-high-k-approx2}
\end{align}
The $O(k^{-1})$ terms are stable under the $k\partial_k$ and $r\partial_r$ operators. 
\end{lem}

\begin{proof}
By \eqref{eq:W0}, we have
\[
\Bigl(-\partial_r^2+\frac{15}{4r^2}-W_0(r)\Bigr)\Psi_1(r,k)=k^2\Psi_1(r,k).
\]
By \eqref{Uaexp0} and \eqref{Uaexpinfty},
\[
W_0(r)=\frac32+O(r^2),\qquad r\to0,
\]
and
\[
W_0(r)=\frac{3}{r^2}+O(r^{-1/2}e^{-r}),\qquad r\to\infty.
\]
In particular,
\[
|W_0(r)|\lesssim \mathbf 1_{(0,1]}(r)+r^{-2}\mathbf 1_{[1,\infty)}(r),
\qquad
\int_{R_k}^\infty |W_0(s)|\,ds\lesssim 1
\]
uniformly in $k\ge1$.
Since $\sqrt r\,H_2^{(1)}(kr)$ solves
\[
\Bigl(-\partial_r^2+\frac{15}{4r^2}\Bigr)f=k^2f,
\]
variation of constants gives
\begin{equation}\label{eq:Psi1-high-k-Volterra}
\Psi_1(r,k)=\Psi_{1,0}(r,k)-\int_r^\infty G_{k,2}^+(r,s)W_0(s)\Psi_1(s,k)\,ds,
\end{equation}
where
\[
G_{k,2}^+(r,s):=\frac{\Psi_{1,0}(r,k)\overline{\Psi_{1,0}(s,k)}-\overline{\Psi_{1,0}(r,k)}\Psi_{1,0}(s,k)}{-2i}
=\frac{\pi}{2}\sqrt{rs}\,\Bigl[J_2(kr)Y_2(ks)-Y_2(kr)J_2(ks)\Bigr].
\]
For $r\ge R_k$ and $s\ge r$ one has $kr,ks\ge k^{1/2}$, and therefore the large-argument bounds for Bessel and Hankel functions imply
\[
|\Psi_{1,0}(r,k)|\lesssim k^{-1/2},
\qquad
|\Psi_{1,0}'(r,k)|\lesssim k^{1/2},
\]
\[
|G_{k,2}^+(r,s)|\lesssim k^{-1},
\qquad
|\partial_r G_{k,2}^+(r,s)|\lesssim 1.
\]
Applying standard Volterra estimates to \eqref{eq:Psi1-high-k-Volterra} therefore yields
\[
\sup_{r\ge R_k}|\Psi_1(r,k)-\Psi_{1,0}(r,k)|\lesssim
k^{-1}\Bigl(\int_{R_k}^\infty |W_0(s)|\,ds\Bigr)\sup_{s\ge R_k}|\Psi_1(s,k)|
\lesssim k^{-3/2},
\]
and, after differentiating \eqref{eq:Psi1-high-k-Volterra},
\[
\sup_{r\ge R_k}|\Psi_1'(r,k)-\Psi_{1,0}'(r,k)|\lesssim
\Bigl(\int_{R_k}^\infty |W_0(s)|\,ds\Bigr)\sup_{s\ge R_k}|\Psi_1(s,k)|
\lesssim k^{-1/2}.
\]
Since $kr\ge k^{1/2}\to\infty$, the standard large-argument asymptotics also give the lower bounds
\[
|\Psi_{1,0}(r,k)|\gtrsim k^{-1/2},
\qquad
|\Psi_{1,0}'(r,k)|\gtrsim k^{1/2},
\]
uniformly for $r\ge R_k$. Therefore the additive estimates above are equivalent to \eqref{eq:Psi1-high-k-approx1}, \eqref{eq:Psi1-high-k-approx2}. The stability under $k\partial_k$ and $r\partial_r$ is proved by differentiating \eqref{eq:Psi1-high-k-Volterra} exactly as in the small-$k$ case.
\end{proof}

\subsubsection{The spectral measure}
We now derive asymptotics for the spectral measure associated with $\mathcal{H}_1$. The following proposition is the main result of this section.

\begin{prop}\label{SMOp1_1}
The generalized eigenfunctions associated to $\mathcal{H}_1$ are given by
\begin{align}\label{SMOp1_2}
\Phi_1(r,k) = 2\Re (a_1(k) \Psi_1(r,k)).
\end{align}
Here $\Psi_1$ is the Weyl solution at infinity from Lemmas \ref{lemWeyl1} and \ref{lem:Weyl1-high-k-H2}, and $a_1(k)$ is smooth, always non-zero, and satisfies
\begin{align}\label{SMOp1_3}
|a_1(k)| \simeq \dfrac{1}{k \langle k \rangle}.
\end{align}
In fact, we have the sharper statements
\begin{equation}\label{eq:a-smallk}
a_1(k)=c_*\sqrt{\frac{2}{\pi}}\,e^{-3\pi i/4}k^{-1}
\bigl(1+O(k|\log k|)\bigr)
\qquad k\to0
\end{equation}
and
\begin{equation}\label{eq:a-largek}
a_1(k)=4\sqrt{\frac{2}{\pi}}\,e^{-5\pi i/4}k^{-2}\bigl(1+O(k^{-1})\bigr),
\qquad k\to\infty.
\end{equation}
Moreover, $a_1(k)$ satisfies the symbol type bounds
\begin{align}\label{SMOp1_13}
\vert (k \partial_k )^na_1(k)\vert \lesssim \dfrac{1}{k \langle k \rangle}
\end{align}
for every $n\ge0$. The spectral density with respect to the energy variable
$\lambda=k^2$ is given by
\begin{align}\label{SMOp1_4}
 \frac{1}{4\pi |a_1(k)|^2} \simeq k^2\langle k^2\rangle,
\end{align}
with consistent symbol-type upper bounds on the derivatives.
\end{prop}

\begin{proof}
The identities \eqref{SMOp1_2} and \eqref{SMOp1_4} are part of the general Weyl--Titchmarsh theory recalled in \eqref{def_aj}. In particular,
\begin{align}\label{SMOp1_15}
a_1(k) =\dfrac{W(\Phi_1(\cdot,k),\overline{\Psi_1(\cdot,k)})}{W(\Psi_1(\cdot,k),\overline{\Psi_1(\cdot,k)})}= \frac{i}{2} W(\Phi_1(\cdot,k),\overline{\Psi_1(\cdot,k)}).
\end{align}
Smoothness of $a_1$ on $(0,\infty)$ follows from smooth dependence of ODE solutions on the parameter $k$. Moreover, $a_1(k)\neq0$ for all $k>0$: indeed, if $a_1(k_0)=0$, then $W(\Phi_1(\cdot,k_0),\overline{\Psi_1(\cdot,k_0)})=0$, so $\Phi_1(\cdot,k_0)$ and $\overline{\Psi_1(\cdot,k_0)}$ are linearly dependent. Since $\Phi_1(\cdot,k_0)$ is real-valued, this would force $\Psi_1(\cdot,k_0)$ and $\overline{\Psi_1(\cdot,k_0)}$ to be linearly dependent, contradicting the normalization $W(\overline{\Psi_1},\Psi_1)=2i$.

We first treat $k\to0$. Set $R_k:=k^{-1/2}$. Since $kR_k=k^{1/2}\to0$, the small-argument asymptotics of $H_1^{(1)}$ give
\[
H_1^{(1)}(z)=-\frac{2i}{\pi z}\bigl(1+O(z^2|\log z|)\bigr),
\qquad
\frac{d}{dz}H_1^{(1)}(z)=\frac{2i}{\pi z^2}\bigl(1+O(z^2|\log z|)\bigr),
\]
as $z\to0$. Therefore, \eqref{eq:Psi10dom} and \eqref{eq:Psip10dom} imply
\begin{align*}
\overline{\Psi_1(R_k,k)}
&=\sqrt{\frac{2}{\pi}}\,e^{-i\pi/4}k^{-1}R_k^{-1/2}
\bigl(1+O(k|\log k|)\bigr),\\
\partial_r\overline{\Psi_1(R_k,k)}
&=-\frac12\sqrt{\frac{2}{\pi}}\,e^{-i\pi/4}k^{-1}R_k^{-3/2}
\bigl(1+O(k|\log k|)\bigr).
\end{align*}
Here the error coming from Lemma~\ref{lemWeyl1} is $O(e^{-R_k/2})=O(e^{-1/(2\sqrt{k})})$, hence negligible compared to $O(k|\log k|)$.
On the other hand, \eqref{Phi1Op1_2}, \eqref{Phi1_0}, \eqref{Phi1Op1_3}, and \eqref{Phi1Op1_4} yield
\begin{align*}
\Phi_1(R_k,k)&=\Phi_1^{(0)}(R_k)+R_k^{3/2}\sum_{j\ge1}(kR_k)^{2j}\Phi_{1,j}(R_k)
=c_*R_k^{3/2}\bigl(1+O(k)\bigr),\\
\Phi_1'(R_k,k)&
=\frac32c_*R_k^{1/2}\bigl(1+O(k)\bigr).
\end{align*}
Therefore,
\begin{align*}
W\bigl(\Phi_1(\cdot,k),\overline{\Psi_1(\cdot,k)}\bigr)
&=\Phi_1(R_k,k)\,\partial_r\overline{\Psi_1(R_k,k)}
-\Phi_1'(R_k,k)\,\overline{\Psi_1(R_k,k)}\\
&=-2c_*\sqrt{\frac{2}{\pi}}\,e^{-i\pi/4}k^{-1}
\bigl(1+O(k|\log k|)\bigr).
\end{align*}
Using \eqref{SMOp1_15}, we obtain
\[
a_1(k)=\frac{i}{2}W\bigl(\Phi_1(\cdot,k),\overline{\Psi_1(\cdot,k)}\bigr)
=c_*\sqrt{\frac{2}{\pi}}\,e^{-3\pi i/4}k^{-1}
\bigl(1+O(k|\log k|)\bigr),
\]
which is \eqref{eq:a-smallk}.

We now turn to $k\to\infty$. Again set $R_k:=k^{-1/2}$. Then $R_k<1$ and $kR_k=k^{1/2}\to\infty$. By Lemma~\ref{lem:Phi1-high-k-J2},
\[
\Phi_1(R_k,k)=\Phi_{1,0}(R_k,k)+O(k^{-7/2}),
\qquad
\Phi_1'(R_k,k)=\Phi_{1,0}'(R_k,k)+O(k^{-5/2}).
\]
Similarly, Lemma~\ref{lem:Weyl1-high-k-H2} gives
\[
\Psi_1(R_k,k)=\Psi_{1,0}(R_k,k)+O(k^{-3/2}),
\qquad
\Psi_1'(R_k,k)=\Psi_{1,0}'(R_k,k)+O(k^{-1/2}).
\]
By the standard large-argument bounds we have
\[
|\Phi_{1,0}(R_k,k)|\lesssim k^{-5/2},\qquad |\Phi_{1,0}'(R_k,k)|\lesssim k^{-3/2},
\]
\[
|\Psi_{1,0}(R_k,k)|\lesssim k^{-1/2},\qquad |\Psi_{1,0}'(R_k,k)|\lesssim k^{1/2}.
\]
Consequently,
\[
W\bigl(\Phi_1(\cdot,k),\overline{\Psi_1(\cdot,k)}\bigr)
=
W\bigl(\Phi_{1,0}(\cdot,k),\overline{\Psi_{1,0}(\cdot,k)}\bigr)+O(k^{-3}).
\]
Since
\[
W\bigl(\sqrt r\,J_2(kr),\sqrt r\,H_2^{(2)}(kr)\bigr)= -\frac{2i}{\pi},
\]
we obtain
\begin{align*}
W\bigl(\Phi_{1,0}(\cdot,k),\overline{\Psi_{1,0}(\cdot,k)}\bigr)
&=8k^{-2}e^{-i5\pi/4}\sqrt{\frac{\pi}{2}}
W\bigl(\sqrt r\,J_2(kr),\sqrt r\,H_2^{(2)}(kr)\bigr)
\\
&=-8i\sqrt{\frac{2}{\pi}}\,e^{-i5\pi/4}k^{-2}. 
\end{align*}
Thus
\[
a_1(k)=\frac{i}{2}W\bigl(\Phi_1(\cdot,k),\overline{\Psi_1(\cdot,k)}\bigr)
=
4\sqrt{\frac{2}{\pi}}\,e^{-i5\pi/4}k^{-2}\bigl(1+O(k^{-1})\bigr),
\]
which is \eqref{eq:a-largek}.
 We leave the symbol-type bounds to the reader. 
\end{proof}
Formula \eqref{SMOp1_2} allows us to express $\Phi_1$ in terms of $\Psi_1$ and $\overline{\Psi_1}$. Conversely, $\Psi_1$ can be expressed in terms of $\Phi_1$ and $\Theta_1$. The precise formula is recorded in the following lemma.
\begin{lemma}\label{lem:Psi1toTheta1Phi1}
    The Weyl solution at infinity $\Psi_1$ satisfies
    \begin{equation}\label{eq:Psi1toPhi1Theta1}
        \Psi_1(r,k)=W(\Theta_1,\Psi_1)\Phi_1(r,k)+W(\Psi_1,\Phi_1)\Theta_1(r,k).
    \end{equation}
    Moreover, 
    \begin{equation}\label{eq:Psi1Theta1Wornskian1}
        |W(\Theta_1,\Psi_1)|\simeq k\jap{k},\qquad |W(\Psi_1,\Phi_1)|\simeq \frac{1}{k\jap{k}}.
    \end{equation}
\end{lemma}
\begin{proof}
    The identity \eqref{eq:Psi1toPhi1Theta1} follows immediately from the facts that $\Phi_1$ and $\Theta_1$ form a fundamental system for $\calH_1-\jap{k}^2$ and  $W(\Theta_1,\Phi_1)=1$. The estimates \eqref{eq:Psi1Theta1Wornskian1} on the Wronskians are proved as in Proposition~\ref{SMOp1_1} and using also the relation~\eqref{eq:Theta1varpar1}. Indeed, the only new estimate is for $W(\Theta_1,\Psi_1)$ where we use~\eqref{eq:Theta1varpar1} to conclude that $|W(\Theta_1,\Psi_1)(k)|=|\Psi_1(k^{-1},k)||\Phi_1(k^{-1},k)|^{-1}$. This gives the desired estimate.
\end{proof}

\subsection{Linear spectral analysis for \texorpdfstring{$\mathcal{H}_2$}{H2}}\label{subsec:calH2}

Recall that $\mathcal{H}_2 = r^{1/2} \mathcal{L}_2 r^{-1/2}$ can be written as
\begin{align*}
\mathcal{H}_2 := -\partial_r^2-\tfrac{1}{4r^2}+U^2(r).
\end{align*}
We seek to construct generalized eigenfunctions of
\begin{align}\label{L2evalue}
\mathcal{H}_2 f = (k^2 + 1)f, \quad k \geq 0.
\end{align}
Unlike in the case of $\mathcal{H}_1$, there is no turning point.  After moving the threshold to the right-hand side, the spectral equation is a short-range perturbation of the exact Bessel operator
\[
-\partial_r^2-\frac{1}{4r^2}.
\]
Accordingly, the same model solutions built from $J_0$, $Y_0$, and $H_0^{(1)}$ will be used both in the threshold regime $k\to0$ and in the high-energy regime $k\to\infty$.

We set
\begin{equation}\label{eq:W2-def}
W_2(r):=1-U^2(r).
\end{equation}
Then \eqref{Uaexp0} and \eqref{Uaexpinfty} imply
\begin{equation}\label{eq:W2-asymp}
W_2(r)=1+O(r^2),\qquad r\to0,
\qquad\text{and}\qquad
W_2(r)=O(r^{-1/2}e^{-r}),\qquad r\to\infty.
\end{equation}
Hence \eqref{L2evalue} is equivalent to
\begin{equation}\label{eq:H2-shifted}
-y''(r)-\frac{1}{4r^2}y(r)-W_2(r)y(r)=k^2 y(r).
\end{equation}

\subsubsection{Fundamental solutions at zero energy}
We look for $\Phi_2^{(0)}$ solving $\mathcal{H}_2 \Phi_2^{(0)} = \Phi_2^{(0)}$, and we expect that $\Phi_2^{(0)}(r) \approx r^{1/2}$ for small $r$. The equation is approximated by
\[
-y''(r)-\frac{1}{4r^2}y(r)=0,
\]
for which a fundamental system is given by
\begin{align}\label{fsys_h2_inf_app}
y_1(r)=r^{1/2}, \qquad y_2(r)=r^{1/2}\log(r).
\end{align}
We choose the normalization so that $W[y_1,y_2](r)=1$.

\begin{lem}\label{lem_p2_t2}
There exists a fundamental system of real-valued solutions of $\mathcal{H}_2f=f$ with the following asymptotic behavior:
\begin{align}\label{Phi2_0}
\Phi_{2}^{(0)} (r) & = \begin{cases}
r^{1/2}+O(r^{5/2}), & r\lesssim 1,
\\ \widetilde{c}_{1,*}r^{1/2}\log(r)+\widetilde{c}_{2,*}r^{1/2}+O(e^{-r}\log(r)),  &  r \gtrsim 1,
\end{cases}
\\ \Theta_{2}^{(0)} (r) & = \begin{cases}
- r^{1/2}\log(r) + O(r^{5/2}\log(r)), & r\lesssim 1,
\\ \widetilde{c}_{3,*}r^{1/2}\log(r)+\widetilde{c}_{4,*}r^{1/2}+O(e^{-r}\log(r)), &   r \gtrsim 1.
\end{cases} \label{Theta2_0}
\end{align}
Here $\widetilde{c}_{1,*}\widetilde{c}_{4,*}-\widetilde{c}_{2,*}\widetilde{c}_{3,*}=1$ and $\widetilde{c}_{1,*}\neq0$. Moreover, $\Phi_2^{(0)}$ and $\Theta_2^{(0)}$ are normalized so that $W[\Theta_{2}^{(0)},\Phi_{2}^{(0)}]=1$.
\end{lem}

\begin{proof}
At zero energy, \eqref{eq:H2-shifted} becomes
\begin{equation}\label{L2_eqrewri}
-y''(r)-\frac{1}{4r^2}y(r)=W_2(r)y(r).
\end{equation}
By \eqref{eq:W2-asymp}, $W_2(r)=O(r^{-1/2}e^{-r})$ for $r\gtrsim1$. Using the comparison system \eqref{fsys_h2_inf_app} and the identity $W[y_1,y_2]=1$, we define
\begin{align*}
\Phi_{2,\infty}^{(0)}(r) & = r^{1/2} +\int_r^\infty \Bigl(r^{1/2}\log(r)s^{1/2}-r^{1/2}s^{1/2}\log(s)\Bigr)W_2(s)\Phi_{2,\infty}^{(0)}(s)\,ds,
\\
\Theta_{2,\infty}^{(0)}(r)& = r^{1/2}\log(r)+\int_r^\infty \Bigl(r^{1/2}\log(r)s^{1/2}-r^{1/2}s^{1/2}\log(s)\Bigr)W_2(s)\Theta_{2,\infty}^{(0)}(s)\,ds.
\end{align*}
Standard Volterra theory gives a fundamental system of real-valued solutions satisfying
\begin{align}\label{FS0infty_H2}
\Phi_{2,\infty}^{(0)}(r) & = r^{1/2}+O(e^{-r}),
\\
\Theta_{2,\infty}^{(0)}(r) & = r^{1/2}\log(r)+O(e^{-r}\log r).
\end{align}
Near $r=0$, \eqref{L2_eqrewri} is a bounded perturbation of
\[
-y''(r)-\frac{1}{4r^2}y(r)=0,
\]
whose fundamental solutions are $r^{1/2}$ and $r^{1/2}\log r$. Since $\mathcal H_2$ is the radial form of the selfadjoint operator $-\Delta+U^2$ on $\mathbb R^2$, the admissible Weyl solution at the origin is the one behaving like $r^{1/2}$; the logarithmic branch is excluded by the domain condition. Thus there exists a real-valued solution $\Phi_2^{(0)}$ with
\[
\Phi_2^{(0)}(r)=r^{1/2}+O(r^{5/2}),\qquad r\to0.
\]
Extending this solution to $(0,\infty)$ and comparing with \eqref{FS0infty_H2}, we obtain
\begin{align}\label{asymp_phi_2_0_H2}
\Phi_2^{(0)}(r)=\begin{cases}
r^{1/2}+O(r^{5/2}), & r\ll1,
\\ \widetilde{c}_{1,*}r^{1/2}\log(r)+\widetilde{c}_{2,*}r^{1/2}+O\big( e^{-r}\log(r)\big), & r\gg 1.
\end{cases}
\end{align}
The fact that $\widetilde{c}_{1,*}\neq0$ is exactly \cite[Lemma~6.3]{LPPSS1PartI}.

Finally, define $\Theta_2^{(0)}$ by the condition $W[\Theta_2^{(0)},\Phi_2^{(0)}]=1$. Near $r=0$ this gives
\[
\Theta_2^{(0)}(r)=r^{1/2}\log r+O(r^{5/2}\log r),
\]
while at infinity $\Theta_2^{(0)}$ is a real linear combination of the two solutions in \eqref{FS0infty_H2}, with coefficients $\widetilde c_{3,*},\widetilde c_{4,*}$ satisfying
\[
\widetilde{c}_{1,*}\widetilde{c}_{4,*}-\widetilde{c}_{2,*}\widetilde{c}_{3,*}=1.
\]
This proves the lemma.
\end{proof}

\begin{rem}
As noted, $\calH_2$ is limit-circle at $0$, so one cannot select the Weyl solution there by imposing an $L^2$ condition alone.  Instead, the selfadjointness of $\calL_2$ on $\mathbb R^2$ selects the branch $\Phi_2^{(0)}(r)\sim r^{1/2}$; the logarithmic branch is not in the operator domain.  This is the analogue, for $\mathcal H_2$, of the regular boundary condition at the origin.
\end{rem}

\subsubsection{Solving from $r=0$}
Using the fundamental system at zero energy, we first construct the regular solution $\Phi_2(r,k)$ in the regime $kr\lesssim1$.

\begin{lem}\label{lemFou2}
For all $r>0$ and $k \geq 0$ we have the expansion
\begin{align}\label{Fou21}
\Phi_2(r,k) = \Phi_2^{(0)}(r) + \sqrt{r} \sum_{j\geq 1} (k r)^{2j} \Phi_{2,j}(r),
\end{align}
which converges absolutely. The expansion converges uniformly if $k r \leq 1$, and $\Phi_{2,j}$ are smooth and satisfy, for some absolute $C>0$ and all $j\geq 1$,
\begin{align}\label{Fou22}
\begin{split}
& | \Phi_{2,j}(u) | \leq \frac{C^j}{j! }, \quad u \lesssim 1,
\\
& | \Phi_{2,j}(u) | \leq \frac{C^j}{j!} \log(u), \quad u \gtrsim 1.
\end{split}
\end{align}
Moreover,
\begin{align}\label{Fou23}
\Phi_{2,1}(u) = \left\{
\begin{array}{ll}
-\dfrac{1}{4} + O(u^2), &  u \lesssim  1,
\\[1ex]
-\dfrac{1}{4}\Big(\widetilde{c}_{1,*}\log(u)-\widetilde{c}_{1,*}+\widetilde{c}_{2,*}\Big) + O(u^{-2} \,\log u), &  u \gtrsim 1,
\end{array}
\right.
\end{align}
and in \eqref{Fou22}--\eqref{Fou23} we also have consistent symbol-type bounds for the derivatives, namely
\begin{align}\label{Fou24}
 | (u\partial_u)^m\Phi_{2,j}(u) | \leq \frac{C_m^j}{j! } u \langle u\rangle^{-1},
\end{align}
for all $m\ge1$.
\end{lem}

\begin{proof}
We seek $\Phi_2(r,k)$ solving $\mathcal H_2\Phi_2(r,k)=(1+k^2)\Phi_2(r,k)$ with the regular behavior at $r=0$.  We use the ansatz
\[
\Phi_2(r,k) = \sqrt{r}\sum_{j\ge0} k^{2j} f_j(r), \qquad f_0(r):=r^{-1/2}\Phi_2^{(0)}(r).
\]
Substituting into the equation yields the recursion
\begin{align}\label{Fou2rec}
(\mathcal{H}_2-1) \big(\sqrt{r}\, f_j(r)\big) = \sqrt{r}\, f_{j-1}(r), \quad j\geq 0, \qquad f_{-1}(r)=0.
\end{align}
Comparing with \eqref{Fou21} gives $f_j(r)=r^{2j}\Phi_{2,j}(r)$.

Let
\[
K_0(r,s):=\Theta_2^{(0)}(s)\Phi_2^{(0)}(r)-\Theta_2^{(0)}(r)\Phi_2^{(0)}(s), \qquad 0<s\le r.
\]
Since $W[\Theta_2^{(0)},\Phi_2^{(0)}]=1$, variation of constants applied to \eqref{Fou2rec} yields
\begin{equation}\label{eq:Fou2-rec-int}
\sqrt r\,f_j(r)=-\int_0^r K_0(r,s)\sqrt s\, f_{j-1}(s)\,ds, \qquad j\ge1.
\end{equation}

We first analyze $f_1$.  For $r\lesssim1$, Lemma~\ref{lem_p2_t2} gives
\[
K_0(r,s)=\sqrt{rs}\log\Big(\frac{r}{s}\Big)\Bigl(1+O(r^2+s^2)\Bigr),
\]
and therefore
\begin{align*}
f_1(r)
&=-\int_0^r s\log\Big(\frac{r}{s}\Big)\Bigl(1+O(r^2+s^2)\Bigr)\,ds
= -\frac14 r^2+O(r^4).
\end{align*}
Since $f_1(r)=r^2\Phi_{2,1}(r)$, this proves the first asymptotic in \eqref{Fou23}.

For $r\gtrsim1$, we use \eqref{Phi2_0}--\eqref{Theta2_0} and the identity
\[
\widetilde c_{1,*}\widetilde c_{4,*}-\widetilde c_{2,*}\widetilde c_{3,*}=1
\]
to obtain
\[
\frac{1}{\sqrt r}K_0(r,s)=s\log\Big(\frac{r}{s}\Big)+O\big(\sqrt{s}\,e^{-s}\log(r)\log(s)\big), \qquad 1\le s\le r.
\]
Since
\[
f_0(s)=\widetilde c_{1,*}\log s+\widetilde c_{2,*}+O(s^{-1/2}e^{-s}\log s),
\]
we infer from \eqref{eq:Fou2-rec-int} that
\begin{align*}
f_1(r)
&=-\int_1^r s\log\Big(\frac{r}{s}\Big)\Bigl(\widetilde c_{1,*}\log s+\widetilde c_{2,*}\Bigr)\,ds
+O(\log r)
\\
&=-\frac{r^2}{4}\Bigl(\widetilde c_{1,*}\log(r)-\widetilde c_{1,*}+\widetilde c_{2,*}\Bigr)+O(\log r),
\end{align*}
which yields the second asymptotic in \eqref{Fou23}.

We next prove \eqref{Fou22} by induction.  For $r\lesssim1$, assume
\[
|f_m(r)|\le \frac{C^m}{m!} r^{2m}, \qquad m\le j-1.
\]
Using again the kernel expansion near $0$ and \eqref{eq:Fou2-rec-int}, we obtain
\begin{align*}
|f_j(r)|
&\lesssim \int_0^r s\log\Big(\frac{r}{s}\Big)\frac{C^{j-1}}{(j-1)!}s^{2j-2}\,ds
\lesssim \frac{C^j}{j!}r^{2j},
\end{align*}
for $C$ large enough.  Dividing by $r^{2j}$ gives the first line in \eqref{Fou22}.

For $r\gtrsim1$, assume
\[
|f_m(r)|\le \frac{C^m}{m!} r^{2m}\log r, \qquad m\le j-1.
\]
Using the large-$r$ kernel asymptotic displayed above, \eqref{eq:Fou2-rec-int} implies
\begin{align*}
|f_j(r)|
&\lesssim \int_1^r s\log\Big(\frac{r}{s}\Big)\frac{C^{j-1}}{(j-1)!}s^{2j-2}\log s\,ds
+O(1)
\\
&\lesssim \frac{C^j}{j!} r^{2j}\log r.
\end{align*}
Dividing by $r^{2j}$ proves the second line in \eqref{Fou22}.  Absolute convergence of \eqref{Fou21} follows immediately from the factorial bounds.

Finally, the derivative estimates \eqref{Fou24} follow by differentiating the recursion \eqref{eq:Fou2-rec-int} and arguing exactly as in the proof of Lemma~\ref{Phi1Op1_1}; we omit the routine details.
\end{proof}

For the high-energy connection problem we also need a second representation of the regular solution, now built from $J_0$. In analogy with our analysis of~$\calH_1$, for $k\ge1$ we define 
\begin{equation}\label{eq:Phi20-J0}
\Phi_{2,0}(r,k):=\sqrt r\,J_0(kr)
\end{equation}
and the Green kernel
\begin{equation}\label{eq:Gk0-reg}
G_{k,0}(r,s):=\frac{\pi}{2}\sqrt{rs}\,\Bigl[J_0(kr)Y_0(ks)-Y_0(kr)J_0(ks)\Bigr],
\qquad 0<s\le r\le1.
\end{equation}
Then the Weyl solution at zero satisfies on $(0,1]$ the Volterra equation
\begin{equation}\label{eq:Phi2-high-k-Volterra}
\Phi_2(r,k)=\Phi_{2,0}(r,k)+\int_0^r G_{k,0}(r,s)W_2(s)\Phi_2(s,k)\,ds.
\end{equation}

\begin{lem}\label{lem:Phi2-high-k-J0}
One has
\begin{align}
&\sup_{0<r\le1}\omega(kr)^{-1}\Bigl|\Phi_2(r,k)-\Phi_{2,0}(r,k)\Bigr|\lesssim k^{-3/2},\label{eq:Phi2-high-k-bd1}\\
&\sup_{0<r\le1}\Bigl|\partial_r\Phi_2(r,k)-\partial_r\Phi_{2,0}(r,k)\Bigr|\lesssim k^{-1/2}\label{eq:Phi2-high-k-bd1-der},
\end{align}
with
\[
\omega(x):=
\begin{cases}
x^{1/2}, & 0<x\le1,
\\
1, & x\ge1. 
\end{cases}
\]
\end{lem}

\begin{proof}
The functions $\sqrt r\,J_0(kr)$ and $\sqrt r\,Y_0(kr)$ solve
\[
\Bigl(-\partial_r^2-\frac{1}{4r^2}\Bigr)f=k^2f,
\]
with Wronskian
\[
W\bigl(\sqrt r\,J_0(kr),\sqrt r\,Y_0(kr)\bigr)=\frac{2}{\pi}.
\]
Hence \eqref{eq:Gk0-reg} is the backwards Green's kernel and \eqref{eq:Phi2-high-k-Volterra} holds.
To estimate the Volterra equation, set $x:=kr$ and
\[
\phi(x,k):=k^{1/2}\Phi_2(x/k,k), \qquad 0<x\le k.
\]
Then \eqref{eq:Phi2-high-k-Volterra} becomes
\begin{equation}\label{eq:phi2-scaled-Volterra}
\phi(x,k)=\sqrt x\,J_0(x)+k^{-2}\int_0^x \mathcal G_0(x,y)W_2(y/k)\phi(y,k)\,dy,
\end{equation}
where
\[
\mathcal G_0(x,y):=\frac{\pi}{2}\sqrt{xy}\,\Bigl[J_0(x)Y_0(y)-Y_0(x)J_0(y)\Bigr],
\qquad 0<y\le x\le k.
\]
Using the standard bounds
\[
|J_0(z)|\lesssim
\begin{cases}
1,&0<z\le1,
\\
z^{-1/2},&z\ge1,
\end{cases}
\qquad
|Y_0(z)|\lesssim
\begin{cases}
1+|\log z|,&0<z\le1,
\\
z^{-1/2},&z\ge1,
\end{cases}
\]
we obtain
\[
|\mathcal G_0(x,y)|\lesssim
\begin{cases}
\sqrt{xy}\,\log(x/y),&0<y\le x\le1,
\\
\sqrt y\,(1+|\log y|),&0<y\le1\le x,
\\
1,&1\le y\le x.
\end{cases}
\]
Since $W_2$ is bounded on $[0,1]$, these kernel bounds imply
\[
\omega(x)^{-1}\int_0^x |\mathcal G_0(x,y)|\,\omega(y)\,dy\lesssim 1+x\lesssim k,
\qquad 0<x\le k.
\]
Therefore the Volterra operator on the right-hand side of \eqref{eq:phi2-scaled-Volterra} has norm $O(k^{-1})$ on
\[
X:=\Bigl\{f\in C((0,k]):\ \|f\|_X:=\sup_{0<x\le k}\omega(x)^{-1}|f(x)|<\infty\Bigr\},
\]
and hence
\[
\|\phi(\cdot,k)-\sqrt{\cdot}\,J_0\|_X\lesssim k^{-1}.
\]
Rescaling back to $r=x/k$ gives \eqref{eq:Phi2-high-k-bd1}.
 Stability under $k\partial_k$ and \eqref{eq:Phi2-high-k-bd1-der} follow by differentiating \eqref{eq:phi2-scaled-Volterra} and using similar bounds for $\partial_x\calG_0$.
\end{proof}
As for $\Theta_1$ in \eqref{eq:Theta1varpar1}, we let
\begin{equation}\label{eq:Theta2varpar1}
    \Theta_2(r,k):= \Phi_2(r,k)\int_r^{1/k}\frac{\ud s}{(\Phi_2(s,k))^2},
\end{equation}
where the normalization is chosen so that $W(\Theta_2,\Phi_2)=1$. Note that $\Theta_2(r,k)$ also satisfies $(\calH_2-\jap{k}^2)\Theta_2(r,k)=0$.
\subsubsection{Solving from $r=\infty$}
We now discuss the outgoing solution at infinity. Let $0<k\le1$ and define
\begin{equation}\label{eq:Psi20-def}
\Psi_{2,0}(r,k):=e^{i\pi/4}\sqrt{\frac{\pi r}{2}}\,H_0^{(1)}(kr).
\end{equation}
Set
\begin{equation}\label{eq:Rk-H2-small}
R_k:=k^{-1}|\log k|^{-2}.
\end{equation}
Then, for $r\ge R_k$, the Weyl solution at infinity satisfies the Volterra equation
\begin{equation}\label{eq:Psi2-low-k-Volterra}
\Psi_2(r,k)=\Psi_{2,0}(r,k)-\int_r^\infty G_{k,0}^+(r,s)W_2(s)\Psi_2(s,k)\,ds,
\end{equation}
where
\begin{align*}
G_{k,0}^+(r,s) &:=\frac{\Psi_{2,0}(r,k)\overline{\Psi_{2,0}(s,k)}-\overline{\Psi_{2,0}(r,k)}\Psi_{2,0}(s,k)}{-2i}\\
&=\frac{\pi}{2}\sqrt{rs}\,\Bigl[J_0(kr)Y_0(ks)-Y_0(kr)J_0(ks)\Bigr].
\end{align*}

\begin{lem}\label{lem:Weyl2-low-k-H0}
As $k\to0$, uniformly for $r\ge R_k$,
\begin{align}
\Psi_2(r,k)&=\Psi_{2,0}(r,k)+O\!\big(e^{-r/2}\sqrt r\,\langle\log(kr)\rangle\big),\label{eq:Psi2-low-k-H0-1}
\\
\Psi_2'(r,k)&=\Psi_{2,0}'(r,k)+O\!\big(e^{-r/2}r^{-1/2}\langle\log(kr)\rangle\big),\label{eq:Psi2-low-k-H0-2}
\end{align}
and the $O$-terms are stable under finitely many $k\partial_k$ derivatives.  In particular, at $r=R_k$,
\begin{align}
\Psi_2(R_k,k)&=\Psi_{2,0}(R_k,k)\bigl(1+O(e^{-R_k/2})\bigr),\label{eq:Psi2-low-k-atRk-1}
\\
\Psi_2'(R_k,k)&=\Psi_{2,0}'(R_k,k)\bigl(1+O(e^{-R_k/2})\bigr).\label{eq:Psi2-low-k-atRk-2}
\end{align}
\end{lem}

\begin{proof}
Since $\sqrt r\,H_0^{(1)}(kr)$ solves
\[
\Bigl(-\partial_r^2-\frac{1}{4r^2}\Bigr)f=k^2f,
\]
variation of constants gives \eqref{eq:Psi2-low-k-Volterra}.  For $r\ge R_k$ and $s\ge r$, the standard bounds for $J_0$ and $Y_0$ imply
\[
|G_{k,0}^+(r,s)|\lesssim \sqrt{rs}\,\langle\log(kr)\rangle \langle\log(ks)\rangle,
\]
and similarly
\[
|\partial_r G_{k,0}^+(r,s)|\lesssim r^{-1/2}s^{1/2}\,\langle\log(kr)\rangle \langle\log(ks)\rangle.
\]
Together with $W_2(s)=O(s^{-1/2}e^{-s})$, this yields
\[
\int_r^\infty |G_{k,0}^+(r,s)|\,|W_2(s)|\,ds \lesssim e^{-r/2}\sqrt r\,\langle\log(kr)\rangle,
\]
and
\[
\int_r^\infty |\partial_r G_{k,0}^+(r,s)|\,|W_2(s)|\,ds \lesssim e^{-r/2}r^{-1/2}\,\langle\log(kr)\rangle.
\]
Standard Volterra theory therefore gives \eqref{eq:Psi2-low-k-H0-1}--\eqref{eq:Psi2-low-k-H0-2}.  Since at $r=R_k$ the model terms satisfy
\[
|\Psi_{2,0}(R_k,k)|\simeq \sqrt{R_k}\,\langle\log(kR_k)\rangle,
\qquad
|\Psi_{2,0}'(R_k,k)|\simeq R_k^{-1/2}\,\langle\log(kR_k)\rangle,
\]
the relative forms \eqref{eq:Psi2-low-k-atRk-1}, \eqref{eq:Psi2-low-k-atRk-2} follow.  Stability under $k\partial_k$ is obtained by differentiating \eqref{eq:Psi2-low-k-Volterra}.
\end{proof}

Finally, for the high-energy connection problem we solve from infinity using the same Hankel model, now at the fixed matching point $r=1$.   

\begin{lem}\label{lem:Weyl2-high-k-H0}
Let $\Psi_{2,0}$ be defined by \eqref{eq:Psi20-def}. As $k\to\infty$,
\begin{align}
\sup_{r\ge1}\bigl|\Psi_2(r,k)-\Psi_{2,0}(r,k)\bigr|&\lesssim k^{-3/2},\label{eq:Psi2-high-k-H0-1}
\\
\sup_{r\ge1}\bigl|\Psi_2'(r,k)-\Psi_{2,0}'(r,k)\bigr|&\lesssim k^{-1/2},\label{eq:Psi2-high-k-H0-2}
\end{align}
and the $O$-terms are stable under finitely many $k\partial_k$ and $r\partial_r$ derivatives.
\end{lem}

\begin{proof}
For $r\ge1$ and $s\ge r$ one has $kr,ks\ge k$, so the large-argument bounds for Hankel and Bessel functions imply
\[
|\Psi_{2,0}(r,k)|\lesssim k^{-1/2},
\qquad
|\Psi_{2,0}'(r,k)|\lesssim k^{1/2},
\]
\[
|G_{k,0}^+(r,s)|\lesssim k^{-1},
\qquad
|\partial_r G_{k,0}^+(r,s)|\lesssim 1.
\]
Since $W_2\in L^1([1,\infty))$, standard Volterra estimates applied to \eqref{eq:Psi2-low-k-Volterra} yield
\[
\sup_{r\ge1}|\Psi_2(r,k)-\Psi_{2,0}(r,k)|\lesssim
k^{-1}\Bigl(\int_1^\infty |W_2(s)|\,ds\Bigr)\sup_{s\ge1}|\Psi_2(s,k)|
\lesssim k^{-3/2},
\]
and, after differentiating \eqref{eq:Psi2-low-k-Volterra},
\[
\sup_{r\ge1}|\Psi_2'(r,k)-\Psi_{2,0}'(r,k)|\lesssim
\Bigl(\int_1^\infty |W_2(s)|\,ds\Bigr)\sup_{s\ge1}|\Psi_2(s,k)|
\lesssim k^{-1/2}.
\]
The derivative stability follows as before.
\end{proof}

\subsubsection{The spectral measure}\label{sec_specm_2}
We now derive asymptotics for the spectral measure associated to $\mathcal{H}_2$.  Recall that
\[
W\big(\Theta_2(r,k),\Phi_2(r,k)\big)=1
\qquad \text{and} \qquad
W\big(\Psi_2(r,k),\overline{\Psi_2(r,k)}\big)=-2i,
\]
the first identity by definition and the second by the normalization of the Weyl solution at infinity.

\begin{prop}\label{propSM2}
The generalized eigenfunctions associated with $\mathcal{H}_2$ are given by
\begin{align}\label{propSM0}
\Phi_2(r,k) = 2\Re \big( a_2(k) \Psi_2(r,k) \big).
\end{align}
Here $a_2(k)$ is smooth, always non-zero, and satisfies
\begin{align}\label{propSMa}
|a_2(k)| \simeq \left\{
\begin{array}{ll}
\jap{\log(k)}, &  k \lesssim 1,
\\ 1, & k \gtrsim 1.
\end{array} \right.
\end{align}
In fact, we have the sharper statements
\begin{equation}\label{eq:a2-smallk}
a_2(k)=\frac{e^{-i\pi/4}}{\sqrt{2\pi}}
\Biggl[
\widetilde c_{2,*}-\widetilde c_{1,*}\Bigl(\log\frac{k}{2}+\gamma+i\frac{\pi}{2}\Bigr)
\Biggr]
+O\bigl(|\log k|^{-1}\bigr),
\qquad k\to0,
\end{equation}
and
\begin{equation}\label{eq:a2-largek}
a_2(k)=\frac{e^{-i\pi/4}}{\sqrt{2\pi}}\bigl(1+O(k^{-1})\bigr),
\qquad k\to\infty.
\end{equation}
Moreover, for every $n\ge1$,
\begin{align}\label{propSMa_dk}
\vert (k \partial_k)^na_2(k)\vert \lesssim_n 1.
\end{align}
The spectral density with respect to the energy variable
$\lambda=k^2$ is given by
\begin{align}\label{propSMrho}
 \frac{1}{4\pi |a_2(k)|^2} \simeq \left\{
\begin{array}{ll}
\jap{\log(k)}^{-2} , &  k \lesssim 1,
\\ 1, & k \gtrsim 1.
\end{array} \right.
\end{align}
In fact,
\[
\frac{1}{4\pi |a_2(k)|^2}=\frac{1}{2\widetilde c_{1,*}^2|\log k|^2}\Bigl(1+O(|\log k|^{-1})\Bigr),
\qquad k\to0,
\]
and
\[
\frac{1}{4\pi |a_2(k)|^2}=\frac12\bigl(1+O(k^{-1})\bigr),
\qquad k\to\infty.
\]
\end{prop}

\begin{proof}
Since $(\Psi_2,\overline{\Psi_2})$ is a fundamental system at infinity and $\Phi_2$ is real-valued, there exists a unique coefficient $a_2(k)$ such that \eqref{propSM0} holds.  Moreover,
\begin{align}\label{SMW_a2}
a_2(k) := \dfrac{W(\Phi_2(\cdot,k),\overline{\Psi_2(\cdot,k)})}{W(\Psi_2(\cdot,k),\overline{\Psi_2(\cdot,k)})}
=\dfrac{i}{2}W(\Phi_2(\cdot,k),\overline{\Psi_2(\cdot,k)}).
\end{align}

We first consider $k\to0$.  Let $R_k$ be as in \eqref{eq:Rk-H2-small}. Since $kR_k=|\log k|^{-2}\to0$ and $R_k\to\infty$, Lemma~\ref{lemFou2} and \eqref{Phi2_0} give
\begin{align}
\Phi_2(R_k,k)
&=\sqrt{R_k}\Bigl(\widetilde c_{1,*}\log R_k+\widetilde c_{2,*}\Bigr)+O\bigl(\sqrt{R_k}\,|\log k|^{-3}\bigr),\label{eq:Phi2-smallk-match1}
\\
\Phi_2'(R_k,k)
&=R_k^{-1/2}\Bigl(\frac12\widetilde c_{1,*}\log R_k+\widetilde c_{1,*}+\frac12\widetilde c_{2,*}\Bigr)
+O\bigl(R_k^{-1/2}|\log k|^{-3}\bigr).\label{eq:Phi2-smallk-match2}
\end{align}
On the other hand, Lemma~\ref{lem:Weyl2-low-k-H0} together with the standard expansion
\[
H_0^{(2)}(z)=1-\frac{2i}{\pi}\Bigl(\log\frac{z}{2}+\gamma\Bigr)+O(z^2|\log z|),
\qquad z\to0,
\]
yields
\begin{align}
\overline{\Psi_2(R_k,k)}
&=C_0\sqrt{R_k}\Bigl(1-\frac{2i}{\pi}\bigl(\log R_k+\beta_k\bigr)\Bigr)
+O\bigl(\sqrt{R_k}\,|\log k|^{-4}\log|\log k|\bigr),\label{eq:Psi2-smallk-match1}
\\
\partial_r\overline{\Psi_2(R_k,k)}
&=C_0R_k^{-1/2}\Bigl(\frac12-\frac{i}{\pi}\bigl(\log R_k+\beta_k\bigr)-\frac{2i}{\pi}\Bigr)
+O\bigl(R_k^{-1/2}|\log k|^{-4}\log|\log k|\bigr),\label{eq:Psi2-smallk-match2}
\end{align}
where
\[
C_0:=e^{-i\pi/4}\sqrt{\frac{\pi}{2}},
\qquad
\beta_k:=\log\frac{k}{2}+\gamma.
\]

Set
\[
A:=\widetilde c_{1,*}, \qquad B:=\widetilde c_{2,*}.
\]
A direct calculation shows that
\begin{equation}\label{eq:Wronskian-log-identity}
W\!\left(\sqrt r\,(A\log r+B),\,C_0\sqrt r\Bigl(1-\frac{2i}{\pi}(\log r+\beta)\Bigr)\right)
=
-i\sqrt{\frac{2}{\pi}}\,e^{-i\pi/4}\Bigl(B-A\Bigl(\beta+i\frac{\pi}{2}\Bigr)\Bigr).
\end{equation}
Applying \eqref{eq:Wronskian-log-identity} with $\beta=\beta_k$, and using \eqref{eq:Phi2-smallk-match1}--\eqref{eq:Psi2-smallk-match2}, we obtain
\[
W\bigl(\Phi_2(\cdot,k),\overline{\Psi_2(\cdot,k)}\bigr)
=
-i\sqrt{\frac{2}{\pi}}\,e^{-i\pi/4}
\Biggl[
\widetilde c_{2,*}-\widetilde c_{1,*}\Bigl(\log\frac{k}{2}+\gamma+i\frac{\pi}{2}\Bigr)
\Biggr]
+O(|\log k|^{-1}).
\]
Inserting this into \eqref{SMW_a2} gives \eqref{eq:a2-smallk}.

We next consider $k\to\infty$.  We evaluate the Wronskian at the fixed point $r=1$.  By Lemma~\ref{lem:Phi2-high-k-J0},
\[
\Phi_2(1,k)=\Phi_{2,0}(1,k)+O(k^{-3/2}),
\qquad
\Phi_2'(1,k)=\Phi_{2,0}'(1,k)+O(k^{-1/2}),
\]
while Lemma~\ref{lem:Weyl2-high-k-H0} yields
\[
\Psi_2(1,k)=\Psi_{2,0}(1,k)+O(k^{-3/2}),
\qquad
\Psi_2'(1,k)=\Psi_{2,0}'(1,k)+O(k^{-1/2}).
\]
Since
\[
|\Phi_{2,0}(1,k)|+|\Psi_{2,0}(1,k)|\lesssim k^{-1/2},
\qquad
|\Phi_{2,0}'(1,k)|+|\Psi_{2,0}'(1,k)|\lesssim k^{1/2},
\]
we get
\[
W\bigl(\Phi_2(\cdot,k),\overline{\Psi_2(\cdot,k)}\bigr)
=
W\bigl(\Phi_{2,0}(\cdot,k),\overline{\Psi_{2,0}(\cdot,k)}\bigr)+O(k^{-1}).
\]
Finally,
\[
W\bigl(\sqrt r\,J_0(kr),\sqrt r\,H_0^{(2)}(kr)\bigr)= -\frac{2i}{\pi},
\]
and therefore
\[
W\bigl(\Phi_{2,0}(\cdot,k),\overline{\Psi_{2,0}(\cdot,k)}\bigr)
=
e^{-i\pi/4}\sqrt{\frac{\pi}{2}}\,
W\bigl(\sqrt r\,J_0(kr),\sqrt r\,H_0^{(2)}(kr)\bigr)
=
-i\sqrt{\frac{2}{\pi}}\,e^{-i\pi/4}.
\]
Using \eqref{SMW_a2} once more, we arrive at \eqref{eq:a2-largek}.

The two sharper formulas imply \eqref{propSMa}.  Moreover, $a_2(k)$ cannot vanish: indeed, if $a_2(k_0)=0$ for some $k_0>0$, then \eqref{propSM0} would force $\Phi_2(\cdot,k_0)\equiv0$, which is impossible.  The derivative bounds \eqref{propSMa_dk} follow by differentiating the Volterra and power-series constructions above; on compact $k$-intervals they follow from smooth dependence on parameters.  Finally, \eqref{propSMrho} and the sharper asymptotics for $|a_2|^{-2}$ are immediate from the asymptotics for $a_2$.
\end{proof}
As in the case of $\Psi_1$ in Lemma~\ref{lem:Psi1toTheta1Phi1}, we can express $\Psi_2$ in terms of $\Phi_2$ and $\Theta_2$. This is the content of the next lemma.
\begin{lemma}\label{lem:Psi2toTheta1Phi2}
    The Weyl solution at infinity $\Psi_2$ satisfies
    \begin{equation}\label{eq:Psi2toPhi2Theta2}
        \Psi_2(r,k)=W(\Theta_2,\Psi_2)\Phi_2(r,k)+W(\Psi_2,\Phi_2)\Theta_2(r,k).
    \end{equation}
    Moreover, 
    \begin{equation}\label{eq:Psi2Theta2Wornskian2}
        |W(\Theta_2,\Psi_2)|\simeq \frac{\log(2+k)}{\jap{\log k}}\qquad |W(\Psi_2,\Phi_2)|\simeq \frac{\jap{\log k}}{\log(2+k)}.
    \end{equation}
\end{lemma}
\begin{proof}
    The identity \eqref{eq:Psi2toPhi2Theta2} follows immediately from the facts that $\Phi_2$ and $\Theta_2$ form a fundamental system for $\calH_2-\jap{k}^2$ and  $W(\Theta_2,\Phi_2)=1$. The estimates \eqref{eq:Psi2Theta2Wornskian2} on the Wronskians are proved as in Proposition~\ref{propSM2} and using also the relation~\eqref{eq:Theta2varpar1}. Indeed the new estimate is for $W(\Theta_2,\Psi_2)$ where~\eqref{eq:Theta2varpar1} implies $|W(\Theta_2,\Psi_2)(k)|=|\Psi_2(k^{-1},k)||\Phi_2(k^{-1},k)|^{-1}$. The desired estimate follows from this relation.  
\end{proof}

\subsection{Distorted Fourier transform for the matrix operator \texorpdfstring{$\mathbf{H}$}{H}}\label{sec:spectral_fourier_bfH}

The purpose of this section is to pull back the distorted Fourier transforms for $\mathcal{H}_1$ and $\mathcal{H}_2$ in order to derive a matrix distorted Fourier transform for $\mathbf{H}$. Recall that the conjugated operators are
\begin{equation}
    \begin{aligned}
        \bfH &:= r^{\frac12} \cdot \bfM \cdot r^{-\frac12}, \quad & & \quad \bfL= \mathcal{B}^*\mathcal{B}, \\ 
        \calH &:= r^{\frac12} \cdot \calL \cdot r^{-\frac12}, \quad & & \quad  \calL= \mathcal{B}\mathcal{B}^*,
    \end{aligned}
\end{equation}
and that $\bfH$ is a $4\times 4$ matrix Schr\"odinger operator, while $\bfL$, $\calH$, and $\calL$ are all $2\times 2$ operators. Since $\bfH$ is self-adjoint, Stone's formula gives (formally)
\begin{equation}\label{eq:StonebfHlambda1}
 e^{it\bfH} P_c \bmf = \frac{1}{2\pi i} \int_0^\infty e^{it(1+\lambda)} \Bigl( \bigl( \bfH - (1+\lambda+i0) \bigr)^{-1} - \bigl( \bfH - (1+\lambda-i0) \bigr)^{-1} \Bigr) P_c \bmf \, \ud \lambda.
\end{equation}
Below we determine the integral kernel of the resolvent of $\bfH$ and compute the jump of the resolvent across the continuous spectrum. This leads to an oscillatory integral representation of the evolution $e^{it\bfH} P_c$ in terms of the distorted Fourier transform associated with $\bfH$.

At least formally, we expect that this representation can be obtained from the corresponding representations for the scalar operators $\calH_j$, $j=1,2$, in view of the intertwining relation
\begin{equation} \label{equ:relation_bfH_calH}
    \bigl( r^{\frac12} \cdot \calB \cdot r^{-\frac12} \bigr) \, \bfH  = \calH \, \bigl( r^{\frac12} \cdot \calB \cdot r^{-\frac12} \bigr).
\end{equation}
More precisely, suppose the distorted Fourier transform and its inverse associated with $\calH_j$ are
\begin{align*}
\begin{split}
\wtilcalF_j[f](\lambda) = \int_0^\infty \phi_j(r,\lambda) f(r)\, \ud r
\qquad \text{and} \qquad
\wtilcalF_j^{-1}[g](r) = \int_0^\infty \phi_j(r,\lambda) g(\lambda) \, \rho_j(\lambda) \, \ud \lambda,
\end{split}
\end{align*}
where
\[
  \rho_j(\lambda)=
  \frac{1}{4\pi|a_j(\sqrt\lambda)|^2},
  \qquad
  \phi_j(r,\lambda)=\Phi_j(r,\sqrt\lambda),
\]
and
\[
  \mathcal H_j\phi_j(\cdot,\lambda)
  =(1+\lambda)\phi_j(\cdot,\lambda),
  \qquad \lambda\geq0.
\]
Using \eqref{equ:relation_bfH_calH} and the representations of the scalar evolutions $e^{it\calH_j}$ in terms of the corresponding scalar distorted Fourier transforms (here 
$P_c^{\calH_j}$ denotes the projection onto the continuous spectrum of $\calH_j$),
\begin{equation}
    e^{it\calH_j} P_c^{\calH_j} f = \int_0^\infty e^{it(1+\lambda)} \phi_j(r,\lambda) \wtilcalF_j[f](\lambda) \rho_j(\lambda) \, \ud \lambda,
\end{equation}
we formally expect that the evolution $e^{it\bfH} P_c^\bfH \bmf$, where $P_c^\bfH$ denotes the projection to the continuous spectrum of $\bfH$, can be represented as
\begin{align*}
\begin{split}
 \small{e^{it\bfH} P_c^\bfH \bmf = \bigl( r^{\frac12} \cdot \calB^{-1} \cdot r^{-\frac{1}{2}} \bigr) \int_0^\infty e^{it(1+\lambda)} \pmat{\phi_1(r,\lambda) \rho_1(\lambda) \wtilcalF_1 & 0 \\ 0 & \phi_2(r,\lambda) \rho_2(\lambda) \wtilcalF_2} \bigl( r^{\frac12} \cdot \calB \cdot r^{-\frac12} \bigr) P_c^\bfH \bmf \, \ud \lambda.}
\end{split}
\end{align*}
Instead of justifying and unwinding this expression directly, we develop the full setup of the distorted Fourier transform for the matrix operator $\bfH$. In doing so we will still borrow from the corresponding analysis for $\calH_j$.

\subsubsection{Weyl solutions near zero for $\bfH$} \label{subsec:bfH_weyl_solutions_near_zero}

Using the identity 
\begin{equation}
    \bfH \bigl( r^{\frac12} \cdot \calB^\ast \cdot r^{-\frac12} \bigr) = \bigl( r^{\frac12} \cdot \calB^\ast \cdot r^{-\frac12} \bigr) \calH,
\end{equation}
we obtain from the Weyl solutions $\Phi_j(r,z)$, $1 \leq j \leq 2$, for the scalar operators $\calH_j$ that two Weyl solutions for the operator $\bfH$ are given by \begin{equation}\begin{aligned}\label{green7}
\upi_1&:=(\upi_{1,1},\upi_{1,2})^t=r^{1/2}\mathcal{B}^*\big( r^{-1/2}\Phi_1,\, 0\big)^t=\big(-\partial_r\Phi_1-\tfrac{1}{2r}\Phi_1-b\Phi_1,\, U\Phi_1\big)^t,
\\ \uptta_1&:=(\uptta_{1,1},\uptta_{1,2})^t=r^{1/2}\mathcal{B}^*\big( r^{-1/2}\Theta_1,\, 0\big)^t=\big(-\partial_r\Theta_1-\tfrac{1}{2r}\Theta_1-b\Theta_1,\, U\Theta_1\big)^t,
\\ \upsi_1&:=(\upsi_{1,1},\upsi_{1,2})^t=r^{1/2}\mathcal{B}^*\big( r^{-1/2}\Psi_1,\, 0\big)^t=\big(-\partial_r\Psi_1-\tfrac{1}{2r}\Psi_1-b\Psi_1,\, U\Psi_1\big)^t,
\\ 
\upi_2&:=(\upi_{2,1},\upi_{2,2})^t=r^{1/2}\mathcal{B}^*\big(0, \, r^{-1/2}\Phi_2\big)^t=\big(U\Phi_2,\, -\partial_r\Phi_2+\tfrac{1}{2r}\Phi_2\big)^t,
\\
\uptta_2&:=(\uptta_{2,1},\uptta_{2,2})^t=r^{1/2}\mathcal{B}^*\big(0, \, r^{-1/2}\Theta_2\big)^t=\big(U\Theta_2,\, -\partial_r\Theta_2+\tfrac{1}{2r}\Theta_2\big)^t,
\\ 
\upsi_2&:=(\upsi_{2,1},\upsi_{2,2})^t=r^{1/2}\mathcal{B}^*\big(0, \, r^{-1/2}\Psi_2\big)^t=\big(U\Psi_2,\, -\partial_r\Psi_2+\tfrac{1}{2r}\Psi_2\big)^t.
\end{aligned}
\end{equation}
They satisfy
\[
\bfH(\upi_{i,1},\upi_{i,2})^t=(1+k^2)(\upi_{i,1},\upi_{i,2})^t, \quad \hbox{and}\quad \bfH(\upsi_{i,1},\upsi_{i,2})^t=(1+k^2)(\upsi_{i,1},\upsi_{i,2})^t.
\]
We also define the $2\times2$ matrix solutions \begin{align*}
\upi:=\big(\upi_1 \, \, \, \upi_2\big), \qquad \uptta:=\big(\uptta_1 \, \, \, \uptta_2\big), \qquad  \upsi_+:=\big(\upsi_1 \, \, \, \upsi_2\big), \qquad \hbox{and}\qquad \upsi_-:=\big(\overline{\upsi}_1 \, \, \, \overline{\upsi}_2\big).
\end{align*}
From the asymptotics for $\Phi_j(r,z)$, $1 \leq j \leq 2$, and the asymptotics for $U(r)$, $a_\theta(r)$, we obtain that near $r=0$,
\begin{equation}
    \begin{aligned}
    \upi_1(r,z) &= \begin{pmatrix} -4 r^{\frac32} \bigl( 1 + \calO(|z|^2 r^{2}) \bigr) \\ d_1 r^{\frac72} \bigl( 1 + \calO((1+|z|^2) r^{2}) \bigr)  \end{pmatrix}, & &  \uptta_1(r,z) = \begin{pmatrix} \calO\bigl( (1+|z|^2) r^{-\frac12} \bigr) \\ \frac14 d_1 r^{-\frac12} \bigl( 1 + \calO\bigl( (1+|z|^2) r^2 \bigr) \bigr) \end{pmatrix}
    \\ \upi_2(r,z) &= \begin{pmatrix} d_1 r^{\frac32} \bigl( 1 + \calO\bigl( (1+|z|^2) r^2 \bigr) \bigr) \\ \calO\bigl(|z|^2 r^{\frac32} \bigr) \end{pmatrix}, & & 
    \uptta_2(r,z) = \begin{pmatrix} -d_1 r^{\frac32} \log(r) \bigl( 1 + \calO\bigl( (1+|z|^2) r^2 \bigr) \bigr)  \\ -r^{-\frac12} \bigl( 1 + \calO( r^2 |z|^2 ) \bigr) \end{pmatrix}.
    \end{aligned}
\end{equation}
Note that $\upi$ is the subordinate solution (near $r=0$) to $\bfH \upi = (1+z^2) \upi$ with $\upi(\cdot,z) \in L^2$ near $r=0$. $\upi$ is unique up to invertible linear combinations of the columns. 
Instead, $\uptta_1(\cdot,z)$, $\uptta_2(\cdot,z)$ are not in $L^2$ near $r=0$.
Importantly, by considering the asymptotics as $r \to 0$, it follows that $\upi_1(r,z)$, $\upi_2(r,z)$, $\uptta_1(r,z)$, $\uptta_2(r,z)$ are linearly independent vector-valued functions on the half-line.
By the same reasoning as above, we obtain the Weyl solution near infinity for $\bfH$ (unique up to multiplication by an invertible matrix) 
\begin{equation} \label{equ:asymptotics_Psi_pm}
    \upsi_{\pm}(r,z) \sim \begin{pmatrix} 
                            (\mp i) z^{\frac12} e^{\pm i zr} & z^{-\frac12} e^{\pm i zr} \\
                            z^{-\frac12} e^{\pm i zr} & (\mp i) z^{\frac12} e^{\pm i zr}
                         \end{pmatrix}.
\end{equation}  

\subsubsection{Green's kernel}

In order to solve $(\bfH -1-z^2) \bmf = \bmg$ for $\pm \Im(z) > 0$, we define the Green's functions
\begin{equation}\label{eq:G+}
 \calG_\pm(r,s; z) := \upsi_\pm(r,z) \bfS_{\pm}(s,z) \mathds{1}_{[0 < s \leq r]} + \upi(r,z) \bfT_{\pm}(s,z) \mathds{1}_{[r \leq s < \infty]}.
\end{equation}
Here the matrices $\bfS_{\pm}(r,z)$ and $\bfT_{\pm}(r,z)$ are required to satisfy
\begin{equation}\nonumber
 \begin{aligned}
  \upsi_\pm(r,z) \bfS_{\pm}(r,z) - \upi(r,z) \bfT_{\pm}(r,z) &= 0, \\
  -(\partial_r \upsi_\pm)(r,z) \bfS_{\pm}(r,z) + (\partial_r \upi)(r,z) \bfT_{\pm}(r,z) &= \Id.
 \end{aligned}
\end{equation}
Then a solution to $(\bfH -1-z^2) \bmf = \bmg$ for $\pm \Im(z) > 0$ is given by
\begin{equation}\nonumber
 \bmf(r) := \int_0^\infty \calG_\pm(r,s;z) \bmg(s) \, \ud s.
\end{equation}
Equivalently, $\bfS_\pm(r,z)$ and $\bfT_\pm(r,z)$ satisfy
\begin{equation} \label{equ:4times4system_Green_ST}
 \begin{aligned}
  \begin{bmatrix} \upsi_\pm(r,z) & - \upi(r,z) \\ (\partial_r \upsi_\pm)(r,z) & -(\partial_r \upi)(r,z) \end{bmatrix} \begin{bmatrix} \bfS_{\pm}(r,z) \\ \bfT_{\pm}(r,z) \end{bmatrix} = \begin{bmatrix} 0 \\ -\Id \end{bmatrix}.
 \end{aligned}
\end{equation}
A decisive question is the invertibility of the $4\times4$ matrix on the left-hand side of \eqref{equ:4times4system_Green_ST}.

\subsubsection{Matrix Wronskians}
Define the (real) vector inner product
\begin{equation}\nonumber
 \langle \bm{v}, \bm{w} \rangle := v_1 w_1 + v_2 w_2, \quad \bm{v} = \begin{bmatrix} v_1 \\ v_2 \end{bmatrix}, \quad \bm{w} = \begin{bmatrix} w_1 \\ w_2 \end{bmatrix}.
\end{equation}
We introduce the Wronskian
\begin{equation}\nonumber
 W[\bma, \bmb] := \langle \bm{a}, \bm{b}' \rangle - \langle \bm{a}', \bm{b} \rangle, \quad \bma(r) = \begin{bmatrix} a_1(r) \\ a_2(r) \end{bmatrix}, \quad \bmb(r) = \begin{bmatrix} b_1(r) \\ b_2(r) \end{bmatrix}.
\end{equation}
We also introduce the following shorthand notation for our linearized operator:
\begin{equation}\nonumber
 \begin{aligned}
  \bfH =  (-\partial_r^2) \mathbb{I} + \calV(r), \quad \mathbb{I} := \begin{bmatrix} 1 & 0 \\ 0 & 1 \end{bmatrix}, \quad \calV(r) := (1+\frac{3}{4r^2})\mathbb{I}+\bfV_0.
 \end{aligned}
\end{equation}

\begin{lemma} \label{lem:constant_wronskian_vectorial}
Let $z \in \bbC$. Suppose $(\bfH -1-z^2) \bma = (\bfH -1-z^2) \bmb = 0$. Then
\begin{equation}\nonumber
  \frac{\ud}{\ud r} W[\bma, \bmb] = 0.
\end{equation}
\end{lemma}
\begin{proof}
Using that $\calV(r)$ is a symmetric matrix, the claim follows by direct computation.
\end{proof}

Next, we introduce the matrix Wronskian
\begin{equation}\nonumber
 \begin{aligned}
  \calW[A,B] := A^t B' - A'^t B.
 \end{aligned}
\end{equation}
for $2\times2$ matrices $A(r)$ and $B(r)$.

\begin{lemma} \label{lem:constant_wronskian_matrix}
 Let $z \in \bbC$. Suppose $(\bfH -1-z^2)A = (\bfH -1-z^2)B = 0$ for  $2\times2$ matrices $A(r)$ and $B(r)$. Then we have
 \begin{equation}\nonumber
  \frac{\ud}{\ud r} \calW[A,B] = 0.
 \end{equation}
\end{lemma}
\begin{proof}
 Denote by $\bma_1(r)$ and $\bma_2(r)$ the columns of $A(r)$, and by $\bmb_1(r)$ and $\bmb_2(r)$ the columns of $B(r)$. Then we have
 \begin{equation}\nonumber
    \calW[A,B] = \begin{bmatrix} W[\bma_1, \bmb_1] & W[\bma_1, \bmb_2] \\ W[\bma_2, \bmb_1] & W[\bma_2, \bmb_2] \end{bmatrix},
 \end{equation}
 and the assertion is a direct consequence of Lemma~\ref{lem:constant_wronskian_vectorial}.
\end{proof}

Next, we point out the following inversion identity
\begin{lemma} \label{lem:inverse_4times4}
 Let $A(r)$ and $B(r)$ be two ($r$-dependent) $2\times2$ matrices. Suppose that $\calW[A,A] = \calW[B,B] = 0$. Moreover, suppose that $D := \calW[A,B]$ is invertible. Then we have
 \begin{equation}\nonumber
 \begin{aligned}
  \begin{bmatrix} A & B \\ A' & B' \end{bmatrix}^{-1} = \begin{bmatrix} D^{-t} & 0 \\ 0 & D^{-1} \end{bmatrix} \begin{bmatrix} 0 & \mathbb{I} \\- \mathbb{I} & 0 \end{bmatrix} \begin{bmatrix} A^t & A'^t \\ B^t & B'^t \end{bmatrix} \begin{bmatrix} 0 & -\mathbb{I} \\ \mathbb{I} & 0 \end{bmatrix}.
 \end{aligned}
\end{equation}
\end{lemma}
\begin{proof}
 Under the assumption that $\calW[A,A] = \calW[B,B] = 0$, we have that
 \begin{equation} \label{equ:4times4_inversion_observation}
 \begin{aligned}
 \begin{bmatrix} 0 & \mathbb{I} \\ -\mathbb{I} & 0 \end{bmatrix} \begin{bmatrix} A^t & A'^t \\ B^t & B'^t \end{bmatrix} \begin{bmatrix} 0 & -\mathbb{I} \\ \mathbb{I} & 0 \end{bmatrix} \begin{bmatrix} A & B \\ A' & B' \end{bmatrix} = \begin{bmatrix} - \calW[B,A] & 0 \\ 0 & \calW[A,B] \end{bmatrix}.
 \end{aligned}
 \end{equation}
 Noting that $D^t = -\calW[B,A]$, the assertion follows from the preceding identity \eqref{equ:4times4_inversion_observation}.
\end{proof}

Let us now return to the question of the invertibility of the $4\times4$ matrix on the left-hand side of \eqref{equ:4times4system_Green_ST}.
Here we have $(\bfH -1-z^2) \upsi_\pm(\cdot,z) = (\bfH -1-z^2) \upi(\cdot,z) = 0$.
Thus, by Lemma~\ref{lem:constant_wronskian_matrix} the matrix Wronskians
\begin{equation}
 \calW[\upsi_\pm(\cdot,z), \upsi_\pm(\cdot, z)], \quad \calW[\upi(\cdot,z), \upi(\cdot, z)], \quad \calW[\upsi_\pm(\cdot,z), \upi(\cdot, z)]
\end{equation}
are constant in $r$. Therefore, based on the asymptotics for $\upsi_\pm(r,z)$ and $\upi(r,z)$ from the previous sections, we have
\begin{align*}
\begin{split}
 \calW[\upsi_+(\cdot,z), \upsi_+(\cdot, z)]= \calW[\upsi_-(\cdot,z), \upsi_-(\cdot, z)]= \calW[\upi(\cdot,z), \upi(\cdot, z)]=0.
\end{split}
\end{align*}
Let
\begin{align}\label{green5}
\begin{split}
D_{\pm}(z) := \calW[\upsi_\pm(\cdot,z),-\upi(\cdot,z)]=-\calW[\upsi_\pm(\cdot,z), \upi(\cdot,z)].
\end{split}
\end{align}
The invertibility of the matrix $D_\pm(z)$ will be established in Lemma~\ref{green9} below. Assuming it for the moment, it follows from \eqref{equ:4times4system_Green_ST} and Lemma~\ref{lem:inverse_4times4} that
\begin{align*}
    \begin{split}
        \bfS_{\pm}(r,z) = - D_{\pm}^{-t}(z) \upi^t(r,z), \qquad \bfT_{\pm}(z) = - D_{\pm}^{-1}(z) \upsi_\pm^t(r,z).
    \end{split}
\end{align*}

\subsubsection{Solving for the Green's kernel}
By the preceding considerations, for $\pm \Im(z) > 0$ the Green's function for $\bfH-1-z^2$ is given by
\begin{equation}\label{eq:calGpmsummary1}
 \calG_\pm(r,s; z) := - \upsi_\pm(r,z) D_\pm^{-t}(z)\upi^{t}(s,z) \mathds{1}_{[0 < s \leq r]} - \upi(r,z) D_{\pm}^{-1}(z)\upsi_\pm^t(s,z) \mathds{1}_{[r \leq s < \infty]}.
\end{equation}
In what follows, for $k \geq 0$ we use the shorthand notation
\begin{equation}
    \upsi_\pm(\cdot,k) \equiv \upsi_\pm(\cdot, k \pm i 0), \quad \upi(\cdot,k) \equiv \upi(\cdot, k \pm i 0).
\end{equation}
Thus, for $k \geq 0$ we obtain
\begin{equation} \label{equ:calG_jump_definition}
\begin{aligned}
\calG(r,s,k) &:= \calG_{+}(r,s,k+i0)-\calG_{-}(r,s,k-i0)\\
    &= - \Big(\upsi_+(r,k) D_{+}^{-t}(k) \upi^t(s,k) - \upsi_-(r,k) D_{-}^{-t}(k) \upi^t(s,k) \Big) \mathds{1}_{[0 < s \leq r]} 
    \\
    &\quad \, - \Big( \upi(r,k) D_{+}^{-1}(k) \upsi_{+}^t(s,k) - \upi(r,k) D_{-}^{-1}(k) \upsi_{-}^{t}(s,k) \Big) \mathds{1}_{[r \leq s < \infty]} 
    \\
    &= - \Big(\upsi_+(r,k) D_{+}^{-t}(k) - \upsi_-(r,k) D_{-}^{-t}(k) \Big) \upi^t(s,k) \mathds{1}_{[0 < s \leq r]} 
    \\
    &\quad \, - \upi(r,k) \Big( D_{+}^{-1}(k) \upsi_{+}^t(s,k) - D_{-}^{-1}(k) \upsi_{-}^{t}(s,k) \Big) \mathds{1}_{[r \leq s < \infty]}.
\end{aligned}
\end{equation}
Note that $\calG^t(r,s,k) = \calG(s,r,k)$. Now since the columns of $\upi(r,k)$ and $\uptta(r,k)$ form a fundamental system for $\bfH-1-k^2$, see the end of Subsection~\ref{subsec:bfH_weyl_solutions_near_zero} for the justification of their linear independence, we can find $2 \times 2$-matrices $M_\pm(k)$ and $N_\pm(k)$ such that
\begin{equation} \label{equ:Psipm_linear_combination}
    \upsi_\pm(r,k) =  \upi(r,k) M_\pm(k) + \uptta(r,k) N_\pm(k).
\end{equation}
Set 
\begin{equation*}
    D(k) := \calW\bigl[ \uptta(\cdot, k),\upi(\cdot,k) \bigr].
\end{equation*}
Using that $\calW[ \upi  M_\pm, \upi ] = M_\pm^t \calW[\upi , \upi] = 0$, we find
\begin{align*}
\begin{split}
    D_\pm(k) &= -\calW[\upsi_\pm(\cdot,k), \upi(\cdot,k)] = -\calW[\uptta(\cdot,k) N_\pm(k), \upi(\cdot,k)] \\ 
    &= - N_\pm^t(k) \calW[\uptta(\cdot,k), \upi(\cdot,k)] = - N_\pm^t(k) D(k).
\end{split}
\end{align*} 
Assuming also that $D(k)$ is invertible, the assumed invertibility of $D_\pm(k)$ implies that $N_\pm(k)$ is invertible as well. Therefore,
\begin{equation} \label{equ:Dpminvtranspose_formula}
    D_\pm^{-t}(k) = - N_\pm^{-1}(k) D^{-t}(k).
\end{equation}
Inserting \eqref{equ:Psipm_linear_combination} and \eqref{equ:Dpminvtranspose_formula} back into \eqref{equ:calG_jump_definition}, we find that in the regime $0 < s \leq r$, the terms with only one factor of $\upi$ cancel out.
It follows that for suitable $\calM_>(s,k)$ and $\calM_<(s,k)$,
\begin{align*}
\begin{split}
    \calG(r,s,k) = \upi(r,k) \calM_<(s,k) \mathds{1}_{[0 < s \leq r]} + \upi(r,k) \calM_>(s,k) \mathds{1}_{[r \leq s < \infty]}.
\end{split}
\end{align*}
But then by continuity at $r=s$ and by invertibility of $\upi(r,k)$, we must have $\calM_<(r,k)=\calM_>(r,k)$, so for a suitable continuous $\calM(s,k)$,
\begin{align*}
\begin{split}
    \calG(r,s,k)=\upi(r,k)\calM(s,k).
\end{split}
\end{align*}
Recalling the relation $\calG^t(r,s,k)=\calG(s,r,k)$, we get
\begin{align*}
\begin{split}
    \upi^{-1}(r,k)\calM^t(r,k)=\calM(s,k) \upi^{-t}(s,k) =: C(k) \quad   \Rightarrow \quad \calM(s,k) = C(k) \upi^t(s,k),
\end{split}
\end{align*}
where $C(k)$ satisfies $C^t(k) = C(k)$. The above relation implies the identity
\begin{align*}
\begin{split}
    \calG(r,s,k) = \upi(r,k) C(k) \upi^t(s,k).
\end{split}
\end{align*}
On the other hand, recall that for $s \leq r$ we can also write
\begin{align*}
\begin{split}
    \calG(r,s,k) = - \Bigl(\upsi_{+}(r,k) D_{+}^{-t}(k) - \upsi_{-}(r,k) D_{-}^{-t}(k) \Bigr) \upi^t(s,k).
\end{split}
\end{align*}
Comparing the two expressions gives
\begin{align*}
\begin{split}
    \upi(r,k) C(k) = - \Bigl( \upsi_{+}(r,k) D_{+}^{-t}(k) - \upsi_{-}(r,k) D_{-}^{-t}(k) \Bigr).
\end{split}
\end{align*}
Taking the Wronskian with $\upsi_{+}(\cdot,k)$, and noting that $\calW[\upsi_+ D_{+}^{-t}, \upsi_{+}] = D_+^{-1} \calW[ \upsi_+, \upsi_{+}] = 0$, we conclude
\begin{equation}
    C(k) \calW\bigl[ \upi(\cdot,k), \upsi_+(\cdot,k) \bigr] = D_-^{-1}(k) \calW\bigl[ \upsi_-(\cdot,k), \upsi_+(\cdot,k) \bigr],
\end{equation}
whence 
\begin{equation}
    C(k) = D_{-}^{-1}(k) \calW[\upsi_{-}(\cdot, k), \upsi_{+}(\cdot, k)] D_{+}^{-t}(k).
\end{equation}
Next, we observe that by the asymptotics \eqref{equ:asymptotics_Psi_pm}, we have 
\begin{equation}
    \begin{aligned}
        \calW\bigl[\upsi_-(\cdot,k), \upsi_+(\cdot,k)\bigr] = \kappa(k) \bbI, \quad \kappa(k) = 2i \bigl(1+k^2\bigr).
    \end{aligned}
\end{equation}
It follows that 
\begin{equation}
    C(k) = \kappa(k) D_{-}^{-1}(k) D_{+}^{-t}(k),
\end{equation}
whence
\begin{align*}
\begin{split}
    \calG(r,s,k) = \kappa(k) \upi(r,k) D_{-}^{-1}(k) D_{+}^{-t}(k) \upi^t(s,k).
\end{split}
\end{align*}
Finally, we recall that for $k \geq0$ we have by construction $\upsi_{-}(r,k) = \overline{\upsi_{+}(r,k)}$ and that $\upi(r,k)$ is real-valued, which implies $D_{-}(k) = \overline{D_{+}(k)}$.
Putting everything together, and changing variables to $\lambda=k^2$  in \eqref{eq:StonebfHlambda1}, we find that 
\begin{equation}
    \begin{aligned}
        &\bigl\langle e^{it\bfH} P_c^\bfH \bmf, \bmg \bigr\rangle
        = \Bigl\langle \frac{1}{\pi i} \int_0^\infty e^{it(1+k^2)}
        \, \upi(r,k) D_{-}^{-1}(k)
        \biggl( \int_0^\infty D_{+}^{-t}(k) \upi^t(s,k) \bmf(s) \, \ud s \biggr)
        \, \kappa(k)k \, \ud k, \bmg \Bigr\rangle \\
        & = \int_0^\infty e^{it(1+k^2)}
        \biggl( \int_0^\infty D_{+}^{-t}(k) \upi^t(s,k) \bmf(s) \, \ud s \biggr)
        \overline{\biggl( \int_0^\infty D_{+}^{-t}(k) \upi^t(r,k) \bmg(r) \, \ud r \biggr)}
        \, \frac{2(1+k^2)}{\pi} \, k \, \ud k.
    \end{aligned}
\end{equation}
where we used $\langle \bma, \bmb \rangle = a_1 \overline{b_1} + a_2 \overline{b_2}$ as the definition of the vector inner product.
Alternatively, we can write
\begin{equation}
    \begin{aligned}
        \bigl(e^{it\bfH} P_c^\bfH \bmf\bigr)(r) = \int_0^\infty e^{it(1+k^2)} \upi(r,k) D_{-}^{-1}(k) \biggl( \int_0^\infty D_{+}^{-t}(k) \upi^t(s,k) \bmf(s) \, \ud s \biggr) \, \frac{2(1+k^2)}{\pi} \, k \, \ud k.
    \end{aligned}
\end{equation}
We introduce
\begin{equation}\label{green6}
    E^\bfH(r,k) := \sqrt{2k} \sqrt{\frac{1+k^2}{\pi}} \upi(r,k) \overline{D_+(k)}^{-1} \equiv \bigl[ E_1^\bfH(r,k) \, E_2^\bfH(r,k) \bigr],
\end{equation}
where $E_1^\bfH(r,k)$ and $E_2^\bfH(r,k)$ denote the two column vectors of the $2 \times 2$ matrix $E^\bfH(r,k)$. From the preceding discussion we obtain the following representation formula for the linear Klein--Gordon evolution:
\begin{equation}
    \begin{aligned}
        \bigl(e^{it\sqrt{\bfH}} P_c^\bfH \bmf\bigr)(r) &= \int_0^\infty e^{it\sqrt{1+k^2}} \biggl( \int_0^\infty \overline{E_1^\bfH(s,k)} \cdot \bmf(s) \, \ud s \biggr) E_1^\bfH(r,k) \, \ud k \\
        &\quad + \int_0^\infty e^{it\sqrt{1+k^2}} \biggl( \int_0^\infty \overline{E_2^\bfH(s,k)} \cdot \bmf(s) \, \ud s \biggr) E_2^\bfH(r,k) \, \ud k.
    \end{aligned}
\end{equation}

\begin{lem}\label{green9}
Under our current hypotheses and definitions,
\begin{align}\label{green10}
D_+(k)=\left[  \begin{matrix}
 2i(1+k^2)\overline{a}_1(k) & 0 \\ 0 &  2i(1+k^2)\overline{a}_2(k)
\end{matrix}  \right].
\end{align}
\end{lem}

Before getting into the proof, note that substituting \eqref{green10} into \eqref{green6} yields
\begin{align}\label{dFT_EbfH}
E^\bfH(r,k)  =  \dfrac{i\sqrt{k}}{\sqrt{2\pi}\langle k\rangle}\left[  \begin{matrix}
a_1^{-1}\Big(-\partial_r\Phi_1-\tfrac{1}{2r}\Phi_1-b\Phi_1\Big) & a_2^{-1}U\Phi_2 \\ a_1^{-1}U\Phi_1 & a_2^{-1}\Big(-\partial_r\Phi_2+\tfrac{1}{2r}\Phi_2\Big)
\end{matrix} \right].
\end{align}
\begin{proof}[Proof of Lemma~\ref{green9}]
From Sections~\ref{subsec:calH1} and~\ref{subsec:calH2} we know that the scalar Weyl solutions satisfy
\begin{align*}
\Psi_1(r,k)&= k^{-1/2}e^{ik r}\Big(1+\dfrac{3i}{8k r}+O\big((k r)^{-2}\big)\Big),
\\ \Psi_2(r,k)&=k^{-1/2}e^{ik r}\Big(1-\dfrac{i}{8k r}+O\big((k r)^{-2}\big)\Big),
\\ \Phi_i(r,k) &= a_i(k) \Psi_i(r,k)+\overline{a}_i(k)\overline{\Psi}_i(r,k),
\end{align*}
and that, from \eqref{green5} and \eqref{green6},
\begin{align*}
D_\pm(k)=-\mathcal W[\upsi_\pm(\cdot,k),\upi(\cdot,k)],
\qquad
E^{\bfH}(r,k)=\sqrt{\frac{2k}{\pi}}\langle k\rangle
\upi(r,k)\overline{D_+(k)}^{-1}.
\end{align*}
Using Lemma~\ref{lem:constant_wronskian_matrix} we compute the Wronskian at $r=+\infty$. Using \eqref{green7}, we write
\[
\upsi_\pm= \left[  \begin{matrix}
-\Psi_{1,\pm}' & \Psi_{2,\pm}
\\ \Psi_{1,\pm} & -\Psi_{2,\pm}'
\end{matrix} \right]+O(r^{-1}), \qquad \upi=\left[  \begin{matrix}
-\Phi_1' & \Phi_2
\\ \Phi_1 & -\Phi_2'
\end{matrix} \right]+O(r^{-1}).
\]
Then, by definition,
\begin{align*}
\calW[\upsi_+,\upi]&=\upsi^t_+\upi'-\upsi'^t_+\upi = \left[ \begin{matrix}  w_{1,1} & w_{1,2} \\ w_{2,1} & w_{2,2}  \end{matrix} \right] +O(r^{-1}),
\end{align*}
where
\begin{align*}
w_{1,1}& :=\Psi_{1,+}'\Phi_1''-\Psi_{1,+}''\Phi_1' + \Psi_{1,+}\Phi_1'-\Psi_{1,+}'\Phi_1, & & w_{1,2} := \Psi_{1,+}''\Phi_2-\Psi_{1,+}\Phi_2'',
\\ w_{2,2}&:=\Psi_{2,+}'\Phi_2''-\Psi_{2,+}''\Phi_2'+\Psi_{2,+}\Phi_2'-\Psi_{2,+}'\Phi_2, & & w_{2,1}:=\Psi_{2,+}''\Phi_1-\Psi_{2,+}\Phi_1''.
\end{align*}
Then, for $rk\gg 1$ we have
\begin{align*}
w_{1,1}&=-2i(1+k^2)\overline{a}_1(k)+O(r^{-1}), & &   w_{1,2}=O(r^{-1}),
\\ w_{2,2}&=-2i(1+k^2)\overline{a}_2(k)+O(r^{-1}),  & & w_{2,1}=O(r^{-1}).
\end{align*}
Taking the limit $r\to\infty$ concludes the proof.
\end{proof}
\subsection{Proof of Theorem~\ref{propdFT}} Theorem~\ref{propdFT} can be proved using what was developed in the previous few subsections. 
\begin{proof}[Proof of Theorem~\ref{propdFT}]
We begin by relating the matrix $E$ in the statement of Theorem~\ref{propdFT} and $E^\bfH$ constructed in the previous subsection. Since $\bfH$ and $\bfM$ are related by
\begin{align*}
    \bfH(r^{\frac{1}{2}}\bmf)=r^{\frac{1}{2}}\bfM\bmf,
\end{align*}
and since the measures used in the statement of Theorem~\ref{propdFT} are $r\ud r$ and $k\ud k$, we see that $E(r,k)=(rk)^{-\frac{1}{2}}E^\bfH(r,k)$. Then, in view of \eqref{dFT_EbfH},
\begingroup
\small
\setlength{\arraycolsep}{3pt}
\begin{equation}\label{dFT_E}
E(r,k)= \frac{i}{\sqrt{2\pi r}\langle k\rangle}
\begin{bmatrix}
 a_1^{-1}(k)\Bigl(\begin{aligned}[t]
 -\partial_r\Phi_1(r,k)-\tfrac{1}{2r}\Phi_1(r,k)\\
 -b(r)\Phi_1(r,k)\Bigr)
\end{aligned}
 & a_2^{-1}(k)U(r)\Phi_2(r,k) \\
 a_1^{-1}(k)U\Phi_1(r,k)
 & a_2^{-1}(k)\Bigl(-\partial_r\Phi_2(r,k)+\tfrac{1}{2r}\Phi_2(r,k)\Bigr)
\end{bmatrix}.
\end{equation}
\endgroup
The first three statements (i)--
(iii) are direct consequences of the Stone formula in the previous subsection. The asymptotics follow from formula~\eqref{dFT_E} and the corresponding asymptotics for $\Phi_j$ and $\Psi_j$, $j=1,2$, in Sections~\ref{subsec:calH1} and~\ref{subsec:calH2}.
We only provide details for the lower-right entries in \eqref{eq:Eklein} and \eqref{eq:Egross}. In view of Lemma~\ref{lemFou2} in the region where $rk\leq1$ and $r$ is small we have 
\begin{equation}\label{eq:Phi2ord}
-\partial_r\Phi_2(r,k)+\frac{1}{2r}\Phi_2(r,k)
=\frac{1+k^2}{2}r^{3/2}+O\bigl((1+k^4)r^{7/2}\bigr)
\qquad (r\to0).
\end{equation}
The desired asymptotics for $E_{2,2}$ in the region where $rk\leq 1$ and $r\leq 1$ follow from this expansion and \eqref{eq:a2-smallk} and \eqref{eq:a2-largek}. For $rk\ge1$ and $r\leq 1$, we use the high-energy $J_0$ comparison solution $\Phi_{2,0}(r,k)=\sqrt r\,J_0(kr)$ from Lemma~\ref{lem:Phi2-high-k-J0}. It satisfies 
\[
-\partial_r\Phi_{2,0}(r,k)+\frac1{2r}\Phi_{2,0}(r,k)=k\sqrt r\,J_1(kr).
\]
By Lemma~\ref{lem:Phi2-high-k-J0} and $|J_1(z)|\lesssim z^{-1/2}$ for $z\ge1$,
\[
\Bigl| -\partial_r\Phi_2(r,k)+\frac1{2r}\Phi_2(r,k)\Bigr|\lesssim k^{1/2},
\qquad rk\ge1,
\]
and therefore
\[
|E_{2,2}(r,k)|
\lesssim \frac{k^{1/2}}{\sqrt r\,\jap{k}}
\sim (rk)^{-1/2},
\qquad rk\ge1.
\]
For $r\ge1$ and $rk\le1$, write $\Phi_2(r,k)=\sqrt r\,G(r,k)$. Then
\[
-\partial_r\Phi_2(r,k)+\frac1{2r}\Phi_2(r,k)=-\sqrt r\,\partial_r G(r,k).
\]
Using the large-$r$ expansion from Lemma~\ref{lemFou2},
\[
G(r,k)=\widetilde c_{1,*}\log r+\widetilde c_{2,*}
+\sum_{j\ge1}(rk)^{2j}\Phi_{2,j}(r)+O(e^{-r}\jap{\log r}),
\]
with $|r\partial_r\Phi_{2,j}(r)|\lesssim1$, we get
\[
\partial_r G(r,k)
=\widetilde c_{1,*}r^{-1}+O\bigl(k^2r\jap{\log r}\bigr)+O\bigl(r^{-1}e^{-r}\jap{\log r}\bigr),
\]
whence
\[
|E_{2,2}(r,k)|
\lesssim \frac{1+(rk)^2\jap{\log r}}{r\,\jap{\log k}},
\qquad r\ge1,\ rk\le1.
\]
For $r\geq1$ and $rk\geq1$ the claimed asymptotics follow from similar considerations using also \eqref{propSM0} as well as the estimates in Lemmas~\ref{lem:Weyl2-low-k-H0} and~\ref{lem:Weyl2-high-k-H0}.
\end{proof}
In a few instances we will need more precise asymptotics for $E(r,k)$ in the region where $rk\geq1$. Before stating these asymptotics in the next lemma we introduce some notation that will be used later as well. Note that \eqref{eq:a-smallk} implies that 
\begin{equation}\label{eq:a1a2inv-small}
a_1(k)^{-1}=c_*^{-1}\sqrt{\frac{\pi}{2}}\,e^{3\pi i/4}\,k\bigl(1+O(k|\log k|)\bigr),
\qquad k\to0,
\end{equation}
while \eqref{eq:a2-smallk} yields
\begin{equation}\label{eq:a2inv-small}
a_2(k)^{-1}
=
\sqrt{2\pi}\,e^{i\pi/4}\,L_2(k)^{-1}\bigl(1+O(\jap{\log k}^{-2})\bigr),
\qquad k\to0,
\end{equation}
where
\[
L_2(k):=\widetilde c_{2,*}-\widetilde c_{1,*}\Bigl(\log\frac{k}{2}+\gamma+i\frac{\pi}{2}\Bigr).
\]
We set
\begin{equation}\label{eq:eta12-def}
\eta_j(k):=\frac{\overline{a_j(k)}}{a_j(k)},\qquad j=1,2.
\end{equation}
Then $|\eta_j(k)|=1$ for every $k>0$.
\begin{lemma}\label{lem:Erkgeq1_1}
   In the region $\{rk\geq1\}$, the entries $E_{i,j}$ of the matrix $E$ in Theorem~\ref{propdFT} satisfy
\begin{align}
E_{1,1}(r,k)
&=
\frac{k^{1/2}}{\sqrt{2\pi}\,\jap{k}\,r^{1/2}}
\Bigl(e^{ikr}-\eta_1(k)e^{-ikr}\Bigr)
+O\!\bigl(r^{-3/2}k^{-1/2}\jap{k}^{-1}\bigr),
\label{eq:E11grenz}
\\
E_{1,2}(r,k)
&=
\frac{i}{\sqrt{2\pi}\,\jap{k}\,k^{1/2}r^{1/2}}
\Bigl(e^{ikr}+\eta_1(k)e^{-ikr}\Bigr)
+O\!\bigl(r^{-3/2}k^{-3/2}\jap{k}^{-1}\bigr),
\label{eq:E21grenz}
\\
E_{2,1}(r,k)
&=
\frac{i}{\sqrt{2\pi}\,\jap{k}\,k^{1/2}r^{1/2}}
\Bigl(e^{ikr}+\eta_2(k)e^{-ikr}\Bigr)
+O\!\bigl(r^{-3/2}k^{-3/2}\jap{k}^{-1}\bigr),
\label{eq:E12grenz}
\\
E_{2,2}(r,k)
&=
\frac{k^{1/2}}{\sqrt{2\pi}\,\jap{k}\,r^{1/2}}
\Bigl(e^{ikr}-\eta_2(k)e^{-ikr}\Bigr)
+O\!\bigl(r^{-3/2}k^{-1/2}\jap{k}^{-1}\bigr).
\label{eq:E22grenz}
\end{align}
Moreover,
\begin{equation}\label{eq:eta12-small}
|i+\eta_1(k)|\lesssim k\,\jap{\log k},
\qquad
|1+i\eta_2(k)|\lesssim \jap{\log k}^{-1},
\qquad 0<k\le1.
\end{equation}
\end{lemma}
\begin{proof}
Since $U(r)=1+O(e^{-r})$ and $b(r)=O(e^{-r})$, and since
\[
\Psi_j(r,k)=k^{-1/2}e^{ikr}\bigl(1+O((kr)^{-1})\bigr),
\qquad
\Psi_j'(r,k)=ik^{1/2}e^{ikr}\bigl(1+O((kr)^{-1})\bigr)
\qquad (r\to\infty),
\]
using also the relations~\eqref{SMOp1_2} and~\eqref{propSM0} we obtain  the desired asymptotics \eqref{eq:E11grenz}--\eqref{eq:E22grenz}.
For \eqref{eq:eta12-small}, note that \eqref{eq:a-smallk} gives
\[
\eta_1(k)=-i+O\bigl(k\,\jap{\log k}\bigr),
\]
while \eqref{eq:a2-smallk} yields
\[
\eta_2(k)= i\,\frac{\overline{L_2(k)}}{L_2(k)}+O\bigl(\jap{\log k}^{-2}\bigr),
\]
and therefore
\[
1+i\eta_2(k)
=
1-\frac{\overline{L_2(k)}}{L_2(k)}+O\bigl(\jap{\log k}^{-2}\bigr)
=
\frac{L_2(k)-\overline{L_2(k)}}{L_2(k)}+O\bigl(\jap{\log k}^{-2}\bigr)
=
O\bigl(\jap{\log k}^{-1}\bigr).
\]
This proves \eqref{eq:eta12-small}.
\end{proof}

\begin{rem}
    The proof of Lemma~\ref{lem:Erkgeq1_1} in fact shows that in the region where $rk\geq1$ and $k\leq1$
    \begin{align*}
    &E_{1,1}(r,k)=\frac{e^{\frac{i\pi}{4}}}{\jap{k}}kJ_0(rk)+kH_0^{(2)}(rk)\big(\uppsi_{1,1}^1(r,k)+\uppsi_{1,1}^2(k)\big)+H_1^{(2)}(rk)\uppsi_{1,1}^3(r,k)\\
    &\phantom{E_{1,1}(r,k)=}+J_1(rk)\uppsi_{1,1}^4(r),\\
    &E_{1,2}(r,k)=\frac{e^{\frac{i\pi}{4}}}{\jap{k}}J_1(rk)(1+\uppsi_{1,2}^1(r))+H_1^{(2)}(rk)\big(\uppsi_{1,2}^2(r,k)+\uppsi_{1,2}^3(k)\big),\\
    & E_{2,1}(r,k)=\frac{e^{\frac{3i\pi}{4}}}{\jap{k}}J_0(rk)\big(1+\uppsi_{2,1}^1(r)\big)\\
    &\phantom{E_{2,1}=}-\frac{e^{\frac{i\pi}{4}}}{\jap{k}}\frac{\pi \tilde{c}_{1,\ast}}{\tilde{c}_{2,\ast}\log(k/2)}H_0^{(2)}(rk)\big(1+\uppsi_{2,1}^2(k)\big)\big(1+\uppsi_{2,1}^3(r)\big)+\uppsi_{2,1}^4(r,k),\\
      & E_{2,2}(r,k)=\frac{e^{\frac{3i\pi}{4}}}{\jap{k}}kJ_1(rk)-\frac{e^{\frac{i\pi}{4}}}{\jap{k}}\frac{k\pi \tilde{c}_{1,\ast}}{\tilde{c}_{2,\ast}\log(k/2)}H_1^{(2)}(rk)\big(1+\uppsi_{2,2}^1(k)\big)+\uppsi_{2,2}^2(r,k).
\end{align*}
Here the functions $\uppsi_{\ell,m}^j$ satisfy 
\begin{align*} |            &\uppsi_{1,1}^1(r,k)|+|\uppsi_{1,1}^3(r,k)|+|\uppsi_{1,1}^4(r)|\lesssim \sqrt{r}e^{-r/2},\quad |\uppsi_{1,1}^2(k)|\lesssim k\jap{\log k},\\
&|\uppsi_{1,2}^1(r)|+|\uppsi_{1,2}^2(r,k)|\lesssim \sqrt{r}e^{-r/2},\qquad |\uppsi_{1,2}^3(k)|\lesssim k\jap{\log k},\\
&|\uppsi_{2,1}^1(r)|+|\uppsi_{2,1}^3(r)|\lesssim r^{-\frac{1}{2}}e^{-r},\qquad  |\uppsi_{2,1}^4(r,k)|\lesssim \sqrt{r}e^{-r/2}\jap{\log(kr)},\\
&|\uppsi_{2,1}^2(k)|=|\uppsi_{2,2}^1(k)|\lesssim \jap{\log k}^{-2}, \qquad |\uppsi_{2,2}^2(r,k)|\lesssim \sqrt{r}e^{-r/2}\jap{\log(kr)}.
\end{align*}
Moreover these estimates on $\uppsi_{\ell,m}^j$ are stable with respect to differentiation in $\partial_r$ and $k\partial_k$ for any finite number of derivatives. This follows from the relations \eqref{SMOp1_2} and~\eqref{propSM0},  the estimates \eqref{eq:eta12-small}, and Lemmas~\ref{lemWeyl1} and~\ref{lem:Weyl2-low-k-H0}, as well as the relations \eqref{eq:Besselders1} for derivatives of Bessel functions.
\end{rem}

\begin{rem}\label{rem:Bessels1}
    For future reference we record some asymptotics for the Bessel functions $J_n(\rho)$ and $Y_n(\rho)$. The corresponding asymptotics for $H_{n}{(j)}(\rho)$, $j=1,2$, can be read off using the formula
    \begin{align*}
        H_{2}^{(j)}(\rho)=J_n(\rho)+(-1)^{j-1}iY_n(\rho), \qquad j=1,2.
    \end{align*}
    In stating these asymptotics we use the notation $y\sim \sum_{n=m}^\infty \alpha_n$ to mean that $y-\sum_{n=m}^\ell \alpha_n=O(\alpha_{\ell+1})$. For $\rho\leq1$
    \begin{align*}
        &J_n(\rho)=(\rho/2)^n\sum_{j=0}^{n-1}\frac{(-\rho^2/4)^j}{j!\Gamma(n+j+1)},\\
        &Y_n(\rho)=-\frac{(\rho/2)^{-n}}{\pi}\sum_{j=0}^{n-1}\frac{(n-j-1)!}{j!}(\rho^2/4)^{j}+\frac{2}{\pi}\log(\rho/2)J_n(\rho)\\
        &\phantom{Y_n(\rho)=}-\frac{(\rho/2)^{n}}{\pi}\sum_{j=0}^\infty \big(\uppsi(j+1)+\uppsi(n+j+1)\big)\frac{(-\rho^2/4)^j}{j!(n+j)!},
    \end{align*}
    where $\uppsi$ is as in \cite[(6.3.2)]{AbStebook}. See \cite[(9.1.10--9.1.11)]{AbStebook}. For $\rho\geq1$
    \begin{align*}
        &J_n(\rho)=\sqrt{\frac{2}{\pi\rho}}\big(\cos(\rho-\omega_n)P_n(\rho)-\sin(\rho-\omega_n)Q_n(\rho)\big),\\
        &Y_n(\rho)=\sqrt{\frac{2}{\pi\rho}}\big(\cos(\rho-\omega_n)Q_n(\rho)+\sin(\rho-\omega_n)P_n(\rho)\big),
    \end{align*}
    where $\omega_n=\frac{\pi}{4}+\frac{n\pi}{2}$ and with the notation $(n,2j)=\frac{\Gamma(\nu+2j+1/2)}{(2j)!\Gamma(n-2k+1/2)}$,
    \begin{align*}
        P_n(\rho)\sim \sum_{j=0}^\infty(-1)^j (n,2j)(2\rho)^{-2j},\qquad Q_n(\rho)\sim\sum_{j=0}^\infty (-1)^j(n,2j+1)(2\rho)^{-2j-1}.
    \end{align*}
    See \cite[(9.2.5--9.2.10)]{AbStebook}.
\end{rem}





\section{Estimates for the Linear Evolution} \label{sec:linear_decay_estimates}

    


In this section we establish decay and energy estimates for  solutions of 
\begin{align}\label{KGj'}
\partial_t^2\boldsymbol{v}+\mathbf{L}\boldsymbol{v}=\bmF, \qquad \boldsymbol{v}(0)=\boldsymbol{v}_0=P_c^\bfL\boldsymbol{v}_0, \qquad \partial_t\boldsymbol{v}(0)=\boldsymbol{v}_1=P_c^\bfL\boldsymbol{v}_1,\qquad \bmF=P_c^\bfL\bmF.
\end{align}
In view of the diagonal structure of $\bfM$ corresponding estimates for the Klein-Gordon evolution associated with $\bfM$ follow as a corollary. We begin by defining the Littlewood-Paley (L-P) frequency projections with respect to $\bfL$ in Section~\ref{ssecLP}. While the final estimates, to be used in \cite{LPPSS3}, are stated without frequency localizations, the L-P projections play an important role in their proofs. In Section~\ref{ssecKG} we prove a number of dispersive $L^\infty_{r}$ estimates. Section~\ref{subec:localenergy} is devoted to local energy decay estimates, both for fixed time and in the time-integrated sense. Here we also prove a limiting absorption estimate that will be useful in the analysis of the internal mode in the nonlinear paper \cite{LPPSS3}.

\smallskip
\subsection{Littlewood-Paley projections}\label{ssecLP}
Let $\fy,\fy_0:[0,\infty)\to\bbR$ be a smooth non-negative functions such that $\mathrm{supp}\,\fy_0\subseteq [0,2]$, $\mathrm{supp}\, \fy\subseteq[\frac{1}{2},2]$, and $\{\fy_\ell\}_{\ell\geq0}$ form a partition of unity for $[0,\infty)$, where $\fy_\ell(\cdot):=\fy(2^{-\ell}\cdot)$ for any positive integer $\ell>0$. We define the Littlewood--Paley projection $P_\ell^\bfL$, $\ell \in \bbN_{0}$, relative to the operator $\bfL$ by the relation
\begin{equation}
    \wtilcalF_\bfL[P_\ell^\bfL\bmf](k)=\fy_\ell(k)\wtilcalF_\bfL[\bmf](k),
\end{equation}
where $\bmf\in L^2_{r\ud r}$. Note that $P_\ell^\bfL P_{\ell'}^\bfL=0$ if $|\ell-\ell'|>1$ and $P_c^\bfL\bmf=\sum_{\ell\geq0}P_\ell^\bfL \bmf$ almost everywhere. Similarly, for any interval $I\in[0,\infty)$ we denote by $\fy_I$ a smooth positive function supported in $2I\cup2^{-1}I$ and equal to one on $I$. Accordingly we define $P_I^\bfL$ by
\begin{equation}
    \wtilcalF_\bfL[P_I^\bfL\bmg](k)=\fy_I(k)\wtilcalF_\bfL[\bmg](k).
\end{equation}
We can then write
\begin{align}  
        P_\ell^\bfL\bmf(r)=\int_0^\infty \bfK_\ell(r,s)\bmf(s)s\ud s,\qquad P_I^\bfL\bmf(r)=\int_0^\infty \bfK_I(r,s)\bmf(s)s\ud s, 
\end{align}
where
\begin{align}
    \begin{split}
        &\bfK_\ell(r,s)=\int_0^\infty E(r,k)\overline{E(s,k)}^t\fy_\ell(k) k \ud k,\\
        &\bfK_I(r,s)=\int_0^\infty E(r,k)\overline{E(s,k)}^t\fy_I(k) k \ud k.
    \end{split}
\end{align}
The following lemma contains the estimates we need on the frequency projections $P_\ell^\bfL$. As usual, the inspection of the proof shows that similar estimates hold if $P_\ell^\bfL$ is replaced by $P_I^\bfL$ where $I$ is an interval with length of order $2^\ell$. 
\begin{lemma}\label{lem:LPbounds1}
    For any $\ell\geq0$
    \begin{align}
            \|P_\ell^\bfL \bmf\|_{L^1_{r\ud r}}\leq C \|\bmf\|_{L^1_{r\ud r}},\label{eq:Pellboound1}
    \end{align}
    where the constant $C$ is independent of $\ell$.
\end{lemma}
The proof of Lemma~\ref{lem:LPbounds1} uses a few facts about the distorted Fourier basis that are established in Section~\ref{sec:working_title_transference}. Their proofs are independent of Lemma~\ref{lem:LPbounds1}. We have decided to state the lemma and its proof here, because it relates more naturally with the material in Section~\ref{sec:linear_decay_estimates}.
\begin{proof}[Proof of Lemma~\ref{lem:LPbounds1}]
    Note that
    \begin{align*}
        P_\ell^\bfL \bmf(r)=\int_0^\infty \bfK(r,s)\bmf(s)s\ud s,
    \end{align*}
    where
    \begin{align*}
        \bfK(r,s)=\int_0^\infty E(r,k)\overline{E(s,k)}^t\fy_\ell(k)k \ud k.
    \end{align*}
    We want to prove that $\sup_{s\geq0}\|\bfK\|_{L^1_{r\ud r}}\lesssim 1$.
    Let $\{\chi_0,\chi_1\}$ be a partition of unity for $[0,\infty)$ with $\chi_0$ supported in $[0,2]$ and $\chi_1$ in $[1,\infty)$. Then we write
    \begin{align*}
        \bfK(r,s)= \sum_{i,j=0}^1 \bfK_{i,j}(r,s),
    \end{align*}
    with 
        \begin{align*}
        \bfK_{i,j}(r,s)=\int_0^\infty E(r,k)\overline{E(s,k)}^t\fy_\ell(k) \chi_i(rk)\chi_j(sk)k \ud k.
    \end{align*}
    We will prove that
    \begin{align*}
        \sup_{s\geq0}\|\bfK_{1,1}\|_{L^1_{r\ud r}}\lesssim 1,\qquad \sup_{s\geq0}\|\bfK_{0,1}\|_{L^1_{r\ud r}}\lesssim 1,\qquad \sup_{s\geq0}\|\bfK_{0,0}+\bfK_{1,0}\|_{L^1_{r\ud r}}\lesssim 1.
    \end{align*}
    Starting with $\bfK_{1,1}$ using Theorem~\ref{propdFT} we write
    \begin{align*}
        \bfK_{1,1}(r,s)=\sum_{\iota_1,\iota_2=\pm}\int_0^\infty e^{i(\iota_1r+\iota_2s)k}\frac{\psi_{\iota_1,\iota_2}(r,s,k)}{\sqrt{rk}\sqrt{sk}}\chi_1(rk)\chi_1(sk)\fy_\ell(k)k \ud k,
    \end{align*}
    where $\psi$ satisfies $|(k\partial_k)^m\psi_{\iota_1,\iota_2}|\lesssim 1$ for $m=0,1,2$. We only consider the case $\iota_1=-\iota_2=+$, the other cases being similar or easier. With $\psi=\psi_{+,-}$ we want to prove that
    \begin{equation}\label{eq:LPboundednesstemp1}
        \int_0^\infty\Big|\int_0^\infty e^{i(r-s)k}\frac{\psi(r,s,k)}{\sqrt{rk}\sqrt{sk}}\chi_1(rk)\chi_1(sk)(rk)\fy_\ell(k) \ud k\Big| \ud r\lesssim 1.
    \end{equation}
    Consider first the case $\ell\geq1$. If $r\geq 2s$ then $r-s\simeq r$ and after two integration by parts in $k$, and with $\bffy_\ell$ a function whose support is contained in that of $\fy_\ell$, we bound the integral by
    \begin{align*}
        \int_{2^{-\ell-1}}^\infty \int_0^\infty \bffy_\ell(k) (rk)^{-\frac{3}{2}}\ud k \ud r\lesssim \int_{2^{-\ell-1}}^\infty(2^\ell r)^{-\frac{3}{2}}2^\ell \ud r\lesssim 1.
    \end{align*}
    For $r\leq 2s$ we first consider the region $[s-2^{-\ell},s+2^{-\ell}]$. Here we simply bound the integrand in \eqref{eq:LPboundednesstemp1} by $\fy_\ell(k)$ which shows that the integral is bounded by direct integration. On the interval $[s+2^{-\ell},2s]$ we integrate by parts in $k$ twice and observe that the result can be bounded by
    \begin{align*}
        \int_{s+2^{-\ell}}^{2s}\int_{\{k\simeq2^{\ell}\}}k^{-2}(r-s)^{-2}\ud k \ud r\lesssim 2^{-\ell}\int_{s+2^{-\ell}}^{2s}\frac{\ud r}{(r-s)^2}\lesssim 1.
    \end{align*}
    The interval $0\leq r \leq s-2^{-\ell}$ is treated similarly. In the case $\ell=0$ we continue to use the notation $\bffy_\ell$ and also let $\bfchi_1$ denote a function supported in $[1,\infty)$. Then for $r\geq 2s$ we again integrate by parts twice and bound \eqref{eq:LPboundednesstemp1} by
      \begin{align*}
        \int_{0}^\infty \int_0^\infty \bffy_0(k)\bfchi_1(rk)\bfchi_1(sk)(sk)^{-\frac{1}{2}} (rk)^{-\frac{3}{2}} \ud r \ud k\lesssim \int_{s^{-1}}^2 s^{-1}k^{-2}\ud k\lesssim 1.
    \end{align*}
    For $s+1\leq r \leq 2s$ we use the more refined structure of the Fourier basis $E(r,k)$ and $E(s,k)$ in the region where $rk\geq 1$, $sk\geq 1$, and $k\leq 1$ from Theorem~\ref{propdFT}. We consider only the most difficult contributions in the product $E(r,k)\overline{E(s,k)}^t$, namely those of $J_0(kr)J_0(ks)$ (which is similar to $J_1(kr)J_1(ks)$) and $(\log(k/2))^{-1}\big(H_0^{(2)}(kr)J_0(ks)-\overline{H_0^{(2)}(ks)}J_0(kr)\big)$ (which is more difficult than $\big(H_1^{(2)}(kr)J_2(ks)+H_1^{(2)}(ks)J_2(kr)\big)\uppsi_{1,2}^3(k)$). Here we have dropped the $\jap{k}^{-1}$ factors as they are smooth with bounded derivatives of all order for small $k$. Starting with $J_0(rk)J_0(sk)$, in view of the large parameter asymptotics of Bessel functions from Remark~\ref{rem:Bessels1} for suitable constant $\upalpha_j$ we write
    \begin{align*}
        k(rs)^{\frac{1}{2}}J_0(kr)J_0(ks)&=\big(\upalpha_0\cos(kr)+\upalpha_1\sin(kr)(kr)^{-1})\big)\big(\upalpha_0\cos(ks)+\upalpha_1\sin(ks)(ks)^{-1}\big)\\
        &\quad+O((kr)^{-2})+O((ks)^{-2})\\
        &=\frac{\upalpha_0^2}{2}\big(\cos(k(r+s))+\cos(k(r-s))\big)+\frac{\upalpha_0\upalpha_1}{2(rk)}\big(\sin(k(r+s))+\sin(k(r-s))\big)\\
        &\quad+\frac{\upalpha_0\upalpha_1}{2(sk)}\big(\sin(k(r+s))-\sin(k(r-s))\big)+O((rk)^{-2})+O((sk)^{-2}).
    \end{align*}
 The trigonometric functions with argument $k(r+s)$ are easier to handle by integration by parts so we ignore those. For the contributions of the $O((rk)^{-2})$ and $O((rsk)^{-2})$ it suffices to observe that
 \begin{align*}
     \int_{s+1}^{2s}\int_{1/s}^1(sk)^{-2}\ud k \ud r+\int_0^{s-1}\int_{1/r}^1\sqrt{\frac{r}{s}}(kr)^{-2}\ud k \ud r\lesssim 1.
 \end{align*}
 For the contribution of $\cos(k(r-s))$ we present the argument on the interval $[s+1,2s]$, the case of $[0,s-1]$ being more of the same. We need to show that the integral
 \begin{equation}
     \int_{s+1}^{2s}\Big|\int_0^\infty \cos(k(r-s))\chi_1(rk)\chi_1(sk)\fy_0(k)\ud k\Big|\ud r
 \end{equation}
 is uniformly bounded for $s\geq1$. In the inner integral we integrate by parts to write
 \begin{align}
     \int_0^\infty \cos(k(r-s))\chi_1(rk)\chi_1(sk)\fy_0(k)\ud k=-\int \frac{\sin(k(r-s))}{r-s}\partial_k\big(\chi_1(rk)\chi_1(sk)\fy_0(k)\big)\ud k.
 \end{align}
 If the derivative falls on $\chi_1(rk)$ or $\chi_1(sk)$ then we in the support of the resulting function $k\simeq s^{-1}$ so we bound $\frac{\sin(k(r-s))}{r-s}$ by $s^{-1}$ which is sufficient. If the derivative falls on $\fy_0(k)$ then we note that in the support of $\partial_k\fy_0(k)$, $k\simeq1$, so we integrate by parts again to write
 \begin{equation}
     \int \frac{\sin(k(r-s))}{r-s}\chi_1(rk)\chi_1(sk)\partial_k\fy_0(k)\ud k=\int \frac{\cos(k(r-s))}{(r-s)^2}\partial_k\big(\chi_1(rk)\chi_1(sk)\partial_k\fy_0(k)\big)\ud k\lesssim (r-s)^{-2},
 \end{equation}
 which is integrable on $[s+1,2s]$. 
 Next, we consider $(\log(k/2))^{-1}\big(H_0^{(2)}(kr)J_0(ks)-\overline{H_0^{(2)}(ks)}J_0(kr)\big)$
 which reduces to analyzing $(\log(k/2))^{-1}\big(Y_0(kr)J_0(ks)+Y_0(ks)J_0(kr)\big)$. Here an inspection of the asymptotics of Bessel functions from Remark~\ref{rem:Bessels1} reveals that up to error of order $O((rk)^{-2})$ and $O(
 (sk)^{-2})$ the resulting trigonometric functions have $k(r+s)$ and not $k(r-s)$ as their arguments. The desired bound then follows as before. For the error terms in the product $E(r,k)\overline{E(s,k)}^t$ where exact trigonometric cancellations may not be present, we consider the most difficult sample error $\chi_1(sk)\chi_1(rk)k^{-1}(rs)^{-\frac{1}{2}}\jap{\log k}^{-2}e^{ik(r-s)}$. Here we split the $k$ integral at $(r-s)^{-1}$. For $k\leq (r-s)^{-1}$ we integrate directly while for $k\geq (r-s)^{-1}$ we integrate by parts twice in $k$. For $k\leq (r-s)^{-1}$ this gives
  \begin{align*}
     \int_{s+1}^{2s}\int_0^1\Big|\chi_{1}(rk)\chi_1(sk)\chi_0((r-s)k)\frac{e^{ik(r-s)}}{\jap{\log k}^{2}}\Big|\ud k\,\ud r\lesssim \int_{s+1}^{2s}\frac{\ud r}{(r-s)\jap{\log(r-s)}^2}\lesssim 1,
 \end{align*}
 where the bound is independent of $s\geq1$. For $k\geq (r-s)^{-1}$ using integration by parts twice in $k$ we observe that
 \begin{align*}
     &\int_{s+1}^{2s}\Big|\int_0^1\chi_{1}(rk)\chi_1(sk)\chi_1((r-s)k)\frac{e^{ik(r-s)}}{\jap{\log k}^{2}}\ud k\Big|\ud r\\
     &\lesssim \int_{s+1}^{2s}\int_{0}^1\Big|\partial_k^2\Big(\chi_{1}(rk)\chi_1(sk)\chi_1((r-s)k)\jap{\log k}^{-2}\Big)(r-s)^{-2} \ud k\,\ud r\\
     &\lesssim \int_{s+1}^{2s}\big((r-s)^{-1}\jap{\log(r-s)}^{-2}+(r-s)^{-2}\big)\ud r,
 \end{align*}
 which is bounded independently of $s\geq1$. This completes the analysis on $[s+1,2s]$, and the analysis on $[0,s-1]$ is similar.

 For the other kernels $\bfK_{i,j}$, $(i,j)\neq(1,1)$, we only present the details for the case $\ell=0$ which is more delicate. For $\bfK_{0,1}$, using the asymptotics in from Theorem~\ref{propdFT} for both $E(r,k)$ and $\overline{E(s,k)}^t$, it suffices to observe that
 \begin{align*}
     \sup_{s\geq0}\int_0^\infty \int_0^1 \chi_0(rk)\chi_1(sk)(sk)^{-\frac{1}{2}}\frac{\log(r+2)}{\log(k/2)}rk \,\ud k\, \ud r\lesssim 1.
 \end{align*}
 For $\bfK_{0,0}+\bfK_{1,0}$ in the region $r\leq \max\{1,s\}$ we again use the asymptotics from Theorem~\ref{propdFT}. For $r\geq \max\{1,s\}$ we use the more refined representations from Lemmas~\ref{lem:Fbasisrepalt1} and~\ref{lem:bfqbfpbds1}, as well as the relations~\eqref{eq:Besselders1} for $E(r,k)$, while for $\overline{E(s,k)}^t$ we use the asymptotics in the region where $sk\leq1$ from Theorem~\ref{propdFT}. We only consider the contribution of the $\bfq$ coefficients in Lemma~\ref{lem:Fbasisrepalt1}, because in view of Lemma~\ref{lem:bfqbfpbds1} the Volterra coefficients $\bfp^Y$ and $\bfp^J$ come with extra $r$-decay that make the estimates easier. Then with $K_n$ representing $J_n$ or $Y_n$ we consider the representative kernel $$\frac{K_0(rk)\chi_0(sk)s\log (2+s)}{\jap{s}\log(k/2)\upalpha_K(k)}\uppsi(s,k),$$ where $\upalpha_J(k)\equiv1$ and $\upalpha_Y(k)=\log(k/2)$, and $|(s^{-1}\partial_k)^{j}\uppsi(s,k)|\lesssim 1$ for $j=0,1,2$. Using $(rk)K_0(rk)=r^{-1}\partial_k\big((rk)K_1(rk)\big)$ and integrating by parts we get
 \begin{align*}
     &\int_{\max\{1,s\}}^\infty\Big|\int_0^1 (rk)K_0(rk)\fy_0(k)\chi_0(sk)\uppsi(s,k)\frac{\log(2+s)}{\log (k/2)\upalpha_K(k)}\ud k \Big| \ud r\\
     &= \int_{\max\{1,s\}}^\infty\Big|\int_0^1 (\chi_0(rk)+\chi_1(rk))kK_1(rk)\partial_k\big(\fy_0(k)\chi_0(sk)\uppsi(s,k)\frac{\log(2+s)}{\log (k/2)\upalpha_K(k)}\big)\ud k \Big| \ud r\\
     &\leq I+II,
 \end{align*}
 with $I$ corresponding to the integral with $\chi_0(rk)$ and $II$ to the integral with $\chi_1(rk)$. For $I$ if $\partial_k$ falls on $\fy_0$ then $k\simeq1$ and hence $r\simeq 1$ and the boundedness of the resulting integral follows. If the derivative falls on $\chi_0(sk)$ then we must have $s\gtrsim 1$ and $r\lesssim s$ for the resulting integral to be nonzero. In this case it suffices to observe that (since $rk\leq1$ for $I$, $K_1=Y_1$ is the more difficult choice if we drop the favorable term $1/\upalpha_K$)
 \begin{equation*}
     \int_s^{2s}\int_{1/s}^{2/s}\frac{\ud k \,\ud r}{rk}\lesssim 1.
 \end{equation*}
 When $\partial_k$ falls on $\uppsi$ or $(\log(k/2))^{-1}$ we observe that
 \begin{equation*}
     \int_{\max\{1,s\}}^\infty\int_0^{1}\chi_0(sk)\chi_0(rk)\Big(\frac{s}{r|\log(k/2)|}+srk^2+\frac{\log(2+s)}{rk|\log(k/2)|^3}+\frac{rk\log(2+s)}{(\log(k/2))^2}\Big)\ud k \,\ud r\lesssim 1.
 \end{equation*}
 For $II$ we integrate by parts another time, by writing $kK_1(rk)=(rk)^{-1}\partial_k\big(k^2 K_2(rk)\big)$. Integrating by parts and considering the different possibilities of where the two factors of $\partial_k$ can fall, it suffices to observe that
 \begin{align*}
     \int_{\max\{1,s\}}^\infty \int_0^1\chi_0(sk)\chi_1(rk)\Big(\frac{s^2k^{\frac{1}{2}}\log(2+s)}{r^{\frac{3}{2}}|\log(k/2)|}+\frac{\log(2+s)}{r^{\frac{3}{2}}k^{\frac{3}{2}}(\log(k/2))^2}\Big)\ud k\, \ud r\lesssim 1
 \end{align*}
 and
 \begin{equation*}
     \int_{\max\{1,s\}}^\infty\int_{1/r}^{2/r}\Big(\frac{k^{\frac{1}{2}}s\log(2+s)}{r^{\frac{1}{2}}|\log(k/2)|}+\frac{k^{\frac{1}{2}}\log(2+s)}{r^{\frac{1}{2}}k(\log(k/2))^2}\Big)\ud k\,\ud r\lesssim 1.\qedhere
 \end{equation*}
\end{proof}
\subsection{Dispersive estimates}\label{ssecKG} The following proposition contains the main $L^\infty_r$ estimate of this paper. We will subsequently derive a few corollaries that are more useful for applications in the nonlinear setting in \cite{LPPSS3}. The imaginary exponent $i\tau$ in the statement of the proposition is included simply for the purpose of later applying complex interpolation.
\begin{prop}\label{propKG1dec}
For any $\sigma, \tau\in\bbR$,
\begin{align*}
\|e^{\pm it\sqrt{\bfL}}P_c^\bfL\bmf\|_{L^\infty_r} \leq  \frac{C_{\sigma,\tau}}{\jap{t}}\sum_{\ell\geq0} 
2^{-\sigma\ell} {\big\| \bfL^{1+\frac{\sigma+i\tau}{2}}P_\ell^{\mathbf{L}} \bmf \big\|}_{L^1_{rdr}}.
\end{align*}
The constant $C_{\sigma,\tau}$ can be chosen independently of $\sigma$ for $\sigma$ in any bounded interval, and grows at most polynomially in $\tau$. 
\end{prop}
Accepting Proposition~\ref{propKG1dec} for the moment, we prove the corollaries that will be used in \cite{LPPSS3}. For the applications in \cite{LPPSS3} it is convenient to work with non-equivariant functions of the form $e^{i\theta}\bmf$, where $\bmf$ is radial. We use the notation $W^{s,p}_x(\bbR^2)$ to denote the usual $L^p$ Sobolev spaces of regularity $s$ for functions on $\bbR^2$. We will also need to use the non-radial analogue of $\bfL$ in a few places. This is defined as 
\begin{equation}
\label{eq:Lnr}
    \bfL_\nr:=(-\Delta+1)\Id_{2\times2}+\bfV_0, 
\end{equation}  where $\bfV_0$, defined in \eqref{eq:Lzerodefintro1} is viewed as a radial potential on $\bbR^2$, and $-\Delta$ denotes the full Laplacian on $\bbR^2$. The two operators are related by the relation $\bfL_\nr e^{\pm i\theta}\bmf=e^{\pm i\theta}\bfL\bmf$.
\begin{cor}\label{cor:Lpdisp1}
For any $p\in(2,\infty)$, $\gamma>0$, and with $q^{-1}+p^{-1}=1$,
\begin{equation}\label{eq:Lpdispersivebound1}
    \|e^{\pm i t\sqrt{\bfL}}P_c^\bfL \bmf\|_{L^p_{r\ud r}}\lesssim \jap{t}^{-1+\frac{2}{p}}\|e^{i\theta}\bmf\|_{W^{2+\gamma-\frac{4}{p},q}_x(\bbR^2)},
\end{equation}
and 
\begin{equation}\label{eq:Linftydispersiveboundnonsharp1}
    \|e^{\pm i t\sqrt{\bfL}}P_c^\bfL \bmf\|_{L^\infty_r}\lesssim \jap{t}^{-1+\frac{2}{p}}\|e^{i\theta}\bmf\|_{W^{2+\gamma-\frac{2}{p},q}_x(\bbR^2)}.
\end{equation}
The implicit constants in these estimates are allowed to depend on $p$ and $\gamma$. Additionally, for $j=0,1$,
\begin{equation}\label{eq:Linftydispersiveboundnonsharp2}
    \|e^{\pm i t\sqrt{\bfL}}\bfL^{-\frac{j}{2}}P_c^\bfL \bmf\|_{L^\infty_r}\lesssim \jap{t}^{-1}\sum_{\ell=0}^{4-j}\big\|\partial_x^\ell\big(e^{i\theta}\bmf\big)\big\|_{L^1_x(\bbR^2)}.
\end{equation}
Moreover, for any $p\in(2,\infty)$ and for $j=0,1$,
    \begin{equation}\label{eq:Linftyweighteddispersivebound1}
        \|e^{\pm it\sqrt{\bfL}}\bfL^{-\frac{j}{2}}P_c^\bfL\bmf\|_{L^\infty_{r}}\lesssim \jap{t}^{-1+\frac{2}{p}}\|\jap{x}\jap{D}^{2-j} (e^{i\theta}\bmf)\|_{L^2_x(\bbR^2)}.
    \end{equation}
\end{cor}
\begin{proof}
    First using interpolation we prove an $L^p_{r\ud r}$ estimate for frequency localized functions. For any $\ell\geq0$ and $\delta\in\bbR$, and for $z\in\{\Re z\in[0,1]\}$ consider (we suppress the dependence on $t$, $\ell$, and $\delta$ from the notation $T^\pm_z$)
    \begin{align*}
        T^{\pm}_z:=e^{\pm it \sqrt{\bfL}}2^{(1-z+\delta)\ell}\bfL^{-1-\frac{1-z+\delta}{2}}P_\ell^\bfL P_c^\bfL.
    \end{align*}
    By Proposition~\ref{propKG1dec} and Lemma~\ref{lem:LPbounds1},
    \begin{align*}
        \|T^\pm_{1+i\tau}\|_{ L^1_{r\ud r}\to L^\infty_r}\leq C(\tau)\jap{t}^{-1},   
    \end{align*}
    where the constant $C(\tau)$ grows at most polynomially with respect to $\tau\in\bbR$. On the other hand by Theorem~\ref{propdFT}, in particular the Plancherel statement,
    \begin{align*}
        \|T^\pm_{i\tau}\|_{ L^2_{r\ud r}\to L^2_{r\ud r}}\leq C 2^{-2\ell},   
    \end{align*}
    with a constant that is independent of $\tau\in\bbR$. By interpolation, cf. \cite[Theorem 1.3.7]{GrafakosCFA},
    \begin{equation}
        \|T^\pm_{1-\frac{2}{p}}\|_{L^q_{r\ud r}\to L^p_{r\ud r}}\lesssim \jap{t}^{-1+\frac{2}{p}}2^{-\frac{4\ell}{p}},\qquad p\in(2,\infty), \quad \frac{1}{p}+\frac{1}{q}=1.
    \end{equation}
    In particular, if for a given $p\in(2,\infty)$ and $\gamma\geq0$ we choose $\delta=-\frac{6}{p}+\gamma$ then we have shown that 
    \begin{equation}
        \|e^{\pm i t\sqrt{\bfL}}P_\ell^\bfL P_c^\bfL \bmf\|_{L^p_{r\ud r}}\lesssim 2^{-\gamma\ell}\jap{t}^{-1+\frac{2}{p}}\|\bfL^{1-\frac{2}{p}+\frac{\gamma}{2}}\bmf\|_{L^q_{r\ud r}}.
    \end{equation}
    If $\gamma>0$ the right-hand side can be summed up in $\ell\geq0$. Since $\bfL^{1-\frac{2}{p}+\frac{\gamma}{2}}\bmf=e^{-i\theta}\bfL_\nr^{1-\frac{2}{p}+\frac{\gamma}{2}}(e^{i\theta}\bmf)$ we conclude that
    \begin{equation}
        \|e^{\pm i t\sqrt{\bfL}}P_c^\bfL \bmf\|_{L^p_{r\ud r}}\lesssim \jap{t}^{-1+\frac{2}{p}}\|e^{i\theta}\bmf\|_{W^{2+\gamma-\frac{4}{p},q}_x(\bbR^2)},\qquad \gamma>0,\quad p\in(2,\infty),\quad \frac{1}{p}+\frac{1}{q}=1.
    \end{equation}
    For the $L^\infty$ bound \eqref{eq:Linftydispersiveboundnonsharp1} we use Sobolev and elliptic regularity as well as \eqref{eq:Lpdispersivebound1} with $\gamma$ replaced by $\frac{\gamma}{2}$. Indeed,
    \begin{align*}
       \|e^{\pm i t\sqrt{\bfL}}P_c^\bfL\bmf\|_{L^\infty_{r}}&\lesssim \|e^{\pm i t\sqrt{\bfL_\nr}}e^{i\theta}P_c^\bfL\bmf\|_{L^\infty_{x}(\bbR^2)}\lesssim \|e^{\pm i t\sqrt{\bfL_\nr}}e^{i\theta}\bmf\|_{W^{\frac{2}{p},p}_{x}(\bbR^2)}\\
       &\lesssim \|\bfL_\nr^{\frac{1}{p}+\frac{\gamma}{4}}e^{\pm i t\sqrt{\bfL_\nr}}e^{i\theta}\bmf\|_{L^p_{x}(\bbR^2)}+\|e^{\pm i t\sqrt{\bfL_\nr}}e^{i\theta}\bmf\|_{L^p_{x}(\bbR^2)}\\
       &\lesssim \jap{t}^{-1+\frac{2}{p}}\big(\|e^{i\theta}\bfL^{\frac{1}{p}+\frac{\gamma}{4}}\bmf\|_{W_x^{2+\frac{\gamma}{2}-\frac{4}{p},q}(\bbR^2)}+\|e^{i\theta}\bmf\|_{W_x^{2+\frac{\gamma}{2}-\frac{4}{p},q}(\bbR^2)}\big)\\
       &\lesssim \jap{t}^{-1+\frac{2}{p}}\big(\|\bfL_\nr^{1+\frac{\gamma}{2}-\frac{1}{p}}(e^{i\theta}\bmf)\|_{L^q_x(\bbR^2)}+\|e^{i\theta}\bmf\|_{L^q_{x}(\bbR^2)}\big)\\
       &\lesssim \jap{t}^{-1+\frac{2}{p}}\|e^{i\theta}\bmf\|_{W_x^{2+\gamma-\frac{2}{p},q}(\bbR^2)}.
    \end{align*}
    Using this estimate with $\gamma=\frac{2}{p}$, applying H\"older with $\frac{1}{2}+\frac{1}{m}=\frac{1}{q}$ (to $\bmf$ and $\bfL^{-\frac{1}{2}}\bmf$) and noting that $\jap{x}^{-1}\in L^m_x(\bbR^2)$ gives \eqref{eq:Linftyweighteddispersivebound1}. Here we also use that $\|e^{i\theta}\bmf\|_{W^{2,q}_x(\bbR^2)}\simeq \|\jap{D}^2(e^{i\theta}\bmf)\|_{L^q_x(\bbR^2)}$ and that $\bfL^{-\frac{j}{2}}_\nr$ is bounded in $L^q_x(\bbR^2)$. Finally, \eqref{eq:Linftydispersiveboundnonsharp2} follows directly from Proposition~\ref{propKG1dec} by choosing $\sigma=1,2$ and summing up dyadically. 
\end{proof}
\begin{cor}\label{cor34KG1dec}
    Suppose $\bmv$ is a solution of \eqref{KGj'}. For any $p\in(2,\infty)$, $\gamma>0$, and with $q^{-1}+p^{-1}=1$, (the implicit constant is allowed to depend on $p$ and $\gamma$)
    \begin{align}
        \sup_{r\geq0}|\bmv(t,r)|&\lesssim \jap{t}^{-1+\frac{2}{p}}\|e^{i\theta}\bmv_0\|_{W^{2+\gamma-\frac{2}{p},q}_x(\bbR^2)}+\jap{t}^{-1+\frac{2}{p}}\|e^{i\theta}\bmv_1\|_{W^{1+\gamma-\frac{2}{p},q}_x(\bbR^2)}\\
        &\quad+\int_0^t\jap{t-s}^{-1+\frac{2}{p}}\|e^{i\theta}\bmF\|_{W^{1+\gamma-\frac{2}{p},q}_x(\bbR^2)}\ud s,
    \end{align}
    and 
        \begin{align}
        \sup_{r\geq0}|\bmv(t,r)|&\lesssim \jap{t}^{-1}\sum_{\ell=0}^4\|\partial_x^\ell(e^{i\theta}\bmv_0)\|_{L^1_x(\bbR^2)}+\jap{t}^{-1}\sum_{\ell=0}^3\|\partial_x^\ell(e^{i\theta}\bmv_1)\|_{L^1_x(\bbR^2)}\\
        &\quad+\sum_{\ell=0}^3\int_0^t\jap{t-s}^{-1}\|\partial_x^\ell(e^{i\theta}\bmF)\|_{L^1_x(\bbR^2)}\ud s.
    \end{align}
\end{cor}
\begin{proof}
    The desired result then follows from the usual representation formula for the solution $\bmv$ of \eqref{KGj'}, and the estimates \eqref{eq:Linftydispersiveboundnonsharp1} and \eqref{eq:Linftydispersiveboundnonsharp2}. Here for the initial evolution of the time derivative and the Duhamel term we also use the boundedness of $\bfL_{\nr}^{-\frac{1}{2}}$ in $L^q_x(\bbR^2)$.
\end{proof}

We now turn to the proof of Proposition~\ref{propKG1dec}.

\begin{proof}[Proof of Proposition~\ref{propKG1dec}]
Let $\gamma=1+\frac{\sigma+i\tau}{2}$. In the notation introduced at the beginning of Section~\ref{ssecLP} let $\bfP_\ell^\bfL:=P_{[2^{\ell-1},2^{\ell+1}]}^\bfL$ for $\ell>0$ and $\bfP_0^\bfL:=P_{[0,1]}^\bfL$. Denote the corresponding kernels by $\bffy_\ell:=\fy_{[2^{\ell-1},2^\ell]}$, for $\ell>0$, and $\bffy_0=\fy_{[0,1]}$. Note that $\bfP_\ell^\bfL P_\ell^\bfL=P_\ell^\bfL$, so we can write
\begin{align*}
    e^{\pm i t\sqrt{\bfL}}\bfP_c^\bfL\bmf=\sum_{\ell\geq0}e^{\pm i t\sqrt{\bfL}}P_\ell^\bfL\bfP_c^\bfL\bmf=\sum_{\ell\geq0}e^{\pm i t\sqrt{\bfL}}\bfL^{-\gamma}\bfP_\ell^\bfL \bfP_c^\bfL \bfL^\gamma P_\ell^\bfL\bmf.
\end{align*}
Using the distorted Fourier transform we have
\begin{align*}
    e^{\pm i t\sqrt{\bfL}}\bfL^{-\gamma}\bfP_\ell^\bfL \bfP_c^\bfL \bfL^\gamma P_\ell^\bfL\bmf(r)=\int_0^\infty \bfK^j_{\pm,\ell}(t,r,s)\bfL^\gamma P_\ell^\bfL\bmf(s)s\ud s,\qquad \gamma=1+\frac{\sigma+i\tau}{2},
\end{align*}
where 
\begin{align*}
    \bfK_{\pm,\ell}(t,r,s)=\int_0^\infty \jap{k}^{-2-\sigma-i\tau}e^{\pm it\jap{k}}E(r,k)\overline{E(s,k)}^t\bffy_\ell(k)k\ud k.
\end{align*}
The proposition will follow if we can show that
\begin{align*}
\sup_{r,s\geq 0}\big| \mathbf{K}_{\pm,\ell}(t,r,s) \big| \lesssim 2^{-\sigma\ell} \jap{t}^{-1}. 
\end{align*}
The estimates for the two signs $\pm$ are almost identical so we only consider the $+$ sign. In our proof we will perform up to one integration by parts in $k$. When the derivative falls on $\jap{k}^{-2-\sigma-i\tau}$ this can produce a constant that grows at most linearly in $\sigma$ and $\tau$. This explains the claim about  the constant $C_{\sigma,\tau}$ in the statement of the proposition.
By symmetry we may assume, without loss of generality, that $r\geq s$.
We let $\chi:[0,\infty)\to[0,1]$,  be a smooth cutoff function equal to $1$ in $[0,1]$, decreasing, and vanishing on $[2,\infty)$, and denote $\chi^c := 1-\chi$. We split the kernel into three 
main pieces:
\begin{align*}
\nonumber
\mathbf{K}_{+,\ell}(t,r,s) & = \mathbf{K}_{1,\ell}(t,r,s) + \mathbf{K}_{2,\ell}(t,r,s) + \mathbf{K}_{3,\ell}(t,r,s),
\\
\mathbf{K}_{1,\ell}(t,r,s) & := \int_0^\infty \jap{k}^{-2-\sigma-i\tau}e^{it\jap{k}}E(r,k)\overline{E(s,k)}^t\bffy_\ell(k)\chi^c(sk )k\ud k,
\\
\mathbf{K}_{2,\ell}(t,r,s) & := \int_0^\infty \jap{k}^{-2-\sigma-i\tau}e^{ it\jap{k}}E(r,k)\overline{E(s,k)}^t\bffy_\ell(k) \chi(rk )\chi(sk )k\ud k,
\\
\mathbf{K}_{3,\ell}(t,r,s) & :=  \int_0^\infty \jap{k}^{-2-\sigma-i\tau}e^{ it\jap{k}}E(r,k)\overline{E(s,k)}^t\bffy_\ell(k) \chi^c(rk )\chi(sk )k\ud k.
\end{align*}
We will prove that 
\begin{align}\label{prKG1main}
\sup_{r,s\geq 0} \big|2^{(2+\sigma)\ell} \mathbf{K}_{j,\ell}(r,s) \big| \lesssim 2^{2\ell} \jap{t}^{-1}, \qquad j=1,2,3,\qquad \ell\geq0.   
\end{align}

\noindent
{\it Estimate of $\mathbf{K}_{1,\ell}$.}
Since we have assumed that $r\geq s$, on the support of $\mathbf{K}_{1,\ell}$ we have $rk  \geq sk  \gtrsim 1$. Using the asymptotics for the Fourier basis $E$ from Theorem~\ref{propdFT} we need to estimate (having artificially inserted $2^{(2+\sigma)\ell}$)
\begin{align}\label{prKG16}
\begin{split}
L_{\eps_1,\eps_2}(t,r,s) := 
  \int_0^\infty \frac{1}{\sqrt{rs }}   e^{\eps_1 irk }  \psi_{\eps_1}(r,k) e^{\eps_2 isk }  \psi_{\eps_2}(s,k)   e^{it\sqrt{k^2+1}} \, \bffy_\ell(k) \chi^c(sk ) \frac{2^{(2+\sigma)\ell} }{\jap{k}^{2+\sigma}} \,  dk, 
\end{split}
\end{align}
where $\eps_1,\eps_2 \in \{+,-\}$ and $\psi_{\eps_m}$  satisfy $|k^j\partial_k^j\psi_{\eps_m}|\lesssim 1$ for $j=0,1,2$. 
We only consider the case $(\eps_1,\eps_2) = (-,+)$ since the other cases are similar or easier.
Therefore writing
\begin{align}\label{prKG1L}
\begin{split}
L(t,r,s) & := \int_0^\infty \frac{1}{\sqrt{rk }} \frac{1}{\sqrt{sk }} e^{i t S(r,s,k,t)}  A(r,s,k)    k \, dk ,
\\ S(r,s,k,t) & := \frac{k}{t}(s-r) + \sqrt{k^2+1},
\\ A(r,s,k) & := \big( 1 + \phi_-(r,k)\big) \big( 1 + \phi_{+}(s,k)\big)(2^\ell\jap{k}^{-1})^{2+\sigma}\bffy_\ell(k) \chi^c(sk ),
\end{split}
\end{align}
we aim to prove that
\begin{align}\label{prKG1main1}
\sup_{r,s\geq 0} \big| L(t,r,s) \big| \lesssim 2^{2\ell} \jap{t}^{-1}. 
\end{align}
%
%
First note that we may restrict to $t\geq1$,
because otherwise \eqref{prKG1main1} follows by crudely bounding the integrand in absolute value. 
For $t\geq1$ we observe that according to \eqref{prKG1L} and in the support of the integral,
\begin{align*}
& \partial_k  S(r,s,k,t) = \frac{1}{t}(s-r) + \frac{k}{\sqrt{k^2+1}},
\qquad \partial_k^2 S(r,s,k,t) = (k^2+1)^{-3/2} \approx 2^{-3\ell}.
\end{align*}
In what follows we will often omit some of the arguments $(r,s,k,t)$ for ease of notation,
and when this causes no confusion. The phase $S$ has a unique stationary point, which we denote by $k_0$, whenever $\frac{1}{t}(s-r) \in (-1,0)$. 

First consider the case $\ell=0$.  If $k_0s\leq\frac{1}{2}$ then $|\partial_kS|\simeq k$ for $k$ in the support of $A$ and we can bound $L$ using one integration by parts by
\begin{align*}
   t^{-1}+\frac{1}{t\sqrt{rs}} \int_{1/s}^1\frac{\ud k}{k^2}\lesssim t^{-1}.
\end{align*}
Here we have used that in the support of $A$ we have $sk\geq1$. If $k_0>2$ then $|\partial_kS|$ is bounded below by a constant (independent of $r,s,t$) and we again get the desired bound by integration by parts. We now assume that $\frac{1}{2s}\leq k_0\leq 2$. Then since $s^{-1}\lesssim k_0\simeq t^{-1}(r-s)\leq rt^{-1}$, we have $(rs)^{-1}\lesssim t^{-1}$. It follows that on the interval $[k_0-\frac{1}{\sqrt{rs}},k_0+\frac{1}{\sqrt{rs}}]$ we bound the integral in $L$ crudely by bounding the integrand by $(rs)^{-\frac{1}{2}}$. Outside of this interval we observe that
\begin{align*}
    |\partial_kS(k)|=|\partial_kS(k)-\partial_kS(k_0)|\simeq |k_0-k|.
\end{align*}
After integration by parts, and observing that the boundary terms are bounded by $t^{-1}$, we can bound $L$ in the interval $[\frac{1}{s},k_0-\frac{1}{\sqrt{rs}}]$ (if non-empty) by
\begin{align*}
    t^{-1}+ \frac{1}{t\sqrt{st}}\int_{1/s}^{k_0/2}\frac{\ud k}{k^2}+\frac{1}{t\sqrt{rs}}\int_{k_0/2}^{k_0-\frac{1}{\sqrt{rs}}}\frac{\ud k}{(k_0-k)^2}\lesssim t^{-1}.
\end{align*}
A similar argument can be used in the region $[k_0+\frac{1}{\sqrt{rs}},1]$ by dividing the integral into the regions $[k_0+\frac{1}{\sqrt{rs}},2k_0]$ and $[2k_0,1]$.  


Next we consider $\ell\geq1$. In what follows we will assume that the stationary point $k_0$ is in the support of the integral, otherwise the desired bound can be obtained more easily. In particular, we have $k_0 \approx 2^\ell$ and $r \gtrsim t$.
Let $q_0$ be the smallest integer such that $2^{q_0} \geq t^{-1/2}2^{3\ell/2}$. 
We decompose the support of the integral according to the distance of $k$ to $k_0$ as follows:
\begin{align}\label{prKGmain1q}
\begin{split}
& L(t,r,s) = \sum_{q =q_0}^{\ell+10} L_{q}(t,r,s), 
\\
& L_{q}(t,r,s) := \int_0^\infty \frac{1}{\sqrt{rk }} \frac{1}{\sqrt{sk }} e^{i t S} A(r,s,k)
  \varphi_q^{(q_0)}(k-k_0)  k \, dk,
\end{split}
\end{align}
where $\fy_q^{(q_0)}(k-k_0)$ is a cutoff to $\{|k-k_0|\lesssim 2^{q_0}\}$ when $q=q_0$ and to $\{|k-k_0|\simeq 2^q\}$ if $q>q_0$. Here the summation is understood to have only the term $q=q_0$ if $\ell+10\leq q_0$. Note that the estimate for the term with $q=q_0$ follows by direct integration, using  that $r^{-\frac{1}{2}}\lesssim t^{-\frac{1}{2}} $,  $sk\geq1$, and $2^{q_0}\approx t^{-1/2}2^{3\ell/2}$, giving
\begin{align*}
 \vert L_{q_0}\vert \lesssim t^{-1/2} 
  2^{\ell/2} 
  2^{q_0} \lesssim t^{-1}2^{2\ell}.
\end{align*}
For $q>q_0$ observe that on the support of $L_q$ we have
\begin{align}\label{prKG1main1S}
| \partial_k  S| \approx |k-k_0| \langle k\rangle^{-3} \approx 2^q 2^{-3\ell} \gtrsim t^{-1/2}2^{-3\ell/2}. 
\end{align}
Integrating by parts in $k$ gives
\begin{align*}
L_{q}(t,r,s) = \frac{i}{t} \int_0^\infty e^{i t S} \partial_k  \Big[ \frac{1}{\partial_k  S}
	 \frac{1}{\sqrt{rs}} \,  A \, \varphi_q^{(q_0)}(k-k_0) \,   \Big] \, dk ,
\end{align*}
so that (in what follows we will sometimes  omit the variable $t$ for lighter notation)
\begin{align}\label{prKG1main1q}
\begin{split}
| L_{q} | &\lesssim t^{-1} (I + II + III), \qquad
\\ I(t,r,s) & = \int_0^\infty 
  \frac{|\partial_k^2 S|}{(\partial_k  S)^2}
	\frac{1}{\sqrt{rs}}  \, | A | \,\varphi_q^{(q_0)}(k-k_0)  \, dk ,	
\\
II(t,r,s) & = \int_0^\infty \frac{1}{|\partial_k  S|}
	\frac{1}{\sqrt{rs}} \, | A |
	\,  \partial_k  \varphi_q^{(q_0)}(k-k_0) \, dk ,
\\
III(t,r,s) & = \int_0^\infty \frac{1}{|\partial_k  S|}
	\frac{1}{\sqrt{rs}}  \, \big| \partial_k   A
	 \big| \varphi_q^{(q_0)}(k-k_0) \, dk .
\end{split}
\end{align}
Since $r \gtrsim t$ and $sk  \gtrsim 1$, using also \eqref{prKG1main1S} we obtain
\begin{align*}
| I | \lesssim \int_0^\infty \langle k\rangle^3 2^{-2q}
	\frac{1}{\sqrt{t }}  \frac{1}{\sqrt{s}} \chi^c(sk ) \varphi_q(k-k_0) \, \bffy_\ell(k) \, dk \lesssim 2^{\frac{7\ell}{2}}2^{-q}t^{-\frac{1}{2}}. 
\end{align*}
Then, summing over $q\geq q_0$ gives a bound by $2^{2\ell}t^{-1}$. 
Similarly, we can estimate the second term in \eqref{prKG1main1q}, 
using $r \gtrsim t$ and $sk  \gtrsim 1$, by
\begin{align*}
| II | \lesssim \int_0^\infty \langle k\rangle^3 2^{-q}
	\frac{1}{\sqrt{t }}  \frac{1}{\sqrt{s}}  2^{-q} \varphi_{[2^{q-2},2^{q+2}]}(k-k_0) \, \bffy_\ell(k) \, dk  \lesssim 2^{\frac{7\ell}{2}}2^{-q}t^{-\frac{1}{2}},
\end{align*}
which suffices. For the third term in \eqref{prKG1main1q} we have 
\begin{align*}
| III | \lesssim \int_0^\infty \langle k\rangle^3 2^{-q}
	\frac{1}{\sqrt{t }}  \frac{1}{\sqrt{s}} \varphi_q^{(q_0)}(k-k_0) \, 2^{-\ell} \varphi_{[2^{\ell-2},2^{\ell+2}]}(k) \, dk  \lesssim 2^{\frac{5\ell}{2}}t^{-\frac{1}{2}},
\end{align*}
which is sufficient since $\ell \geq q-10$ and the above can be bounded by $2^{10}2^{\frac{7\ell}{2}}2^{-q}t^{-\frac{1}{2}}$.
This concludes the proof of \eqref{prKG1main1} and therefore 
the proof of \eqref{prKG1main} for $j=1$. 

\medskip
\noindent
{\it Estimate of $\mathbf{K}_{2,\ell}$.}
We use the $rk\leq1$ and $sk\leq 1$ asymptotics in Theorem~\ref{propdFT}, and for each column of $E$ we consider the largest possible contribution. Then it suffices to prove that for $j=1,2$,
\begin{equation*}
   \sup_{r,s\geq0}|\bfK_{2,\ell}^j(r,s)|\lesssim 2^{2\ell}\jap{t}^{-1}, 
\end{equation*}
where
\begin{align*}
    &\bfK_{2,\ell}^1(r,s)=\int_0^\infty e^{it\jap{k}}(rk)(sk)\phi_1(r,s,k)\chi(sk)\chi(rk)\bffy_\ell(k)k\ud k,\\
    &\bfK_{2,\ell}^2(r,s)=\int_0^\infty e^{it\jap{k}}\frac{rs}{\jap{r}\jap{s}}\log(2+r)\log(2+s)\jap{k}^2\frac{(\log(2+k))^2}{\jap{\log k}^2}\phi_2(r,s,k)\chi(sk)\chi(rk)\bffy_\ell(k)k\ud k.
\end{align*}
Here $\phi_j(r,s,k)$ satisfy $|\partial_k^m\phi_j(r,s,k)|\lesssim \max\{\jap{r}^m,\jap{s}^m\}$ for $m=0,1$. When $t\leq1$ direct integration shows that $|\bfK_{2,\ell}^j|\lesssim 2^{\ell}$ for $j=1,2$. For $t\geq1$ we write $e^{it\jap{k}}=\frac{\jap{k}}{itk}\partial_k e^{it\jap{k}}$ and integrate by parts. The resulting integral can then be bounded by $t^{-1}2^{2\ell}$ by inspection. Here we only discuss the most delicate case corresponding to $\ell=0$ and $r,s\geq2$ in $\bfK_{2,\ell}^2$. Here depending on whether $\partial_k$ falls on $\jap{\log k}^{-2}$ or $\phi_2(r,s,k)$ we need to observe that
\begin{align*}
    \log(2+r)\log(2+s)\int_0^{\min\{1/r,1/s\}}\frac{\ud k}{k|\log k|^3}\lesssim 1,
\end{align*}
and
\begin{align*}
    \log(2+r)\log(2+s)\max\{r,s\}\int_0^{\min\{1/r,1/s\}}\frac{\ud k}{(\log k)^2}\lesssim 1.
\end{align*}
\noindent
{\it Estimate of $\mathbf{K}_{3,\ell}$.}
Here in the support of the integral $sk  \lesssim 1 \lesssim rk $. We use the estimates for $E(r,k)$ and $E(s,k)$ from Theorem~\ref{propdFT}. For $E(r,k)$ we will only consider the contribution of $e^{irk}$ in the asymptotics in the region $rk\geq1$, because the contribution of $e^{-irk}$ leads to a non-vanishing phase which is easier to estimate. It then suffices to prove that for $j=1,2$,
\begin{align*}
    \sup_{r,s\geq0}|\bfK_{3,\ell}^j(t,r,s)|\lesssim 2^{2\ell}\jap{t}^{-1},
\end{align*}
where with $S(r,k,t):=\jap{k}-\frac{r}{t}k$,
\begin{align*}
    &\bfK_{3,\ell}^1(t,r,s)=\int_0^\infty e^{it S(r,k,t)}(rk)^{-\frac{1}{2}}(sk)A_1(r,s,k)k\ud k,\\
    &\bfK_{3,\ell}^2(t,r,s)=\int_0^\infty e^{it S(r,k,t)}(rk)^{-\frac{1}{2}}\frac{s\log(2+s)\log(2+k)}{\jap{s}\jap{\log k}}\jap{k}A_2(r,s,k)k\ud k.
\end{align*}
Here $A_j$ is of the form
\begin{align*}
    A_j(r,s,k)=\chi^c(rk)\chi(sk)\bffy_\ell(k)\phi_j(s,k)\psi_j(r,k)
\end{align*}
where $\phi_j$ and $\psi_j$ satisfy $|(k\partial_k)^m\phi_j|\lesssim 1$ and $|(k\partial_k)^m\psi_j|\lesssim 1$ for $m=0,1$. As usual the desired bound follows by direct integration if $t\leq1$, so we assume that $t\geq1$. Let $k_0$ denote the critical point $\frac{k_0}{\jap{k_0}}=\frac{r}{t}$ where $\partial_kS$ vanishes. We start with $\ell=0$. If $k_0\geq 2$ then the phase is non-stationary and the result follows by integration by parts. Similarly if $k_0r\leq \frac{1}{2}$ the $|\partial_kS|\simeq k$ and the result again follows by integration by parts. So we assume that $\frac{1}{2r}\leq k_0\leq 2$. 
The integrand in $\bfK_{3,0}^j$ by $r^{-\frac{1}{2}}$, for both $j=1,2$, is bounded by $(rk)^{-\frac{1}{2}}\frac{s}{\jap{s}}k$. Therefore since $k_0\simeq \frac{r}{t}$, the integral on $[k_0-\frac{\jap{s}}{s\sqrt{t}},k_0+\frac{\jap{s}}{s\sqrt{t}}]$ is bounded by
\begin{align*}
    (rk_0)^{-\frac{1}{2}}t^{-\frac{1}{2}}k_0\simeq t^{-1}.
\end{align*}
On $[r^{-1},k_0-\frac{\jap{s}}{s\sqrt{t}}]$ we integrate by parts and bound the result by
\begin{align*}
    t^{-1}+r^{-\frac{1}{2}}t^{-1}\int_{1/r}^{k_0/2}\frac{\ud k}{k^{\frac{3}{2}}} +t^{-1}(rk_0)^{-\frac{1}{2}}\frac{s}{\jap{s}}k_0\int_{k_0/2}^{k_0-\frac{\jap{s}}{s\sqrt{t}}}\frac{\ud k}{(k_0-k)^2}\lesssim t^{-1}.
\end{align*}
A similar estimate, where we divide the region of integration into $[k_0+\frac{\jap{s}}{s\sqrt{t}},2k_0]$ and $[2k_0,1]$, shows that the integral on $[k_0+\frac{\jap{s}}{s\sqrt{t}},1]$ is bounded by $t^{-1}$.

Next we turn to $\ell\geq1$ and consider the only the most difficult case where $k_0\simeq 2^\ell$. In particular, $r \simeq t$.  We also restrict attention to $\bfK_{3,\ell}^1$ as $\bfK_{3,\ell}^2$ satisfies more favorable estimates. Similarly to \eqref{prKG1L} and \eqref{prKGmain1q} we let 
\begin{align}\label{prKG1main3q}
\begin{split}
& \bfK^1_{3,\ell}(t,r,s) = \sum_{q = q_0}^{\ell+10} \bfK_{q}^1(t,r,s),  
\\
& \bfK_{q}^1(t,r,s) := \int_0^\infty e^{it S(r,k,t)}(rk)^{-\frac{1}{2}}(sk)\varphi_q^{(q_0)}(k-k_0)A_1(r,s,k)k\ud k,
\end{split}
\end{align}
where $q_0$ is the smallest integer such that $2^{q_0} \geq t^{-1/2}2^{3\ell/2}$.
For $q=q_0$ by direct integration and using $r\simeq t$ we obtain $|\bfK_{q_0}^1| \lesssim t^{-1}2^{2\ell}$. When $q>q_0$ we integrate by parts and use the fact that since $r\simeq t$, in the support of the integral
\begin{align*}
    |\partial_kS|\simeq 2^{-3\ell}2^q,\qquad |\partial_k^2S|\simeq 2^{-3\ell},\qquad |\partial_k^m\big((sk)(rk)^{-\frac{1}{2}}k A_1\big)|\lesssim 2^{(\frac{1}{2}-m)\ell}t^{-\frac{1}{2}},~~m=0,1.
\end{align*}
Since the length of the interval of integration is bounded by $\min\{2^{\ell},2^q\}$, it follows that for $q>q_0$,
\begin{equation*}
    |\bfK_q^1|\lesssim t^{-\frac{3}{2}}2^{-q}2^{\frac{7\ell}{2}}.
\end{equation*}
Summing up in $q$ and using that $2^{q_0}\simeq 2^{\frac{3\ell}{2}}t^{-\frac{1}{2}}$, we conclude that $\sum_{q > q_0}^{\ell+10} |\bfK_{q}^1|\lesssim 2^{2\ell}t^{-1}$. This concludes the proof of \eqref{prKG1main} for $j=3$, and of the proposition.
\end{proof}


\subsection{(Integrated) local energy decay} \label{subec:localenergy} 
In this subsection we prove a number of local energy decay estimates, both pointwise in time and and integrated in time. We continue to use $\{\chi_0,\chi_1\}$ to denote a partition of unity for $[0,\infty)$ with $\chi_0$ supported in $[0,2]$ and $\chi_1$ in $[1,\infty)$. We also write $\bfI$ for the identity operator.
\begin{lem}\label{locdec_lem2}
Let $k_0>0$ be fixed. The following inequality holds
\begin{align}\label{locdec_3}
\big\Vert \langle r\rangle^{-1}e^{it\sqrt{\textbf{L}}}\chi_1\big(\sqrt{\mathbf{L}-\bfI}/k_0\big)P_c^\bfL\mathbf{f}\big\Vert_{L^2_{r\ud r}}&\lesssim_{k_0} \dfrac{1}{t}\Vert \langle r\rangle \mathbf{f}\Vert_{L^2_{r\ud r}},  \quad t\geq1.
\end{align}
\end{lem}
\begin{proof}
We omit the dependence of constant on $k_0$ in the proof. For $\bmf$ such that $\widetilde\bmf:=\widetilde\calF[\bmf]$ decays sufficiently fast as $k\to\infty$,
\begin{align*}   e^{it\sqrt{\textbf{L}}}\chi_1\big(\sqrt{\mathbf{L}-\bfI}/k_0\big)P_c^\bfL\mathbf{f}(r)&=\int_0^\infty e^{it\jap{k}}E(r,k)\chi_1(k/k_0)\widetilde\bmf(k)k\ud k\\
&=\frac{i}{t}\int_0^\infty E(r,k)\frac{1}{k}\partial_k\big(\jap{k}\chi_1(k/k_0)\widetilde\bmf(k)\big)k \ud k\\
&\quad+\frac{i}{t}\int_0^\infty (\partial_kE(r,k))\frac{\jap{k}}{k}\chi_1(k/k_0)\widetilde\bmf(k)k \ud k\\
&=:I+II.
\end{align*}
For $I$ we simply observe that $\|I\|_{L^2_{r\ud r}}\lesssim t^{-1}\|\jap{r}\bmf\|_{L^2_{r\ud r}}$ by Plancherel in Theorem~\ref{propdFT} and Lemma~\ref{lem:dkbounds1} (which is more refined near $k=0$ than we need here, and whose proof is independent of the material in this section). For $II$ we write
\begin{align*}
    \|t\jap{r}^{-1}II\|_{L^2_{r\ud r}}^2&\lesssim \int_0^{1/k_0}\Big(\int_0^\infty \chi_0(rk)|\partial_k E(r,k)|\chi_1(k/k_0)|\widetilde\bmf(k)|k\ud k\Big)^2r\ud r\\
    &\quad+\int_0^\infty\jap{r}^{-2}\Big(\int_0^\infty \chi_1(rk)\partial_k E(r,k)\chi_1(k/k_0)\frac{\jap{k}}{k}e^{it\jap{k}}\widetilde\bmf(k)k\ud k\Big)^2r\ud r\\
    &=:II_1+II_2.
\end{align*}
For $II_1$ using the asymptotics in the region $\{rk\leq1\}$ from Theorem~\ref{propdFT} we bound $|\partial_kE(r,k)|\lesssim k^{-1}$ and apply Cauchy-Schwarz in the $k$ integral. For $II_2$ the result follows from the asymptotics of $E(r,k)$ in the region $\{rk\geq1\}$ from Theorem~\ref{propdFT} and the Cotlar--Stein criterion: if
\begin{equation}\label{eq:oscKriterium10}
\calK(r,k)=\chi_{\geq1}(rk)(rk)^{-1/2}e^{\pm irk}\mathfrak m(r,k)
\end{equation}
with
\begin{equation}\label{eq:oscKriterium20}
|(r\partial_r)^a(k\partial_k)^b\mathfrak m(r,k)|\lesssim 1,
\qquad 0\le a,b\le2,
\end{equation}
then the operator $T_{\calK}g(r):=\int_0^\infty \calK(r,k)g(k)k\,\ud k$ is bounded from $L^2_{k\ud k}$ to $L^2_{r\ud r}$; see, for instance, \cite[Lemma~6.3]{LSS1}. 
\end{proof}


In the next lemma we incorporate small frequencies as well. 

\begin{lem}\label{locdec_lem3}
For any $\kappa>0$,
\begin{align}\label{locdec_9}
\big\Vert \langle r\rangle^{-(1+\kappa)}e^{\pm it\sqrt{\textbf{L}}}P_c^\bfL\mathbf{f}\big\Vert_{L^2_{r\ud r}}&\lesssim_\kappa \dfrac{1}{t}\Vert \langle r\rangle^{1+\kappa} \mathbf{f}\Vert_{L^2_{r\ud r}},  \quad t\geq 1. 
\end{align}
\end{lem}
\begin{proof} 
We consider only $e^{it\sqrt{\bfL}}$, the case of $e^{-it\sqrt{\bfL}}$ being more fo the same.  For frequencies $k\geq1/2$ bounded away from zero the result follows from Lemma~\ref{locdec_lem2} with $k_0=1/2$. For small frequencies we write 
\begin{equation}\label{locdec_21}
\langle r\rangle^{-(1+\kappa)} \chi_0(\sqrt{\bfL-\bfI})e^{it\sqrt{\textbf{L}}}P_c^\bfL\mathbf{f}  =  \int_0^\infty \int_0^\infty   \dfrac{e^{it\langle k\rangle}}{\langle r\rangle^{1+\kappa} \langle s\rangle} E(r,k) \overline{E(s,k)}^t \bmf(s)\jap{s} \chi_0(2k)(sk) ds \, dk .
\end{equation}
We are planning to integrate by parts in $k$ in \eqref{locdec_21} and then use Cauchy-Schwarz in $s$. Since the $r$ localization weight is stronger than the $s$ localization weight, we assume that $s\geq r\geq 2$ in the remainder of the proof. We define the kernels
\begin{align*}
\nonumber
\mathbf{K}_{1}(t,r,s) & := 
\int_0^\infty e^{ it\sqrt{k^2+1}} E(r,k)\overline{E(s,k)^t} \chi_0(2k) \chi_1(rk)\chi_1(sk )k\,dk  ,
\\
\mathbf{K}_{2}(t,r,s) & := 
\int_0^\infty e^{ it\sqrt{k^2+1}}E(r,k)\overline{E(s,k)^t} \chi_0(2k) \chi_0(rk ) \chi_0(sk )k\,dk  ,
\\
\mathbf{K}_{3}(t,r,s) & :=  
\int_0^\infty e^{ it\sqrt{k^2+1}}E(r,k)\overline{E(s,k)^t}  
  \chi_0(2k) \chi_0(rk )\chi_1(sk )k\,dk  .
\end{align*}
Our aim is to prove that (entry-wise)
\begin{align}\label{prKG1main2}
\big\Vert \langle r\rangle^{-(1+\kappa)} \langle s\rangle^{-1-\kappa} \mathbf{K}_{j}(t,r,s) \big\Vert_{L^2_{r\ud r}L^2_{s\ud s}} \lesssim t^{-1}, \qquad t\geq1, \qquad j=1,2,3.   
\end{align} 
This last estimate along with \eqref{locdec_21} imply \eqref{locdec_9} by Cauchy-Schwarz. We present the proof of \eqref{prKG1main2} for $j=1,2$ and the case $j=3$ follows from combining the arguments used for $j=1,2$.

\smallskip
\noindent
{\it Estimate of $\mathbf{K}_{1}$.} Let $\phi_\ell(k)$, $\ell\leq-1$, denote a cutoff to the region $\{k\simeq 2^\ell\}$. Then in view of the asymptotics for the Fourier basis from Theorem~\ref{propdFT}, it suffices to prove that
\begin{align}\label{locdec_32}
\Vert K_\ell(t,r,s) \Vert_{L^2_{r\ud r}L^2_{s\ud s}(\{s\geq r\geq2\})} \lesssim t^{-1}2^{\kappa\ell},
\end{align}
where (dropping a factor of $\jap{s}^{-\kappa}$)
\begin{align}\label{locdec_31}
\begin{split}
K_\ell(t,r,s) & :=  \dfrac{1}{r^{\frac{3}{2}+\kappa} s^{\frac{3}{2}}}  \int_0^\infty e^{i t S(r,s,k,t)}  A(r,s,k)  \varphi_\ell(k)\, dk ,
\\ S(r,s,k,t) & := \pm\frac{k}{t}(r\pm s) + \sqrt{k^2+1},
\end{split}
\end{align}
and $A\equiv A(r,s,k)$ is supported in $\{sk\geq rk\geq 1\}$ and satisfies 
\begin{align}\label{locdec_33}
|\partial_k^\alpha A(r,s,k)| \lesssim k^{-\alpha}, \qquad \alpha = 0,1.
\end{align}
We consider only the most difficult choice of signs for $S$, namely
$$S(r,s,k,t)  := \jap{k}-\frac{k}{t}(s-r).$$
%
%
%
%
%
%
To prove \eqref{locdec_32}, first observe that we may restrict to the case $k \gtrsim 1/\sqrt{t}$, otherwise the bound follows from \eqref{locdec_33} and $rk ,sk  \geq 1$. Indeed,   
\begin{align}\label{locdec_16}
\bigg[\int_{1/k}^\infty r^{-2-2\kappa}\ud r\bigg]^{1/2}\lesssim  k^{1/2+\kappa}, \qquad \hbox{and}\qquad \bigg[\int_{1/k}^\infty s^{-2}\ud s\bigg]^{1/2}\lesssim  k^{1/2} .
\end{align}
Then, using Minkowski's inequality the contribution corresponding to the region $\{k\leq t^{-\frac{1}{2}}\}$ is bounded by 
\begin{align}\label{locdec_17}
 \int_0^{ t^{-1/2} } k^{1+\kappa} \varphi_\ell(k)  dk  \lesssim \dfrac{1}{t}2^{\kappa\ell}.
\end{align}
Hence, in the sequel we can add to $A$ a cutoff $\chi_1( \sqrt{t}  k)$. Since
\begin{align*}
& \partial_k  S(r,s,k,t) = \frac{1}{t}(r-s) + \frac{k}{\sqrt{k^2+1}},
\qquad \partial_k^2 S(r,s,k,t) = (k^2+1)^{-3/2},
\end{align*}
the phase $S$ has a unique stationary point, which we denote by $k_0$, whenever $\frac{1}{t}(r-s) \in (-1,0)$. In what follows we may assume that the stationary point is in the support of the integral, otherwise the desired bound can be obtained more easily. In particular, we have $k_0 \approx 2^\ell$ and $s \gtrsim r+t k/\sqrt{k^2+1}\gtrsim tk$.
Let $q_0$ be the smallest integer such that $2^{q_0} \geq t^{-1/2}$. 
We decompose the support of the integral according to the distance of $k$ to $k_0$ as  
\begin{align}\label{prKGmain1q2}
\begin{split}
& K_\ell(t,r,s) = \sum_{q \geq q_0}^{\ell+10} K_{\ell,q}(t,r,s), 
\\
& K_{\ell,q}(t,r,s) := \dfrac{1}{r^{\frac{3}{2}+\kappa}s^{\frac{3}{2}}} \int_0^\infty   e^{i t S} A(r,s,k)
  \varphi_q^{(q_0)}(k-k_0) \varphi_\ell(k) \, dk ,
\end{split}
\end{align}
where $\fy_q^{(q_0)}(k-k_0)$ is a cutoff to $\{|k-k_0|\lesssim 2^{q_0}\}$ when $q=q_0$ and to $\{|k-k_0|\simeq 2^q\}$ if $q>q_0$. Note that the estimate for the term with $q=q_0$ follows by direct integration. Indeed, using that $s \gtrsim  tk$, we see that \begin{align}\label{locdec_34}
\bigg[\int_{tk}^\infty \dfrac{1}{s^2}ds\bigg]^{1/2}\lesssim \dfrac{1}{(tk)^{1/2}}, \quad \hbox{and} \quad \bigg[\int_{1/k}^\infty \dfrac{1}{r^{2+2\kappa}}dr\bigg]^{1/2}\lesssim k^{1/2+\kappa},
\end{align}
and hence using $2^{q_0}\approx t^{-1/2}$,
\[
\Vert K_{\ell,q_0}\Vert_{L^2_{r\ud r}L^2_{s\ud s}} \lesssim \dfrac{1}{t^{1/2}}\int_0^\infty k^{\kappa}\varphi_q^{(q_0)}(k-k_0)dk  \lesssim \dfrac{1}{t}2^{\kappa\ell}.
\]
Next, observe that on the support of $K_{\ell,q}$ we have (recall that $k<1/2$)
\begin{align}\label{prKG1main1S2}
| \partial_k  S| \approx |k-k_0| \langle k\rangle^{-3} \approx 2^q  \gtrsim  t^{-1/2}\ . 
\end{align}
Now, for $q>q_0$, integrating by parts in $k$ gives
\begin{align*}
K_{\ell,q}(t,r,s) = \frac{i}{t} \int_0^\infty e^{i t S} \partial_k  \Big[ \frac{1}{\partial_k  S} \,  A \, \varphi_q^{(q_0)}(k-k_0)\varphi_\ell(k) \Big] \, dk ,
\end{align*}
so that $| K_{\ell,q} | \lesssim t^{-1} (\mathbb{I} + \mathbb{II} + \mathbb{III})$,
\begin{align}\label{prKG1main1q2}
\begin{split}
\mathbb{I}(r,s) & = \dfrac{1}{r^{\frac{3}{2}+\kappa}s^{\frac{3}{2}}} \int_0^\infty 
  \frac{|\partial_k^2 S|}{(\partial_k  S)^2}
 \, | A | \, \varphi_q(k-k_0)  \varphi_\ell(k) \, dk ,	
\\
\mathbb{II}(r,s) & = \dfrac{1}{r^{\frac{3}{2}+\kappa}s^{\frac{3}{2}}} \int_0^\infty \frac{1}{|\partial_k  S|}
	 \, | A | \, \varphi_\ell(k) \partial_k  \varphi_q(k-k_0)  \, dk ,
\\
\mathbb{III}(r,s) & = \dfrac{1}{r^{\frac{3}{2}+\kappa}s^{\frac{3}{2}}} \int_0^\infty \frac{1}{|\partial_k  S|}
	 \, \big| \partial_k  \big( A  \,\varphi_\ell(k)\big)
	 \big| \varphi_q(k-k_0) \, dk .
\end{split}
\end{align}
Since $s \gtrsim t k$, $rk  \gtrsim 1$ and $k\lesssim 1$, using also \eqref{prKG1main1S2}, \eqref{locdec_34} and \eqref{locdec_33} we obtain
\begin{align*}
\Vert \mathbb{I}(r,s) \Vert_{L^2_{r\ud r}L^2_{s\ud s}} \lesssim \dfrac{1}{t^{1/2}}\int_0^\infty 2^{-2q}k^{\kappa} \varphi_q^{(q_0)}(k-k_0)  \varphi_\ell(k) \, dk   \lesssim \dfrac{1}{t^{1/2}} \cdot 2^{-q}2^{\kappa\ell}.
\end{align*}
Then, summing over $q\geq q_0$ gives a bound by $t^{-1}2^{\kappa\ell}$. 
The second term in \eqref{prKG1main1q2} is estimated in a similar way.
Next we use \eqref{locdec_33} to bound the third term in \eqref{prKG1main1q2} by
\begin{align*}
\Vert \mathbb{III}(r,s) \Vert_{L^2_{r\ud r}L^2_{s\ud s}}\lesssim \dfrac{1}{t^{1/2}}\int_0^\infty 2^{-q} k^{\kappa-1}  \varphi_q^{(q_0)}(k-k_0)  \varphi_\ell(k) \, dk   \lesssim \dfrac{1}{t^{1/2}} \cdot 2^{-(1-\kappa)\ell},
\end{align*}
which suffices since $\ell\geq q-10$. 

\smallskip
\noindent
{\it Estimate of $\mathbf{K}_{2}$.} We writer $e^{it\jap{k}}=\frac{\jap{k}}{ikt}\partial_ke^{it\jap{k}}$and integrate by parts in $k$. In view of the asymptotics of the Fourier basis from Theorem~\ref{propdFT} we replace the $E(r,k)\overline{E(s,k)}^t$ by the worst possible contribution $\frac{\log r\log s}{(\log k)^2}A(r,s,k)$ where $A$ satisfies \eqref{locdec_33}. Then after integration by parts in $k$ the result is bounded in $L^2_{r\ud r}L^2_{s\ud s}(\{s\geq r\geq 2\})$ by
\begin{equation*}
   \Big(\int_2^\infty\int_r^\infty r^{-1-2\kappa}s^{-1-2\kappa}\Big(\int_0^{1/s}\frac{\log r \log s}{tk |\log k|^2}\ud k\Big)^2\ud s\ud r\Big)^{\frac{1}{2}}\lesssim  \Big(\int_2^\infty\int_r^\infty (rs)^{-1-2\kappa}\frac{(\log r)^2 }{ t^2}\ud s\ud r\Big)^{\frac{1}{2}}\lesssim t^{-1},
\end{equation*}
proving \eqref{prKG1main2} for $j=2$.
\end{proof}

%
%
%
Next we prove an integrated local energy decay (ILED) estimate and its dual, where the key point in the following proof is to take the Fourier transform in time.
\begin{lem}\label{lem:LEDallfereqs1}
For any $\gamma_1>\frac{1}{2}$ and $\gamma_2\geq0$ with $\gamma_1+\gamma_2\geq1$ and $(\gamma_1,\gamma_2)\neq(1,0)$,
\begin{equation}\begin{aligned}\label{local_smoothing1}
\big\Vert \langle r\rangle^{-\gamma_1} e^{\pm it \sqrt{\mathbf{L}}} \mathbf{L}^{-\frac{\gamma_2}{2}} (\mathbf{L-\bfI})^{\frac{\gamma_2}{2}} \, P_c^\bfL\mathbf{f}\big\Vert_{L^2_t( \bbR;L^2_{r\ud r})}\lesssim_{\gamma_1,\gamma_2}  \Vert  P_c^\bfL\mathbf{f}\Vert_{L^2_{r\ud r}},
\end{aligned}\end{equation}
and  
\begin{equation}\label{eq:ILEDDuhamel1}
    \|\int_0^t e^{\mp is\sqrt{\bfL}}\mathbf{L}^{-\frac{\gamma_2}{2}} (\mathbf{L-\bfI})^{\frac{\gamma_2}{2}}P_c^\bfL\bmF(s,\cdot)\ud s\|_{L^2_{r\ud r}}\lesssim_{\gamma_1,\gamma_2} \|\jap{r}^{\gamma_1}P_c^\bfL\bmF\|_{L^2_s ([0,t];L^2_{r\ud r})}.
\end{equation}
\end{lem}
\begin{proof}
We consider only $e^{it\sqrt{\bfL}}$ and the treatment of $e^{-it\sqrt{\bfL}}$ is similar. Estimate \eqref{eq:ILEDDuhamel1} follows from \eqref{local_smoothing1} by duality. For \eqref{local_smoothing1} we begin by writing $L^2_tL^2_{r\ud r}$ on the left-hand side as $L^2_{r\ud r}L^2_t$. We consider only $e^{it\sqrt{\bfL}}$ the case of $e^{-it\sqrt{\bfL}}$ begin more of the same. Expressing $e^{it \sqrt{\mathbf{L}}} \mathbf{L}^{-\frac{\gamma_2}{2}} (\mathbf{L-\bfI})^{\frac{\gamma_2}{2}} \, P_c^\bfL\mathbf{f}$ as the inverse distorted Fourier transform of its distorted Fourier transform, changing variables to $\lambda=\jap{k}$, applying Plancherel in time, and changing variable back to $k$, we see that for a suitable constant $C>0$ and with $\widetilde\bmf=\widetilde\calF[\bmf]$,
\begin{equation*}
    \|e^{it \sqrt{\mathbf{L}}} \mathbf{L}^{-\frac{\gamma_2}{2}} (\mathbf{L-\bfI})^{\frac{\gamma_2}{2}} \, P_c^\bfL\mathbf{f}\|_{L^2_t}= C \Big(\int_0^\infty \big(\frac{k}{\jap{k}}\big)^{2\gamma_2}\jap{k}|E(r,k)\widetilde\bmf(k)|^2k\ud k\Big)^{\frac{1}{2}},
\end{equation*}
and hence
\begin{equation*}
    \|\jap{r}^{-\gamma_1}e^{it \sqrt{\mathbf{L}}} \mathbf{L}^{-\frac{\gamma_2}{2}} (\mathbf{L-\bfI})^{\frac{\gamma_2}{2}} \, P_c^\bfL\mathbf{f}\|_{L^2_tL^2_{r\ud r}}^2\lesssim \int_0^\infty \int_0^\infty \jap{r}^{-2\gamma_1}\big(\frac{k}{\jap{k}}\big)^{2\gamma_2}\jap{k}|E(r,k)\widetilde\bmf(k)|^2k r\,\ud k\ud r.
\end{equation*}
When $k\geq1$ we observe that in view of the asymptotics in Theorem~\ref{propdFT}, $|\sqrt{r\jap{k}}E(r,k)|\lesssim 1$, and therefore since $2\gamma_1>1$,
\begin{equation*}
    \int_0^\infty \int_1^\infty \jap{r}^{-2\gamma_1}\big(\frac{k}{\jap{k}}\big)^{2\gamma_2}\jap{k}|E(r,k)\widetilde\bmf(k)|^2k r\,\ud k\ud r\lesssim \|\widetilde\bmf\|_{L^2_{k\ud k}}^2\int_0^\infty\frac{\ud r}{\jap{r}^{2\gamma_1}}\lesssim \|\widetilde\bmf\|_{L^2_{k\ud k}}^2.
\end{equation*}
When $k\leq1$ and $rk\leq1$ then in view of Theorem~\ref{propdFT} we bound $|E(r,k)|\lesssim 1$ and 
\begin{equation*}
    \int_0^1 \int_0^{1/k} \jap{r}^{-2\gamma_1}\big(\frac{k}{\jap{k}}\big)^{2\gamma_2}\jap{k}|E(r,k)\widetilde\bmf(k)|^2k r\,\ud r\ud k\lesssim \int_0^1|\widetilde\bmf(k)|^2k^{2(\gamma_1+\gamma_2)-1}\ud k\lesssim \|\widetilde\bmf\|_{L^2_{k\ud k}}^2,
\end{equation*}
where we have used the assumption that $\gamma_1+\gamma_2\geq1$. When $k\leq1$ and $rk\geq1$ we  use Theorem~\ref{propdFT} to bound $|\sqrt{rk}E(r,k)|\lesssim 1$ and
\begin{equation*}
    \int_0^1 \int_{1/k}^\infty \jap{r}^{-2\gamma_1}\big(\frac{k}{\jap{k}}\big)^{2\gamma_2}\jap{k}|E(r,k)\widetilde\bmf(k)|^2k r\,\ud r\ud k\lesssim \int_0^1|\widetilde\bmf(k)|^2k^{2(\gamma_1+\gamma_2)-1}\ud k\lesssim \|\widetilde\bmf\|_{L^2_{k\ud k}}^2,
\end{equation*}
where we have again used that $\gamma_1+\gamma_2\geq1$.
\end{proof}
The next lemma contains an inhomogeneous ILED estimate.
\begin{lemma}\label{lem:ILEDDuhamel}
    Suppose $\bmv$ satisfies
    \begin{equation*}
        (\partial_t^2+\bfL)P_c^\bfL\bmv=P_c^\bfL\bmF,\qquad P_c^\bfL\bmv(0)=\partial_tP_c^\bfL\bmv(0)=0,\qquad \mathrm{on}~\bbR_+^2.
    \end{equation*}
    Then for any $\kappa>0$ and $0<T\leq \infty$, 
\begin{equation}\label{eq:ILEDDuhamel2}
    \|\jap{r}^{-\frac{1}{2}-\kappa}\big(|\nabla_x(e^{i\theta}P_c^\bfL\bmv)|+|r^{-1}(e^{i\theta}P_c^\bfL\bmv)|\big)\|_{L^2_{t}([0,T];L^2_x(\bbR^2))}
    \lesssim \|\jap{r}^{1+\kappa}e^{i\theta}P_c^\bfL\bmF\|_{L^2_{t}([0,T];L^2_x(\bbR^2))}.
\end{equation}
\end{lemma}
\begin{proof}
    We prove the estimate for $T=\infty$. The estimate for finite $T$ follows from this case by extending $\bmF$ by zero outside of $[0,T]$. Note that it suffices to replace \eqref{eq:ILEDDuhamel2} by
    \begin{align}
\begin{split}\label{eq:ILEDDuhamel2temp1}
    &\|\jap{r}^{-\frac{1}{2}-\kappa}(|\partial_rP_c^\bfL\bmv|+|r^{-1}P_c^\bfL\bmv|)\|_{L^2_{t}(\bbR;L^2_{r\ud r})}\lesssim \|\jap{r}^{1+\kappa} P_c^\bfL\bmF\|_{L^2_{t}(\bbR;L^2_{r\ud r})}.
    \end{split}
\end{align}
Estimate \eqref{eq:ILEDDuhamel2temp1} in turn follows from the limiting absorption estimate 
\begin{equation}\label{eq:ILEDDuhamel2temp2}
    \|\jap{r}^{-\frac{1}{2}-\kappa}\big(|\partial_r(\bfL-(\tau+i0)^2)^{-1}P_c^\bfL\bmv|+|r^{-1}(\bfL-(\tau+i0)^2)^{-1}P_c^\bfL\bmv|\big)\|_{L^2_{r\ud r}}\lesssim \|\jap{r}^{1+\kappa}P_c^\bfL\bmv\|_{L^2_{r\ud r}},
\end{equation}
uniformly in $\tau\in\bbR$. To see this suppose we know \eqref{eq:ILEDDuhamel2temp2}. To prove \eqref{eq:ILEDDuhamel2temp1} we extend $F$ to be identically zero for negative times and the domain of the time variable to $\bbR$. By Fatou's lemma it then suffices replace $\bmv$ and $\bfF$ by $e^{-\epsilon t}\bmv$ and $e^{-\epsilon t}\bmF$, $\epsilon>0$, in \eqref{eq:ILEDDuhamel2temp1} and take the limit as $\epsilon\to0^+$. The resulting estimate follows by applying Plancherel in time and using \eqref{eq:ILEDDuhamel2temp2}. To prove \eqref{eq:ILEDDuhamel2temp2} we first conjugate the measure to $\ud r$. More precisely, with $\bfLambda:=r^{\frac{1}{2}}\cdot\bfL\cdot r^{-\frac{1}{2}}$, estimate \eqref{eq:ILEDDuhamel2temp2} is equivalent to  
\begin{equation}\label{eq:ILEDDuhamel2temp3}
    \|\jap{r}^{-\frac{1}{2}-\kappa}\big(|\partial_r(\bfLambda-(\tau+i0)^2)^{-1}P_c^\bfLambda\bmv|+|r^{-1}(\bfLambda-(\tau+i0)^2)^{-1}P_c^\bfLambda\bmv|\big)\|_{L^2_{\ud r}}\lesssim \|\jap{r}^{1+\kappa}P_c^\bfLambda\bmv\|_{L^2_{\ud r}},
\end{equation}
where $P_c^\bfLambda$ is the orthogonal projection onto the continuous spectrum of $\bfLambda$. Let $k$ be defined by the relation $k^2+1=\tau^2$, $\Im \,k\geq0$.
Then in view of \eqref{eq:calGpmsummary1},
\begin{equation*}
    (\bfLambda-(\tau+i0)^2)^{-1}P_c^\bfLambda\bmv(r)=\int_0^\infty\calG_\bfLambda(r,s,k)P_c^\bfLambda\bmv(s)\ud s,
\end{equation*}
where the Green's function $\calG_\bfLambda(r,s,k)$ is given by either of the identical upper or lower two by two blocks in the diagonal of $\calG_+$ in \eqref{eq:calGpmsummary1}. By Cauchy-Schwarz, \eqref{eq:ILEDDuhamel2temp3} follows if we can show that
\begin{equation}\label{eq:ILEDDuhamel2temp4}
    \|\jap{r}^{-\frac{1}{2}-\kappa}\jap{s}^{-1-\kappa}\partial_r\calG_\bfLambda(r,s,k)\|_{L^2_{\ud r}L^2_{\ud s}}+\|\jap{r}^{-\frac{1}{2}-\kappa}\jap{s}^{-1-\kappa}r^{-1}\calG_\bfLambda(r,s,k)\|_{L^2_{\ud r}L^2_{\ud s}}\lesssim 1,
\end{equation}
uniformly in $k$.  Since the estimates in the various regions are similar, we only provide the details for $\tau\geq1$ and in two cases: {\bf{(i)}} $s\leq r\leq k^{-1}$ and the contribution of $\Phi_j$ to $\Psi_j$, $j=1,2$, in \eqref{eq:Psi1toPhi1Theta1} and \eqref{eq:Psi2toPhi2Theta2}. {\bf{(ii)}}  $r\leq s\leq k^{-1}$ and the contribution of $\Phi_j$ to $\Psi_j$, $j=1,2$, in \eqref{eq:Psi1toPhi1Theta1} and \eqref{eq:Psi2toPhi2Theta2}. We will use the formulas \eqref{eq:Theta1varpar1}, \eqref{SMOp1_2}, \eqref{eq:Psi1toPhi1Theta1}, \eqref{eq:Psi1Theta1Wornskian1}, \eqref{eq:Theta2varpar1}, \eqref{propSMa}, \eqref{eq:Psi2toPhi2Theta2}, \eqref{eq:Psi2Theta2Wornskian2}, \eqref{eq:calGpmsummary1}, \eqref{green10} and the asymptotics for $\Phi_j$ and $\Psi_j$ proved in Sections~\ref{subsec:calH1} and~\ref{subsec:calH2}. For {\bf{(i)}}, arguing as in the proof of equation~\ref{eq:Eklein} in Theorem~\ref{propdFT}, for $s\leq r\leq k^{-1}$ we can bound (the matrix estimates here are entry-wise)
{\renewcommand{\arraystretch}{1.2}
\begin{align*}
    |\partial_r\calG_\bfLambda(r,s,k)|&\lesssim \pmat{\dfrac{\sqrt{r}k\jap{k}}{\jap{r}}&\dfrac{\sqrt{r}\log(2+r)}{\jap{r}}\\[3ex]\dfrac{r^{\frac{5}{2}}k\jap{k}}{\jap{r}^2}&\dfrac{\sqrt{r}\log(2+r)\jap{k}^2}{\jap{r}^2}}\cdot\\&\qquad \qquad\qquad\pmat{\dfrac{s^{\frac{3}{2}}k}{\jap{s}\jap{k}}&\dfrac{s^{\frac{7}{2}}k}{\jap{s}^2\jap{k}}\\[3ex]
    \dfrac{s^{\frac{3}{2}}\log(2+s)\log(2+k)}{\jap{s}\jap{k}^2\jap{\log k}}&\dfrac{s^{\frac{3}{2}}\log(2+s)\log(2+k)}{\jap{s}^2\jap{\log k}}}.
\end{align*}}
The same estimate holds with the left-hand side replaced by $|r^{-1}\calG_\bfLambda(r,s,k)|$. Estimate \eqref{eq:ILEDDuhamel2temp4} then follows by carrying out the matrix multiplication.  Similarly for {\bf{(ii)}},  $|\partial_r\calG_\bfLambda(r,s,k)|$
and $|r^{-1}\calG_\bfLambda(r,s,k)|$ are bounded by
{\renewcommand{\arraystretch}{1.2}
\begin{align*}
    \pmat{\dfrac{\sqrt{r}k}{\jap{r}\jap{k}}&\dfrac{\sqrt{r}\log(2+r)\log(2+k)}{\jap{r}\jap{k}^2\jap{\log k}^2}\\[3ex]
    \dfrac{r^{\frac{5}{2}}k}{\jap{r}^2\jap{k}}&\dfrac{\sqrt{r}\log(2+r)\log(2+k)}{\jap{r}^2\jap{\log k}^2}}\pmat{\dfrac{s^{\frac{3}{2}}k\jap{k}}{\jap{s}}&\dfrac{s^{\frac{7}{2}}k\jap{k}}{\jap{s}^2}\\[3ex]\dfrac{s^{\frac{3}{2}}\log(2+s)\jap{\log k}}{\jap{s}}&\dfrac{s^{\frac{3}{2}}\log(2+s)\jap{k}^2\jap{\log k}}{\jap{s}^2}}.
\end{align*}}
The desired estimate  \eqref{eq:ILEDDuhamel2temp4} again follows by carrying out the matrix multiplication.
\end{proof}
We end this section with a local decay estimate which includes a singular multiplier on the distorted Fourier side.
\begin{lemma}\label{lem:LAP1}
    Let $(\sqrt{\bfL}-2\lambda+i0^+)^{-1}\bmf:=\lim_{\epsilon\to 0^+}(\sqrt{\bfL}-2\lambda+i\epsilon)^{-1}\bmf$ for $\bmf$ satisfying $\jap{r}^2\bmf\in L^2_{r\ud r}$. Then for $t\geq0$
    \begin{equation*}
        \sum_{q=0,1}\|\jap{r}^{-2}\partial_r^qe^{it\sqrt{\bfL}}(\sqrt{\bfL}-2\lambda+i0^+)^{-1}P_c^\bfL\bmf\|_{L^2_{r\ud r}}\lesssim \frac{1}{\jap{t}}\|\jap{r}^2P_c^\bfL\bmf\|_{L^2_{r\ud r}}.
    \end{equation*}
\end{lemma}
\begin{proof}
    For $t\leq1$ the estimate can be proved  using a similar argument as in the proof of \eqref{eq:ILEDDuhamel2temp2}. We therefore assume that $t\geq1$.  Starting with $q=0$ we write, for $\epsilon>0$, and a small constant $0<\delta\ll1$,
    \begin{align*}
        \jap{r}^{-2}e^{it\sqrt{\bfL}}(\sqrt{\bfL}-2\lambda+i\epsilon)^{-1}P_c^\bfL\bmf&=\jap{r}^{-2}\chi_0(\sqrt{\bfL-2\lambda\bfI}/\delta)e^{it\sqrt{\bfL}}(\sqrt{\bfL}-2\lambda+i\epsilon)^{-1}P_c^\bfL\bmf\\
        &\quad+\jap{r}^{-2}\chi_1(\sqrt{\bfL-2\lambda\bfI}/\delta)e^{it\sqrt{\bfL}}(\sqrt{\bfL}-2\lambda+i\epsilon)^{-1}P_c^\bfL\bmf\\
        &=:I+II.
    \end{align*}
    The term $II$ can be estimated using similar arguments to the ones in Lemmas~\ref{locdec_lem3} and~\ref{lem:LEDallfereqs1}, so we focus on $I$. With $\chi_\lambda(k):=\chi_0(\sqrt{k^2-2\lambda+1}/\delta)$, this can be written as
    \begin{align*}
        I&=\int_0^\infty\frac{\chi_\lambda(k)}{\jap{k}-2\lambda+i\epsilon}e^{it\jap{k}}\jap{r}^{-2} E(r,k)\widetilde\bmf(k) k\,\ud k\\
        &=-i \int_t^\infty \int_0^\infty e^{i\tau(\jap{k}-2\lambda+i\epsilon)}\chi_\lambda(k)e^{it(2\lambda-i\epsilon)}\jap{r}^{-2} E(r,k)\widetilde\bmf(k) k \,\ud k\, \ud \tau.
    \end{align*}
    Noting that the phase is non-singular in the support of $\chi_\lambda$ we integrate by parts twice in $k$ and use the notation $\calL^\ast\widetilde\bmg:=\partial_k\big(\frac{i\jap{k}}{ k}\widetilde\bmg\big)$ to get
    \begin{align*}
        \|I\|_{L^2_{r\ud r}}\lesssim \int_t^\infty \frac{1}{\tau^2}\Big\|\int_0^\infty e^{i\tau\jap{k}}(\calL^\ast)^2\big(\chi_\lambda(k)\jap{r}^{-2} E(r,k)\widetilde\bmf(k) k\big) \,\ud k\Big\|_{L^2_{r\ud r}} \ud \tau.
    \end{align*}
    Since $k$ is bounded away from zero in the supported of $\chi_\lambda$ using similar arguments to those in the proof of Lemma~\ref{locdec_lem2} we conclude that
    \begin{equation*}
        \|I\|_{L^2_{r\ud r}}\lesssim \int_t^\infty \tau^{-2}\sum_{j=0}^2\|\bfchi_\lambda(k)\partial_k^j\widetilde\bmf\|_{L^2_{k\ud k}}\ud \tau\lesssim t^{-1}\sum_{j=0}^2\|\bfchi_\lambda(k)\partial_k^j\widetilde\bmf\|_{L^2_{k\ud k}},
    \end{equation*}
    where $\bfchi_\lambda$ is a compactly supported cutoff to a region bounded away from zero. It remains to show that $\|\bfchi_\lambda(k)\partial_k^j\widetilde\bmf\|_{L^2_{k\ud k}}\lesssim \|\jap{r}^jP_c^\bfL\bmf\|_{L^2_{r\ud r}}$. But this again follows using the asymptotics of the Fourier basis from Theorem~\ref{propdFT} and writing
    \begin{equation*}
        \bfchi_\lambda(k)\widetilde\bmf(k)=\int_0^\infty\bfchi_\lambda(k)\overline{E(r,k)}^t\bmf(r)r\ud r
    \end{equation*}
    and differentiating in $k$. We refer to the proof of Lemma~\ref{lem:dkbounds1} for details in a more delicate setup where $k$ is not bounded away from zero (and hence one has to distinguish between $\partial_k$ and $\partial_k+\frac{1}{k}$). 

    The case $q=1$ is similar. Indeed, we first commute $\partial_r$ and the weight $\jap{r}^{-2}$, the commutator being bounded by what we already proved. Then by elliptic regularity and introducing $e^{i\theta}$ to view the functions as defined on $\bbR^2$, we replace $\partial_r$ by $\bfL_{\nr}^{\frac{1}{2}}$. See \eqref{eq:Lnr}. Finally we commute $\bfL_{\nr}^{\frac{1}{2}}$ and the weight using Lemma~\ref{lem:Lnrxcomm1} in the appendix. From this point the argument follows exactly as before.
\end{proof}

\section{Transference Relations} \label{sec:working_title_transference}

The goal of this section is to relate the physical Lorentz boosts
\[
\Omega_n^\pm:=r\partial_t+t\Bigl(\partial_r\pm\frac{n}{r}\Bigr)
\]
to the corresponding operators on the distorted Fourier side (note that $\widetilde\Omega_0^+=\widetilde\Omega_0^-$),
\[
\widetilde\Omega_n^\pm:=t k+\partial_t\Bigl(\partial_k\pm\frac{n}{k}\Bigr).
\]
Specifically, we will prove Theorem~\ref{thm:transference1} starting with a number of more basic lemmas. We will always work with the $2\times2$ operator $\bfL$, and the corresponding statements for
$\bfM=\pmat{\bfL&0\\0&\bfL}$ are obtained component-wise.

 We begin with introducing some notation. For a vector-valued radial function $\bff=(f_1,f_2)^t$ we write
\[
\Omega_1^\pm\bff:=(\Omega_1^\pm f_1,\Omega_1^\pm f_2)^t.
\]
 On the distorted-Fourier side we will sometimes need to apply $\widetilde\Omega_n^{\pm}$ for different choices of $n$ to the different components of a vector-valued function. For this we introduce the notation
 \[
 \widetilde{\bfOmega}_{n,m}^{\kappa_1,\kappa_2}:=\pmat{\kappa_1\widetilde\Omega_n^{\kappa_1}&0\\0&\kappa_2\widetilde\Omega_m^{\kappa_2}}.
 \]
 In the proof of Theorem~\ref{thm:transference1} we will also need the non-equivariant forms of the Lorentz boosts. For this we use the notation
\begin{equation}
    \begin{aligned}
        \Omega &:= x\partial_t+t\nabla_x, \quad x \in \bbR^2, \\
        \whatOmega &:= \partial_t\nabla_\xi+t\xi, \quad \xi \in \bbR^2.
    \end{aligned}
\end{equation}
Note that if $\psi(t,r)$ is a time-dependent radially symmetric function then
\begin{equation} \label{equ:equivboost1}
    \Omega\big(e^{in\theta}\psi(t,r)\big) = \frac{1}{2} \begin{pmatrix} e^{i(n-1)\theta}\Omega_n^{+}\psi(t,r)+e^{i(n+1)\theta}\Omega_n^{-}\psi(t,r) \\ ie^{i(n-1)\theta}\Omega_{n}^{+}\psi(t,r)-ie^{i(n+1)\theta}\Omega_n^{-} \psi(t,r) \end{pmatrix}.
\end{equation}
 In our error analysis, we will also have occasion to use the operator
 \begin{equation}\label{eq:tildecalMdef1}
     \widetilde\calM_k:=\pmat{k^{-1}&0\\0&(k\log k)^{-1}}.
 \end{equation}
 In applications in \cite{LPPSS3}, the transference relations are most relevant in the region $r\geq1$. We therefore work with the decomposition of $\bfL$ from \eqref{eq:Linfdefintro1} which we recall:
\begin{align*}
\bfL=\Linf+\Vinf,
\end{align*}
where
\begin{align*}
\Linf:=\pmat{-\Delta+1&0\\0&-\Delta+\frac{1}{r^2}+1},
\qquad
\Vinf:=\pmat{\tfrac{(1-a)^2}{r^2}-\tfrac{3}{2}(1-U^2(r))&-\tfrac{2(1-a)U}{r}\\-\tfrac{2(1-a)U}{r}&U^2-1}.
\end{align*}
The potential $\Vinf$ is singular at the origin but satisfies $\Vinf(r)=O(r^{-1/2}e^{-r})$ as $r\to\infty$. 

Let $E(r,k)$ denote the matrix-valued generalized eigenfunction kernel for $\bfL$ from Theorem~\ref{propdFT}:
\begin{equation*} 
\bfL E(r,k)=\jap{k}^2 E(r,k),
\end{equation*}
and
\[
\wtilcalF_\bfL[\bff](k)=\int_0^\infty \overline{E(r,k)}^{\,t}\bff(r)\,r\,\ud r,
\qquad
\wtilcalF_\bfL^{-1}[\widetilde\bff](r)=\int_0^\infty E(r,k)\widetilde\bff(k)\,k\,\ud k.
\]
We also introduce the matrix-valued Bessel functions
\begin{align}\label{eq:GJYdef1}
\bfJ_{n,m}(\tau):=\pmat{J_n(\tau)&0\\0&J_m(\tau)},
\qquad
\bfY_{n,m}(\tau):=\pmat{Y_n(\tau)&0\\0&Y_m(\tau)}.
\end{align}
When the specific choice is irrelevant, we write $\bfK_{n,m}$, $\bfK\in\{\bfJ,\bfY\}$. All matrix inequalities below are understood component-wise, and $\mathbf 1$ denotes the $2\times2$ matrix all of whose entries are equal to~$1$. Throughout this section we use $\chi_{\leq1}$ and $\chi_{\geq1}$ to denote smooth nonnegative radial cutoffs on $[0,\infty)$ supported in $[0,2]$ and $[\frac{1}{2},\infty)$, respectively, with $\chi_{\leq1}+\chi_{\geq1}=1$.

\begin{lemma}\label{lem:Fbasisrepalt1}
For every $k>0$ the Fourier basis $E(\cdot,k)$ admits the representation
\begin{align}\label{eq:Ereplarge1}
E(r,k)=\bfJ_{0,1}(k r)\,S^J(r,k)+\bfY_{0,1}(k r)\,S^Y(r,k),
\end{align}
where
\begin{align*}
S^J(r,k)
&=\bfq^J(k)+\frac{\pi}{2}\int_r^\infty\bfY_{0,1}(k s)\Vinf(s)E(s,k)\,s\,\ud s
=:\bfq^J(k)+\bfp^J(r,k),
\\
S^Y(r,k)
&=\bfq^Y(k)-\frac{\pi}{2}\int_r^\infty\bfJ_{0,1}(k s)\Vinf(s)E(s,k)\,s\,\ud s
=:\bfq^Y(k)+\bfp^Y(r,k).
\end{align*}
\end{lemma}

\begin{proof}
Since
\[
(\Linf-\jap{k}^2)E(r,k)=-\Vinf(r)E(r,k)
\]
and $\bfJ_{0,1}(k r)$, $\bfY_{0,1}(k r)$ solve the homogeneous equation
\[
(\Linf-\jap{k}^2)\bfK_{0,1}(k r)=0,
\qquad \bfK\in\{\bfJ,\bfY\},
\]
variation of constants gives
\[
E(r,k)=\bfJ_{0,1}(k r)c_J(r,k)+\bfY_{0,1}(k r)c_Y(r,k)
\]
with
\[
\partial_r c_J(r,k)=-\frac{\pi}{2}\bfY_{0,1}(k r)\Vinf(r)E(r,k)\,r,
\qquad
\partial_r c_Y(r,k)=\frac{\pi}{2}\bfJ_{0,1}(k r)\Vinf(r)E(r,k)\,r.
\]
%
The coefficients are now chosen so that the variation--of--constants integrals vanish at $+\infty$, which yields the formulas for $\bfp^J$ and $\bfp^Y$ and leaves the corresponding integration constants as $\bfq^J(k)$ and $\bfq^Y(k)$.
\end{proof}

The next lemma contains the bounds on the matrix coefficients in Lemma~\ref{lem:Fbasisrepalt1}. 
%
%
\begin{lemma}\label{lem:bfqbfpbds1}
For $r\geq1$ the coefficient matrices from Lemma~\ref{lem:Fbasisrepalt1} satisfy
\begingroup\small
\begin{align}
\label{eq:qJgross}
\bfq^J(k)
&=
\left(\begin{matrix}
 k\jap{k}^{-1}q^J_{11} & \jap{k}^{-1}q^J_{21}\\[1ex]
 \jap{k}^{-1}q^J_{12} & k\jap{k}^{-1}q^J_{22}
\end{matrix}\right),
\\ \\
\label{eq:qYgross}
\bfq^Y(k)
&=
\left(\begin{matrix}
 \min\{k^2\jap{\log k},1\}q^Y_{11} & \min\{\jap{\log k}^{-1},k^{-1}\}q^Y_{21}\\[1ex]
 \min\{k\jap{\log k},k^{-1}\} q^Y_{12} & \min\{k\jap{\log k}^{-1},1\}q^Y_{22}
\end{matrix}\right),
\\ \\
\bfp^J(r,k)
&=
\left\{\begin{array}{ll}
\pmat{(k\jap{\log k} r^{3/2}e^{-r/2}+k^{\frac{1}{2}}\jap{k}^{-1}e^{-k^{-1}})p^{J,0}_{11}&(\jap{\log r} r^{\frac{1}{2}}e^{-r/2}\\& +k^{-1}\jap{k}^{-1}e^{-k^{-1}})p^{J,0}_{21}\\[1.5ex](r^{\frac{1}{2}}e^{-r/2}+k^{-1}\jap{k}^{-1}e^{-k^{-1}})p^{J,0}_{12}&(k^{-1}\jap{\log k}^{-1}\jap{\log r}r^{-\frac{1}{2}}e^{-r/2}\\& +k^{\frac{1}{2}}\jap{k}^{-1}e^{-k^{-1}})p^{J,0}_{22}},
& rk\leq 1 \\ & \\
k^{-1}r^{-1/2}e^{-r/2}\,\chi_{\geq1}(rk)\pmat{p^{J,1}_{11}&p^{J,1}_{21}\\[1ex]p^{J,1}_{12}&p^{J,1}_{22}},
& rk\geq 1,
\end{array}\right.
\label{eq:pJgross}\\[0.5ex]
\bfp^Y(r,k)
&=
\left\{\begin{array}{ll}
\pmat{(k r^{3/2}e^{-r/2}+k^{\frac{1}{2}}\jap{k}^{-1}e^{-k^{-1}})p^{Y,0}_{11}&(\jap{\log k}^{-1}\jap{\log r} r^{\frac{1}{2}}e^{-r/2}\\& +k^{-1}\jap{k}^{-1}e^{-k^{-1}})p^{Y,0}_{21}\\[1.5ex](k^2r^{\frac{5}{2}}e^{-r/2}+k^{-1}\jap{k}^{-1}e^{-k^{-1}})p^{Y,0}_{12}&(k\jap{\log k}^{-1}\jap{\log r}r^{\frac{3}{2}}e^{-r/2}\\ & +k^{\frac{1}{2}}\jap{k}^{-1}e^{-k^{-1}})p^{Y,0}_{22}},
& rk\leq 1,\\ & \\ 
k^{-1}r^{-1/2}e^{-r/2}\,\chi_{\geq1}(rk)\pmat{p^{Y,1}_{11}&p^{Y,1}_{21}\\[1ex] p^{Y,1}_{12}&p^{Y,1}_{22}},
& rk\geq 1.
\end{array}\right.
\label{eq:pYgross}
\end{align}
\endgroup
Here the coefficients $q^K_{ij}(k)$ and $p^{K,m}_{ij}(r,k)$, for $K=J,Y$, satisfy $|(k\partial_k)^\ell q^K_{ij}|\lesssim 1$ when $k\geq1$ for $\ell=0,1,2$. For $k<1$ the coefficients are uniformly bounded. The coefficients multiplying leading order terms with no logarithms have bounded derivatives, while the ones multiplying logarithmic factors produce a factor of $k^{-1}\jap{\log k}$ upon differentiation. For instance $|\partial_k^\ell q^Y_{21}|\lesssim k^{-\ell}\jap{\log k}^{-1}$, for $\ell=1,2$.
\end{lemma}

\begin{proof}
In this proof for the entries of a $2\times2$ matrix $M$ we use the convention $M=\pmat{M_{11}&M_{21}\\M_{12}&M_{22}}$. We treat the end point coefficients $\bfq^K$ and the Volterra remainders $\bfp^K$, $K=J,Y$, separately.
\smallskip
\noindent
{\it 1. The endpoint coefficients.}
Since $\Vinf(r)=O(r^{-1/2}e^{-r})$, the Volterra remainders in \eqref{eq:Ereplarge1} satisfy
\[
\bfp^J(r,k),\ \bfp^Y(r,k)\longrightarrow 0
\qquad (r\to\infty),
\]
component-wise. Therefore, the endpoint coefficients can be obtained from the Bessel Wronskians. With $\bfJ_{0,1}(\tau)=\diag(J_0(\tau),J_1(\tau))$, $\bfY_{0,1}(\tau)=\diag(Y_0(\tau),Y_1(\tau))$, and with $\bfJ_{0,1}'$, $\bfY_{0,1}'$ denoting derivatives with respect to the scalar variable $\tau$, one has
\begin{equation}\label{eq:qJ01formel}
\bfq^J(k)
=
\frac{\pi}{2}\lim_{r\to\infty}
\Bigl[
 rk\bfY_{0,1}'(rk)E(r,k)-r\bfY_{0,1}(rk)\partial_r E(r,k)
\Bigr],
\end{equation}
\begin{equation}\label{eq:qY01formel}
\bfq^Y(k)
=
\frac{\pi}{2}\lim_{r\to\infty}
\Bigl[
 r\bfJ_{0,1}(rk)\partial_r E(r,k)-rk\bfJ_{0,1}'(rk)E(r,k)
\Bigr].
\end{equation}
These matrix Wronskian identities follow immediately from the scalar relations
\[
W_r\bigl(J_\nu(kr),Y_\nu(kr)\bigr)=\frac{2}{\pi r},
\qquad \nu=0,1.
\]
To evaluate \eqref{eq:qJ01formel}, \eqref{eq:qY01formel}, let
\[
\delta_0:=\frac{\pi}{4},
\qquad
\delta_1:=\frac{3\pi}{4}.
\]
If
\[
f(r)=r^{-1/2}\Bigl(c_+(k)e^{ikr}+c_-(k)e^{-ikr}\Bigr)+O(r^{-3/2}),
\]
then the standard large--argument asymptotics of $J_\nu$ and $Y_\nu$ imply
\begin{align}
\frac{\pi}{2}\lim_{r\to\infty}
\Bigl[rkY_\nu'(kr)f(r)-rY_\nu(kr)f'(r)\Bigr]
&=
\frac{\sqrt{2\pi k}}{2}
\Bigl(e^{i\delta_\nu}c_+(k)+e^{-i\delta_\nu}c_-(k)\Bigr),
\label{eq:qJallggrenz}
\\
\frac{\pi}{2}\lim_{r\to\infty}
\Bigl[rJ_\nu(kr)f'(r)-rkJ_\nu'(kr)f(r)\Bigr]
&=
\frac{i\sqrt{2\pi k}}{2}
\Bigl(e^{i\delta_\nu}c_+(k)-e^{-i\delta_\nu}c_-(k)\Bigr),
\label{eq:qYallggrenz}
\end{align}
for $\nu=0,1$. Here, the $O(r^{-3/2})$ remainder and its derivative contribute $o(1)$ to the limits, because
\[
rJ_\nu(kr)=O\bigl(r^{1/2}k^{-1/2}\bigr),
\qquad
rkJ_\nu'(kr)=O\bigl(r^{1/2}k^{1/2}\bigr),
\]
and similarly for $Y_\nu$.
Applying \eqref{eq:E11grenz}, \eqref{eq:E21grenz}, \eqref{eq:E12grenz}, \eqref{eq:E22grenz}, \eqref{eq:qJallggrenz}, \eqref{eq:qYallggrenz} to \eqref{eq:qJ01formel}, \eqref{eq:qY01formel}
 row by row now yields
\begin{equation}\label{eq:qJ01klargrenz}
\bfq^J(k)
=
\frac{1}{2\,\jap{k}}
\begin{pmatrix}
 k\bigl(e^{i\pi/4}-e^{-i\pi/4}\eta_1(k)\bigr)
 &
 i\bigl(e^{i\pi/4}+e^{-i\pi/4}\eta_2(k)\bigr)
\\[1ex]
 i\bigl(e^{i3\pi/4}+e^{-i3\pi/4}\eta_1(k)\bigr)
 &
 k\bigl(e^{i3\pi/4}-e^{-i3\pi/4}\eta_2(k)\bigr)
\end{pmatrix},
\end{equation}
\begin{equation}\label{eq:qY01klargrenz}
\bfq^Y(k)
=
\frac{1}{2\,\jap{k}}
\begin{pmatrix}
 k e^{i\pi/4}\bigl(i+\eta_1(k)\bigr)
 &
 e^{-i3\pi/4}\bigl(1+i\eta_2(k)\bigr)
\\[1ex]
 e^{-i3\pi/4}\bigl(i+\eta_1(k)\bigr)
 &
 k e^{-i3\pi/4}\bigl(1+i\eta_2(k)\bigr)
\end{pmatrix}.
\end{equation}
For example, the $(2,1)$ entry comes from the first row with $\nu=0$ and
\[
c_+(k)=\frac{i}{\sqrt{2\pi}\,\jap{k}\,k^{1/2}},
\qquad
c_-(k)=\frac{i\eta_2(k)}{\sqrt{2\pi}\,\jap{k}\,k^{1/2}},
\]
so that \eqref{eq:qYallggrenz} gives
\[
(\bfq^Y(k))_{21}
=
\frac{i\sqrt{2\pi k}}{2}
\Bigl(e^{i\pi/4}c_+(k)-e^{-i\pi/4}c_-(k)\Bigr)
=
\frac{1}{2\,\jap{k}}e^{-3\pi i/4}\bigl(1+i\eta_2(k)\bigr).
\]
The other three entries are obtained analogously.
In particular,
\begin{equation}\label{eq:qJ01schranke}
|\bfq^J(k)|
\lesssim
\pmat{k\jap{k}^{-1} & \jap{k}^{-1}\\ \jap{k}^{-1} & k\jap{k}^{-1}},
\end{equation}
while
\begin{equation}\label{eq:qY01schranke}
|\bfq^Y(k)|
\lesssim
\frac{1}{\jap{k}}
\pmat{k|i+\eta_1(k)| & |1+i\eta_2(k)|\\ |i+\eta_1(k)| & k|1+i\eta_2(k)|}.
\end{equation}
Thus, \eqref{eq:qJgross} is immediate from $|\eta_j(k)|=1$. For \eqref{eq:qYgross}, first note that \eqref{eq:qY01schranke} gives, for $0<k\le1$,
\[
|(\bfq^Y(k))_{11}|
\lesssim k\,|i+\eta_1(k)|,
\qquad
|(\bfq^Y(k))_{12}|
\lesssim |i+\eta_1(k)|,
\]
and
\[
|(\bfq^Y(k))_{21}|
\lesssim |1+i\eta_2(k)|,
\qquad
|(\bfq^Y(k))_{22}|
\lesssim k\,|1+i\eta_2(k)|.
\]
Using \eqref{eq:eta12-small} we therefore obtain the sharper low-energy bounds
\[
|(\bfq^Y(k))_{11}|
\lesssim k^2\jap{\log k},
\qquad
|(\bfq^Y(k))_{12}|
\lesssim k\jap{\log k},
\]
and
\[
|(\bfq^Y(k))_{21}|
\lesssim \jap{\log k}^{-1},
\qquad
|(\bfq^Y(k))_{22}|
\lesssim k\,\jap{\log k}^{-1}.
\]
For $k\ge1$, we simply use $|\eta_j(k)|=1$ in \eqref{eq:qY01schranke}, which gives
\[
|(\bfq^Y(k))_{11}|+|(\bfq^Y(k))_{22}|\lesssim1,
\qquad
|(\bfq^Y(k))_{12}|+|(\bfq^Y(k))_{21}|\lesssim k^{-1}.
\]
Hence,
\[
|\bfq^Y(k)|
\lesssim
\pmat{\min\{k^2\jap{\log k},1\} & \min\{\jap{\log k}^{-1},k^{-1}\}\\ \min\{k\jap{\log k},k^{-1}\} & \min\{k\jap{\log k}^{-1},1\}},
\]
which proves \eqref{eq:qYgross}.

\smallskip
\noindent
{\it 2. The  Volterra remainders.}
Finally, to prove \eqref{eq:pJgross} and \eqref{eq:pYgross}, we return to the Volterra formulas in Lemma~\ref{lem:Fbasisrepalt1}. Since
\[
\Vinf(s)=O(s^{-1/2}e^{-s}),
\]
it is enough to combine the standard small-- and large--argument bounds for $J_0,J_1,Y_0,Y_1$ with \eqref{eq:Egross}. Assume first that $rk\le1$. We split the integral at $s=k^{-1}$.
On the interval $r\le s\le k^{-1}$ we use
\[
|J_0(k s)|\lesssim1,\qquad |J_1(k s)|\lesssim k s,
\qquad
|Y_0(k s)|\lesssim \jap{\log(k s)},\qquad |Y_1(k s)|\lesssim (k s)^{-1},
\]
together with the first line of \eqref{eq:Egross}. A direct inspection of the matrix product then gives the desired bound.
On the interval $s\ge k^{-1}$ one has $k s\ge1$, so both $\bfJ_{0,1}(k s)$ and $\bfY_{0,1}(k s)$ are $O((k s)^{-1/2})$. Then we get the desired bound by using the second line of \eqref{eq:Egross} and inspecting the matrix product. 
If $rk\ge1$, then $k s\ge1$ for every $s\ge r$, and we use the large--argument bounds for $\bfJ_{0,1}$ and $\bfY_{0,1}$ together with the second line of \eqref{eq:Egross} throughout:
\[
|\bfJ_{0,1}(k s)\Vinf(s)E(s,k)s|
+
|\bfY_{0,1}(k s)\Vinf(s)E(s,k)s|
\lesssim
k^{-1}s^{-1/2}e^{-s}\mathbf 1.
\]
Integrating from $r$ to $\infty$ gives the desired bound. The differentiated bounds follow in the same way, since $(k\partial_k)^j\bfJ_{0,1}(k)$ and $(k\partial_k)^j\bfY_{0,1}(k)$ satisfy the same small-- and large--argument estimates, and the corresponding differentiated bounds for $E$ were already noted above.
\end{proof}

Before stating the main transference identity in the next lemma, we recall the standard Bessel identities
\begin{align}\label{eq:Besselders1}
\tau^n\frac{\ud}{\ud\tau}(\tau^{-n}K_n(\tau))=-K_{n+1}(\tau),
\qquad
\tau^{-(n+1)}\frac{\ud}{\ud\tau}(\tau^{n+1}K_{n+1}(\tau))=K_n(\tau),
\end{align}
valid for $K_n\in\{J_n,Y_n\}$ and all integers $n\ge0$.  
The transference identity \eqref{eq:trans1} in Lemma~\ref{lem:trans1} relates the action of $\Omega_1^+$ to that of $\widetilde{\bfOmega}_{n,m}^{\kappa_1,\kappa_2}$ in the region $\{r\geq1\}$. A similar identity could be derived for $\Omega_{1}^{-}$, but in our applications in \cite{LPPSS3} we can simply write 
$$\Omega_1^-=\Omega_1^+-\frac{2t}{r},$$ and treat  $\frac{2t}{r}$ as an error, using the decay of $r^{-1}$.
\begin{lemma}\label{lem:trans1}
For $K=J,Y$, let $S^K_\diag$ and $S^K_\off$ denote the diagonal and off diagonal parts of $S^K$ respectively, that is, 
\[
S^K_\diag=\pmat{S^K_{1,1}&0\\0&S^K_{2,2}},\quad S^K_\off=\pmat{0&S^K_{2,1}\\S^K_{1,2}&0},\qquad \mathrm{where}\qquad S^K=\pmat{S^K_{1,1}&S^K_{2,1}\\S^K_{1,2}&S^K_{2,2}}.
\]
For $\ast\in\{\diag,\off\}$ define the operators
\begingroup
\allowdisplaybreaks
\begin{align*}
&\calT_\ast[\widetilde\bff](t,r):=\int_0^\infty \chi_{\geq1}(r)\big(\bfJ_{1,0}(rk)S^J_{\ast}(r,k)+\bfY_{1,0}(rk)S^Y_{\ast}(r,k)\big)\widetilde\bff(t,k)\,k\,\ud k,
\\
&\calR_{\ast}[\widetilde\bff](t,r):=\int_0^\infty \chi_{\geq1}(r)\big(\bfJ_{1,0}(rk)\partial_k S^J_{\ast}(r,k)+\bfY_{1,0}(rk)\partial_k S^Y_{\ast}(r,k)\big)\widetilde\bff(t,k)\,k\,\ud k.
\end{align*}
\endgroup
For functions $\bff=(f_1,f_2)^t$ so that $\widetilde\bff:=\wtilcalF_\bfL[\bff]\in C_c^2((0,\infty);\bbC^2)$,  and with $\sigma_3=\pmat{1&0\\0&-1}$,
\begin{align}\label{eq:trans1}
\begin{split}
\chi_{\geq1}(r)\Omega_1^+\bff(t,r)
&=\calT_{\off}[\widetilde\bfOmega_{1,0}^{+,-}\widetilde\bff](t,r)+\calT_{\diag}[\widetilde\bfOmega_{0,1}^{-,+}\widetilde\bff](t,r)-\sigma_3\calR_{\off}[\partial_t\widetilde\bff](t,r)-\sigma_3\calR_{\diag}[\partial_t\widetilde\bff](t,r)\\
&\quad+\chi_{\geq1}(r)\frac{t}{r}(f_1(t,r),0)^t.
\end{split}
\end{align}
\end{lemma}

\begin{proof}
 Write
\begin{align}
\chi_{\geq1}(r)\Omega_1^+\bff(t,r)
&=
\chi_{\geq1}(r)\Omega_1^+\int_0^\infty E(r,k)\widetilde\bff(t,k)\,k\,\ud k\\
&=\sum_{\ast\in\{\diag,\off\}}\chi_{\geq1}(r)\Omega_1^+\int_0^\infty\big(\bfJ_{0,1}(k r)\,S^J_\ast(r,k)+\bfY_{0,1}(k r)\,S^Y_\ast(r,k)\big)\widetilde\bff(t,k)\,k\,\ud k\\
 &=:I_\diag+I_\off.
\end{align}
For $K_0,K_1\in\{J_0,Y_0\}\times\{J_1,Y_1\}$ and any scalar $S=S(r,k)$, \eqref{eq:Besselders1} yields
\begin{align}
(\partial_r+\tfrac1r)\bigl(K_0(k r)S\bigr)
&=-k K_1(k r)S+\frac1r K_0(k r)S+K_0(k r)\partial_rS,\label{eq:transgross1}\\
(\partial_r+\tfrac1r)\bigl(K_1(k r)S\bigr)
&=k K_0(k r)S+K_1(k r)\partial_rS,\label{eq:transgross3}\\
rk K_0(k r)S
&=\partial_k\bigl(K_1(k r)S\,k\bigr)-k K_1(k r)\partial_k S,\label{eq:transgross5}\\
rk K_1(k r)S
&=-(\partial_k-\tfrac1k)\bigl(K_0(k r)S\,k\bigr)
+k K_0(k r)\partial_k S.\label{eq:transgross6}\\
\end{align}
The Volterra formulas for $\bfp^J$ and $\bfp^Y$ imply
\[
\bfJ_{0,1}(k r)\partial_r S^J(r,k)+\bfY_{0,1}(k r)\partial_r S^Y(r,k)=0.
\]
Applying \eqref{eq:transgross1}, \eqref{eq:transgross3}, \eqref{eq:transgross5}, and \eqref{eq:transgross6}  to the two diagonal entries separately, and again integrating by parts in $k$ with no boundary terms because of the compact support of $\widetilde\bff$, we find
\begin{align*}
I_\diag
&=
\chi_{\geq1}(r)\int_0^\infty
\bigl(\bfJ_{1,0}(rk)S^J(r,k)+\bfY_{1,0}(rk)S^Y(r,k)\bigr)
\widetilde{\bfOmega}_{0,1}^{-,+}\widetilde\bff(t,k)\,k\,\ud k
\\
&\quad+
\chi_{\geq1}(r)\int_0^\infty
\bigl(\bfJ_{1,0}(rk)\partial_k S^J(r,k)+\bfY_{1,0}(rk)\partial_k S^Y(r,k)\bigr)
\wtilcalF_\bfL[\partial_t\bff](t,k)\,k\,\ud k
\\
&\quad+\chi_{\geq1}(r)\frac{t}{r}(f_1(t,r),0)^t
\\
&=\calT_{\diag}[\widetilde{\bfOmega}_{0,1}^{-,+}\widetilde\bff](t,r)-\sigma_3\calR_{\diag}[\partial_t\widetilde\bff](t,r)+\chi_{\geq1}(r)\frac{t}{r}(f_1(t,r),0)^t.
\end{align*}
The computation for $I_\off$ is similar.
\end{proof}

The next lemma gives the $L^2$ bounds for the operators in Lemma~\ref{lem:trans1}. We will use the notation $\widetilde\calM_k$ introduced in \eqref{eq:tildecalMdef1}, and write $\chi_{\leq\frac{1}{2}}$ for a smooth cutoff to frequencies $k\leq \frac{1}{2}$.

\begin{lemma}\label{lem:calTcalRbounds1}
The operators from Lemma~\ref{lem:trans1} satisfy
\begin{align}\label{eq:calTcalRbounds1}
\begin{split}
&\|\calT_{\off}[\widetilde{\bfOmega}_{1,0}^{+,-}\widetilde\bff]-\sigma_3\calR_{\off}[\partial_t\widetilde\bff]\|_{L^2_{r\ud r}}\lesssim \|\widetilde{\bfOmega}_{1,0}^{+,-}\widetilde\bff\|_{L^2_{k\ud k}}+\|\chi_{\leq \frac{1}{2}}\widetilde\calM_k\partial_t\widetilde\bff\|_{L^2_{k\ud k}}+\|\partial_t\widetilde\bff\|_{L^2_{k\ud k}},
\\
&\|\calT_{\diag}[\widetilde{\bfOmega}_{0,1}^{-,+}\widetilde\bff]-\sigma_3\calR_{\diag}[\partial_t\widetilde\bff]\|_{L^2_{r\ud r}}\lesssim \|\widetilde{\bfOmega}_{1,0}^{+,-}\widetilde\bff\|_{L^2_{k\ud k}}+\|\chi_{\leq \frac{1}{2}}\widetilde\calM_k\partial_t\widetilde\bff\|_{L^2_{k\ud k}}+\|\partial_t\widetilde\bff\|_{L^2_{k\ud k}}.
\end{split}
\end{align}
\end{lemma}

\begin{proof}
For the most part the estimates for $\calT_\ast$ and $\calR_\ast$ are proved independently. 
We repeatedly use a few elementary devices:\\
\noindent {\bf{(i)}} In the region $rk\le1$ we use the one--dimensional Hardy inequality 
\begin{equation*}
    \int_0^{\sigma_0} |f(\sigma)|^2\sigma^{-3}\ud \sigma\lesssim \int_0^{\sigma_0}|f'(\sigma)|^2\sigma^{-1}\ud\sigma,
\end{equation*}
for and $\sigma_0>0$ and functions satisfying $\lim_{\sigma\to0}(\sigma^{-1}f(\sigma))=0$. This follows by writing $\sigma^{-3}=-\frac{1}{2}\frac{\ud}{\ud \sigma}\sigma^{-2}$ and integrating by parts.\\
\noindent {\bf{(ii)}} In the region where $rk\leq1$ we use that the operator $\calT f(r):=\int_0^\infty \calK(r,k)f(k)k \ud k$, $\calK(r,k):=\chi_{\geq2}(r)\chi_{\leq1}(rk)\frac{1}{rk\jap{\log k}}$, is bounded from $L^2_{k \ud k}$ to $L^2_{r\ud r}$. Here $\chi_{\geq2}$ is a smooth cutoff to the region $\{r\geq2\}$. This can be seen by changing variables to $r=e^{\tau}$ and $k=e^{-\sigma}$. Then the desired boundedness property is equivalent to the estimate
\begin{equation*}
    \Big(\int_{\log 2}^\infty\Big(\int_{\tau}^\infty \frac{g(\sigma)}{\sigma}\ud \sigma\Big)^2\ud \tau\Big)^{\frac{1}{2}}\lesssim\Big(\int_{\log 2}^\infty g^2(\sigma)\ud \sigma\Big)^{\frac{1}{2}}.
\end{equation*}
This estimate follows by changing variables to $u=\tau^{-1}\sigma$, applying Minkowski, and then changing variable to $\sigma=\tau u$.
\\ 
\noindent {\bf{(iii)}} In the oscillatory region $rk\ge1$ we use the standard Cotlar--Stein criterion: if
\begin{equation}\label{eq:oscKriterium1}
\calK(r,k)=\chi_{\geq1}(rk)(rk)^{-1/2}e^{\pm irk}\mathfrak m(r,k)
\end{equation}
with
\begin{equation}\label{eq:oscKriterium2}
|(r\partial_r)^a(k\partial_k)^b\mathfrak m(r,k)|\lesssim 1,
\qquad 0\le a,b\le2,
\end{equation}
then the operator $T_{\calK}g(r):=\int_0^\infty \calK(r,k)g(k)k\,\ud k$ is bounded from $L^2_{k\ud k}$ to $L^2_{r\ud r}$; see, for instance, \cite[Lemma~6.3]{LSS1}. We will verify \eqref{eq:oscKriterium1}, \eqref{eq:oscKriterium2} explicitly for the oscillatory $\bfq$--terms below.\\ 

Turning to \eqref{eq:calTcalRbounds1} we first treat the $\bfq$-terms in the region $rk\le1$.
For $\calT_{\off}$ in the region $rk\le1$, the relevant model kernels are
\[
\calK_1(r,k):=\chi_{\geq1}(r)\chi_{\leq1}(rk),
\qquad
\calK_2(r,k):=\chi_{\geq1}(r)\chi_{\leq1}(rk)\frac{1}{rk\jap{\log k}}.
\]
For the kernel $\calK_1$ if
\[
T_{\calK_1}g(r):=\chi_{\geq1}(r)\int_0^{1/r} g(k)\,k\,\ud k,
\]
then with $\sigma=r^{-1}$ and $H(\sigma):=\int_0^\sigma g(k)k\,\ud k$, we can apply the one-dimensional Hardy inequality from (i) above. 
\[
\|T_{\calK_1}g\|_{L^2_{r\ud r}}^2
=
\int_0^1 |H(\sigma)|^2 \sigma^{-3}\,\ud \sigma
\lesssim
\int_0^1 |H'(\sigma)|^2 \sigma^{-1}\,\ud \sigma
=
\int_0^1 |g(\sigma)|^2 \sigma\,\ud \sigma,
\]
For $\calK_2$ we use the argument from item (ii). For $\calR_{\off}$ the relevant kernel is $$\chi_{\leq1}(rk)\chi_{\geq1}(r)\frac{1}{rk\jap{\log k}}\frac{1}{k\jap{\log k}}.$$
Using the same argument as for $\calK_2$ above we then bound $\|\calR_\off[\chi_{\leq1}(rk)\partial_t\widetilde\bmf]\|_{L^2_{r\ud r}}$ by $\|\chi_{\leq1}(k)(k\jap{\log k})^{-1}\partial_t\widetilde\bmf\|_{L^2_{k\ud k}}$. 
In the case of $\calT_{\diag}$ the representative kernels are
\begin{equation*}
    \calK_1'(r,k):=\chi_{\geq1}(r)\chi_{\leq1}(rk)k,\qquad \calK_2'(r,k):=\chi_{\geq1}(r)\chi_{\leq1}(rk)\frac{k\log k}{r}.
\end{equation*}
The extra vanishing at $k=0$ in these kernels allows us to bound the contribution of $\calT_\diag$ by
\begin{align*}
    \|\chi_{\leq1}(k)k\widetilde\bfOmega_{0,1}^{-,+}\widetilde\bmf\|_{L^2_{k\ud k}}\lesssim \|\widetilde\bfOmega_{1,0}^{+,-}\widetilde\bmf\|_{L^2_{k\ud k}}+\|\partial_t\widetilde\bmf\|_{L^2_{k\ud k}}.
\end{align*}
Similarly, in view of the extra vanishing at $k=0$ the contribution of $\calR_\diag$ is simply bounded by~$\|\partial_t\widetilde\bmf\|_{L^2_{k\ud k}}$.

Next we consider the oscillatory region $rk\ge1$.
For $u\ge1$ and $\nu\in\{0,1\}$, the standard large--argument expansions of the Bessel functions can be written in the form
\[
K_\nu(u)=u^{-1/2}e^{iu}m_{\nu,+}(u)+u^{-1/2}e^{-iu}m_{\nu,-}(u),
\qquad K_\nu\in\{J_\nu,Y_\nu\},
\]
where
\[
|(u\partial_u)^\ell m_{\nu,\pm}(u)|\lesssim1,
\qquad 0\le \ell\le2.
\]
Combining this with Lemma~\ref{lem:bfqbfpbds1}, every scalar kernel arising from a $\bfq$--term for $\calT_\off$ in the region $rk\ge1$ is therefore a finite sum of terms of the form
\[
\calK(r,k)=\chi_{\geq1}(rk)(rk)^{-1/2}e^{\pm irk}\mathfrak m(r,k),
\]
where $\mathfrak m$ is smooth and satisfies
\[
|(r\partial_r)^a(k\partial_k)^b\mathfrak m(r,k)|\lesssim1,
\qquad 0\le a,b\le2.
\]
Hence every oscillatory $\bfq$--kernel for $\calT_\off$ satisfies \eqref{eq:oscKriterium1}, \eqref{eq:oscKriterium2}, and the Cotlar--Stein criterion applies. This proves the bounds for $\calT_\off$. For $\calR_\off$ we observe that the kernel estimates are worst by at worst a factor of $k^{-1}\jap{\log k}^{-1}$ in the region $k\leq1$, and therefore the desired estimate follows by the same arguments.
%
%
For $\calT_\diag$ the extra vanishing at $k=0$ allows us to bound this contribution by
\begin{align*}
    \|k\jap{k}^{-1}\widetilde\bfOmega_{0,1}^{-,+}\widetilde\bmf\|_{L^2_{k\ud k}}\lesssim \|\widetilde\bfOmega_{1,0}^{+,-}\widetilde\bmf\|_{L^2_{k\ud k}}+\|\partial_t\widetilde\bmf\|_{L^2_{k\ud k}}.
\end{align*}
Similarly the contribution of $\calR_\diag$ is simply bounded by $\|\partial_t\widetilde\bmf\|_{L^2_{k\ud k}}$.

It remains to treat the contribution of the Volterra terms $\bfp^J,\bfp^Y$. For these, in view of Lemma~\ref{lem:bfqbfpbds1}, except for the contributions $(\bfp^J)_{22}$, $(\bfp^J)_{11}$, and $(\bfp^Y)_{11}$ in the region $rk\leq 1$, all other contributions are more favorable compared with the corresponding ones for $\bfq^J$ and $\bfq^Y$. For these exceptional entries, all of which appear in $\calT_\diag$ and $\calR_\diag$, we need to observe a leading order in $k$ cancellation in $\calT_\diag[\widetilde\bfOmega_{0,1}^{-,+}\widetilde\bmf]-\sigma_3\calR_\diag[\partial_t\widetilde\bmf]$ to turn $\widetilde\bfOmega_{0,1}^{-,+}$ into $\widetilde\bfOmega_{1,0}^{+,-}$. Indeed, concentrating on the leading order contribution in $k$ and suppressing the decaying in $r$ factors from the Volterra integrals, it suffices to make the following observations. For  $Y_1(rk)\big((\bfp^Y)_{11}\widetilde\Omega_0+\partial_k(\bfp^Y)_{11}\partial_t\big)$ note that
\begin{align*}
   \frac{1}{rk}\big(k\widetilde{\Omega}_0 +(\partial_kk)\partial_t\big)=\frac{1}{r}\widetilde{\Omega}_1.
\end{align*}
For $J_1(rk)\big((\bfp^J)_{11}\widetilde\Omega_0+\partial_k(\bfp^J)_{11}\partial_t\big)$ note that
\begin{align*}
    (rk)\big(k\log k \widetilde\Omega_0+\partial_k(k\log k )\partial_t\big)=(rk)\big(k\log k \,\widetilde\Omega_1+\partial_t\big).
\end{align*}
For $J_0(rk)\big((\bfP^J)_{22}\widetilde\Omega_1+\partial_k(\bfp^J)_{22}\partial_t\big)$ note that
\begin{align*}
    \frac{1}{k\log k}\widetilde\Omega_1+\partial_k\big(\frac{1}{k\log k}\big)\partial_t=\frac{1}{k\log k}\widetilde\Omega_0-\frac{1}{k\log k}\frac{1}{k\log k}\partial_t.
\end{align*}
Combining these observations with the $r$ decay from the Volterra integrals and the arguments used to treat the contribution of the $\bfq$-terms
gives the desired estimates.
\end{proof}

We next prove some weighted estimates that are useful for estimating the contributions of $\widetilde{\calM}_k$ and $\widetilde{\bfOmega}_{1,0}^{+,-}$.

\begin{lemma}\label{lem:dkbounds1}
For $\widetilde\bmf=(\widetilde\bmf_1,\widetilde\bmf_2)^t=\widetilde\calF_\bfL[\bmf]$, 
\begin{equation}\label{eq:dkbounds1}
\|((\partial_k+k^{-1}))\widetilde\bmf_1,\partial_k\widetilde\bmf_2)^t\|_{L^2_{k\ud k}}\lesssim \|\jap{r}\bmf\|_{L^2_{r\ud r}},
\end{equation}
and
\begin{equation}\label{eq:dkbounds2}
    \|\chi_{\leq\frac{1}{2}}\widetilde\calM_k\widetilde\bmf\|_{L^2_{k\ud k}}\lesssim \|\jap{r}\bmf\|_{L^2_{r\ud r}}.
\end{equation}
\end{lemma}
\begin{proof}
For \eqref{eq:dkbounds1} we begin by writing
\begin{align*}
    \widetilde\bmf(k) = \int_0^\infty \overline{E(r,k)}^t\bmf(r) r \ud r,
\end{align*}
so that with $\bmf=(\bmf_1,\bmf_2)^t$,
\begin{align*}
    &\widetilde\bmf_1(k)=\int_0^\infty \overline{E_{1,1}}(r,k)\bmf_1(r) r \ud r+\int_0^\infty \overline{E_{1,2}}(r,k)\bmf_2(r) r \ud r=: \widetilde{\bmg}_1(k)+\widetilde{\bmh}_1(k),\\
    &\widetilde\bmf_2(k)=\int_0^\infty \overline{E_{2,1}}(r,k)\bmf_1(r) r \ud r+\int_0^\infty \overline{E_{2,2}}(r,k)\bmf_2(r) r \ud r=: \widetilde{\bmg}_2(k)+\widetilde{\bmh}_2(k).
\end{align*}
In the region $k\geq\frac{1}{2}$ the factor $k^{-1}$ is bounded and the effect of $\partial_k$ falling on the Fourier basis is at worst to produce a factor of $r$ in the region where $rk\geq1$. Therefore the desired estimate in the region $k\geq\frac{1}{2}$ follow from Plancherel in Theorem~\ref{propdFT} and the Cotlar-Stein argument from item (iii) (which is symmetric with respect to $r$ and $k$) described in the beginning of the proof of Lemma~\ref{lem:calTcalRbounds1}. The same considerations also apply in the region $rk\geq1$ and $k\leq\frac{1}{2}$, as $k^{-1}$ is bounded by $r$ in this region. Therefore, we  focus entirely on the region where $k\leq\frac{1}{2}$ and $rk\leq 1$. In the region where $r\leq1$ and $k\leq\frac{1}{2}$ since $|E(r,k)|$ is uniformly bounded in view of \eqref{eq:Eklein} and \eqref{eq:Egross}, we simply use Cauchy-Schwarz to bound
\begin{align*}
    \int_0^{\frac{1}{2}}\Big|\int_0^1\overline{E_{\ell,m}}(r,k)\bmf_j(r)r\ud r\Big|^2k \ud k\lesssim \int_0^1(\bmf_j(r))^2r \ud r.
\end{align*}
When $r\geq1$ we start with the main terms which are the off-diagonal ones $\widetilde\bmh_1$ and $\widetilde\bmg_2$. In view of \eqref{eq:Eklein} and \eqref{eq:Egross}, for $\widetilde\bmh_1$ we need to prove that
\begin{align*}
    \int_0^{\frac{1}{2}}\Big|\int_2^\infty\chi_{\leq1}(rk)r\bmf_2(r) r \ud r\Big|^2 k \ud k\lesssim \int_2^\infty(r\bmf_2(r))^2r\ud r.
\end{align*}
This follows from the one-dimensional Hardy inequality form item (i) in the beginning of the proof of Lemma~\ref{lem:calTcalRbounds1}. See the treatment of the kernel $\calK_1$ in that proof. For $\widetilde\bmg_2$, in view of \eqref{eq:Eklein} and \eqref{eq:Egross}, we need to prove that (here we use the boundedness of $(\log r)|\log k|^{-1}$ when $rk\leq1$)
\begin{align*}
    \Big(\int_0^{\frac{1}{2}}\Big|\int_2^\infty \frac{\chi_{\leq1}(rk)}{rk|\log k|}r\bmf_1(r) r\ud r\Big|^2 k \ud k\Big)^{\frac{1}{2}}\lesssim \Big(\int_2^\infty (r\bmf_1(r))^2r\ud r\Big)^{\frac{1}{2}}.
\end{align*}
This follows from a variant of the argument in item (iii) in the beginning of the proof of Lemma~\ref{lem:calTcalRbounds1}. Indeed, changing variables to $r=e^{\tau}$ and $k=e^{-\sigma}$ we need to prove that (here $F(\tau)=r^2\bmf_1(r)\vert_{r=e^{\tau}}$)
\begin{align*}
    \Big(\int_{\log 2}^\infty\Big|\int_{\log 2}^\sigma F(\tau)\frac{\ud \tau}{\sigma}\Big|^2 \ud\sigma\Big)^{\frac{1}{2}}\lesssim \Big(\int_{\log 2}^\infty F^2(\tau)\ud \tau\Big)^{\frac{1}{2}}.
\end{align*}
This estimate follows by changing variables to $u=\frac{\tau}{\sigma}$ in the inner integral, applying Minkowski, and again changing variables to $\tau=u\sigma$ in the inner integral. Finally, we consider the diagonal terms $\widetilde{\bmg}_1$ and $\widetilde\bmh_2$. Here by \eqref{eq:Eklein} and \eqref{eq:Egross} the diagonal entries of $E(r,k)$ have better $r$ decay in the region $rk\leq1$. It follows that for $\widetilde\bmg_1$ it suffices to prove that 
\begin{align*}
    \int_0^{\frac{1}{2}}\Big|\int_2^\infty\chi_{\leq1}(rk)\bmf_1(r) r \ud r\Big|^2 k \ud k\lesssim \int_2^\infty \bmf_1^2(r)r\ud r,
\end{align*}
and for $\widetilde\bmh_2$ it suffices to prove that
\begin{align*}
    \Big(\int_0^{\frac{1}{2}}\Big|\int_2^\infty \chi_{\leq1}(rk)\big((rk)^{-1}(\log k)^{-2}+rk\big)\bmf_2(r) r\ud r\Big|^2 k \ud k\Big)^{\frac{1}{2}}\lesssim \Big(\int_2^\infty \bmf_2^2(r)r\ud r\Big)^{\frac{1}{2}}.
\end{align*}
These estimates follow from similar arguments as before.
This proves \eqref{eq:dkbounds1}. Estimate \eqref{eq:dkbounds2} follows from \eqref{eq:dkbounds1} and the following Hardy inequalities in the region where $k\leq\frac{1}{2}$:
\begin{equation}\label{eq:Hardydeg1}
    \int_0^{k_0} |k^{-1}f(k)|^2k \ud k\lesssim \int_0^{k_0} |(\partial_k+k^{-1})f(k)|^2k\ud k,
\end{equation}
 for any $k_0>0$ and $f$ satisfying $f(0)=0$, and
 \begin{equation}\label{eq:Hardydeg0}
    \int_0^1 |\chi_{\leq1}(k)k^{-1}(\log k)^{-1} f(k)|^2 k \ud k\lesssim \int_0^{1} |\partial_kf(k)|^2k \ud k+\int_0^{1} |f(k)|^2k \ud k.
\end{equation}
Estimate \eqref{eq:Hardydeg1} follows by writing $k^{-1}f^2(k)=-\frac{1}{2}(kf(k))^2\frac{\ud}{\ud k}k^{-2}$ and integrating by parts. Estimate \eqref{eq:Hardydeg0} follows by writing $(\chi_{\leq1}(k))^2k^{-1}(\log k)^{-2} f^2(k)=-(\chi_{\leq1}(k))^2f^2(k) \partial_k (\log k)^{-1}$ and integrating by parts. Note that to apply \eqref{eq:Hardydeg1} we need $\widetilde\bmf_1(k)$ to vanish at $k=0$. This is guaranteed by the vanishing of $\overline{E_{1,1}}(r,0)$ and $\overline{E_{1,2}}(r,0)$ which can be seen from \eqref{eq:Eklein}.
\end{proof}

We now have all the necessary ingredients to prove Theorem~\ref{thm:transference1}. Inn the proof we will use the notation introduced at the beginning of this section regarding Lorentz boosts.

\begin{proof}[Proof of Theorem~\ref{thm:transference1}]
We prove the theorem with $\bfM$ replaced by $\bfL$. The result for $\bfM$ follows from the block diagonal structure of $\bfM$. We present the derivation of the weighted energy estimate for $\widehat{\bmg}_{\real}(t,\xi)$, the case of $\widehat{\bmg}_{\imag}(t,\xi)$ proceeding analogously. Throughout the proof we consider times $0 \leq t \leq T$. We will use the notation $\widehat\calF$ for the standard Fourier transform on $\bbR^2$.
    Using the relation
    \begin{align} 
        i\jap{\xi} \Bigl( \nabla_\xi - it \frac{\xi}{\jap{\xi}} \Bigr) = \widehat{\Omega}-\nabla_\xi(\partial_t-i\jap{\xi})-\frac{i\xi}{\jap{\xi}},
    \end{align}
    we obtain by direct computation that 
    \begin{equation} \label{equ:weighted_energy_g_jxi_nablaxi_comp}
        \begin{aligned}
            \jxi^2 \nabla_\xi \widehat{\bmg}_\real(t,\xi) &= e^{-it\jxi} \jxi^2 \, \Bigl( \nabla_\xi - it \frac{\xi}{\jap{\xi}} \Bigr) \, \Bigl( (2i\jxi)^{-1} (\pt + i\jxi) \whatcalF\bigl[ \Re\bigl( e^{i\theta} P_c^{\bfL} \bmv(t) \bigr) \bigr](\xi) \Bigr) \\ 
            &= -\frac{1}{2i} e^{-it\jxi} \frac{\xi}{\jxi} (\pt + i \jxi) \whatcalF\bigl[ \Re\bigl( e^{i\theta} P_c^{\bfL} \bmv(t) \bigr) \bigr](\xi) \\ 
            &\quad + \frac12 e^{-it\jxi} \nabla_\xi (\pt^2 + \jxi^2) \whatcalF\bigl[ \Re\bigl( e^{i\theta} P_c^{\bfL} \bmv(t) \bigr) \bigr](\xi) \\ 
            &\quad - \frac12 e^{-it\jxi} (\pt+i\jxi) \, \whatOmega \, \whatcalF\bigl[ \Re\bigl( e^{i\theta} P_c^{\bfL} \bmv(t) \bigr) \bigr](\xi) \\
            &=: I(t,\xi) + II(t,\xi) + III(t,\xi).
        \end{aligned}
    \end{equation}
    The first term $I(t,\xi)$ can be rewritten as $I(t,\xi) = -\xi \, \widehat{\bmg}_\real(t,\xi)$, whence 
    \begin{equation} \label{equ:weighted_energy_black_box_termI_final_bound}
        \bigl\|I(t,\xi)\bigr\|_{L^2_\xi} \lesssim \bigl\| \bmg_\real(t)\bigr\|_{H^1_x}.
    \end{equation}
    For the second term $II(t,\xi)$ on the right-hand side of \eqref{equ:weighted_energy_g_jxi_nablaxi_comp} we recall the operator $\bfL_\nr:=-\Delta+1+\bfV_0$ which was introduced in Section~\ref{ssecKG}. Here $\Delta$ denotes the full Laplacian on $\bbR^2$ and $\bfV_0$ defined in~\eqref{eq:Lzerodefintro1} is viewed as a radial function on $\bbR^2$, so that $e^{\pm i\theta}\bfL \bfP_c\bmv=\bfL_\nr(e^{\pm i\theta}\bfP_c\bmv)$. 
    It follows that
    \begin{equation}
        \begin{aligned}
            \bigl\|II(t,\xi)\bigr\|_{L^2_\xi} &\lesssim \bigl\| \jx \bigl( \pt^2 - \Delta + 1 \bigr) \Re\bigl( e^{i\theta} P_c^{\bfL} \bmv(t) \bigr) \bigr\|_{L^2_x} \\ 
            &\lesssim \bigl\| \jx \bigl( \pt^2 + \bfL_{\mathrm{nr}} \bigr) \Re\bigl( e^{i\theta} P_c^{\bfL} \bmv(t) \bigr) \bigr\|_{L^2_x} + \bigl\| \jx \bfV_0 \Re\bigl( e^{i\theta} P_c^{\bfL} \bmv(t) \bigr) \bigr\|_{L^2_x} \\ 
            &\lesssim \bigl\| \jap{r} P_c^{\bfL} \bmF(t) \bigr\|_{L^2_{r\ud r}} + \bigl\| \jap{r}^{-1} P_c^{\bfL} \bmv(t) \bigr\|_{L^2_{r\ud r}},            
        \end{aligned}
    \end{equation}
    where we invoked the spatial decay estimate $|\bfV_0(r)| \lesssim \jap{r}^{-2}$ and used that
    \begin{equation}
        \begin{aligned}
            ( \pt^2 + \bfL_{\mathrm{nr}} ) \Re\bigl( e^{i\theta} P_c^{\bfL} \bmv(t) \bigr) = \Re \bigl( e^{i\theta} (\pt^2 + \bfL) P_c^{\bfL} \bmv(t) \bigr) = \Re \bigl( e^{i\theta} P_c^{\bfL} \bmF(t) \bigr).
        \end{aligned}
    \end{equation}
    Finally, using H\"older's inequality and recalling the definition of the profiles \eqref{equ:weighted_estimate_black_box_def_profile}, we obtain 
    \begin{equation} \label{equ:weighted_energy_black_box_termII_final_bound}
        \begin{aligned}
            \bigl\|II(t,\xi)\bigr\|_{L^2_\xi} \lesssim \bigl\| \jap{r} \bmF(t) \bigr\|_{L^2_{r\ud r}} + \bigl\| P_c^{\bfL} \bmv(t) \bigr\|_{L^{\infty-}_{r\ud r}} 
            \lesssim \bigl\| \jap{r} \bmF(t) \bigr\|_{L^2_{r\ud r}} + \sum_{\ast \in \{\real,\imag\}} \bigl\| e^{it\jD} \bmg_\ast \bigr\|_{L^{\infty-}_x}.
        \end{aligned}
    \end{equation}
    We now turn to estimating the main third term $III(t,\xi)$ on the right-hand side of \eqref{equ:weighted_energy_g_jxi_nablaxi_comp}. Let $\bfP_\nr^+$ denote the projection onto the orthogonal complement of the span of any potential eigenfunctions of $\bfL_\nr$ with negative eigenvalue.  Using the elliptic regularity estimate $\|\jD f\|_{L^2_x} \lesssim \|f\|_{L^2_x} + \|\bfL_{\mathrm{nr}}^{\frac12}\bfP_\nr^+ f\|_{L^2_x}$ we have 
    \begin{equation} \label{equ:weighted_energy_g_jxi_term3_decomp}
        \begin{aligned}
            \bigl\|III(t,\xi)\bigr\|_{L^2_\xi} &\leq \bigl\| (\pt + i\jD) \, \Omega \, \Re\bigl( e^{i\theta} P_c^{\bfL} \bmv(t) \bigr) \bigr\|_{L^2_x} \\
            &\lesssim \bigl\| \Omega \, \Re\bigl( e^{i\theta} P_c^{\bfL} \bmv(t) \bigr) \bigr\|_{L^2_x} + \bigl\| \pt \, \Omega \, \Re\bigl( e^{i\theta} P_c^{\bfL} \bmv(t) \bigr) \bigr\|_{L^2_x} \\
            &\quad + \Bigl\| \Omega \, \bfL_{\mathrm{nr}}^{\frac12} \Re\bigl( e^{i\theta} P_c^{\bfL} \bmv(t) \bigr) \Bigr\|_{L^2_x} + \Bigl\| \bigl[ \bfL_{\mathrm{nr}}^{\frac12}\bfP_\nr^+, \Omega \bigr] \, \Re\bigl( e^{i\theta} P_c^{\bfL} \bmv(t) \bigr) \Bigr\|_{L^2_x}+\jap{t}\|e^{it\jap{D}}\bmg_\real(t)\|_{L_x^{\infty-}} \\ 
            &=: III_1(t) + III_2(t) + III_3(t) + III_4(t)+\jap{t}\|e^{it\jap{D}}\bmg_\real(t)\|_{L_x^{\infty-}}.
        \end{aligned}
    \end{equation}
    Here we have used that $\bfP_\nr^+(e^{\pm i\theta}\bmv(t,r))=e^{\pm i\theta}\bmv(t,r)$, which follows from the positivity of $\bfL$. The last term in the estimate above is an acceptable error in \eqref{equ:weighted_estimate_black_box_bound}. It appears because we have replaced $\Omega \, \bfL_{\mathrm{nr}}^{\frac12} \Re\bigl( e^{i\theta} P_c^{\bfL} \bfP_\nr^c\bmv(t) \bigr)$ by $\Omega \, \bfL_{\mathrm{nr}}^{\frac12} \Re\bigl( e^{i\theta} P_c^{\bfL} \bmv(t) \bigr)$. For the remaining terms we begin with estimating the first term $III_1(t)$ on the right-hand side of \eqref{equ:weighted_energy_g_jxi_term3_decomp}.
    Here we distinguish the regions $r \lesssim 1$ and $r \gtrsim 1$. Using the identity \eqref{equ:equivboost1} for the latter, we find 
    \begin{equation} \label{equ:weighted_energy_g_jxi_term3_1_decomp}
        \begin{aligned}
            &III_1(t) = \bigl\| \Omega \, \Re\bigl( e^{i\theta} P_c^{\bfL} \bmv(t) \bigr) \bigr\|_{L^2_x} \\
            &\lesssim \bigl\| \chi_{< 1}(r) \, (x \pt + t \nabla_x) \, \Re\bigl( e^{i\theta} P_c^{\bfL} \bmv(t) \bigr) \bigr\|_{L^2_x} + \bigl\| \chi_{\geq 1}(r) \, \Omega \, \Re\bigl( e^{i\theta} P_c^{\bfL} \bmv(t) \bigr) \bigr\|_{L^2_x} \\ 
            &\lesssim \bigl\| \chi_{< 1}(r) \, (x \pt + t \nabla_x) \, \Re\bigl( e^{i\theta} P_c^{\bfL} \bmv(t) \bigr) \bigr\|_{L^2_x} + \bigl\| \chi_{\geq 1}(r) \Omega_1^- P_c^{\bfL} \bmv(t) \bigr\|_{L^2_{r\ud r}} \\
            &\quad + \bigl\| \chi_{\geq 1}(r) \Omega_1^+ P_c^{\bfL} \bmv(t) \bigr\|_{L^2_{r\ud r}}.
        \end{aligned}
    \end{equation}
    For the first term on the right-hand side of \eqref{equ:weighted_energy_g_jxi_term3_1_decomp}, recalling the definition of the profile $\bmg_\real(t)$ from \eqref{equ:weighted_estimate_black_box_def_profile}, we obtain by H\"older's inequality 
    \begin{equation}
        \begin{aligned}
            &\bigl\| \chi_{< 1}(r) \, (x \pt + t \nabla_x) \, \Re\bigl( e^{i\theta} P_c^{\bfL} \bmv(t) \bigr) \bigr\|_{L^2_x} \lesssim \bigl\| \bmg_\real(t) \bigr\|_{H^1_x} + t \, \bigl\| \jD e^{it\jD} \bmg_\real(t) \bigr\|_{L^{\infty-}_x}.
        \end{aligned}
    \end{equation}
    For the second term on the right-hand side of \eqref{equ:weighted_energy_g_jxi_term3_1_decomp}, we use the simple relation $\Omega_1^- = \Omega_1^+ - 2 t r^{-1}$ to find that
    \begin{equation}
        \begin{aligned}
            \bigl\| \chi_{\geq 1}(r) \Omega_1^- P_c^{\bfL} \bmv(t) \bigr\|_{L^2_{r\ud r}} &\lesssim \bigl\| \chi_{\geq 1}(r) \Omega_1^+ P_c^{\bfL} \bmv(t) \bigr\|_{L^2_{r\ud r}} + t \bigl\| r^{-1} \chi_{\geq 1}(r) P_c^{\bfL} \bmv(t) \bigr\|_{L^2_{r\ud r}}. 
        \end{aligned}
    \end{equation}
    The last term can easily be dealt with using H\"older's inequality. We obtain that
    \begin{equation}
        \begin{aligned}
            &t \, \bigl\| r^{-1} \chi_{\geq 1}(r) P_c^{\bfL} \bmv(t) \bigr\|_{L^2_{r\ud r}} \lesssim t \, \bigl\| P_c^{\bfL} \bmv(t) \bigr\|_{L^{\infty-}_{r\ud r}} \lesssim \sum_{\ast \in \{\real,\imag\}} t \, \bigl\| e^{it\jD} \bmg_\ast(t) \bigr\|_{L^{\infty-}_x}.
        \end{aligned}
    \end{equation}
    We are thus left to estimate the main last term $\bigl\| \chi_{\geq 1}(r) \Omega_1^+ P_c^{\bfL} \bmv(t) \bigr\|_{L^2_{r\ud r}}$ on the right-hand side of \eqref{equ:weighted_energy_g_jxi_term3_1_decomp}. By Lemmas~\ref{lem:trans1} and~\ref{lem:calTcalRbounds1},  Plancherel, and H\"older's inequality, we have 
    \begin{equation}
        \begin{aligned}
            \bigl\| \chi_{\geq 1}(r) \Omega_1^+ P_c^{\bfL} \bmv(t) \bigr\|_{L^2_{r\ud r}} &\leq \bigl\| \calT_{\mathrm{off}}\bigl[\widetilde\bfOmega_{1,0}^{+,-} \widetilde{\bmv}\bigr](t,r) - \sigma_3 \calR_{\mathrm{off}}\bigl[\pt \widetilde{\bmv}\bigr](t,r) \bigr\|_{L^2_{r\ud r}} \\
            &\quad + \bigl\| \calT_{\mathrm{diag}}\bigl[\widetilde\bfOmega_{0,1}^{-,+} \widetilde{\bmv}\bigr](t,r) - \sigma_3 \calR_{\mathrm{diag}}\bigl[\pt \widetilde{\bmv}\bigr](t,r) \bigr\|_{L^2_{r \ud r}} \\
            &\quad + t \, \bigl\| r^{-1} \chi_{\geq 1}(r) P_c^{\bfL} \bmv(t) \bigr\|_{L^2_{r\ud r}} \\ 
            &\lesssim \bigl\|\bigl( \widetilde{\bfOmega}_{1,0}^{+,-}\widetilde{\bmv}\bigr)(t,k)\bigr\|_{L^2_{k\ud k}} + \bigl\|\bigl( \chi_{\leq 1}(k) \widetilde{\calM}_k \partial_t \widetilde{\bmv}\bigr)(t,k)\bigr\|_{L^2_{k\ud k}} \\ 
            &\quad + \bigl\|\P_c^{\bfL} \pt \bmv(t)\bigr\|_{L^2_{r\ud r}} + t \, \bigl\| P_c^{\bfL} \bmv(t) \bigr\|_{L^{\infty-}_{r\ud r}} \\ 
            &\lesssim \bigl\|\bigl( \widetilde{\bfOmega}_{1,0}^{+,-}\widetilde{\bmv}\bigr)(t,k)\bigr\|_{L^2_{k\ud k}} + \bigl\|\bigl( \chi_{\leq 1}(k) \widetilde{\calM}_k \partial_t \widetilde{\bmv}\bigr)(t,k)\bigr\|_{L^2_{k\ud k}} \\ 
            &\quad + \sum_{\ast \in \{\real,\imag\}} \Bigl( \bigl\| \bmg_\ast(t) \bigr\|_{H^1_x} + t \, \bigl\| e^{it\jD} \bmg_\ast(t) \bigr\|_{L^{\infty-}_x} \Bigr).
        \end{aligned}
    \end{equation}
    It thus remains to bound the first two terms on the right-hand side of the preceding estimate.
    To this end we note that by \eqref{equ:weighted_estimate_black_box_evol_equ}, $\widetilde{\bmv}(t)$ satisfies $(\pt^2 + \jap{k}^2) \widetilde{\bmv} = \widetilde{F}$ on the distorted Fourier side and that both $\widetilde{\bfOmega}_{1,0}^{+,-}$ and $\chi_{\leq 1}(k) \widetilde{\calM}_k \partial_t$ commute with $(\pt^2 + \jap{k}^2)$. From the Duhamel formula   
    \begin{equation}
        \begin{aligned}
            \bigl(\widetilde{\bfOmega}_{1,0}^{+,-}\widetilde{\bmv}\bigr)(t,k) &= \cos(t\jap{k}) \bigl( \widetilde{\bfOmega}_{1,0}^{+,-} \widetilde{\bmv}\bigr)(0,k) + \frac{\sin(t\jap{k})}{\jap{k}} \bigl( \widetilde{\bfOmega}_{1,0}^{+,-} \pt \widetilde{\bmv}\bigr)(0,k) \\ 
            &\quad + \int_0^t \frac{\sin\bigl((t-s)\jap{k}\bigr)}{\jap{k}} \bigl( \widetilde{\bfOmega}_{1,0}^{+,-} \widetilde{\bmF}\bigr)(s,k) \, \ud s,
        \end{aligned}
    \end{equation}
    we obtain 
    \begin{equation} \label{equ:weighted_estimate_black_box_wtilbfOmega_duhamel_bound}
        \begin{aligned}
            \bigl\|\bigl(\widetilde{\bfOmega}_{1,0}^{+,-}\widetilde{\bmv}\bigr)(t,k)\bigr\|_{L^2_{k\ud k}} &\lesssim \bigl\| \bigl( \widetilde{\bfOmega}_{1,0}^{+,-} \widetilde{\bmv}\bigr)(0,k) \bigr\|_{L^2_{k\ud k}} + \bigl\|\jap{k}^{-1} \bigl( \pt \widetilde{\bfOmega}_{1,0}^{+,-} \widetilde{\bmv} \bigr)(0,k) \bigr\|_{L^2_{k\ud k}} \\
            &\quad + \biggl\| \int_0^t \frac{\sin\bigl((t-s)\jap{k}\bigr)}{\jap{k}} \bigl( \widetilde{\bfOmega}_{1,0}^{+,-} \widetilde{\bmF}\bigr)(s,k) \, \ud s \biggr\|_{L^2_{k\ud k}}.
        \end{aligned}
    \end{equation}
    By Lemma~\ref{lem:dkbounds1}, Plancherel, and inserting the evolution equation~\eqref{equ:weighted_estimate_black_box_evol_equ}, the first two terms on the right-hand side of \eqref{equ:weighted_estimate_black_box_wtilbfOmega_duhamel_bound} are bounded by
    \begin{equation}
        \begin{aligned}
            &\bigl\| \bigl( \widetilde{\bfOmega}_{1,0}^{+,-} \widetilde{\bmv}\bigr)(0, k) \bigr\|_{L^2_{k\ud k}} + \bigl\|\jap{k}^{-1} \bigl( \pt \widetilde{\bfOmega}_{1,0}^{+,-} \widetilde{\bmv} \bigr)(0, k) \bigr\|_{L^2_{k\ud k}} \\
            &\lesssim \biggl\| \begin{pmatrix} \partial_k + \frac{1}{k} & 0 \\ 0 & \partial_k \end{pmatrix} (\pt \widetilde{\bmv})(0, k) \biggr\|_{L^2_{k \ud k}} + \biggl\| \begin{pmatrix} \partial_k + \frac{1}{k} & 0 \\ 0 & \partial_k \end{pmatrix} (\pt^2 \widetilde{\bmv})(0, k) \biggr\|_{L^2_{k \ud k}} + \bigl\|\jap{k}^{-1} k \, \widetilde{\bmv}(0, k) \bigr\|_{L^2_{k\ud k}} \\
            &\lesssim \bigl\| \jap{r} P_c^{\bfL} (\pt \bmv)(0) \bigr\|_{L^2_{r \ud r}} + \bigl\| \jap{r} P_c^{\bfL} (\pt^2 \bmv)(0) \bigr\|_{L^2_{r \ud r}} + \| P_c^{\bfL} \bmv(0) \|_{L^2_{r\ud r}} \\ 
            &\lesssim \bigl\| \jap{r} P_c^{\bfL} (\pt \bmv)(0) \bigr\|_{L^2_{r \ud r}} + \bigl\| \jap{r} \bfL P_c^{\bfL} \bmv(0) \bigr\|_{L^2_{r \ud r}} + \bigl\| \jap{r} \bmF(0) \bigr\|_{L^2_{r\ud r}} + \| P_c^{\bfL} \bmv(0) \|_{L^2_{r\ud r}}.
        \end{aligned}
    \end{equation}
    To estimate the last term on the right-hand side of \eqref{equ:weighted_estimate_black_box_wtilbfOmega_duhamel_bound}, we break up $\widetilde{\bfOmega}_{1,0}^{+,-}$ and then use Lemma~\ref{eq:Lzerodefintro1} as well as the dual ILED estimate \eqref{eq:ILEDDuhamel1} with $\gamma_2=1$ from Lemma~\ref{lem:LEDallfereqs1} to infer that
    \begin{equation}
        \begin{aligned}
            &\biggl\| \int_0^t \frac{\sin\bigl((t-s)\jap{k}\bigr)}{\jap{k}} \bigl(\widetilde{\bfOmega}_{1,0}^{+,-} \widetilde{\bmF}\bigr)(s,k) \, \ud s \biggr\|_{L^2_{k\ud k}} \\
            &\lesssim \biggl\| \int_0^t e^{\pm i s \jap{k}} \jap{k}^{-1} k \cdot s \cdot \widetilde{\bmF}(s,k) \, \ud s \biggr\|_{L^2_{k\ud k}} + \biggl\| \int_0^t e^{\pm i s \jap{k}} \jap{k}^{-1} \begin{pmatrix} \partial_k + \frac{1}{k} & 0 \\ 0 & \partial_k \end{pmatrix} (\partial_s \widetilde{\bmF})(s,k) \, \ud s \biggr\|_{L^2_{k\ud k}} \\
            &\lesssim \bigl\| s \cdot \jap{r} \bmF(s,r) \bigr\|_{L^2_s([0,t]; L^2_{r\ud r})} + \bigl\| \jap{r} \partial_s \bmF(s,r) \bigr\|_{L^1_s([0,t]; L^2_{r\ud r})}.
        \end{aligned}
    \end{equation}
    Proceeding analogously, using in particular \eqref{eq:dkbounds1} from Lemma~\ref{lem:dkbounds1}, we obtain that
    \begin{equation}
        \begin{aligned}
            &\bigl\| \bigl( \chi_{\leq 1}(k) \widetilde{\calM}_k \partial_t \widetilde{\bmv}\bigr)(t,k) \bigr\|_{L^2_{k\ud k}} \lesssim \bigl\| \bigl( \chi_{\leq 1}(k) \widetilde{\calM}_k \partial_t \widetilde{\bmv}\bigr)(0, k) \bigr\|_{L^2_{k\ud k}} + \bigl\|\jap{k}^{-1} \bigl( \chi_{\leq 1}(k) \widetilde{\calM}_k \partial_t^2 \widetilde{\bmv}\bigr)(0, k) \bigr\|_{L^2_{k\ud k}} \\
            &\quad + \biggl\| \int_0^t \frac{\sin\bigl((t-s)\jap{k}\bigr)}{\jap{k}} \bigl( \chi_{\leq 1}(k) \widetilde{\calM}_k \partial_s \widetilde{\bmF}\bigr)(s,k) \, \ud s \biggr\|_{L^2_{k\ud k}} \\ 
            &\lesssim \bigl\| \jap{r} P_c^{\bfL} (\pt \bmv)(0) \bigr\|_{L^2_{r\ud r}} + \bigl\| \jap{r} P_c^{\bfL} (\pt^2 \bmv)(0) \bigr\|_{L^2_{r\ud r}} 
            + \bigl\| \jap{r} (\partial_s \bmF)(s,r) \bigr\|_{L^1_s([0,t]; L^2_{r\ud r})} \\ 
            &\lesssim \bigl\| \jap{r} P_c^{\bfL} (\pt \bmv)(0) \bigr\|_{L^2_{r\ud r}} + \bigl\| \jap{r} \bfL P_c^{\bfL} \bmv(0) \bigr\|_{L^2_{r\ud r}} + \bigl\| \jap{r} \bmF(0) \bigr\|_{L^2_{r\ud r}} + \bigl\| \jap{r} (\partial_s \bmF)(s,r) \bigr\|_{L^1_s([0,t]; L^2_{r\ud r})}.             
        \end{aligned}
    \end{equation}
    In conclusion, we have found that 
    \begin{equation} \label{equ:weighted_energy_black_box_termIII_1_final_bound}
        \begin{aligned}
            III_1(t) &\lesssim \| P_c^{\bfL} \bmv(0) \|_{L^2_{r\ud r}} + \bigl\| \jap{r} \bfL P_c^{\bfL} \bmv(0) \|_{L^2_{r\ud r}} + \bigl\| \jap{r} P_c^{\bfL} (\pt \bmv)(0) \bigr\|_{L^2_{r \ud r}} + \bigl\| \jap{r} \bmF(0) \bigr\|_{L^2_{r\ud r}} \\
            &\quad  + \sum_{\ast \in \{\real,\imag\}} \Bigl( \bigl\| \bmg_\ast(t) \bigr\|_{H^1_x} + t \, \bigl\| \jD e^{it\jD} \bmg_\ast(t) \bigr\|_{L^{\infty-}_x} \Bigr) \\
            &\quad + \bigl\| s \cdot \jap{r} \bmF(s,r) \bigr\|_{L^2_s([0,t]; L^2_{r\ud r})} + \bigl\| \jap{r} \partial_s \bmF(s,r) \bigr\|_{L^1_s([0,t]; L^2_{r\ud r})}.
        \end{aligned}
    \end{equation}

    Next, we estimate the second term $III_2(t)$ on the right-hand side of \eqref{equ:weighted_energy_g_jxi_term3_decomp}, proceeding similarly as for the term $III_1(t)$.
    Using the identity \eqref{equ:equivboost1}, we have 
    \begin{equation} \label{equ:weighted_energy_g_jxi_term3_2_decomp}
        \begin{aligned}
            III_2(t) &= \bigl\| \pt \, \Omega \, \Re\bigl( e^{i\theta} P_c^{\bfL} \bmv(t) \bigr) \bigr\|_{L^2_x} \\ 
            &\lesssim \bigl\| \nabla_x \Re\bigl( e^{i\theta} P_c^{\bfL} \bmv(t) \bigr) \bigr\|_{L^2_x} + \bigl\| \Omega \, \Re\bigl( e^{i\theta} P_c^{\bfL} \pt \bmv(t) \bigr) \bigr\|_{L^2_x} \\ 
            &\lesssim \|\bmg_\real(t)\|_{H^1_x} + \bigl\| \chi_{< 1}(r) \, (x \pt + t \nabla_x) \, \Re\bigl( e^{i\theta} P_c^{\bfL} \pt \bmv(t) \bigr) \bigr\|_{L^2_x}  \\ 
            &\quad + \bigl\| \chi_{\geq 1}(r) \Omega_1^- P_c^{\bfL} \pt \bmv(t) \bigr\|_{L^2_{r\ud r}} + \bigl\| \chi_{\geq 1}(r) \Omega_1^+ P_c^{\bfL} \pt \bmv(t) \bigr\|_{L^2_{r\ud r}}.
        \end{aligned}
    \end{equation}
    For the second term on the right-hand side of \eqref{equ:weighted_energy_g_jxi_term3_2_decomp}, using H\"older's inequality and inserting the evolution equation \eqref{equ:weighted_estimate_black_box_evol_equ}, we obtain 
    \begin{equation}
        \begin{aligned}
            \bigl\| \chi_{< 1}(r) \, (x \pt + t \nabla_x) \, \Re\bigl( e^{i\theta} P_c^{\bfL} \pt \bmv(t) \bigr) \bigr\|_{L^2_x} &\lesssim \bigl\| P_c^{\bfL} \pt^2 \bmv(t) \bigr\|_{L^2_{r\ud r}} + t \, \bigl\| \jD e^{it\jD} \bmg_\real(t) \bigr\|_{L^{\infty-}_x} \\ 
            &\lesssim \bigl\| \bfL P_c^{\bfL} \bmv(t) \bigr\|_{L^2_{r\ud r}} + \bigl\| \bmF(t) \bigr\|_{L^2_{r\ud r}} + t \, \bigl\| \jD e^{it\jD} \bmg_\real(t) \bigr\|_{L^{\infty-}_x}.
        \end{aligned}
    \end{equation}
    For the third term on the right-hand side of \eqref{equ:weighted_energy_g_jxi_term3_2_decomp}, we again use the relation $\Omega_1^- = \Omega_1^+-2tr^{-1}$ along with H\"older's inequality to conclude that
    \begin{equation}
        \begin{aligned}
            \bigl\| \chi_{\geq 1}(r) \Omega_1^- P_c^{\bfL} \pt \bmv(t) \bigr\|_{L^2_{r\ud r}} &\lesssim \bigl\| \chi_{\geq 1}(r) \Omega_1^+ P_c^{\bfL} \pt \bmv(t) \bigr\|_{L^2_{r\ud r}} + t \, \bigl\| r^{-1} \chi_{\geq 1}(r) P_c^{\bfL} \pt \bmv(t) \bigr\|_{L^2_{r \ud r}} \\
            &\lesssim  \bigl\| \chi_{\geq 1}(r) \Omega_1^+ P_c^{\bfL} \pt \bmv(t) \bigr\|_{L^2_{r\ud r}} + \sum_{\ast \in \{\real,\imag\}} t \, \bigl\| \jD e^{it\jD} \bmg_\ast(t) \bigr\|_{L^{\infty-}_x}.
        \end{aligned}
    \end{equation}
    We are thus left to bound the term $\bigl\| \chi_{\geq 1}(r) \Omega_1^+ P_c^{\bfL} \pt \bmv(t) \bigr\|_{L^2_{r\ud r}}$, which is also the fourth term on the right-hand side of \eqref{equ:weighted_energy_g_jxi_term3_2_decomp}. 
    By Lemmas~\ref{lem:trans1},~\ref{lem:calTcalRbounds1}, Plancherel, H\"older's inequality, and the evolution equation \eqref{equ:weighted_estimate_black_box_evol_equ}, we have 
    \begin{equation}
        \begin{aligned}
            \bigl\| \chi_{\geq 1}(r) \Omega_1^+ P_c^{\bfL} \pt \bmv(t) \bigr\|_{L^2_{r\ud r}} 
            &\lesssim \bigl\|\bigl( \widetilde{\bfOmega}_{1,0}^{+,-} \pt \widetilde{\bmv}\bigr)(t,k)\bigr\|_{L^2_{k\ud k}} + \bigl\|\bigl( \chi_{\leq 1}(k) \widetilde{\calM}_k \partial_t^2 \widetilde{\bmv}\bigr)(t,k)\bigr\|_{L^2_{k\ud k}} \\ 
            &\quad + \bigl\|\P_c^{\bfL} \pt^2 \bmv(t)\bigr\|_{L^2_{r\ud r}} + t \, \bigl\| P_c^{\bfL} \pt \bmv(t) \bigr\|_{L^{\infty-}_{r\ud r}} \\ 
            &\lesssim \bigl\|\bigl( \widetilde{\bfOmega}_{1,0}^{+,-} \pt\widetilde{\bmv}\bigr)(t,k)\bigr\|_{L^2_{k\ud k}} + \bigl\|\bigl( \chi_{\leq 1}(k) \widetilde{\calM}_k \pt^2 \widetilde{\bmv}\bigr)(t,k)\bigr\|_{L^2_{k\ud k}} \\ 
            &\quad + \bigl\| \bfL P_c^{\bfL} \bmv(t) \bigr\|_{L^2_{r\ud r}} + \bigl\| \bmF(t) \bigr\|_{L^2_{r\ud r}} + \sum_{\ast \in \{\real,\imag\}} t \, \bigl\| \jD e^{it\jD} \bmg_\ast(t) \bigr\|_{L^{\infty-}_x}.
        \end{aligned}
    \end{equation}
    Commuting $\widetilde{\bfOmega}_{1,0}^{+,-} \pt$ with the evolution equation for $\widetilde{\bmv}(t)$ on the distorted Fourier side, we obtain as before 
    \begin{equation}
        \begin{aligned}
            \bigl\|\bigl( \widetilde{\bfOmega}_{1,0}^{+,-} \pt\widetilde{\bmv}\bigr)(t,k)\bigr\|_{L^2_{k\ud k}} &\lesssim \bigl\|\bigl( \widetilde{\bfOmega}_{1,0}^{+,-} \pt\widetilde{\bmv}\bigr)(0,k) \bigr\|_{L^2_{k\ud k}} + \bigl\| \jap{k}^{-1} \bigl( \pt \widetilde{\bfOmega}_{1,0}^{+,-} \pt\widetilde{\bmv}\bigr)(0,k) \bigr\|_{L^2_{k\ud k}} \\ 
            &\quad + \biggl\| \int_0^t \frac{\sin\bigl((t-s)\jap{k}\bigr)}{\jap{k}} \bigl( \widetilde{\bfOmega}_{1,0}^{+,-} \partial_s \widetilde{\bmF}\bigr)(s,k) \, \ud s \biggr\|_{L^2_{k\ud k}}.
        \end{aligned}
    \end{equation}
    Inserting the equation $\pt^2 \widetilde{\bmv} = - \jap{k}^2 \widetilde{\bmv} + \widetilde{\bmF}$ and invoking Lemma~\ref{lem:dkbounds1}, we find
    \begin{equation}
        \begin{aligned}
            &\bigl\|\bigl( \widetilde{\bfOmega}_{1,0}^{+,-} \pt\widetilde{\bmv}\bigr)(0,k) \bigr\|_{L^2_{k\ud k}} + \bigl\| \jap{k}^{-1} \bigl( \pt \widetilde{\bfOmega}_{1,0}^{+,-} \pt\widetilde{\bmv}\bigr)(0,k) \bigr\|_{L^2_{k\ud k}} \\
            &\lesssim \biggl\| \begin{pmatrix} \partial_k + \frac{1}{k} & 0 \\ 0 & \partial_k \end{pmatrix} (\pt^2 \widetilde{\bmv})(0,k) \biggr\|_{L^2_{k\ud k}} + \biggl\| \jap{k}^{-1} \begin{pmatrix} \partial_k + \frac{1}{k} & 0 \\ 0 & \partial_k \end{pmatrix} (\pt^3 \widetilde{\bmv})(0,k) \biggr\|_{L^2_{k\ud k}} + \bigl\| (\pt \widetilde{\bmv})(0,k) \bigr\|_{L^2_{k\ud k}} \\ 
            &\lesssim \bigl\| \jap{r} P_c^{\bfL} (\pt^2 \bmv)(0) \bigr\|_{L^2_{r\ud r}} + \biggl\| \begin{pmatrix} \partial_k + \frac{1}{k} & 0 \\ 0 & \partial_k \end{pmatrix} \jap{k} (\pt \widetilde{\bmv})(0,k) \biggr\|_{L^2_{k\ud k}} \\
            &\quad + \biggl\| \begin{pmatrix} \partial_k + \frac{1}{k} & 0 \\ 0 & \partial_k \end{pmatrix} (\pt \widetilde{\bmF})(0,k) \biggr\|_{L^2_{k\ud k}} + \bigl\| (\pt \widetilde{\bmv})(0,k) \bigr\|_{L^2_{k\ud k}} \\      
            &\lesssim \bigl\| \jap{r} \bfL \bfP_c^{\bfL} \bmv(0) \bigr\|_{L^2_{r\ud r}} + \bigl\| \jap{r} \bmF(0) \bigr\|_{L^2_{r \ud r}} + \bigl\| \jap{r} \bfL^{\frac12} P_c^{\bfL} (\pt \bmv)(0) \bigr\|_{L^2_{r\ud r}} + \bigl\| \jap{r} (\pt \bmF)(0) \bigr\|_{L^2_{r\ud r}} + \bigl\| P_c^{\bfL} \pt \bmv(0) \bigr\|_{L^2_{r \ud r}}.
        \end{aligned}
    \end{equation}
    Moreover, breaking up $\widetilde{\bfOmega}_{1,0}^{+,-}$ again and using Lemma~\ref{lem:dkbounds1} as well as the dual ILED estimate \eqref{eq:ILEDDuhamel1} with $\gamma_2=1$ from Lemma~\ref{lem:LEDallfereqs1}, we find that
    \begin{equation}
        \begin{aligned}
            &\biggl\| \int_0^t \frac{\sin\bigl((t-s)\jap{k}\bigr)}{\jap{k}} \bigl( \widetilde{\bfOmega}_{1,0}^{+,-} \partial_s \widetilde{\bmF}\bigr)(s,k) \, \ud s \biggr\|_{L^2_{k\ud k}} \\
            &\quad \lesssim \bigl\| s \cdot \jap{r} \partial_s \bmF(s,r) \bigr\|_{L^2_s([0,t]; L^2_{r\ud r})} + \bigl\| \jap{r} \partial_s^2 \bmF(s,r) \bigr\|_{L^1_s([0,t]; L^2_{r\ud r})}.
        \end{aligned}
    \end{equation}
    
    Similarly, inserting the equation $\pt^2 \widetilde{\bmv} = - \jap{k}^2 \widetilde{\bmv} + \widetilde{\bmF}$, invoking Duhamel's formula for $\widetilde{\bmv}(t,k)$, and using Lemma~\ref{lem:dkbounds1}, we infer that
    \begin{equation}
        \begin{aligned}
            &\bigl\|\bigl( \chi_{\leq 1}(k) \widetilde{\calM}_k \partial_t^2 \widetilde{\bmv}\bigr)(t,k)\bigr\|_{L^2_{k\ud k}} \\
            &\lesssim \bigl\| \chi_{\leq 1}(k) \widetilde{\calM}_k \jap{k}^2 \widetilde{\bmv}(0,k) \bigr\|_{L^2_{k \ud k}} + \bigl\| \chi_{\leq 1}(k) \widetilde{\calM}_k \jap{k} (\pt \widetilde{\bmv})(0,k) \bigr\|_{L^2_{k \ud k}} \\ 
            &\quad + \biggl\| \int_0^t \sin\bigl((t-s)\jap{k}\bigr) \widetilde{\calM}_k \jap{k} \widetilde{\bmF}(s,k) \, \ud s \biggr\|_{L^2_{k\ud k}} + \bigl\| \chi_{\leq 1}(k) \widetilde{\calM}_k \widetilde{\bmF}(t,k) \bigr\|_{L^2_{k\ud k}} \\
            &\lesssim \bigl\| \jap{r} \bfL P_c^{\bfL} \bmv(0) \bigr\|_{L^2_{r \ud r}} + \bigl\| \jap{r} \bfL^{\frac12} P_c^{\bfL} (\pt \bmv)(0) \bigr\|_{L^2_{r \ud r}} + \bigl\| \jap{r} \bfL^{\frac12} P_c^{\bfL} \bmF(s) \bigr\|_{L^1_s([0,t]; L^2_{r\ud r})} + \bigl\| \jap{r} \bmF(t) \bigr\|_{L^2_{r \ud r}}.
        \end{aligned}
    \end{equation}
    In summary, we have found that 
    \begin{align} 
            III_2(t) 
            &\lesssim \bigl\| P_c^{\bfL} \pt \bmv(0) \bigr\|_{L^2_{r \ud r}} + \bigl\| \jap{r} \bfL P_c^{\bfL} \bmv(0) \bigr\|_{L^2_{r\ud r}} + \bigl\| \jap{r} \bfL^{\frac12} P_c^{\bfL} (\pt \bmv)(0) \bigr\|_{L^2_{r \ud r}} \nonumber\\ 
            &\quad + \bigl\| \jap{r} \bmF(0) \bigr\|_{L^2_{r\ud r}} + \bigl\| \jap{r} (\pt \bmF)(0) \bigr\|_{L^2_{r\ud r}} + \bigl\| \jap{r} \bmF(t) \bigr\|_{L^2_{r\ud r}} \label{equ:weighted_energy_black_box_termIII_2_final_bound}\\
            &\quad + \bigl\| \bfL P_c^{\bfL} \bmv(t) \bigr\|_{L^2_{r\ud r}} + \|\bmg_\real(t)\|_{H^1_x} + \sum_{\ast \in \{\real,\imag\}} t \, \bigl\| \jD e^{it\jD} \bmg_\ast(t) \bigr\|_{L^{\infty-}_x}\nonumber \\ 
            &\quad + \bigl\| s \cdot \jap{r} \partial_s \bmF(s,r) \bigr\|_{L^2_s([0,t]; L^2_{r\ud r})} + \bigl\| \jap{r} \partial_s^2 \bmF(s,r) \bigr\|_{L^1_s([0,t]; L^2_{r\ud r})} + \bigl\| \jap{r} \bfL^{\frac12} P_c^{\bfL} \bmF(s) \bigr\|_{L^1_s([0,t]; L^2_{r\ud r})}.\nonumber
    \end{align}
    
    We continue with the estimates for the third term $III_3(t)$ on the right-hand side of \eqref{equ:weighted_energy_g_jxi_term3_decomp}, where we proceed similarly as for the preceding terms $III_1(t)$ and $III_2(t)$.
    Using the identity \eqref{equ:equivboost1}, we first obtain
    \begin{equation} \label{equ:weighted_energy_g_jxi_term3_3_decomp}
        \begin{aligned}
            III_3(t) &= \bigl\| \Omega \, \Re\bigl( e^{i\theta} \bfL^{\frac12} P_c^{\bfL} \bmv(t) \bigr) \bigr\|_{L^2_x} \\             
            &\lesssim \bigl\| \chi_{<1}(r) \bigl(x \pt + t \nabla_x\bigr) \Re\bigl( e^{i\theta} \bfL^{\frac12} P_c^{\bfL} \bmv(t) \bigr) \bigr\|_{L^2_x} + \bigl\| \chi_{\geq 1}(r) \Omega_1^- \bfL^{\frac12} P_c^{\bfL} \bmv(t) \bigr\|_{L^2_{r\ud r}} \\ 
            &\quad + \bigl\| \chi_{\geq 1}(r) \Omega_1^+ \bfL^{\frac12} P_c^{\bfL} \bmv(t) \bigr\|_{L^2_{r\ud r}}.
        \end{aligned}
    \end{equation}    
    For the first term on the right-hand side of \eqref{equ:weighted_energy_g_jxi_term3_3_decomp}, we obtain 
    \begin{equation}
        \begin{aligned}
            \bigl\| \chi_{<1}(r) \bigl(x \pt + t \nabla_x\bigr) \Re\bigl( e^{i\theta} \bfL^{\frac12} P_c^{\bfL} \bmv(t) \bigr) \bigr\|_{L^2_x} &= \bigl\| \chi_{<1}(r) \bigl(x \pt + t \nabla_x\bigr) \bfL_{\mathrm{nr}}^{\frac12} \Re\bigl( e^{i\theta} P_c^{\bfL} \bmv(t) \bigr) \bigr\|_{L^2_x} \\ 
            &\lesssim \bigl\| \bfL_{\mathrm{nr}}^{\frac12} \Re\bigl( e^{i\theta} P_c^{\bfL} \pt \bmv(t) \bigr) \bigr\|_{L^2_x} + t \, \bigl\| \nabla_x \bfL_{\mathrm{nr}}^{\frac12} \Re\bigl( e^{i\theta} P_c^{\bfL} \bmv(t) \bigr) \bigr\|_{L^{\infty-}_x} \\ 
            &\lesssim \bigl\| \bmg_\real(t) \bigr\|_{H^2_x} + t \, \bigl\| \jD^2 e^{it\jD} \bmg_\real(t) \bigr\|_{L^{\infty-}_x}.
        \end{aligned}
    \end{equation}
    For the second term on the right-hand side of \eqref{equ:weighted_energy_g_jxi_term3_3_decomp} we again use the relation $\Omega_1^{-} = \Omega_1^{+}-2tr^{-1}$ and H\"older's inequality to conclude
    \begin{equation}
        \begin{aligned}
            \bigl\| \chi_{\geq 1}(r) \Omega_1^- \bfL^{\frac12} P_c^{\bfL} \bmv(t) \bigr\|_{L^2_{r\ud r}} &\lesssim \bigl\| \chi_{\geq 1}(r) \Omega_1^+ \bfL^{\frac12} P_c^{\bfL} \bmv(t) \bigr\|_{L^2_{r\ud r}} + t \, \bigl\| r^{-1} \chi_{\geq 1}(r) \bfL^{\frac12} P_c^{\bfL} \bmv(t) \bigr\|_{L^2_{r \ud r}}.
        \end{aligned}
    \end{equation}
    We may bound the last term by
    \begin{equation}
        \begin{aligned}
            t \, \bigl\| r^{-1} \chi_{\geq 1}(r) \bfL^{\frac12} P_c^{\bfL} \bmv(t) \bigr\|_{L^2_{r \ud r}} \lesssim t \, \bigl\| r^{-1} \chi_{\geq 1}(r) \bfL^{\frac12}_{\mathrm{nr}} \bigl( e^{i\theta} P_c^{\bfL} \bmv(t) \bigr) \bigr\|_{L^2_x} &\lesssim t \, \bigl\| \bfL_{\mathrm{nr}}^{\frac12} \bigl( e^{i\theta} P_c^{\bfL} \bmv(t) \bigr) \bigr\|_{L^{\infty-}_x} \\ 
            &\lesssim \sum_{\ast \in \{\real,\imag\}} t \, \bigl\| \jD e^{it\jD} \bmg_\ast(t) \bigr\|_{L^{\infty-}_x}.
        \end{aligned}
    \end{equation}
    Hence, it remains to bound the term $\bigl\| \chi_{\geq 1}(r) \Omega_1^+ \bfL^{\frac12} P_c^{\bfL} \bmv(t) \bigr\|_{L^2_{r\ud r}}$.
    By Lemma~\ref{lem:trans1} and~\ref{lem:calTcalRbounds1}, Plancherel, and H\"older's inequality, we have 
    \begin{equation}
        \begin{aligned}
            \bigl\| \chi_{\geq 1}(r) \Omega_1^+ \bfL^{\frac12} P_c^{\bfL} \bmv(t) \bigr\|_{L^2_{r\ud r}} 
            &\lesssim \bigl\|\bigl( \widetilde{\bfOmega}_{1,0}^{+,-} \jap{k} \widetilde{\bmv}\bigr)(t,k)\bigr\|_{L^2_{k\ud k}} + \bigl\|\bigl( \chi_{\leq 1}(k) \widetilde{\calM}_k \jap{k} \partial_t \widetilde{\bmv}\bigr)(t,k)\bigr\|_{L^2_{k\ud k}} \\ 
            &\quad + \bigl\| \bfL^{\frac12} \P_c^{\bfL} \pt \bmv(t)\bigr\|_{L^2_{r\ud r}} + t \, \bigl\| r^{-1} \chi_{\geq 1}(r) \bfL^{\frac12} P_c^{\bfL} \bmv(t) \bigr\|_{L^2_{r\ud r}} \\ 
            &\lesssim \bigl\|\bigl( \widetilde{\bfOmega}_{1,0}^{+,-} \jap{k} \widetilde{\bmv}\bigr)(t,k)\bigr\|_{L^2_{k\ud k}} + \bigl\|\bigl( \chi_{\leq 1}(k) \widetilde{\calM}_k \jap{k} \pt \widetilde{\bmv}\bigr)(t,k)\bigr\|_{L^2_{k\ud k}} \\ 
            &\quad + \sum_{\ast \in \{\real, \imag\}} \Bigl( \bigl\| \bmg_\ast(t) \bigr\|_{H^2_x} + t \, \bigl\| \jD e^{it\jD} \bmg_\ast(t) \bigr\|_{L^{\infty-}_x} \Bigr).
        \end{aligned}
    \end{equation}
    Commuting $\widetilde{\bfOmega}_{1,0}^{+,-} \jap{k}$ with the evolution equation for $\widetilde{\bmv}(t)$ on the distorted Fourier side, we get
    \begin{equation}
        \begin{aligned}
            \bigl\| \bigl( \widetilde{\bfOmega}_{1,0}^{+,-} \jap{k} \widetilde{\bmv}\bigr)(t,k)\bigr\|_{L^2_{k\ud k}} &\lesssim \bigl\| \bigl( \widetilde{\bfOmega}_{1,0}^{+,-} \jap{k} \widetilde{\bmv}\bigr)(0,k)\bigr\|_{L^2_{k\ud k}} + \bigl\| \jap{k}^{-1} \bigl( \pt \widetilde{\bfOmega}_{1,0}^{+,-} \jap{k} \widetilde{\bmv}\bigr)(0,k)\bigr\|_{L^2_{k\ud k}} \\ 
            &\quad + \biggl\| \int_0^t \frac{\sin\bigl((t-s)\jap{k}\bigr)}{\jap{k}} \bigl(\widetilde{\bfOmega}_{1,0}^{+,-} \jap{k} \widetilde{\bmF}\bigr)(s,k) \, \ud s \biggr\|_{L^2_{k \ud k}}. 
        \end{aligned}
    \end{equation}
    Then we find, using Lemma~\ref{lem:dkbounds1} and inserting the equation \eqref{equ:weighted_estimate_black_box_evol_equ},
    \begin{equation}
        \begin{aligned}
            &\bigl\| \bigl( \widetilde{\bfOmega}_{1,0}^{+,-} \jap{k} \widetilde{\bmv}\bigr)(0,k)\bigr\|_{L^2_{k\ud k}} + \bigl\| \jap{k}^{-1} \bigl( \pt \widetilde{\bfOmega}_{1,0}^{+,-} \jap{k} \widetilde{\bmv}\bigr)(0,k)\bigr\|_{L^2_{k\ud k}} \\
            &\lesssim \biggl\| \begin{pmatrix} \partial_k + \frac{1}{k} & 0 \\ 0 & \partial_k \end{pmatrix} \jap{k} (\pt \widetilde{\bmv})(0,k) \biggr\|_{L^2_{k\ud k}} + \biggl\| \begin{pmatrix} \partial_k + \frac{1}{k} & 0 \\ 0 & \partial_k \end{pmatrix} (\pt^2 \widetilde{\bmv})(0,k) \biggr\|_{L^2_{k\ud k}} \\
            &\quad + \bigl\| (\pt^2 \widetilde{\bmv})(0,k) \bigr\|_{L^2_{k\ud k}} + \bigl\| \jap{k} \widetilde{\bmv}(0,k) \bigr\|_{L^2_{k\ud k}} \\ 
            &\lesssim \bigl\| \jap{r} \bfL^{\frac12} P_c^{\bfL} \pt \bmv(0) \bigr\|_{L^2_{r \ud r}} + \bigl\| \jap{r} P_c^\bfL \pt^2 \bmv(0) \bigr\|_{L^2_{r\ud r}} + \bigl\| \bfL^{\frac12} P_c^\bfL \bmv(0) \bigr\|_{L^2_{r\ud r}} \\
            &\lesssim \bigl\| \jap{r} \bfL^{\frac12} P_c^{\bfL} \pt \bmv(0) \bigr\|_{L^2_{r \ud r}} + \bigl\| \jap{r} \bfL P_c^\bfL \bmv(0) \bigr\|_{L^2_{r\ud r}} + \bigl\| \jap{r} P_c^\bfL \bmF(0) \bigr\|_{L^2_{r\ud r}} + \bigl\| \bfL^{\frac12} P_c^\bfL \bmv(0) \bigr\|_{L^2_{r\ud r}}.
        \end{aligned}
    \end{equation}
    Moreover, breaking up $\widetilde{\bfOmega}_{1,0}^{+,-}$ again and using Lemma~\ref{lem:dkbounds1} as well as the dual ILED estimate \eqref{eq:ILEDDuhamel1} with $\gamma_2=1$ from Lemma~\ref{lem:LEDallfereqs1}, we find that
    \begin{equation}
        \begin{aligned}
            &\biggl\| \int_0^t \frac{\sin\bigl((t-s)\jap{k}\bigr)}{\jap{k}} \bigl(\widetilde{\bfOmega}_{1,0}^{+,-} \jap{k} \widetilde{\bmF}\bigr)(s,k) \, \ud s \biggr\|_{L^2_{k \ud k}} \\ 
            &\lesssim \bigl\| s \cdot \jap{r} \bfL^{\frac12} P_c^{\bfL} \bmF(s,r) \bigr\|_{L^2_s([0,t]; L^2_{r\ud r})} + \bigl\| \jap{r} \partial_s \bmF(s,r) \bigr\|_{L^1_s([0,t]; L^2_{r\ud r})}.
        \end{aligned}
    \end{equation}
    Similarly, using Lemma~\ref{lem:dkbounds1}, we obtain 
    \begin{equation}
        \begin{aligned}
            &\bigl\|\bigl( \chi_{\leq 1}(k) \widetilde{\calM}_k \jap{k} \pt \widetilde{\bmv}\bigr)(t,k)\bigr\|_{L^2_{k\ud k}} \\
            &\lesssim \bigl\|\bigl( \chi_{\leq 1}(k) \widetilde{\calM}_k \jap{k} \pt \widetilde{\bmv}\bigr)(0,k)\bigr\|_{L^2_{k\ud k}} + \bigl\| \jap{k}^{-1} \bigl( \chi_{\leq 1}(k) \widetilde{\calM}_k \jap{k} \pt^2 \widetilde{\bmv}\bigr)(0,k)\bigr\|_{L^2_{k\ud k}} \\ 
            &\quad + \biggl\| \int_0^t \frac{\sin\bigl((t-s)\jap{k}\bigr)}{\jap{k}} \chi_{\leq 1}(k) \widetilde{\calM}_k \jap{k} \partial_s \widetilde{\bmF}(s,k) \, \ud s \biggr\|_{L^2_{k\ud k}} \\ 
            &\lesssim \bigl\| \jap{r} \bfL^{\frac12} P_c^\bfL \pt \bmv(0) \bigr\|_{L^2_{r\ud r}} + \bigl\| \jap{r} P_c^{\bfL} \pt^2 \bmv(0) \bigr\|_{L^2_{r\ud r}} + \bigl\| \jap{r} \partial_s \bmF(s) \bigr\|_{L^1_s([0,t]; L^2_{r\ud r})} \\ 
            &\lesssim \bigl\| \jap{r} \bfL^{\frac12} P_c^\bfL \pt \bmv(0) \bigr\|_{L^2_{r\ud r}} + \bigl\| \jap{r} \bfL P_c^{\bfL} \bmv(0) \bigr\|_{L^2_{r\ud r}} + \bigl\| \jap{r} \bmF(0) \bigr\|_{L^2_{r\ud r}}  + \bigl\| \jap{r} \partial_s \bmF(s) \bigr\|_{L^1_s([0,t]; L^2_{r\ud r})}.
        \end{aligned}
    \end{equation}
    In conclusion, we have found that 
    \begin{equation} \label{equ:weighted_energy_black_box_termIII_3_final_bound}
        \begin{aligned}
            III_3(t) 
            &\lesssim \bigl\| \jap{r} \bfL P_c^{\bfL} \bmv(0) \bigr\|_{L^2_{r\ud r}} + \bigl\| \jap{r} \bfL^{\frac12} P_c^\bfL \pt \bmv(0) \bigr\|_{L^2_{r\ud r}} + \bigl\| \jap{r} \bmF(0) \bigr\|_{L^2_{r\ud r}} \\ 
            &\quad + \bigl\| \bmg_\real(t) \bigr\|_{H^2_x} + t \, \bigl\| \jD^2 e^{it\jD} \bmg_\real(t) \bigr\|_{L^{\infty-}_x} + \sum_{\ast \in \{\real, \imag\}} \Bigl( \bigl\| \bmg_\ast(t) \bigr\|_{H^2_x} + t \, \bigl\| \jD^2 e^{it\jD} \bmg_\ast(t) \bigr\|_{L^{\infty-}_x} \Bigr) \\
            &\quad + \bigl\| s \cdot \jap{r} \bfL^{\frac12} P_c^{\bfL} \bmF(s,r) \bigr\|_{L^2_s([0,t]; L^2_{r\ud r})} + \bigl\| \jap{r} \partial_s \bmF(s,r) \bigr\|_{L^1_s([0,t]; L^2_{r\ud r})}. 
        \end{aligned}
    \end{equation}
    
    Finally, it remains to bound the fourth term $III_4(t)$ on the right-hand side of \eqref{equ:weighted_energy_g_jxi_term3_decomp}.
    Invoking the commutator bound (see Lemma~\ref{lem:LnrOmegacomm1})
    \begin{equation}
        \bigr\| \bigl[ \bfL_{\mathrm{nr}}^{\frac12}\bfP_\nr^c, \Omega \bigr] u(t,x) \bigr\|_{L^2_x(\bbR^2)} \lesssim \|\pt u(t)\|_{L^2_x(\bbR^2)} + t \, \bigl\| \jx^{-2} u(t) \bigr\|_{L^2_x(\bbR^2)}
    \end{equation}
    we obtain that
    \begin{equation} \label{equ:weighted_energy_black_box_termIII_4_final_bound}
        \begin{aligned}
            III_4(t) = \Bigl\| \bigl[ \bfL_{\mathrm{nr}}^{\frac12}\bfP_\nr^c, \Omega \bigr] \, \Re\bigl( e^{i\theta} P_c^{\bfL} \bmv(t) \bigr) \Bigr\|_{L^2_x} &\lesssim \bigl\| \Re\bigl( e^{i\theta} P_c^{\bfL} \pt \bmv(t) \bigr) \bigr\|_{L^2_x} + t \, \bigl\| \jx^{-2} \Re\bigl( e^{i\theta} P_c^{\bfL} \bmv(t) \bigr) \bigr\|_{L^2_x} \\ 
            &\lesssim \| \bmg_\real(t) \|_{H^1_x} + t \, \bigl\| e^{it\jD} \bmg_\real(t) \bigr\|_{L^{\infty-}_x}.
        \end{aligned}
    \end{equation}

    Combining the bounds \eqref{equ:weighted_energy_black_box_termI_final_bound}, \eqref{equ:weighted_energy_black_box_termII_final_bound}, \eqref{equ:weighted_energy_black_box_termIII_1_final_bound}, \eqref{equ:weighted_energy_black_box_termIII_2_final_bound}, \eqref{equ:weighted_energy_black_box_termIII_3_final_bound}, \eqref{equ:weighted_energy_black_box_termIII_4_final_bound} and using (see Lemma~\ref{lem:Lnrxcomm1})
    \begin{equation}
        \begin{aligned}
            \bigl\| \jap{r} \bfL P_c^{\bfL} \bmv(0) \bigr\|_{L^2_{r\ud r}} &\lesssim \bigl\| \jap{x} e^{i\theta} P_c^{\bfL} \bmv(0) \bigr\|_{H^2_x}, \\
            \bigl\| \jap{r} \bfL^{\frac12} P_c^{\bfL} (\pt \bmv)(0) \bigr\|_{L^2_{r\ud r}} &\lesssim \bigl\| \jap{x} e^{i\theta} P_c^{\bfL} (\pt \bmv)(0) \bigr\|_{H^1_x},\\ 
            \bigl\| \bfL P_c^{\bfL} \bmv(t) \bigr\|_{L^2_{r\ud r}} &\lesssim \sum_{\ast \in \{\real,\imag\}} \|\bmg_\ast(t)\|_{H^2_x}, \\ 
            \bigl\| \jap{r} \bfL^{\frac12} P_c^{\bfL} \bmF(s,r) \bigr\|_{L^2_{r\ud r}} &\lesssim \bigl\| \jx \jD \bigl( e^{i\theta} \bmF(s,r) \bigr) \bigr\|_{L^2_x},
        \end{aligned}
    \end{equation}   
    we arrive at the asserted weighted energy estimate \eqref{equ:weighted_estimate_black_box_bound} in the statement of the proposition.
\end{proof}


\appendix
\section{Estimates for the Flat Fourier Transform}

\begin{lemma}\label{lem:A2}
We have 
\begin{equation}
    \label{eq:SM0}
    \Big\| e^{it\jap{D}} \sqrt{\frac{|\partial_1|}{\jap{D}}} \, g \Big\|_{L^\infty_{x_1}L^2_{x_2}L^2_t}\lesssim \|g\|_{L^2(\R^2)}, 
\end{equation}
as well as the symmetric version with $x_1$ and~$x_2$ interchanged. 
In particular, for all $\sigma>\frac 12$, 
\begin{equation}
    \label{eq:SM1}
    \Big\| \jap{x}^{-\sigma} e^{it\jap{D}} \sqrt{\frac{|D|}{\jap{D}}} \, g \Big\|_{L^2_{x}L^2_t}\lesssim \| g\|_{L^2(\R^2)}, 
\end{equation}
and dually
    \begin{equation}
        \label{eq:ILED1}
        \Big\| \int_0^T e^{-is\jap{D}} \sqrt{\frac{|D|}{\jap{D}}}\, F(s,\cdot)\, ds \Big\|_{L^2(\R^2)} \leq C_\sigma \| \jap{x}^\sigma F\|_{L^2_{t,x}([0,T]\times \R^2)}
    \end{equation}
    uniformly in $T\ge0$. 
\end{lemma}
\begin{proof}
    The map $\Phi(\xi)=(\xi_2,\jap{\xi})=:\eta$ has Jacobian $J=|\xi_1|\jap{\xi}^{-1}$. 
    We can therefore bound the left-hand side of~\eqref{eq:SM0} 
    \[
    \Big\| \int e^{i(x_1\xi_1+x_2\xi_2)} e^{it\jap{\xi}}  g(\xi_1,\xi_2)\,\sqrt{J}\,  d\xi_1 d\xi_2\Big\|_{L^\infty_{x_1} L^2_{x_2}L^2_t}\lesssim \| g\|_2 
    \]
     by switching to $\eta$-coordinates, applying Plancherel relative to $(x_2,t)$ and~$\eta$, and then switching back to $\xi$. The third bound~\eqref{eq:ILED1} follows by dualizing~\eqref{eq:SM0} and then bounding $L^1_{x_1}$ by a $\jap{x_1}^{-\sigma}L^2_{x_1}$. Switching $x_1$ and~$x_2$ then finishes the proof. The second inequality~\eqref{eq:SM1} follows from~\eqref{eq:ILED1} by duality. 
\end{proof}

\begin{lemma}
    \label{lem:ILED2}
    We have for any $\sigma>1$
    \begin{equation}\label{eq:LSM0}
\big\| \jap{x}^{-\sigma} e^{it\jap{D}} f\big\|_{L_{t,x}^2([0,T]\times \R^2)} \leq C_\sigma  \|f\|_2
    \end{equation}
    or, dually,
    \begin{equation}\label{eq:LSM0*}
\big\| \int_0^T e^{-it\jap{D}} F(t,\cdot)\, dt\big\|_{L^2(\R^2)} \leq C_\sigma  \|\jap{x}^{\sigma}F\|_{L^2_{t,x}([0,T]\times \R^2)}
    \end{equation}
     uniformly in $T\ge0$. 
\end{lemma}
\begin{proof}
Let $\chi$ be a radial bump which equals~$1$ on the unit ball, and let $\chi_0\ge0$ be a bump function on the line which equals~$1$ near~$0$.  Then for any $R\ge1$, 
\begin{align*}
    \Big\| \chi(\cdot/R) e^{it\jap{D}} f\Big\|_{L_{t,x}^2(\R\times\R^2)} &=  \Big\| \int_{\R^2} R^2 \widehat\chi(R(\xi-\eta)) e^{it\jap{\eta}} \hat{f}(\eta)\,d\eta \Big\|_{L^2_tL_\xi^2(\R\times\R^2)}\\
    &\lesssim \limsup_{\eps\to0} \Big\| \int_{\R^2} R^2 \widehat\chi(R(\xi-\eta)) \eps^{-1}\chi_0((\tau-\jap{\eta})/\eps) \hat{f}(\eta)\,d\eta \Big\|_{L^2_\tau L_\xi^2(\R\times \R^2)} \\
    &\lesssim R \|f\|_2.
\end{align*}
The final bound follows from Schur's test via the obvious 
\[
\sup_{\eta,\eps} \int R^2 |\widehat\chi(R(\xi-\eta)) |\eps^{-1}\chi_0((\tau-\jap{\eta})/\eps)\, d\tau \,d\xi \lesssim 1
\]
and the less obvious 
\[
\sup_{\xi,\tau,\eps} \int R^2 |\widehat\chi(R(\xi-\eta)) |\eps^{-1}\chi_0((\tau-\jap{\eta})/\eps)\, d\eta \lesssim R^2
\]
To see this, note that the worst case occurs for $\tau=1$ in which case the $\xi$ plays no role.  The lemma follows by summing over dyadic scales. 
\end{proof}

\begin{cor}
For $\sigma>\frac12$
    \begin{equation}\label{eq:LSM0**1}
\big\| \jap{x}^{-\sigma}\int_0^T e^{i(t-s)\jap{D}} {\partial_j}\jD^{-1} F(s,\cdot)\, ds\big\|_{L^2_{t,x}([0,T]\times \R^2)} \leq C_\sigma  \|\jap{x}^{\sigma}F\|_{L^2_{s,x}([0,T]\times \R^2)}
    \end{equation}
    and
      \begin{equation}\label{eq:LSM0**2}
\big\| \jap{x}^{-\sigma}\int_0^T e^{i(t-s)\jap{D}} {\partial_j}\jD^{-1} F(s,\cdot)\, ds\big\|_{L^2_{t,x}([0,T]\times \R^2)} \leq C_\sigma  \|F\|_{L^1_s([0,T], L^2_x(\R^2))}
    \end{equation}
     uniformly in $T\ge0$. Furthermore, we can replace $T$ by~$t$ as upper limit in the integrals on the left-hand side. 
\end{cor}
\begin{proof}
    The first bound follows by composing the two bounds of Lemma~\ref{lem:A2}: first~\eqref{eq:SM1} and then~\eqref{eq:ILED1}. The second one follows from~\eqref{eq:SM1} and Minkowski. The claim about $T\to t$ in~\eqref{eq:LSM0**2} is obvious. For~\eqref{eq:LSM0**1} by the same argument as in the proof of Lemma~\ref{lem:ILEDDuhamel}, the desired estimate follows from the uniform in $\tau\in\bbR$ and $\epsilon>0$ limiting absorption bound
    \begin{equation}
        \|\jap{x}^{-\sigma}D\jap{D}^{-1}\big(-\Delta+1-(\tau+i\epsilon)^2\big)^{-1}g\|_{L^2_x(\bbR^2)}\lesssim \|\jap{x}^\sigma g\|_{L^2_x(\bbR^2)}.
    \end{equation}
    A proof of this estimate can be found for instance in \cite[Chapter XIII.8, Lemma 4]{RSIV}.
\end{proof}

Finally, we establish pointwise decay. 

\begin{lemma}  \label{lem:decay}
    For every $0 < \delta \ll 1$ there exists $C_\delta \geq 1$ such that for any Schwartz function $g \in \calS(\bbR^2)$ and for any $t \in \bbR$,
    \begin{equation}
        \bigl\| e^{it\jD} g \bigr\|_{L^\infty_x(\bbR^2)} \leq \frac{C_
        \delta}{\jt^{1-\delta}} \Bigl( \bigl\| \jxi^2 \nabla_\xi \widehat{g}(\xi) \bigr\|_{L^2_\xi(\bbR^2)} + \bigl\| \jxi^2 \widehat{g}(\xi) \bigr\|_{L^2_\xi(\bbR^2)} \Bigr).
    \end{equation}    
\end{lemma}
\begin{proof}
    This is standard. We concentrate on $|t|\geq1$. By Corollary~2.38 in~\cite{NSbook}, we conclude that 
    \[
    \bigl\| e^{it\jD} g \bigr\|_{L^\infty_x(\bbR^2)} \leq  Ct^{-1}\|g\|_{B^2_{1,\infty}(\R^2)}
    \]
    where the norm on the right-hand side is from the standard Besov space. Then it follows by interpolation with unitarity on $L^2$ and Sobolev embedding that for some $q=q(\delta)>1$
    \[
    \bigl\| e^{it\jD} g \bigr\|_{L^\infty_x(\bbR^2)} \leq  C_\delta t^{-1+\delta}\|g\|_{W^{2,q}(\R^2)}\leq  C_\delta t^{-1+\delta}\|\jap{r}g\|_{W^{2,2}(\R^2)}
    \]
    as claimed. 
\end{proof}

The following three lemmas present results from pseudo-differential calculus. In fact, with $\bfL_\nr$ defined in~\eqref{eq:Lnr}, we claim that for $\bmu(t,r,\theta)=e^{i\theta}\bmv(t,r)$ (see Lemmas~\ref{lem:LnroverD1},~\ref{lem:Lnrxcomm1},~\ref{lem:LnrOmegacomm1} for the precise statements),
\begin{itemize}
\item $\|\jap{D}^{-1}\bfL_\nr^{\frac{1}{2}}\bmu\|_{L^\infty_x(\bbR^2)}\lesssim \|\bmu\|_{L^\infty_x(\bbR^2)}$.
    \item $\|\jap{x}^{\gamma}\bfL_\nr^{\frac{k}{2}}  \bmu\|_{L^2_{x}(\bbR^2)}\lesssim \|\jap{x}^{\gamma} \bmu\|_{H^k_x(\bbR^2)}$, $\gamma\in\bbR$.
    \item $\bigr\| \bigl[ \bfL_{\mathrm{nr}}^{\frac12}, \Omega \bigr] \bmu(t,x) \bigr\|_{L^2_x(\bbR^2)} \lesssim \|\pt \bmu(t)\|_{L^2_x(\bbR^2)} + t \, \bigl\| \jx^{-2} \bmu(t) \bigr\|_{L^2_x(\bbR^2)}$.
\end{itemize}
Here $\bfL_\nr^{\frac{1}{2}}$ is defined via the functional calculus as the positive root of a positive operator. Indeed the assumption $\bmu=e^{i\theta}\bmv$ is only to guarantee this positivity, which in this case follows from that of $\bfL$. Alternatively we could have restricted to the orthogonal subspace to any potential negative eigenvalue of $\bfL_{\nr}$ which would have to be purely radial. The $\Psi$DO interpretation of these powers goes back to Seeley's 1967 theorem, albeit on compact manifolds. Here we will rely on a more recent reference~\cite{MSS}, which applies to our setting (at least for the scalar case). For the sake of simplicity, we establish these three estimates for a scalar operator and simply note that the matrix case is exactly analogous thanks to~\cite{MSS}. Thus, let
\[
    H=-\Delta+1+V
    \quad \text{on } \R^2,
\]
with $H>0$. Assume that $V$ is smooth and radial and that, for large $r=|x|$,
\[
    V(r)=-r^{-2}+V_1(r),
\]
where $V_1$ decays exponentially together with all derivatives (this choice is precisely the potential in the upper left-hand corner of $\bfL_\nr$). Then $V$ is a classical SG-symbol of order $(0,-2)$, viz. 
\[
    V\in S_{\mathrm{cl}}^{0,-2}.
\]
Consequently,
\[
    h(x,\xi)=\la \xi\ra^2+V(x)
    \in S_{\mathrm{cl}}^{2,0}.
\]
Here $S^{\mu,m}$ denotes the SG-symbol class with $\xi$-order $\mu$ and $x$-order $m$, i.e., 
\[
    |\partial_\xi^\alpha\partial_x^\beta a(x,\xi)|
    \lesssim_{\alpha,\beta}
    \la \xi\ra^{\mu-|\alpha|}
    \la x\ra^{m-|\beta|}.
\]
By the complex-power theorem of 
\cite[Theorem 3.2 and Proposition 3.1]{MSS}, applied to the
parameter-elliptic classical SG-operator $H$, the spectral square root
$H^{1/2}$ is an SG-pseudodifferential operator
\[
    H^{1/2}\in \Op S_{\mathrm{cl}}^{1,0}.
\]
Moreover, comparing \(H\) with the free operator
\[
    H_0=-\Delta+1=\la D\ra^2,
\]
one obtains the sharper expansion
\[
    H^{1/2}-\la D\ra
    \in \Op S_{\mathrm{cl}}^{-1,-2}.
\]
Equivalently, there exists
\begin{equation}\label{eq:AOp}
    A\in \Op S_{\mathrm{cl}}^{-1,-2}
\end{equation}
such that
\[
    H^{1/2}=\la D\ra+A.
\]
The estimate
\[
    \la D\ra^{-1}H^{1/2}=I+\Psi_{1,0}^{-2}
\]
then follows immediately.  Indeed,
\[
    \la D\ra^{-1}H^{1/2}
    =
    I+\la D\ra^{-1}A,
\]
and since
\[
    \la D\ra^{-1}\in \Op S_{\mathrm{cl}}^{-1,0},
    \qquad
    A\in \Op S_{\mathrm{cl}}^{-1,-2},
\]
the SG composition theorem gives
\[
    \la D\ra^{-1}A\in \Op S_{\mathrm{cl}}^{-2,-2}.
\]
In particular,
\[
    S_{\mathrm{cl}}^{-2,-2}\subset S_{1,0}^{-2},
\]
so
\[
    \la D\ra^{-1}H^{1/2}=I+\Psi_{1,0}^{-2}.
\]

\begin{lem}\label{lem:LnroverD1}
    With $\bfL_\nr$ defined as in \eqref{eq:Lnr}, and with $\bmu(r,\theta)=e^{i\theta}\bmv(r)$,
    \begin{equation}
        \|\jap{D}^{-1}\bfL_\nr^{\frac{1}{2}}\bmu\|_{L^\infty_x(\bbR^2)}\lesssim \|\bmu\|_{L^\infty_x(\bbR^2)}.
    \end{equation}
\end{lem}
\begin{proof}
We prove this for $H$ rather than $\bfL_\nr$. 
    We write 
    \[
    \jap{D}^{-1}H^{\frac{1}{2}} = \Id+ \mathrm{Op}(b)
    \]
    where $b\in S^{-2}_{1,0}(\R^2)$. In fact, the principal symbol of $b$ is
    \[
    b(x,\xi)= \frac{V(x)}{2\jap{\xi}^2} + \text{lower order}
    \]
    This follows from~\cite[Theorem 3.2]{MSS}.  Next, we write
    \[
    (\mathrm{Op}(b) f)(x) = \int_{\R^2} K(x,y) f(y)\, dy
    \]
    with kernel
    \[
    K(x,y) = \int_{\R^2} e^{i(x-y)\xi} \; b(x,\xi)\, d\xi
    \]
    We decompose dyadically $K=\sum_{j=0}^\infty K_j$, where for $j\ge1$
    \[
    K_j(x,y) = \int_{\R^2} e^{i(x-y)\xi} \; b(x,\xi)\chi(2^{-j}\xi) \, d\xi
    \]
    with a cutoff $\chi(\xi)$ to the region $|\xi|\simeq 1$.  By integration by parts,
    \[
    |K_j(x,y)|\le C\jap{2^j|x-y|}^{-3}  
    \]
    By Schur's test and summing, this implies that $\| \mathrm{Op}(b) f\|_\infty \leq C\|f\|_\infty$. 
\end{proof}

\begin{lem}\label{lem:Lnrxcomm1}
    For any $\gamma\in\R$, any integer $k\geq0$, and $\bmu(r,\theta)=e^{i\theta}\bmv(r)$, 
    \[\|\jap{x}^{\gamma}\, \bfL_\nr^{\frac{k}{2}} \, \bmu\|_{L^2_{x}(\bbR^2)}\lesssim \|\jap{x}^{\gamma}\bmu\|_{H^k_x(\bbR^2)},\]
    and
    \[\bigl\| \jx^\gamma \bmu \bigr\|_{H^k_x(\bbR^2)} \lesssim \bigl\| \jx^\gamma \bmu \bigr\|_{L^2_x(\bbR^2)} + \bigl\| \jx^\gamma \bfL_{\nr}^{\frac{k}{2}} \bmu \bigr\|_{L^2_x(\bbR^2)}.\]
\end{lem}
\begin{proof}
    The first estimate reduces to the mapping property
    \[
    \jap{x}^{\gamma}\, \bfL_\nr^{\frac{k}{2}} \, \jap{x}^{-\gamma} \, \jap{D}^{-k}:\: L^2(\R^2)\to L^2(\R^2).
    \]
    By means of~\cite{MSS} and the symbol calculus in this paper we deduce that the operator on the left-hand side lies in $\mathrm{Op}(S^{0,0}_{\mathrm{cl}})$. Similarly, the second estimate reduces to the mapping property
    \[\jap{D}^k\jap{x}^{\gamma}(1+\bfL_\nr^{\frac{k}{2}})^{-1}\jap{x}^{-\gamma}:\: L^2(\R^2)\to L^2(\R^2),\]
    which holds by a similar argument.
\end{proof}

\begin{lem}\label{lem:LnrOmegacomm1}
    With $\Omega=t\nabla_x+x\partial_t$, and with $\bfP_\nr^+$ denoting the projection onto the orthogonal complement of the span of any potential eigenfunctions of $\bfL_\nr$ with negative eigenvalue, we have 
    \[
    \bigr\| \bigl[ \bfL_{\mathrm{nr}}^{\frac12}\bfP_\nr^+, \Omega \bigr] u(t,x) \bigr\|_{L^2_x(\bbR^2)} \lesssim \|\pt u(t)\|_{L^2_x(\bbR^2)} + t \, \bigl\| \jx^{-2} u(t) \bigr\|_{L^2_x(\bbR^2)}
    \]
\end{lem}
\begin{proof}
Once again, we use the scalar operator $H$ in place of $\bfL_{\mathrm{nr}}$. For each spatial component
\[
    \Omega_j=x_j\partial_t+t\partial_{x_j},
    \qquad j=1,2,
\]
we have, since $H^{1/2}$ is independent of $t$,
\[
    [H^{1/2},\Omega_j]u
    =
    [H^{1/2},x_j]\partial_tu
    +
    t[H^{1/2},\partial_{x_j}]u.
\]
Using $H^{1/2}=\la D\ra+A$, this becomes
\[
    [H^{1/2},\Omega_j]u
    =
    [\la D\ra,x_j]\partial_tu
    +
    [A,x_j]\partial_tu
    +
    t[A,\partial_{x_j}]u,
\]
because $\la D\ra$ commutes with $\partial_{x_j}$.
First,
\[
    [\la D\ra,x_j]\in \Op S^{0,0},
\]
since commuting with multiplication by $x_j$ differentiates the symbol in
$\xi_j$. Also,
\[
    [A,x_j]\in \Op S_{\mathrm{cl}}^{-2,-2},
\]
because $A\in \Op S_{\mathrm{cl}}^{-1,-2}$, see \eqref{eq:AOp},  and the commutator with $x_j$
again lowers the $\xi$-order by one. Hence
\[
    [H^{1/2},x_j]\in \Op S^{0,0}.
\]
By the SG $L^2$-mapping theorem, every operator in $\Op S^{0,0}$ is
bounded on $L^2(\R^2)$. Therefore
\[
    \|[H^{1/2},x_j]\partial_tu(t)\|_{L^2_x}
    \lesssim
    \|\partial_tu(t)\|_{L^2_x}.
\]
Next, since $A\in \Op S_{\mathrm{cl}}^{-1,-2}$,
\[
    [A,\partial_{x_j}]
    =
    -\Op(\partial_{x_j}a)
\]
at the symbolic level, where $a$ is the full symbol of $A$. Thus
\[
    [A,\partial_{x_j}]
    \in \Op S_{\mathrm{cl}}^{-1,-3}.
\]
Multiplication by $\la x\ra^2$ has SG-order $(0,2)$, so the SG
composition theorem gives
\[
    [A,\partial_{x_j}]\la x\ra^2
    \in
    \Op S_{\mathrm{cl}}^{-1,-1}
    \subset
    \Op S^{0,0}.
\]
Hence
\[
    \|[A,\partial_{x_j}]u(t)\|_{L^2_x}
    =
    \|[A,\partial_{x_j}]\la x\ra^2\,
      \la x\ra^{-2}u(t)\|_{L^2_x}
    \lesssim
    \|\la x\ra^{-2}u(t)\|_{L^2_x}.
\]
Combining the two estimates yields
\[
    \|[H^{1/2},\Omega_j]u(t)\|_{L^2_x(\R^2)}
    \lesssim
    \|\partial_tu(t)\|_{L^2_x(\R^2)}
    +
    |t|\,\|\la x\ra^{-2}u(t)\|_{L^2_x(\R^2)}.
\]
This concludes the proof. 
\end{proof}

We remark that the key input from
\cite[Theorem 3.2 and Proposition 3.1]{MSS} is the refined symbolic
statement
\[
    H^{1/2}-\la D\ra\in \Op S_{\mathrm{cl}}^{-1,-2},
\]
which is stronger than the ordinary Hörmander statement
\[
    \la D\ra^{-1}H^{1/2}=I+\Psi_{1,0}^{-2}.
\]
The latter is not sufficient for our purposes.


\begin{refcontext}[sorting=nyt]
  \printbibliography

@book{AbStebook,
author = {Abramowitz, M. and Stegun, I.A.},
title= {Handbook of mathematical functions with formulas, graphs, and mathematical tables},
year = {1964},
address = {New York},
publisher = {Dover},
}

@misc{Delort16,
  author = {Delort, J. M.},
  title = {Modified scattering for odd solutions of cubic nonlinear Schr\"{o}dinger equations with potential in dimension one},
  note = {Preprint hal-01396705},
  year = {2016}
}

@book{DMKink,
  author = {Delort, J. M. and Masmoudi, N.},
  title = {Long time dispersive estimates for perturbations of a kink solution of one dimensional cubic wave equations},
  publisher = {EMS Press},
  address = {Berlin},
  series = {Mem. Eur. Math. Soc.},
  volume = {1},
  year = {2022},
  pages = {ix+280}
}

@article{PusSof20,
  title={Bilinear Estimates in the Presence of a Large Potential and a Critical NLS in 3D},
  author={F. Pusateri and A. Soffer},
  journal={Memoirs of the American Mathematical Society},
  year={2020},
  %url={https://api.semanticscholar.org/CorpusID:211677571}
}

@misc{PalPus24,
       author = {{Palacios}, J. M. and {Pusateri}, F.},
        title = "{Local Energy control in the presence of a zero-energy resonance}",
note = {Preprint arXiv:2401.02623}
}

@misc{Li25,
      title={Asymptotic stability of solitary waves for the 1D focusing cubic Schr\"odinger equation}, 
      author={Y. Li},
      note = {Preprint arXiv:2510.17763}
}

@article{GerPus22,
author = {Germain, P. and Pusateri, F.},
year = {2022},
month = {07},
pages = {},
title = {Quadratic Klein-Gordon equations with a potential in one dimension},
volume = {10},
journal = {Forum of Mathematics, Pi},
%doi = {10.1017/fmp.2022.9}
}

@misc{ChenLuhr24,
      title={Asymptotic stability of the sine-Gordon kink}, 
      author={G. Chen and J. L{\"u}hrmann},
       note = {Preprint arXiv:2411.07004}
}

@article{ChenPus24,
author = {Chen, G. and Pusateri, F.},
year = {2024},
month = {02},
pages = {},
title = {On the 1d Cubic NLS with a Non-generic Potential},
volume = {405},
journal = {Communications in Mathematical Physics},
%doi = {10.1007/s00220-023-04894-4}
}

@article{ColGer25,
author = {Collot, C. and Germain, P.},
year = {2025},
month = {04},
pages = {},
title = {Asymptotic stability of solitary waves for one-dimensional nonlinear Schrödinger equations},
journal = {Journal of the European Mathematical Society},
%doi = {10.4171/jems/1642}
}

@misc{LiLuhr24,
      title={Asymptotic stability of solitary waves for the 1D focusing cubic Schr\"odinger equation under even perturbations}, 
      author={Y. Li and J. L{\"u}hrmann},
      note = {Preprint arXiv:2408.15427}
}

@article {BusPer92,
    AUTHOR = {Buslaev, V. S. and Perelman, G. S.},
     TITLE = {Scattering for the nonlinear {S}chr\"{o}dinger equation:
              states that are close to a soliton},
   JOURNAL = {Algebra i Analiz},
    VOLUME = {4},
      YEAR = {1992},
    NUMBER = {6},
     PAGES = {63--102},
  MRNUMBER = {1199635},
}

@article{GHW,
  author = {Germain, P. and Hani, Z. and Walsh, S.},
  title = {Nonlinear resonances with a potential: multilinear estimates and an application to NLS},
  journal = {Int. Math. Res. Not. IMRN},
  number = {18},
  year = {2015},
  pages = {8484--8544}
}

@article{GPR18,
  author = {Germain, P. and Pusateri, F. and Rousset, F.},
  title = {The nonlinear Schr\"{o}dinger equation with a potential in dimension 1},
  journal = {Ann. Inst. H. Poincar\'e C},
  volume = {35},
  number = {6},
  year = {2018},
  pages = {1477--1530}
}

@article{GZ,
  author = {Gesztesy, F. and Zinchenko, M.},
  title = {On spectral theory for Schr\"{o}dinger operators with strongly singular potentials},
  journal = {Math. Nachr.},
  volume = {279},
  number = {9--10},
  year = {2006},
  pages = {1041--1082}
}

@misc{CGP,
      title={Estimates for the Gross-Pitaevskii equation linearized around a vortex}, 
      author={C. Collot and P. Germain and E. Pacherie},
      note = {Preprint arXiv:2503.02953}
}

@incollection{GusVort0,
  author = {Gustafson, S.},
  title = {Stability of vortex solutions of the Ginzburg-Landau heat equation},
  booktitle = {Partial differential equations and their applications},
  pages = {159--165},
  series = {CRM Proc. Lecture Notes},
  volume = {12},
  publisher = {Amer. Math. Soc.},
  address = {Providence, RI},
  year = {1997}
}

@article{GusVort,
  author = {Gustafson, S.},
  title = {Dynamic stability of magnetic vortices},
  journal = {Nonlinearity},
  volume = {15},
  number = {5},
  year = {2002},
  pages = {1717--1728}
}

@article{GusSigVort,
  author = {Gustafson, S. and Sigal, I. M.},
  title = {The Stability of Magnetic Vortices},
  journal = {Commun. Math. Phys.},
  volume = {212},
  year = {2000},
  pages = {257--275}
}

@article{GusSig2,
  author = {Gustafson, S. and Sigal, I. M.},
  title = {Effective dynamics of magnetic vortices},
  journal = {Adv. Math.},
  volume = {199},
  number = {2},
  year = {2006},
  pages = {448--498}
}

@book{bookJT,
  author = {Jaffe, A. and Taubes, C.},
  title = {Vortices and monopoles. Structure of static gauge theories},
  series = {Progress in Physics},
  number = {2},
  publisher = {Birkh\"{a}user},
  address = {Boston, MA},
  year = {1980},
  pages = {v+287},
  isbn = {3-7643-3025-2}
}

@misc{KMS,
  author = {Krieger, J. and Miao, S. and Schlag, W.},
  title = {A stability theory beyond the co-rotational setting for critical Wave Maps blow up},
  note = {Preprint arXiv:2009.08843}
}

@article{KST,
  author = {Krieger, J. and Schlag, W. and Tataru, D.},
  title = {Renormalization and blow up for charge one equivariant critical wave maps},
  journal = {Invent. Math.},
  volume = {171},
  number = {3},
  year = {2008},
  pages = {543--615}
}

@article{KriSchNLS,
  author = {Krieger, J. and Schlag, W.},
  title = {Stable manifolds for all monic supercritical focusing nonlinear Schr\"{o}dinger equations in one dimension},
  journal = {J. Amer. Math. Soc.},
  volume = {19},
  number = {4},
  year = {2006},
  pages = {815--920}
}

@book{bookManSut,
  author = {Manton, N. and Sutcliffe, P.},
  title = {Topological solitons},
  series = {Cambridge Monographs on Mathematical Physics},
  publisher = {Cambridge University Press},
  address = {Cambridge},
  year = {2004},
  pages = {xii+493}
}

@book{NSbook,
  author = {Nakanishi, K. and Schlag, W.},
  title = {Invariant manifolds and dispersive Hamiltonian evolution equations},
  series = {Zurich Lectures in Advanced Mathematics},
  publisher = {European Mathematical Society (EMS)},
  address = {Z\"{u}rich},
  year = {2011},
  pages = {vi+253}
}

@article {MSS,
    AUTHOR = {Maniccia, L. and Schrohe, E. and Seiler, J.},
     TITLE = {Complex powers of classical {SG}-pseudodifferential operators},
   JOURNAL = {Ann. Univ. Ferrara Sez. VII Sci. Mat.},
    VOLUME = {52},
      YEAR = {2006},
    NUMBER = {2},
     PAGES = {353--369}
}

@misc{PPGL,
  author = {Palacios, J. M. and Pusateri, F.},
  title = {Linearized dynamic stability for vortices of Ginzburg-Landau evolutions},
  note = {Preprint arXiv:2409.04393}
}

@book{RS,
  author = {Reed, M. and Simon, B.},
  title = {Methods of modern mathematical physics. II. Fourier analysis, self-adjointness},
  publisher = {Academic Press},
  address = {New York-London},
  year = {1975},
  pages = {xv+361},
  msc = {47-02 (81.47)}
}

@article{KS06,
  title={Stable manifolds for all monic supercritical focusing nonlinear Schr{\"o}dinger equations in one dimension},
  author={Krieger, J. and Schlag, W.},
  journal={Journal of the American Mathematical Society},
  year={2006},
  volume={19},
  pages={815-920},
  %url={https://api.semanticscholar.org/CorpusID:120867346}
}

@book{RSIV,
  address = {New York},
  author = {Reed, M. and Simon, B.},
  publisher = {Academic Press},
  title = {Methods of Modern Mathematical Physics. IV Analysis
  of Operators},
  year = 1978
}

@article{Stu,
  author = {Stuart, D.},
  title = {Dynamics of abelian Higgs vortices in the near Bogomolny regime},
  journal = {Comm. Math. Phys.},
  volume = {159},
  number = {1},
  year = {1994},
  pages = {51--91}
}

@incollection{DR1,
	author = {Dafermos, M. and Rodnianski, I.},
	booktitle = {X{VI}th {I}nternational {C}ongress on {M}athematical {P}hysics},
	mrnumber = {2730803},
	pages = {421--432},
	title = {A new physical-space approach to decay for the wave equation with applications to black hole spacetimes},
	year = {2010},
}

@article{Schlue1,
	author = {Schlue, V.},
	journal = {Anal. PDE},
	mrnumber = {3080190},
	number = {3},
	pages = {515--600},
	title = {Decay of linear waves on higher-dimensional {S}chwarzschild black holes},
	volume = {6},
	year = {2013},
}

@article{DSS12,
    author      = "Donninger, R. and Schlag, W. and Soffer, A.",
    title       = "{On pointwise decay of linear waves on a Schwarzschild black hole background}",
    journal     = "Commun. Math. Phys.",
    volume      = "309",
    year        = "2012",
    pages       = "51-86",
}

@misc{OPT26,
      title={The good commutator approach to global asymptotics for the Schr\"odinger equation with variable coefficients}, 
      author={S.-J. Oh and F. Pasqualotto and N. Tang},
      note = {Preprint arXiv:2608.00268}
}

@article {GermPusZhang22,
    AUTHOR = {Germain, P. and Pusateri, F. and Zhang, Z.},
     TITLE = {On 1d quadratic {K}lein-{G}ordon equations with a potential
              and symmetries},
   JOURNAL = {Arch. Ration. Mech. Anal.},
    VOLUME = {247},
      YEAR = {2023},
    NUMBER = {2},
     PAGES = {Paper No. 17, 39},
}

@article{SSS10p2,
 author = {W. Schlag and A. Soffer and W. Staubach},
 journal = {Transactions of the American Mathematical Society},
 number = {1},
 pages = {289--318},
 title = {Decay for the wave and Schr\"odinger evolutions on maniffolds with conical ends, Part II},
 %urldate = {2026-06-03},
 volume = {362},
 year = {2010}
}

@article{Hintz22,
  title={A sharp version of Price's law for wave decay on asymptotically flat spacetimes},
  author={Hintz, P.},
  journal={Communications in Mathematical Physics},
  volume={389},
  number={1},
  pages={491--542},
  year={2022},
}

@article{SSS10p1,
title = "Decay for thewaveand schrodinger evolutions on manifolds with conical ends, Part I",
author = "W. Schlag and A. Soffer and W. Staubach",
year = "2010",
volume = "362",
pages = "19--52",
journal = "Transactions of the American Mathematical Society",
number = "1",
}

@article{MTT12,
title = {Price’s law on nonstationary space–times},
journal = {Advances in Mathematics},
volume = {230},
number = {3},
pages = {995-1028},
year = {2012},
author = {J. Metcalfe and D. Tataru and M. Tohaneanu},
}

@article{DSS11,
title = {A proof of Price's Law on Schwarzschild black hole manifolds for all angular momenta},
journal = {Advances in Mathematics},
volume = {226},
number = {1},
pages = {484-540},
year = {2011},
author = {R. Donninger and W. Schlag and A. Soffer},
}

@article{BVW18,
title = {Asymptotics of scalar waves on long-range asymptotically Minkowski spaces},
journal = {Advances in Mathematics},
volume = {328},
pages = {160-216},
year = {2018},
author = {D. Baskin and A. Vasy and J.Wunsch},
}

@article{BVW15,
 author = {D. Baskin and A. Vasy and J. Wunsch},
 journal = {American Journal of Mathematics},
 number = {5},
 pages = {1293--1364},
 title = {Asymptotics of radiation fields in asymptotically Minkowski space},
 %urldate = {2026-06-03},
 volume = {137},
 year = {2015}
}

@article{AAG18,
title = {Late-time asymptotics for the wave equation on spherically symmetric, stationary spacetimes},
journal = {Advances in Mathematics},
volume = {323},
pages = {529-621},
year = {2018},
author = {Y. Angelopoulos and S. Aretakis and D. Gajic},
}

@misc{LukOh24,
      title={Late time tail of waves on dynamic asymptotically flat spacetimes of odd space dimensions}, 
      author={J. Luk and S.-J. Oh},
      note = {Preprint arXiv:2404.02220}
}

@article{Mos1,
	author = {Moschidis, G.},
	journal = {Ann. PDE},
	mrnumber = {3493208},
	number = {1},
	pages = {Art. 6, 194},
	title = {The {$r^p$}-weighted energy method of {D}afermos and {R}odnianski in general asymptotically flat spacetimes and applications},
	volume = {2},
	year = {2016},
	}

@misc{KSW25,
      title={The cubic NLS on the line with an inverse square potential}, 
      author={J. Krieger and W. Schlag and K. Widmayer},
      note = {Preprint arXiv:2508.01919} 
}

@misc{Stewart24,
      title={Asymptotics for the cubic 1D NLS with a slowly decaying potential}, 
      author={G. Stewart},
      note = {Preprint arXiv:2408.03391} 
}

@article{DKSW16,
	author = {Donninger, R. and Krieger, J. and Szeftel, J. and Wong, W. W. Y.},
	%doi = {10.1215/00127094-3167383},
	fjournal = {Duke Mathematical Journal},
	%issn = {0012-7094},
	journal = {Duke Math. J.},
	%mrclass = {53C42 (35B30 35B35 35B40 35B44 35L70 53A15 53C44)},
	mrnumber = {3474816},
	%mrreviewer = {David James Hartley},
	number = {4},
	pages = {723--791},
	title = {Codimension one stability of the catenoid under the vanishing mean curvature flow in {M}inkowski space},
	%url = {https://doi.org/10.1215/00127094-3167383},
	volume = {165},
	year = {2016},
	%bdsk-url-1 = {https://doi.org/10.1215/00127094-3167383}
    }

@article{DonKri16,
	author = {Donninger, R. and Krieger, J.},
	%doi = {10.1090/memo/1142},
	fjournal = {Memoirs of the American Mathematical Society},
	%isbn = {978-1-4704-1873-1; 978-1-4704-2877-8},
	%issn = {0065-9266},
	journal = {Mem. Amer. Math. Soc.},
	%mrclass = {35L10 (35B40 35B45 42B35 53C44)},
	mrnumber = {3478759},
	%mrreviewer = {Marcio Antonio Jorge Silva},
	number = {1142},
	pages = {v+80},
	title = {A vector field method on the distorted {F}ourier side and decay for wave equations with potentials},
	%url = {https://doi.org/10.1090/memo/1142},
	volume = {241},
	year = {2016},
	%bdsk-url-1 = {https://doi.org/10.1090/memo/1142}
    }

@article{CGV13,
author = {Cuccagna, S. and Georgiev, V. and Visciglia, N.},
title = {Decay and Scattering of Small Solutions of Pure Power NLS in $\mathbb{R}$ with $p>3$ and with a Potential},
journal = {Communications on Pure and Applied Mathematics},
volume = {67},
number = {6},
pages = {957-981},
year = {2014}
}

@ARTICLE{RodTao15,
       author = {{Rodnianski}, I. and {Tao}, T.},
        title = "{Effective Limiting Absorption Principles, and Applications}",
      journal = {Communications in Mathematical Physics},
     %keywords = {Helmholtz Equation, Pseudodifferential Operator, Decay Estimate, Strichartz Estimate, Limit Absorption Principle},
         year = 2015,
        month = jan,
       volume = {333},
       number = {1},
        pages = {1-95},
          %doi = {10.1007/s00220-014-2177-8},
       %adsurl = {https://ui.adsabs.harvard.edu/abs/2015CMaPh.333....1R},
      %adsnote = {Provided by the SAO/NASA Astrophysics Data System}
}

@article{KriSch05,
  title={Non-generic blow-up solutions for the critical focusing NLS in 1-D},
  author={Krieger, J. and Schlag, W.},
  journal={Journal of the European Mathematical Society},
  year={2005},
  volume={11},
  pages={1-125},
  %url={https://api.semanticscholar.org/CorpusID:119641236}
}

@article{Naumkin16,
  title = {Sharp asymptotic behavior of solutions for cubic nonlinear Schr{\"o}dinger equations with a potential},
  author = {Naumkin, I. P.},
  journal = {Journal of Mathematical Physics},
  volume = {57},
  number = {5},
  pages = {051501},
  year = {2016},
  %doi = {10.1063/1.4947470},
  %url = {https://doi.org}
}

@article{Naumkin18,
  title={Nonlinear Schr{\"o}dinger equations with exceptional potentials},
  author={Naumkin, I. P.},
  journal={Journal of Differential Equations},
  volume={265},
  number={9},
  pages={4575--4631},
  year={2018},
  %publisher={Elsevier}
}

@misc{LPPSS1PartI,
shorthand = {L{\"u}h+26a},
  author = {L\"uhrmann, J. and Palacios, J. M.  and Pusateri, F. and Schlag, W. and Shahshahani, S.},
  title  = {Asymptotic stability of the degree-one vortex in the abelian Yang--Mills--Higgs model: Spectral Theory and Numerics},
  note   = {Preprint},
  year   = {2026}
}

@misc{LPPSS3,
shorthand = {L{\"u}h+26c},
  author       = {L{\"u}hrmann, J. and Palacios, J. M. and Pusateri, F. and Schlag, W. and Shahshahani, S.},
  title        = {Asymptotic stability of the degree-one vortex in the abelian Yang-Mills-Higgs model},
  year         = {2026},
  note         = {Preprint}
}

@misc{LSS1,
      title={On the Gross-Pitaevskii evolution linearized around the degree-one vortex}, 
      author={L\"uhrmann, J. and Schlag, W. and Shahshahani, S.},
      note = {Preprint arXiv:2503.07345}
}

@book{GrafakosCFA,
author = {Grafakos, L.},
year = {2014},
month = {01},
pages = {},
title = {Classical Fourier Analysis},
volume = {249}}
\end{refcontext}

\end{document}